\documentclass[11pt]{book}
\usepackage{etex}
\usepackage[utf8]{inputenc}
\usepackage[catalan,english]{babel}
\usepackage{amssymb,amsmath,amsthm}
\usepackage{graphicx,caption}
\usepackage[twoside,top=3.5cm,bottom=3.5cm,outer=2cm,inner=3cm,headsep=1.5cm]{geometry}
\usepackage{float}
\usepackage{enumerate}
\usepackage{tikz}
\usetikzlibrary{patterns}
\usepackage{pgfplots}
\usepackage[all]{xy}
\usepackage{comment}
\usepackage{cite}
\usepackage{mathabx}
\usepackage{fncychap}
\usepackage{setspace}
\usepackage{MnSymbol}
\pgfplotsset{compat=1.10}

\usepackage{fancyhdr}
\numberwithin{equation}{section}

\def\lgem{\discretionary{l-}{l}{\hbox{l$\cdot$l}}}
\newcommand{\extp}{\@ifnextchar^\@extp{\@extp^{\,}}}
\def\extp^#1{\mathop{\bigwedge\nolimits^{\!#1}}}
\makeatother

\theoremstyle{plain}
\newtheorem{teo}{Theorem}[section]
\newtheorem{prop}{Proposition}[section]
\newtheorem{cor}{Corollary}[section]
\newtheorem{lema}{Lemma}[section]
\newtheorem{conj}{Conjecture}[section]
\newtheorem*{lema*}{Lemma}

\newtheorem{theorem}{Theorem}

\theoremstyle{definition}
\newtheorem{defi}{Definition}[section]

\newtheorem{nota}{Notation}[section]
\newtheorem{exe}{Example}[section]

\newtheorem{rem}{Remark}[section]

\makeatletter
\ChNumVar{\Large}
\ChNameVar{\centering\Large\bfseries}
\ChTitleVar{\centering\huge\bfseries}
\ChNameUpperCase
\ChTitleAsIs
\ChRuleWidth{0.5pt}
\renewcommand{\DOCH}{%
    \raggedleft
    \CNV\FmN{\@chapapp}\space \CNoV\thechapter
    \par\nobreak
    \vskip -18\p@}
\renewcommand{\DOTI}[1]{%
    \CTV\raggedleft\mghrulefill{\RW}\par\nobreak
    \vskip 5\p@
    \CTV\FmTi{#1}\par\nobreak
    \vskip -5\p@
    \mghrulefill{\RW}\par\nobreak
    \vskip 20\p@}
\renewcommand{\DOTIS}[1]{%
    \CTV\raggedleft\mghrulefill{\RW}\par\nobreak
    \vskip 5\p@
    \CTV\FmTi{#1}\par\nobreak
    \vskip -5\p@
    \mghrulefill{\RW}\par\nobreak
    \vskip 20\p@} 
\makeatother

\begin{document}

\selectlanguage{catalan}

\begin{titlepage}
\begin{center}
\begin{figure}[htb]
\begin{center}
\includegraphics[width=10cm]{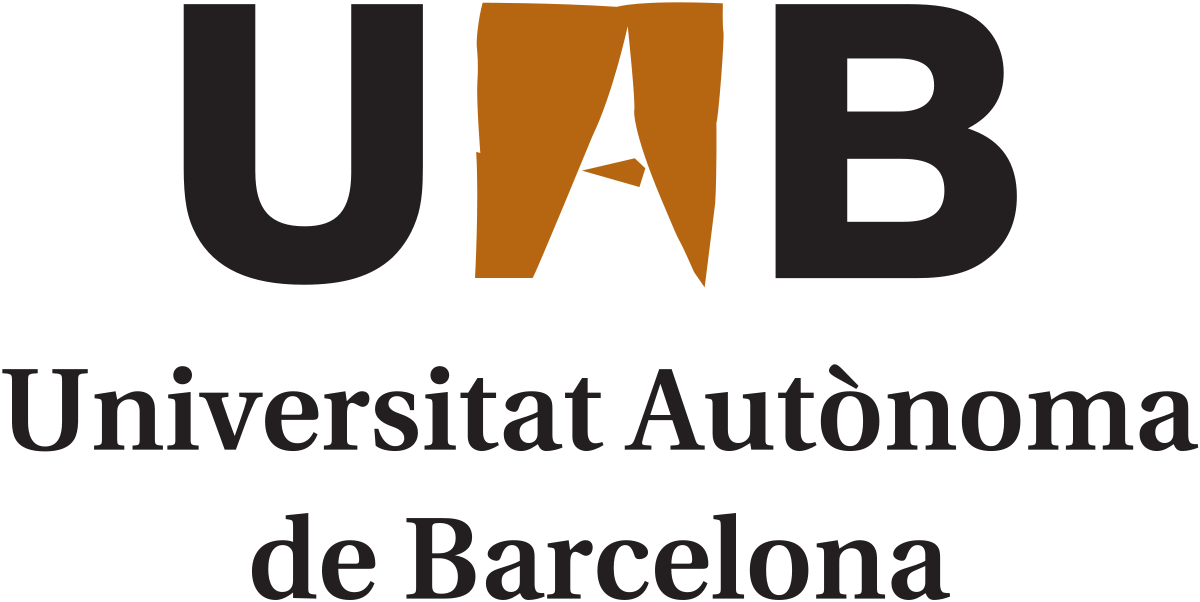}
\end{center}
\end{figure}
\textsc{\LARGE Departament de Matemàtiques} \\
{\Large Doctorat en Matemàtiques} \\
\today \\
\vspace*{0.5in}
\textsc{\Large Tesis Doctoral} \\
\begin{Huge}
\textbf{Trivial 2-cocycles for invariants of mod p homology spheres and Perron's conjecture} \\
\end{Huge}

\vspace*{5cm}

\begin{minipage}{.45\linewidth}
\begin{flushleft}
\emph{Candidat:} \\
Ricard Riba Garcia
\end{flushleft}
\end{minipage}
\hfill
\begin{minipage}{.45\linewidth}
\begin{flushright}
\emph{Director de la tesis:} \\
Dr. Wolfgang Pitsch
\end{flushright}
\end{minipage}

\end{center}
\end{titlepage}

\newpage
\mbox{}
\thispagestyle{empty}
\newpage

\pagenumbering{Roman}



\mbox{}
\vfill
\hfill
\begin{minipage}{.60\linewidth}

CERTIFICO que la present Memòria
ha estat realitzada per en Ricard Riba Garcia,
sota la meva supervisió i que constitueix la seva Tesi per a aspirar al grau de Doctor en Matemàtiques per la Universitat Autònoma de Barcelona.
\vspace*{0.3 cm}
Bellaterra, \today.

\vspace*{2.5 cm}
\begin{flushright}
{\Large \textit{Dr. Wolfgang Pitsch}}
\end{flushright}

\end{minipage}

\thispagestyle{empty}

\newpage
\mbox{}
\thispagestyle{empty}
\newpage


\chapter*{Agraïments}

Sempre he pensat que quan un escriu un text amb passió i i\lgem usió, d'alguna manera els diferents moments que ha viscut en el transcurs de la escriptura queden immortalitzats entre les linees d'aquell text.
Cada vegada que llegeixo aquesta memòria puc sentir tots els bons moments que vaig viure quan la vaig escriure.
Aquests agraïments només són un esbós de tota la gratitud que sento quan llegeixo aquesta memòria.

En primer lloc agraeixo a tots els membres del Grup de Topologia algebraica de la UAB el haver-me acollit com un membre més d'aquesta gran família. En especial, al meu director de tesis, el Dr. Wolfgang Pitsch,
per tota l'ajuda prestada durant el transcurs d'aquesta tesis, per encoratjar-me a seguir endavant en els moments més difícils de la tesis,
per ser crític i incentivar-me a millorar la meva recerca i
per aconsellar-me tant en l'àmbit de la tesis com a fora d'aquesta.
Ha estat un gran honor ser el teu alumne.

També vull agrair als meus companys de doctorat tots els bons moments que hem passat junts.

Carlos i Teresa, gràcies per haver-me escoltat totes les vegades que ho he necessitat, per totes les xerrades que hem tingut i per ajudar-me quan ho he necessitat.

Odi, Laura, Luca Matteo, Ales, Lui, Laurent, gràcies per tots els moments divertits que hem passat junts prenent cervesa després dels seminaris informals.

Leonardo, Murilo, Petr, Carles, Juan Luis, gràcies per tots els cafès que hem fet al despatx, per haver-me tret un somriure cada dia i per haver estat els meus fidels companys de viatge durant aquests quatre anys.

També vull agrair al Professor Louis Funar, de l'Institut Fourier de l'Universitat Grenoble I, al Professor David Marin, de la Universitat Autònoma de Barcelona, i al Professor Luis Paris, de la Université du Bourgogne, el haver acceptat ser els membres del tribunal de la meva defensa de tesis i els Professor Ignasi Mundet, de la Universitat de Barcelona, i al Professor Joackim Kock, de la Universitat Autònoma de Barcelona, el haver acceptat ser els suplents d'aquest tribunal.

Finalment, i de manera molt especial, també vull agrair a la meva família el haver estat sempre al meu costat. En particular, al meu pare Ricard per haver-me mostrat el món de les matemàtiques des de ben petit i incentivar-me a seguir descobrint aquest món sense fi, a la meva mare Isabel i la meva germana Meritxell per tot l'afecte que m'han mostrat cada dia i als meus avis Ramón i Cristina per ser la veu de l'experiència i haver-me ensenyat tot allò que no es pot aprendre en els llibres.

A tots vosaltres, gràcies per ser-hi.

\selectlanguage{english}
\newpage
\thispagestyle{empty}

\tableofcontents
\newpage
\mbox{}
\thispagestyle{empty}
\newpage
\mbox{}
\thispagestyle{empty}
\newpage
\pagenumbering{arabic}

\chapter*{Introduction}
\markboth{Introduction}{Introduction}
\addcontentsline{toc}{chapter}{Introduction}

At the begining of the 18th century, Leonhard Euler published a paper on the solution of the Königsberg bridge problem entitled: \textit{''Solutio problematis ad geometriam situs pertinentis'',} which translates into English as \textit{''The solution of a problem relating to the geometry of position''.} 
Perhaps, his work deserves to be considered as the beginnings of topology.
In that paper, he demonstrated that it was impossible to find a route through the town of Königsberg (now Kaliningrad) that would cross each of its seven bridges exactly once. This result did not depend on the lengths of the bridges, nor on their distance from one another, but only on connectivity properties: which bridges connect to which islands or riverbanks. This problem in introductory mathematics called Seven Bridges of Königsberg led to the branch of mathematics known as graph theory. 

 The paper not only shows that the problem of crossing the seven bridges in a single journey is impossible, but generalises the problem to show that, in today's notation,

\textit{A graph has a path traversing each edge exactly once if exactly two vertices have odd degree.}

The next step in freeing mathematics from being a subject about measurement was also due to Leonhard Euler. In 1750 he wrote a letter to Christian Goldbach which, as well as commenting on a dispute Goldbach was having with a bookseller, gives Euler's famous formula for a polyhedron:
$$v - e + f = 2,$$
where $v$ is the number of vertices of the polyhedron, $e$ is the number of edges and $f$ is the number of faces.
Two years later, Leonhard Euler published details of his formula in two papers, the first admits that Euler cannot prove the result but the second gives a proof based on dissecting solids into tetrahedral slices.
Following his ideas, at the begining of the 19th century, Antoine-Jean Lhuilier, who worked for most of his life on problems relating to Euler's formula, published an important work in which he noticed that Euler's formula didn't work for solids with holes in them. In particular, he proved that if a solid has $g$ holes then
$$v - e + f = 2 - 2g.$$
This was the first known result on a topological invariant.
In fact, this result was the first step to get the actual classification of compact connected surfaces.
Then a natural question that arise from the classification of surfaces is if there exists a similar result for $3$-manifolds.
In particular, we concern about how to decide whenever two $3$-manifolds are homeomorphic or not.

To deal with such problem, one tries to construct invariants of $3$-manifolds, that is, a function which sends each $3$-manifold to an element of an abelian group in such a way that if two $3$-manifolds are homeomorphic then this function takes the same value on these $3$-manifolds.

A first difficulty is to find a convenient way to describe $3$-manifolds. A natural way to do this is using the ''cutting and paste'' techniques.
These techniques consist on gluing several pieces to get a $3$-manifold.
When one uses these techniques one can either consider ''complicated'' pieces
and ''simple'' gluing maps, like in Thurston's decomposition, or ''simple'' pieces
and ''complicated'' gluing maps. In the last case one finds the theory of Heegaard splittings, in which the pieces are two handlebodies and the glueing map is an element of the mapping class group.
In this thesis we will use the Heegaard splittings theory.

Consider a standardly embedded surface $\Sigma_g$ in the $3$-sphere $\mathbf{S}^3.$ This surface separates $\mathbf{S}^3$ in to two handlebodies of the same genus $\mathcal{H}_g,$ $-\mathcal{H}_g.$ If we glue the boundaries of these handlebodies by an element $f$ of the mapping class group of $\Sigma_g$ we get another $3$-manifold $\mathcal{H}_g \cup_{f} -\mathcal{H}_g.$ In fact, every $3$-manifold can be obtained in this way. By technical reasons we consider $\Sigma_{g,1}$ the surface $\Sigma_g$ with a marked disk and $\mathcal{M}_{g,1}$ its mapping class group.
The embedding $\Sigma_g\hookrightarrow \mathbf{S}^3$ induces the following three natural subgroups of $\mathcal{M}_{g,1}:$
\begin{itemize}
\item $\mathcal{A}_{g,1}=$ subgroup of restrictions of diffeomorphisms of the outer handlebody $-\mathcal{H}_g,$
\item $\mathcal{B}_{g,1}=$ subgroup of restrictions of diffeomorphisms of the inner handlebody $\mathcal{H}_g,$
\item $\mathcal{AB}_{g,1}=$ subgroup of restriction of diffeomorphisms of the $3$-sphere $\mathbf{S}^3.$
\end{itemize}
Denote by $\mathcal{V}^3$ the set of oriented diffeomorphism classes of compact, closed and oriented smooth 3-manifolds.
The main advantage of using the Heegaard splittings theory instead of other theories, is the existence of the following bijection
\begin{align*}
\lim_{g\to \infty}\mathcal{A}_{g,1}\backslash\mathcal{M}_{g,1}/\mathcal{B}_{g,1} & \longrightarrow \mathcal{V}^3, \\
\phi & \longmapsto S^3_\phi=\mathcal{H}_g \cup_{\iota_g\phi} -\mathcal{H}_g,
\end{align*}
proved by J. Singer in 1933.
This bijection allows us to translate topological problems into algebraic problems on $\mathcal{M}_{g,1}.$
In particular, let $\mathcal{S}^3\subset\mathcal{V}^3$ be the subset of all integral homology $3$-spheres, if we restrict the previous bijection to the Torelli group $\mathcal{T}_{g,1},$ which is the subgroup formed by the elements that act trivially in the first homology group of the surface $\Sigma_{g,1},$ one has the following bijection:
$
\lim_{g\to \infty}\mathcal{A}_{g,1}\backslash\mathcal{T}_{g,1}/\mathcal{B}_{g,1} \simeq \mathcal{S}^3.
$
Moreover, in \cite{pitsch}, W. Pitsch proved that the induced equivalence relation on $\mathcal{T}_{g,1},$ which is given by:
$
\phi \sim \psi \quad\Leftrightarrow \quad \exists \zeta_a \in \mathcal{A}_{g,1}\;\exists \zeta_b \in \mathcal{B}_{g,1} \quad \text{such that} \quad \zeta_a \phi \zeta_b=\psi,$
can be rewritten as follows:
\begin{lema*}
Two maps $\phi, \psi \in \mathcal{M}_{g,1}[d]$ are equivalent if and only if there exists a map $\mu \in \mathcal{AB}_{g,1}$ and two maps $\xi_a\in \mathcal{A}_{g,1}[d]$ and $\xi_b\in\mathcal{B}_{g,1}[d]$ such that $\phi=\mu \xi_a\psi \xi_b\mu^{-1}.$
\end{lema*}
As a consequence, he obtained the following bijection:
$$
\lim_{g\to \infty}(\mathcal{TA}_{g,1}\backslash\mathcal{T}_{g,1}/\mathcal{TB}_{g,1})_{\mathcal{AB}_{g,1}} \simeq \mathcal{S}^3,
$$
where $\mathcal{TA}_{g,1}=\mathcal{T}_{g,1} \cap \mathcal{A}_{g,1},$
$\mathcal{TB}_{g,1}=\mathcal{T}_{g,1} \cap \mathcal{B}_{g,1}$ and $(\;)_{\mathcal{AB}_{g,1}}$ are the $\mathcal{AB}_{g,1}$-coinvariants.

This bijection allows us to define an invariant of integral homology $3$-spheres $F:\mathcal{S}^3\rightarrow A,$ as family of functions $\{F_g\}_g$ on the Torelli group with the following properties:

\begin{enumerate}[i)]
\item $F_{g+1}(x)=F_g(x) \quad \text{for every }x\in \mathcal{T}_{g,1},$
\item $F_g(\xi_a x\xi_b)=F_g(x) \quad \text{for every } x\in \mathcal{T}_{g,1},\;\xi_a\in \mathcal{TA}_{g,1},\;\xi_b\in \mathcal{TB}_{g,1},$
\item $F_g(\phi x \phi^{-1})=F_g(x)  \quad \text{for every }  x\in \mathcal{T}_{g,1}, \; \phi\in \mathcal{AB}_{g,1}.$
\end{enumerate}

In addition, if we consider the associated trivial $2$-cocycles $\{C_g\}_g,$ which measure the failure of the maps $\{F_g\}_g$ to be homomorphisms of groups, i.e.
\begin{align*}
C_g: \mathcal{T}_{g,1}\times \mathcal{T}_{g,1} & \longrightarrow A \\
 (\phi,\psi) & \longmapsto F_g(\phi)+F_g(\psi)-F_g(\phi\psi),
\end{align*}
then the family of $2$-cocycles $\{C_g\}_g$ inherits the following properties:
\begin{enumerate}[(1)]
\item The $2$-cocycles $\{C_g\}_g$ are compatible with the stabilization map, i.e. the following diagram of maps commutes:
$$
\xymatrix@C=15mm@R=13mm{ \mathcal{T}_{g,1}\times \mathcal{T}_{g,1} \ar@{->}[rd]_{C_g}\ar@{^{(}->}[r] & \mathcal{T}_{g+1,1}\times \mathcal{T}_{g+1,1} \ar@{->}[d]^{C_{g+1}} \\
 & A}
$$
\item The $2$-cocycles $\{C_g\}_g$ are invariant under conjugation by elements in $\mathcal{AB}_{g,1},$ i.e. for every $\phi \in \mathcal{AB}_{g,1},$ $$C_g(\phi - \phi^{-1},\phi - \phi^{-1})=C_g(-,-),$$
\item If $\phi\in \mathcal{TA}_{g,1}$ or $\psi \in \mathcal{TB}_{g,1}$ then $C_g(\phi, \psi)=0.$
\end{enumerate}

In 2008, W. Pitsch, using this setting, gave a new tool to get invariants with values in an abelian group without $2$-torsion of integral homology $3$-spheres as trivializations of $2$-cocyles on $\mathcal{T}_{g,1}.$ Moreover, using this tool, he gave a new purely algebraic construction of the Casson's invariant.

In this thesis we generalized this tool to invariants with values in an abelian group without restrictions getting the following result:

\begin{theorem}
\label{teo_ch2}
Let $A$ be an abelian group and $A_2$ the subgroup of elements of order $\leq 2.$
For each $x\in A_2,$ a family of cocycles $\{C_g\}_{g\geq 3}$ on the Torelli groups $\mathcal{T}_{g,1}, g\geq 3,$ satisfying conditions (1)-(3) provides a compatible familiy of trivializations $F_g+\mu_g^x: \mathcal{T}_{g,1}\rightarrow A$ that reassemble into an invariant of homology spheres
$$\lim_{g\to \infty}F_g+\mu_g^x:\mathcal{S}^3\rightarrow A$$
if and only if the following two conditions hold:
\begin{enumerate}[(i)]
\item The associated cohomology classes $[C_g]\in H^2(\mathcal{T}_{g,1};A)$ are trivial.
\item The associated torsors $\rho(C_g)\in H^1(\mathcal{AB}_{g,1},Hom(\mathcal{T}_{g,1},A))$ are trivial.
\end{enumerate}
\end{theorem}
The main difference between this generalization and the tool given in \cite{pitsch} is the presence of the Rohlin invariant which accounts for the failure unicity in the construction of the invariants.

In the same year, B. Perron conjectured an invariant of $\mathbb{Q}$-homology spheres constructed from a Heegaard splitting with gluing map an element of the (mod $p$) Torelli group $\mathcal{M}_{g,1}[p]$ with values on $\mathbb{Z}/p.$
Here we call the (mod $p$) Torelli group the kernel of the canonical map $\mathcal{M}_{g,1}\rightarrow Sp_{2g}(\mathbb{Z}/p),$
which is product of the Torelli subgroup and the subgroup generated by $p$-powers of Dehn twists.

Such conjectured invariant consist on writing an element of the (mod $p$) Torelli group as a product of an element of the Torelli group and an  element of the subgroup generated by $p$-powers of Dehn twists and taking the Casson's invariant modulo $p$ of the part of the Torelli group.

The main purpose of this thesis is to study the Perron's conjecture.

In order to achieve our target, we first study the subclass of $\mathbb{Q}$-homology $3$-spheres, which are the $3$-manifolds whose first homology group is finite. It is well known that there is not a subgroup of $\mathcal{M}_{g,1}$ that parametrizes all the $\mathbb{Q}$-homology $3$-spheres. Nevertheless, the union of all (mod $p$) Torelli groups with $p$ prime parametrizes such manifolds. Moreover we get the following criterion to know whenever a $\mathbb{Q}$-homology $3$-sphere can be constructed as a Heegaard splitting with gluing map an element of (mod $d$) Torelli group.
\begin{theorem}
\label{teo_ch3}
Let $M$ be a $\mathbb{Q}$-homology $3$-sphere and $n=|H_1(M;\mathbb{Z})|.$
Then $M$ has a Heegaard splitting $\mathcal{H}_g\cup_{\iota_gf} -\mathcal{H}_g$ of some genus $g$ with gluing map $f\in \mathcal{M}_{g,1}[d]$ with $d\geq 2,$ if and only if $d$ divides $n-1$ or $n+1.$
\end{theorem}
Denote by $\mathcal{S}^3[d]$ the set of $\mathbb{Q}$-homology $3$-spheres which are homeomorphic to $\mathcal{H}_g \cup_{\iota_g\phi} -\mathcal{H}_g$ for some $\phi\in \mathcal{M}_{g,1}[d].$
The above criterion tell us that, unlikely the case of integral homology $3$-spheres and the Torelli group, in general, $\mathcal{S}^3[d]$ does not coincide with the set of $\mathbb{Z}/d$-homology $3$-spheres.

Nevertheless, we still have the following bijection:
$$
\lim_{g\to \infty}(\mathcal{A}_{g,1}[d]\backslash\mathcal{M}_{g,1}[d]/\mathcal{B}_{g,1}[d])_{\mathcal{AB}_{g,1}} \simeq \mathcal{S}^3[d],
$$
where $\mathcal{A}_{g,1}[d]=\mathcal{M}_{g,1}[d]\cap \mathcal{A}_{g,1}$ and
$\mathcal{B}_{g,1}[d]=\mathcal{M}_{g,1}[d]\cap \mathcal{B}_{g,1}.$

Then using an analogous setting that we used to construct invariants of integral homology $3$-spheres as a trivialization of $2$-cocycles, for the $(\text{mod }d)$ Torelli group we get the following result:

\begin{theorem}
\label{teo_ch5}
Let $A$ an abelian group. For $g\geq 3,$ $d\geq 3$ an odd integer and for $g\geq 5,$ $d\geq 2$ an even integer such that $4 \nmid d.$
For each $x\in A_d,$ a family of 2-cocycles $\{C_g\}_{g\geq 3}$ on the $(\text{mod }d)$ Torelli groups $\mathcal{M}_{g,1}[d],$ with values in $A,$ satisfying conditions analogous to (1)-(3) provides a compatible family of trivializations $F_g+\varphi_g^x: \mathcal{M}_{g,1}[d]\rightarrow A$ that reassembles into an invariant of $\mathbb{Q}$-homology spheres $\mathcal{S}^3[d]$
$$\lim_{g\to \infty}F_g+\varphi_g^x: \mathcal{S}^3[d]\longrightarrow A$$
if and only if the following two conditions hold:
\begin{enumerate}[(i)]
\item The associated cohomology classes $[C_g]\in H^2(\mathcal{M}_{g,1}[d];A)$ are trivial.
\item The associated torsors $\rho(C_g)\in H^1(\mathcal{AB}_{g,1},Hom(\mathcal{M}_{g,1}[d],A))$ are trivial.
\end{enumerate}
\end{theorem}

Unfortunately, in our attempt to use these tools to prove Perron's conjecture, we realised that the conjecture is actually false. What is more, we get the following result:

\begin{theorem}
\label{teo_ch6}
Perron's conjecture is false with an obstruction given by the non-vanishing first characteristic class of surface bundles reduced modulo $p.$
\end{theorem}

\newpage

\section*{Structure of the thesis}

As a starting point, in \textbf{Chapter \ref{chapter: Preliminaries}},
we give the background that we will use throughout this thesis.

In \textbf{Chapter \ref{chapter: trivocicles torelli}},
we extend the tools given in \cite{pitsch} to give a construction of invariants with values on an abelian group without restrictions, from a suitable family of $2$-cocycles on $\mathcal{T}_{g,1},$ proving Theorem \ref{teo_ch2}.
Moreover we give some interesting results about the Luft subgroup $\mathcal{L}_{g,1}$ and the handlebody subgroup $\mathcal{B}_{g,1}.$ 

In \textbf{Chapter \ref{chapter: RHS}}, we show that every Heegaard splitting with gluing map an element of the (mod $p$) Torelli group is a $\mathbb{Q}$-homology $3$-sphere and that every $\mathbb{Q}$-homology $3$-sphere has a Heegaard splitting with gluing map an element of the (mod $p$) Torelli group. In particular, we give a criterion to determine whenever a $\mathbb{Q}$-homology $3$-sphere has a Heegaard splitting with gluing map an element of the (mod $d$) Torelli group, proving Theorem \ref{teo_ch3}.

In \textbf{Chapter \ref{chapter: versal ext}}, we compare the automorphisms of $p$-nilpotent quotients of the free group given by the Stallings or Zassenhaus filtration of the free group. To be more precise,
given $\Gamma$ a free group of finite rank and $\{ \Gamma_k^\bullet\}_k,$ the Stallings or Zassenhaus $p$-central series and
$\mathcal{N}_{k+1}^\bullet=\Gamma/\Gamma_{k+1}^\bullet.$
We show that there is a well-defined homomorphism
$\psi_k^\bullet:Aut(\mathcal{N}_{k+1}^\bullet)\rightarrow Aut(\mathcal{N}_{k}^\bullet)$
that fits into a non central extension
\begin{equation*}
\xymatrix@C=7mm@R=10mm{0 \ar@{->}[r] & Hom(\mathcal{N}^\bullet_1, \mathcal{L}^\bullet_{k+1}) \ar@{->}[r]^-{i} & Aut\;\mathcal{N}^\bullet_{k+1} \ar@{->}[r]^-{\psi_k^\bullet} & Aut \;\mathcal{N}^\bullet_k \ar@{->}[r] & 1,}
\end{equation*}
and we study its splitability.

In \textbf{Chapter \ref{chapter: trivocicles torelli mod p}}, we extend the results obtained in Chapter \ref{chapter: trivocicles torelli} to the case of the (mod $p$) Torelli group,
getting a new tool to construct invariants of $\mathbb{Q}$-homology spheres as a trivializations of certain trivial $2$-cocycles on the (mod $p$) Torelli group, proving Theorem \ref{teo_ch5}.
Moreover, we get a new invariant of $\mathbb{Q}$-homology $3$-spheres as a family of homomorphisms on the (mod $p$) Torelli group.

Finally, in \textbf{Chapter \ref{chapter: obtruction}}, we give our obstruction to Perron's conjecture, proving Theorem \ref{teo_ch6}.

\chapter{Preliminaries}
\label{chapter: Preliminaries}

In this Chapter we introduce basic background material that we will use throughout all this thesis.

\section{Cohomology of groups and group extensions}
In this section we give some definitions and elementary results about cohomology of groups and group extensions. For further information on these topics see \cite{brown} and \cite{stam}.

\subsection{The group algebra}
\begin{defi}
The \textit{group algebra} $\mathbb{Z}G$ is the free $\mathbb{Z}$-module with basis $G$ and with multiplication induced by the multiplication in the group $G.$ Thus, elements are formal linear combinations
$$\sum_{g\in G}a_gg,$$
where $a_g\in \mathbb{Z},$ and where $a_g=0$ for all but finitely many $g\in G,$ and multiplication is given by
$$\sum_{g\in G}a_gg\sum_{h\in G}b_hh=\sum_{g,h\in G}(a_gb_h)g\cdot h.$$
With these definitions, the group algebra $\mathbb{Z}G$ becomes an associative ring with unit.
\end{defi}

\paragraph{Invariants and coinvariants}

Let $M$ be a left $G$-module. The \textit{invariants} of $M$ are the elements of the $\mathbb{Z}$-submodule
$$M^G=\{m\in M\;\mid \;gm=m \text{ for all } g\in G\}.$$
In fact, $M^G$ is the largest submodule of $M$ on which $G$ acts trivially.

The \textit{coinvariants} of $M$ are the elements of the quotient $\mathbb{Z}$-module
$$M_G=M/(gm-m\;\mid\; g\in G,m\in M). $$
Indeed, $M_G$ is the largest quotient of $M$ on which $G$ acts trivially.

\paragraph{Left and right $G$-modules}
Let $M$ be a non-trivial (left) $G$-module given by the action
\begin{align*}
G\times M & \rightarrow M \\
(g,m)&\mapsto gm,
\end{align*}
since $G$ is a group and in particular every element $g\in G$ has an inverse $g^{-1}\in G,$ then we can give to $M$ the structure of a non-trivial (right) $G$-module taking the action
\begin{align*}
M\times G & \rightarrow M \\
(m,g)&\mapsto g^{-1}m.
\end{align*}
Analogously, if $M$ is a non-trivial (right) $G$-module then $M$ admits the structure of a non-trivial (left) $G$-module.
Therefore we will not distinguish between left and right $G$-modules.

\subsection{Homology and Cohomology of groups}

We define the $n$-th cohomology group of $G$ with coefficients in the $G$-module $A$ by
$$H^n(G;A)=Ext^n_{\mathbb{Z}G}(\mathbb{Z};A),$$
where $\mathbb{Z}$ is to be regarded as a trivial $G$-module. In particular $H^0(G;A)=A_G.$

The $n$-th homology group of $G$ with coefficients in the $G$-module $B$ is defined by
$$H_n(G,B)=Tor^{\mathbb{Z}G}_n(B,\mathbb{Z}),$$
where again $\mathbb{Z}$ is to be regarded as trivial $G$-module.
In particular, $H_0(G;B)=B^G.$

The way to compute such groups is taking a $G$-projective resolution $\mathbf{P}$ of the trivial $G$-module $\mathbb{Z},$ form the complexes $Hom_{\mathbb{Z}G}(\mathbf{P},A)$ and $B\otimes_{\mathbb{Z}G} \mathbf{P},$ and compute their homology.

In the next section we will give a standard procedure of constructing such a resolution $\mathbf{P}$ from the group $G.$

\subsection{The bar resolution}

We now describe a particular resolution $B_\bullet\rightarrow \mathbb{Z}$ of the trivial $\mathbb{Z}G$-module $\mathbb{Z},$ called the \textit{bar resolution} or the \textit{standard resolution.} 

Let $B_n$ be the free $\mathbb{Z}$-module with basis all $(n+1)$-tuples $(g_0,\ldots, g_n)\in G^{n+1}.$ Then $B_n$ becomes a $\mathbb{Z}G$-module via the diagonal action
$$g\cdot (g_0,\ldots, g_n)=(gg_0,\ldots,gg_n)$$
for $g\in G.$ As a $\mathbb{Z}G$-module, $B_n$ is free with basis all elements of the form
$$[g_1|\ldots |g_n]=(1,g_1,g_1g_2,\ldots, g_1g_2\cdots g_n)$$
where $g_1,\ldots, g_n\in G.$ In particular, $B_0$ is free with basis $[\;]=(1),$ so we may identify $B_0$ with the $\mathbb{Z}G$-module $\mathbb{Z}G.$

For $0\leq i \leq n,$ define the $i^{th}$ face map $d_i:B_n\rightarrow B_{n-1}$ to be the homomorphism of $\mathbb{Z}$-modules determined by
$$d_i(g_0,\ldots, g_n)=(g_0, \ldots , \widehat{g_i}, \ldots , g_n)$$
where the symbol $\widehat{g_i}$ indicates that $g_i$ is to be omitted.
Clearly, $d_i$ is also a homomorphism of $\mathbb{Z}G$-modules. It is easy to verify that if $i<j$ then
$$d_i\circ d_j=d_{j-1}\circ d_i.$$
If we define $\partial_n: B_n\rightarrow B_{n-1}$ by
$$\partial_n=\sum^n_{i=0}(-1)^id_i,$$
then the above relation implies that $\partial_{n-1}\partial_n=0.$
Moreover $B_0/Im\; \partial_1=\mathbb{Z},$ by a direct computation.
Thus,
$$B_\bullet:\;\xymatrix@C=7mm@R=7mm{
 \cdots \ar@{->}[r] & B_n \ar@{->}[r]^-{\partial_n}
 & B_{n-1} \ar@{->}[r] & \cdots \ar@{->}[r] & B_1 \ar@{->}[r]^-{\partial_1} & B_0 \ar@{->}[r] & 0  }$$
is a free resolution of the (left) $G$-module $\mathbb{Z}.$

In terms of the basis $[g_1|\ldots |g_n]$ for $B_n,$ the face maps $d_i$ take the form:
\begin{align*}
d_0[g_1|\ldots |g_n] &=g_1[g_2|\ldots |g_n],\\
d_i[g_1|\ldots |g_n] &=[g_1|\ldots |g_{i-1}|g_ig_{i+1}|\ldots | g_n] \quad (0<i<n),\\
d_n[g_1|\ldots |g_n] &=[g_1|\ldots |g_{n-1}].
\end{align*}
Thus, for instance
\begin{align*}
\partial_1[g_1] & =g_1[\;]-[\;],\\
\partial_2[g_1|g_2] & =g_1[g_2]-[g_1g_2]+[g_2].
\end{align*}

For any group $G$ we can always take our resolution $F$ to be the bar resolution. In this case we write $C_*(G,M)$ for $F\otimes_G M$ and $C^*(G,M)$ for $Hom_G(F,M).$ Thus an element of $C_n(G,M)$ can be uniquely expressed as a finite sum of elements of the form $m\otimes [g_1|\cdots | g_n],$ i.e., as a formal linear combination with coefficients in $M$ of the symbols $[g_1|\cdots |g_n].$
The boundary operator $\partial: C_n(G,M)\rightarrow C_{n-1}(G,M)$ is given by
\begin{align*}
\partial(m\otimes [g_1|\cdots |g_n]) & =mg_1\otimes [g_2|\cdots | g_n] \\
& -m\otimes [g_1g_2|\cdots |g_n]+\cdots + (-1)^n m\otimes [g_1|\cdots |g_{n-1}].
\end{align*}
Similarly, an element of $C^n(G,M)$ can be regarded as a function $f:G^n\rightarrow M,$ i.e., as a function of $n$ variables from $G$ to $M.$ The coboundary operator $\delta:C^{n-1}(G,M)\rightarrow C^n(G,M)$ is given, up to sign by
\begin{align*}
(\delta f)(g_1,\ldots ,g_n) &= g_1f(g_2,\ldots ,g_n) \\
& -f(g_1g_2,\ldots , g_n)+\cdots + (-1)^n f(g_1,\ldots ,g_{n-1}).
\end{align*}

\subsection{Some results on Cohomology of Groups}

\begin{lema}[Center kills lemma (Lemma 5.4. of\cite{dupont})]
\label{lem_cen_kill}
Let $M$ be a (left) $R[G]$-module ($R$ any commutative ring) and let $\gamma\in G$ be a central element such that for some $r\in R,$ $\gamma x=rx$ for all $x\in M.$
Then $(r-1)$ annihilates $H_*(G,M).$
\end{lema}

\begin{defi}[\cite{brown}, VI.7.]
For $M$ an abelian group, define $M^*=Hom(M,\mathbb{Q}/\mathbb{Z}).$ Let $L\geq 2$ be an integer, observe that if $M$ is a module over $\mathbb{Z}/L,$ then $M^*\cong Hom(M,\mathbb{Z}/L).$
\end{defi}

\begin{lema}[\cite{brown}, VI. Proposition 7.1.]
\label{lema_duality_homology}
Let $G$ be a group and let $M$ be a $G$-module. Then there is a natural isomorphism $H^k(G;M^*)\cong (H_k(G;M))^*$ for every $k\geq 0.$
\end{lema}

\begin{teo}[K. Dekimpe, M. Hartl, S. Wauters \cite{seven}]
\label{teo_seven_terms}
For any group extension $1\rightarrow N \rightarrow G \rightarrow Q \rightarrow 1$
and any $G$-module $M,$ the Lyndon-Hochschild-Serre spectral sequence gives rise to an exact sequence:
$$
\xymatrix@C=10mm@R=13mm{0 \ar@{->}[r] & H^1(Q;M^N) \ar@{->}[r]^{\text{inf}} & H^1(G;M) \ar@{->}[r]^{\text{res}} &H^1(N;M)^Q
\ar@{->}[r]^{\text{tr}} &  H^2(Q;M^N)}
$$
$$
\xymatrix@C=10mm@R=13mm{ \ar@{->}[r]^{\text{inf} \qquad} & H^2(G;M)_1 \ar@{->}[r]^{\rho \qquad} & H^1(Q;H^1(N;M)) \ar@{->}[r]^{\lambda} & H^3(Q;M^N),}
$$
where ''inf'' and ''res'' are respectively the inflation and restriction maps, and $H^2(G,M)_1$ is the kernel of the restriction map $res: H^2(G,M)\rightarrow H^2(N,M).$
\end{teo}

\begin{teo}\label{teo_exact_coef}
Let $0\rightarrow L \rightarrow M \rightarrow N \rightarrow 0$ a short exact sequence of $G$-modules, then there is a long exact sequence:
$$0\rightarrow L^G \rightarrow M^G \rightarrow N^G \rightarrow H^1(G,L)\rightarrow H^1(G,M) \rightarrow H^1(G,N) \rightarrow H^2(G,L)$$
\end{teo}

\begin{lema}[Leedham \cite{leed}]
\label{lema_canvi_coef_coh}
If $R$ is a commutative ring, $A$ is an $RG$-module and $n\geq 0$ then there is a natural isomorphism
$$Ext^n_{RG}(R,A)\cong Ext^n_{\mathbb{Z}G}(\mathbb{Z},A).$$
\end{lema}

\begin{proof}
If $P$ is a projective $\mathbb{Z}G$-module then $P\otimes R$ is a projective $RG$-module as this is true for free modules. Let $\mathbf{P} \rightarrow \mathbb{Z}$ be a projective resolution of $\mathbb{Z}$ as a $\mathbb{Z}G$-module. Then $\mathbf{P}\otimes R \rightarrow R$ is a projective resolution of $R$ as an $RG$-module since $\mathbf{P}\otimes R \rightarrow R$ is exact as its homology groups are $Tor^\mathbb{Z}_n(\mathbb{Z},R)=0.$
Since $Hom_{RG}(P_n\otimes R, A)$ is naturally isomorphic to $Hom_{\mathbb{Z}G}(P_n,A)$ it follows that $Ext^n_{\mathbb{Z}G}(\mathbb{Z},A)\cong Ext^n_{RG}(R,A).$
\end{proof}

\begin{teo}[Universal coefficients Theorem, Theorem 9.4.14 in \cite{leed}]
\label{uct}
Let $G$ be a group, $R$ a Dedekind domain considered as a trivial $G$-module, and $A$ an $RG$-module with trivial $G$-action, then there are natural exact sequences
$$
\xymatrix@C=7mm@R=10mm{ 0 \ar@{->}[r] &  Ext^1_R(H_{n-1}(G;R),A) \ar@{->}[r]^-{\alpha} & H^n(G;A)
\ar@{->}[r]^-{\beta} & Hom_R(H_n(G;R),A) \ar@{->}[r] & 0.}
$$
$$
\xymatrix@C=7mm@R=10mm{ 0 \ar@{->}[r] &  H_n(G;R)\otimes_R A \ar@{->}[r] & H_n(G;A)
\ar@{->}[r] & Tor_1^R (H_{n-1}(G;R),A)\ar@{->}[r] & 0.}
$$
which split, but not naturally.
\end{teo}

\begin{rem}
\label{rem_construction_uct}
The map $\beta: H^n(G;A) \rightarrow Hom_R(H_n(G;R),A)$ of Theorem \eqref{uct} is constructed as follows.
By definition of the cohomology of groups we know that $H^n(G;A)=Ext^n_{\mathbb{Z}G}(\mathbb{Z};A)$ and by Lemma \eqref{lema_canvi_coef_coh} have that $Ext^n_{RG}(R,A)\cong Ext^n_{\mathbb{Z}G}(\mathbb{Z},A).$

Since the bar resolution $C_*(G)\rightarrow \mathbb{Z}$ is a projective resolution of $\mathbb{Z}$ as $\mathbb{Z}G$-module. Then 
$C_*(G)\otimes R \rightarrow R$ is a projective resolution of $R$ as $RG$-module.
Applying the contravariant functor $Hom_{RG}(-,A)$ to the projective resolution $C_*(G)\otimes R,$ we obtain the following chain complex 
\begin{equation*}
\xymatrix@C=3.5mm@R=7mm{ \cdots & Hom_{RG}(C_2(G)\otimes R, A) \ar@{->}[l]_-{\delta} & Hom_{RG}(C_1(G)\otimes R, A) \ar@{->}[l]_-{\delta} & Hom_{RG}(C_0(G)\otimes R, A) \ar@{->}[l]_-{\delta} &  \ar@{->}[l] 0.}
\end{equation*}
Since $A$ is a trivial $RG$-module, we have the following canonic isomorphisms:
\begin{align*}
Hom_{RG}(C_n(G)\otimes R,A)= & Hom_R(C_n(G)\otimes R,A)^G \cong \\
\cong & Hom_R((C_n(G)\otimes R)_G,A)\cong \\
\cong & Hom_R(C_n(G)\otimes_{\mathbb{Z}G} R,A).
\end{align*}
Then the above chain complex becomes
$$\xymatrix@C=4mm@R=7mm{ \cdots & Hom_{R}(C_2(G,R) , A) \ar@{->}[l]_-{\delta} & Hom_{R}(C_1(G,R), A) \ar@{->}[l]_-{\delta} & Hom_{R}(C_0(G,R), A) \ar@{->}[l]_-{\delta} &  \ar@{->}[l] 0.}
$$
Thus we have that
\begin{equation}
\label{iso_explicacio_UCT}
H^n(G;A)\cong H_n(Hom_{R}(C_*(G,R), A)).
\end{equation}

Therefore $\beta: H^n(G;A) \rightarrow Hom_R(H_n(G;R),A)$ in Theorem \eqref{uct} is given as follows. Let $z\in H^n(G;A)$ and $\widetilde{z} \in H_n(Hom_{R}(C_*(G,R), A))$ the homology class that corresponds to $z$ by the isomorphism \eqref{iso_explicacio_UCT}. Take a $n$-cycle $c\in Hom_{RG}(C_n (G,R), A),$ with associated homology class $\widetilde{z}.$
By passage to subquotients, $c$ induces an element of $Hom_R(H_n(G;R),A).$
Then $\beta(z)$ is defined as the evaluation of such $n$-cycle $c,$  on $H_n(G;R).$
\end{rem}

\begin{teo}[Theorem $VI.8.1.$  in \cite{stam}]
\label{teo_5-term}
Let $1\rightarrow N \rightarrow G \rightarrow Q \rightarrow 1$ be an exact sequence of groups. For $Q$-modules $A,$ $B,$ the following sequences are exact and natural
\begin{align*}
& H_2(G;B)\rightarrow H_2(Q;B)\rightarrow B\otimes_Q N_{ab}\rightarrow H_1(G;B)\rightarrow H_1(Q;B)\rightarrow 0, \\
& 0\rightarrow H^1(Q;A)\rightarrow H^1(G;A)\rightarrow Hom_Q(N_{ab},A)\rightarrow H^2(Q;A)\rightarrow H^2(G;A).
\end{align*}
Where natural means :
\begin{enumerate}[i)]
\item A commutative diagram of short exact sequences
\begin{equation}
\xymatrix@C=7mm@R=10mm{0 \ar@{->}[r] & N \ar@{->}[r] \ar@{->}[d]&  G \ar@{->}[r] \ar@{->}[d]& Q \ar@{->}[d] \ar@{->}[r]& 1 \\
0 \ar@{->}[r] & N' \ar@{->}[r] & G' \ar@{->}[r] & Q' \ar@{->}[r] & 1 }
\end{equation}
induces the following commutative diagrams
\begin{equation}
\xymatrix@C=7mm@R=10mm{ H_2(G;B)\ar@{->}[r] \ar@{->}[d]& H_2(Q;B)\ar@{->}[r] \ar@{->}[d]& B\otimes_{Q} N_{ab}\ar@{->}[r] \ar@{->}[d]& H_1(G;B)\ar@{->}[r] \ar@{->}[d]& H_1(Q;B)\ar@{->}[d] \\
H_2(G';B)\ar@{->}[r] & H_2(Q';B)\ar@{->}[r] & B\otimes_{Q'} N'_{ab}\ar@{->}[r] & H_1(G';B)\ar@{->}[r] & H_1(Q';B) }
\end{equation}

\begin{equation}
\xymatrix@C=5mm@R=10mm{ H^1(Q';A)\ar@{->}[r] \ar@{->}[d]& H^1(G';A)\ar@{->}[r] \ar@{->}[d]& Hom_{Q'}(N'_{ab},A)\ar@{->}[r] \ar@{->}[d]& H^2(Q';A)\ar@{->}[r] \ar@{->}[d]& H^2(G';A) \ar@{->}[d] \\
H^1(Q;A)\ar@{->}[r] & H^1(G;A)\ar@{->}[r] & Hom_Q(N_{ab},A)\ar@{->}[r] & H^2(Q;A)\ar@{->}[r] & H^2(G;A) }
\end{equation}

\item A homomorphism of $Q$-modules $g:B\rightarrow B'$ induces the following commutative diagram
\begin{equation}
\xymatrix@C=7mm@R=10mm{ H_2(G;B)\ar@{->}[r] \ar@{->}[d]& H_2(Q;B)\ar@{->}[r] \ar@{->}[d]& B\otimes_{Q} N_{ab}\ar@{->}[r] \ar@{->}[d]& H_1(G;B)\ar@{->}[r] \ar@{->}[d]& H_1(Q;B)\ar@{->}[d] \\
H_2(G;B')\ar@{->}[r] & H_2(Q;B')\ar@{->}[r] & B'\otimes_Q N_{ab}\ar@{->}[r] & H_1(G;B')\ar@{->}[r] & H_1(Q;B') }
\end{equation}
Similarly a homomorphism of $Q$-modules $f:A\rightarrow A'$ induces the following commutative diagram
\begin{equation}
\xymatrix@C=5mm@R=10mm{ H^1(Q;A)\ar@{->}[r] \ar@{->}[d]& H^1(G;A)\ar@{->}[r] \ar@{->}[d]& Hom_Q(N_{ab},A)\ar@{->}[r] \ar@{->}[d]& H^2(Q;A)\ar@{->}[r] \ar@{->}[d]& H^2(G;A) \ar@{->}[d] \\
H^1(Q;A')\ar@{->}[r] & H^1(G;A')\ar@{->}[r] & Hom_Q(N_{ab},A')\ar@{->}[r] & H^2(Q;A')\ar@{->}[r] & H^2(G;A') }
\end{equation}
\end{enumerate}
\end{teo}

\paragraph{Hopf formula modulo $q$}
In \cite{graham} Graham J. Ellis find a formula to compute the homology of groups with coefficients in $\mathbb{Z}/q,$ with $q$ a positive integer.
In particular, following the Exercice II.5.4 in \cite{brown} one can gives an explicit isomorphism which induces the Hopf modulo $q$ formula as follows.
Let $q$ be a positive integer and $G$ be a group with a finite representation $\langle t_1,\ldots ,t_n | r_1,\ldots r_m\rangle.$ We denote by $F$ the free group generated by a family of elements $T=\{t_1,\ldots ,t_n\}$ and by $R$ the normal closure in $F$ of $\{r_1,\ldots, r_m\}.$

Consider the following exact sequence
\begin{equation*}
\xymatrix@C=7mm@R=7mm{1 \ar@{->}[r] &  R \ar@{->}[r]^-{i} & F \ar@{->}[r]^-{\overline{(\;)}}  & G \ar@{->}[r] & 1. }
\end{equation*}
where $\overline{(\;)}$ is the quotient map corresponding to $R.$

For each $g\in G$ choose an element $s(g)\in F$ such that $\overline{s(g)}=g.$
Given $g,h\in G,$ write $s(g)s(h)=s(gh)r(g,h)$ where $r(g,h)\in R.$
There is an abelian group homomorphism $C_2(G)\rightarrow R/[R,R]$ given by $[g|h]\mapsto r(g,h) \text{ mod } [R,R]$ and this induces an isomorphism
$$\phi:H_2(G;\mathbb{Z}/q)
\rightarrow \frac{R\cap ([F,F]F^q)}{[F,R]R^q}$$
by passage to subquotients.

\subsection{Group Extensions}
In this section, we recall some basic results about group extensions, we make a summary of chapter $IV$ in \cite{brown} and section $VI.10$ in \cite{stam}.

An extension of a group $G$ by a group $N$ is a short exact sequence of groups
\begin{equation}
\label{extension_groups}
1\rightarrow N \rightarrow E \rightarrow G \rightarrow 1
\end{equation}
A second extension $1\rightarrow N \rightarrow E' \rightarrow G \rightarrow 1$ of $G$ by $N$ is said to be \textit{equivalent} to \eqref{extension_groups} if there is a homomorphism $E\rightarrow E'$ making the diagram
\begin{equation*}
\xymatrix@C=10mm@R=10mm{
 1 \ar@{->}[r] & N \ar@{->}[r] \ar@{=}[d] & E \ar@{->}[r] \ar@{->}[d] & G \ar@{->}[r] \ar@{=}[d] & 1\\
1  \ar@{->}[r] & N \ar@{->}[r]
 & E' \ar@{->}[r] & G \ar@{->}[r] & 1 }
\end{equation*}
commute. Note that such a map is necessarily an isomorphism.

Consider the case where the kernel $N$ is an abelian group $A$ (written additively). A special feature of this case is that an extension
\begin{equation}
\label{cext}
\xymatrix@C=5mm@R=5mm{
0  \ar@{->}[r] & A \ar@{->}[r]^i
 & E \ar@{->}[r]^\pi & G \ar@{->}[r] & 1 }
\end{equation}
gives rise to an \textit{action} of $G$ on $A,$ making $A$ a $G$-module. For $E$ acts on $A$ by conjugation since $A$ is embedded as a normal subgroup of $E;$ and the conjugation action of $A$ on itself is trivial, so there is an induced action of $E/A=G$ on $A.$
Moreover observe that the $G$-action is trivial if and only if $i(A)$ is central in $E.$
In the case that $i(A)$ is central in $E,$ the extension \eqref{cext} is called a \textit{central extension.}

\paragraph{Split Extensions.}

Fix a $G$-module $A$ and let
\begin{equation}
\label{extension_central}
\xymatrix@C=5mm@R=5mm{
0  \ar@{->}[r] & A \ar@{->}[r]^i
 & E \ar@{->}[r]^\pi & G \ar@{->}[r] & 1 }
\end{equation}
be an extension which gives rise to the given action of $G$ on $A.$ We say that \eqref{extension_central} \textit{splits} if there is a homomorphism $s:G\rightarrow E$ such that $\pi s=id_G.$

\begin{prop}
\label{prop_eqiv_splittings}
The following two conditions on the extension \eqref{extension_central} are equivalent:
\begin{enumerate}[i)]
\item \eqref{extension_central} splits.
\item \eqref{extension_central} is equivalent to the extension
$
\xymatrix@C=5mm@R=5mm{
0  \ar@{->}[r] & A \ar@{->}[r]
 & A\rtimes G \ar@{->}[r] & G \ar@{->}[r] & 1,}
$
where $A \rtimes G$ is the semi-direct product of $G$ and $A$ relative to the given action.
\end{enumerate}
\end{prop}
Observe that Proposition \eqref{prop_eqiv_splittings} says that there
is only one split extension $G$ by $A$ (up to equivalence) associated to the given action of $G$ on $A.$ Nevertheless, there is an interesting ''classification'' problem involving split extensions:
Given that an extension \eqref{extension_central} splits, classify all possible splittings.

In case $G$ acts trivially on $A,$ the group $E$ is isomorphic to the direct product $A\times G$ and then the splittings are obviously in 1-1 correspondence with homomorphisms $G\rightarrow A.$ In the general case, splittings correspond to \textit{derivations} (also called \textit{crossed homomorphisms}). These are functions $d: G\rightarrow A$ satisfying
$$d(gh)=d(g)+g\cdot d(h)$$
for all $g,h\in G.$

Two splittings $s_1,$ $s_2$ will be said to be $A$-\textit{conjugate} if there is an element $a\in A$ such that $s_1(g)=i(a)s_2(g)i(a)^{-1}$ for all $g\in G.$ Since
$(a,1)(b,g)(a,1)^{-1}=(a+b,g\cdot a,g)$
in $A\rtimes G,$ this conjugancy relation becomes
$d_1(g)=a+d_2(g)- g\cdot a$
in terms of the derivations $d_1,$ $d_2$ corresponding $s_1,$ $s_2.$ Thus $d_1$ and $d_2$ correspond to $A$-conjugate splittings if and only if their difference $d_2-d_1$ is a function $G\rightarrow A$ of the form $g\mapsto ga-a$ for some fixed $a\in A.$ Such function is called a \textit{principal derivation.}

Summarizing, the $A$-conjugacy classes of splittings of a split extension of $G$ by $A$ correspond to the elements of the quotient group $Der(G,A)/P(G,A),$ where $Der(G,A)$ is the abelian group of derivations $G\rightarrow A$ and $P(G,A)$ is the group of principal derivations.
In addition, by Exercise 2 of III.1 in \cite{brown} one gets that the quotient $Der(G,A)/P(G,A)$ is canonically isomorphic to the first cohomology group $H^1(G,A).$
Therefore we have the following result:

\begin{prop}
For any $G$-module $A,$ the $A$-conjugancy classes of splittings of the split extension
$$0\rightarrow A \rightarrow A \rtimes G \rightarrow G \rightarrow 1$$
are in 1-1 correspondence with the elements of $H^1(G,A).$
\end{prop}

\paragraph{Classification of extensions with abelian kernel.}

Let $A$ be a fixed $G$-module. All extensions of $G$ by $A$ to be considered in this section will be assumed to give rise to the given action of $G$ on $A.$ To analyse an extension
\begin{equation}
\label{extension_ab_ker}
\xymatrix@C=5mm@R=5mm{
0  \ar@{->}[r] & A \ar@{->}[r]^i
 & E \ar@{->}[r]^\pi & G \ar@{->}[r] & 1, }
\end{equation}
we choose a set-theoretic cross-section of $\pi,$ i.e., a function $s: G\rightarrow E$ such that $\pi s=id_G.$

If $s$ is a homomorphism, then the extension splits and by Proposition \eqref{prop_eqiv_splittings}, we know that \eqref{extension_ab_ker} is equivalent to $0\rightarrow A \rightarrow A \rtimes G \rightarrow G \rightarrow 1$ .
In the general case, however, there is a function $f: G\times G \rightarrow A$ which measures the failure of $s$ to be a homomorphism. Indeed, for any $g,h\in G,$ the elements $s(gh)$ and $s(g)s(h)$ of $E$ both map to $gh$ in $G,$ so they differ by an element of $i(A).$ Thus we can define $f$ by the equation:
\begin{equation}
\label{eq_factor set}
s(g)s(h)=i(f(g,h))s(gh).
\end{equation}
The function $f$ is called the \textit{factor set} associated to \eqref{extension_ab_ker} and $s.$

The $G$-module structure on $A$ and the factor set $f$ classifies all extensions in the following way.
Let $E_f$ be the set $A\times G$ with the group law:
$(a,g)(b,h)=(a+gb+f(g,h),gh),$
then the original extension \eqref{extension_ab_ker} is equivalent to the extension
$0\rightarrow A \rightarrow E_f \rightarrow G \rightarrow 1.$
In particular to be $E_f$ a group $f$ has to satisfy the following identity:
\begin{equation*}
gf(h,k)-f(gh,k)+f(g,hk)-f(g,h)=0
\end{equation*}
for all $g,h,k\in G.$
As a consequence, $f$ can be regarded as a $2$-cocycle of the standard complex $C^*(G,A)$ for computing $H^*(G;A).$

Notice that to get the factor set $f$ we have chosen an arbitrary theoretic section $s.$
However, if we take two sections $s_1,$ $s_2$ with associated factor sets $f_1,$ $f_2,$ then there exists a function $d:G\rightarrow A$ such that
$$f_2(g,h)=d(g)+gd(h)-d(gh)+f_1(g,h).$$
Hence, changing the choice of the section in \eqref{extension_ab_ker} corresponds precisely to modifying the cocycle $f$ in $C^*(G,A)$ by a coboundary.

Therefore we get the following result:
\begin{prop}
\label{prop_class_extensions}
Let $A$ be a $G$-module and let $\mathcal{E}(G,A)$ be the set of equivalence classes of extensions of $G$ by $A$ giving rise to the action of $G$ on $A.$ Then there is a bijection
$\mathcal{E}(G,A)\thickapprox H^2(G,A).$
\end{prop}

Next we show another description of the bijection
given in Proposition \eqref{prop_class_extensions} , due to U. Stammbach  \cite{stam}.

Denote by $[E_i^r],$ the class of $\mathcal{E}(G,A)$ containing the extension
\begin{equation}
\label{ext_stamb}
\xymatrix@C=10mm@R=10mm{
 0 \ar@{->}[r] & A \ar@{->}[r]^-{i} & E \ar@{->}[r]^-{r}  & G \ar@{->}[r] & 1.}
\end{equation}
We now will define a map $\varDelta: \mathcal{E}(G,A)\rightarrow H^2(G,A)$ as follows.
Given an extension \eqref{ext_stamb}, Theorem \eqref{teo_5-term} yields the exact sequence
\begin{equation}
\label{5-term_seq theta}
\xymatrix@C=6mm@R=5mm{0\ar@{->}[r] & H^1(G;A)\ar@{->}[r] & H^1(E;A)\ar@{->}[r] & Hom_G(A,A)\ar@{->}[r]^-{\theta_{E_i^r}} & H^2(G;A)\ar@{->}[r] & H^2(E;A)}.
\end{equation}
Associate with the extension $1\rightarrow A \rightarrow E \rightarrow G \rightarrow 1$ the element
$$\varDelta[E_i^r]=\theta_{E_i^r}(id_A)\in H^2(G;A)$$
The naturality of \eqref{5-term_seq theta} immediately shows that $\theta_{E_i^r}(id_A)\in H^2(G;A)$ does not depend on the extension but only on its equivalence class in $\mathcal{E}(G,A).$ Hence $\varDelta$ is well-defined,
$$\varDelta: \mathcal{E}(G,A)\rightarrow H^2(G,A).$$
In VI.10. of \cite{stam}, U. Stammbach showed moreover that $\varDelta$ is one-to-one.

\paragraph{Maps between extensions.} Next we give some results about maps between two extensions induced by push-outs and pull-backs.

\begin{prop}[Exercice 1a, page 94 in \cite{brown}]
\label{prop_pull-back}
Given an extension $0\rightarrow A \rightarrow E \rightarrow G \rightarrow 1$ and a group homomorphism $\alpha:G'\rightarrow G.$ There is an extension
$0\rightarrow A \rightarrow E' \rightarrow G' \rightarrow 1,$ characterized up to equivalence by the fact that it fits into a commutative diagram
\begin{equation*}
\xymatrix@C=10mm@R=10mm{
 0 \ar@{->}[r] & A \ar@{->}[r]^i & E \ar@{->}[r]^r  & G \ar@{->}[r] & 1\\
0  \ar@{->}[r] & A \ar@{=}[u] \ar@{->}[r]^{i'}
 & E' \ar@{->}[u]^e \ar@{->}[r]^{r'} &\ar@{->}[u]^\alpha G' \ar@{->}[r] & 1. }
\end{equation*}
In fact $(E';r',e)$ is given by the pull-back of $(\alpha,r).$

As a consequence $\alpha$ induces a map $\alpha^*:\mathcal{E}(G,A)\rightarrow \mathcal{E}(G',A),$ which corresponds under the bijection of Proposition \eqref{prop_class_extensions} to $\alpha^*: H^2(G;A)\rightarrow H^2(G';A).$
\end{prop}

\begin{prop}[Exercice 1b, page 94 in \cite{brown}]
\label{prop_push-out}
Given an extension $0\rightarrow A \rightarrow E \rightarrow G \rightarrow 1$ and a homomorphism $f: A\rightarrow A'$ of $G$-modules, there is an extension
$0\rightarrow A' \rightarrow E' \rightarrow G \rightarrow 1,$ characterized up to equivalence by the fact that it fits into a commutative diagram:
\begin{equation*}
\xymatrix@C=10mm@R=10mm{
 0 \ar@{->}[r] & A \ar@{->}[r]^i \ar@{->}[d]^{f} & E \ar@{->}[r]^r \ar@{->}[d]^{e} & G \ar@{->}[r] \ar@{=}[d] & 1\\
0  \ar@{->}[r] & A' \ar@{->}[r]^{i'}
 & E' \ar@{->}[r]^{r'} & G \ar@{->}[r] & 1. }
\end{equation*}
In fact $(E';i',e)$ is given by the push-out of $(f,i).$

As a consequence $f$ induces a map $f_*:\mathcal{E}(G,A)\rightarrow \mathcal{E}(G,A'),$ which corresponds under the bijection of Proposition \eqref{prop_class_extensions} to $f_*:H^2(G,A)\rightarrow H^2(G,A').$
\end{prop}

\begin{cor}
\label{cor_push-pull}
Let $1\rightarrow A \rightarrow E \rightarrow G \rightarrow 1$ be a central extension with associated cohomology class $c,$ let $f: A \rightarrow A$ be a homomorphism of $G$-modules and $\phi :G \rightarrow G$ a group homomorphism,
such that $f_*(c)=\phi^*(c)$ in $H^2(G,A).$
Then there exists a homomorphism $\Phi: E\rightarrow E$ such that the following diagram commutes.
\begin{equation*}
\xymatrix@C=10mm@R=10mm{
 0 \ar@{->}[r] & A \ar@{->}[r]^i \ar@{->}[d]^-{f} & E \ar@{->}[r]^r \ar@{->}[d]^-{\Phi} & G \ar@{->}[r] \ar@{->}[d]^-{\phi} & 1\\
0  \ar@{->}[r] & A \ar@{->}[r]^{i}
 & E \ar@{->}[r]^{r} & G \ar@{->}[r] & 1. }
\end{equation*}
\end{cor}

\begin{proof}
Given an extension $1\rightarrow A \rightarrow E \rightarrow G \rightarrow 1,$ $f: A \rightarrow A$ a homomorphism of $G$-modules and $\phi :G \rightarrow G$ a group homomorphism,
by Propositions \eqref{prop_pull-back}, \eqref{prop_push-out} we have the following commutative diagrams:
\begin{equation*}
\xymatrix@C=10mm@R=10mm{
 0 \ar@{->}[r] & A \ar@{->}[r]^i \ar@{->}[d]^-{f} & E \ar@{->}[r]^r \ar@{->}[d] & G \ar@{->}[r] \ar@{=}[d] & 1\\
0  \ar@{->}[r] & A \ar@{->}[r]
 & E' \ar@{->}[r] & G \ar@{->}[r] & 1. }
\qquad
\xymatrix@C=10mm@R=10mm{
 0 \ar@{->}[r] & A \ar@{->}[r] \ar@{=}[d] & E'' \ar@{->}[r] \ar@{->}[d] & G \ar@{->}[r] \ar@{->}[d]^-{\phi} & 1\\
0  \ar@{->}[r] & A \ar@{->}[r]^{i}
 & E \ar@{->}[r]^{r} & G \ar@{->}[r] & 1. }
\end{equation*}
In addition, since $f_*(c)=\phi^*(c)$ in $H^2(G,A),$ using the bijection between $\mathcal{E}(G,A)$ and $H^2(G;A)$ given by Proposition \eqref{prop_class_extensions}, and Propositions \eqref{prop_pull-back}, \eqref{prop_push-out} we obtain the following commutative diagram
\begin{equation*}
\xymatrix@C=10mm@R=10mm{
 0 \ar@{->}[r] & A \ar@{->}[r] \ar@{=}[d] & E' \ar@{->}[r] \ar@{->}[d] & G \ar@{->}[r] \ar@{=}[d] & 1\\
0  \ar@{->}[r] & A \ar@{->}[r]
 & E'' \ar@{->}[r] & G \ar@{->}[r] & 1. }
\end{equation*}
Finally reassembling the above three diagrams we get the result.
\end{proof}

\begin{lema}
\label{lem-cocy-pull-back}
Let $A$ be a fixed $G$-module,
$$\xymatrix@C=10mm@R=13mm{1 \ar@{->}[r] & A \ar@{->}[r]^{i} & E \ar@{->}[r]^{\pi} & G \ar@{->}[r]& 1 }$$
a central extension with associated cohomology class $[c]\in H^2(G;A),$ and $\pi^*:H^2(G;A) \rightarrow H^2(E;A)$ the induced morphism for $\pi.$ Then the cohomology class $[\pi^*(c)]$ is zero.
\end{lema}

\begin{proof}
Let $[c]\in H^2(G;A),$ the cohomology class associated to
$$\xymatrix@C=10mm@R=13mm{1 \ar@{->}[r] & A \ar@{->}[r]^{i} & E \ar@{->}[r]^{\pi} & G \ar@{->}[r]& 1 }$$
then we have the following commutative diagram
\begin{equation}
\label{diag_com_triv_pullback}
\xymatrix@C=10mm@R=13mm{1 \ar@{->}[r] & A\ar@{=}[d] \ar@{->}[r]^{i} & \tilde{E}\ar@{->}[d]^{\eta} \ar@{->}[r]^{\eta} & E \ar@{->}[d]^-{\pi}\ar@{->}[r]& 1 \\
1 \ar@{->}[r] & A \ar@{->}[r]^{j} & E \ar@{->}[r]^{\pi} & G \ar@{->}[r]& 1}
\end{equation}
where the top central extension is associated to the cohomology class $[\pi^*(c)].$
But such cohomology class is given as follows.
Let $s$ a set theoretic section of $\eta,$ i.e. $\eta\circ s=id,$ then  $[\pi^*(c)]$ is equal to the cohomology class of the $2$-cocycle $h$ defined by
$$h([x|y])=i^{-1}(s(x)s(y)s(xy)^{-1}).$$
Now since $j$ is injective, the diagram \eqref{diag_com_triv_pullback} is commutatvie and $s$ is a theoretic section of $\eta,$ we have that
\begin{align*}
0=h([x|y])\Leftrightarrow & j(i^{-1}(s(x)s(y)s(xy)^{-1}))=1\Leftrightarrow \\
\Leftrightarrow & (\eta\circ i)(i^{-1}(s(x)s(y)s(xy)^{-1}))=1\Leftrightarrow \\ \Leftrightarrow &\eta(s(x)s(y)s(xy)^{-1}))=0\Leftrightarrow xy(xy)^{-1}=1
\end{align*}
Therefore $[\pi^*(c)]\in H^2(E;A)$ is zero.
\end{proof}

\paragraph{Stabilizing automorphisms of an extension}
\begin{defi}[Section 9.1.3. in \cite{rot}]
An automorphism $\varphi$ of a group $E$ \textit{stabilizes} an extension
\begin{equation*}
\xymatrix@C=7mm@R=10mm{0 \ar@{->}[r] & A  \ar@{->}[r] &  E \ar@{->}[r] & G \ar@{->}[r]& 1  }
\end{equation*}
if the following diagram commutes:
\begin{equation*}
\xymatrix@C=7mm@R=10mm{0 \ar@{->}[r] & A \ar@{->}[r]^-{i} \ar@{=}[d]&  E \ar@{->}[r]^-{p} \ar@{->}[d]^-{\varphi} & G \ar@{=}[d] \ar@{->}[r]& 1 \\
0 \ar@{->}[r] & A \ar@{->}[r]^-{i} & E \ar@{->}[r]^-{p} & G \ar@{->}[r] & 1. }
\end{equation*}
The set of all stabilizing automorphisms of an extension of $A$ by $G,$ where $A$ is a $G$-module, is a group under composition and it is denoted by $Stab(G,A).$
In addition, by Corollary 9.16 in \cite{rot}, $Stab(G,A)$ is isomorphic to the group of derivations $Der(G,A)$ via the homomorphism
\begin{align*}
\sigma: Stab(G,A) & \rightarrow Der(G,A) \\
\varphi & \mapsto (d:\;G \rightarrow A),
\end{align*}
where $d(x)=\varphi(s(x))-s(x)$ with $s$ a section.
\end{defi}

\section{Geometric objects}

\subsection{Simple closed curves}

We refer by a \textit{closed curve} in a surface $\Sigma,$ a continuous map $S^1\rightarrow \Sigma.$ We will usually identify a closed curve with its image in $\Sigma.$ A closed curve is called \textit{essential} if it is not homotopic to a point, a puncture, or a boundary component.

A closed curve in $\Sigma$ is $\textit{simple}$ if it is embedded, that is, if the map $S^1\rightarrow \Sigma$ is injective.
Moreover, we say that $\gamma$ is a bounding simple closed curve if its homology class in $H_1(\Sigma)$ is zero, or equivalently, if $\gamma$ separates $\Sigma$ in two disjoint surfaces with boundary component $\gamma.$

Any closed curve $\alpha$ is homotopic to a smooth closed curve $\alpha'.$ What is more, if $\alpha$ is simple, then $\alpha'$ can be chosen to be simple.

\paragraph{Intersection number}

There are two natural ways to count the number of intersection points between two simple closed curves in an oriented surface: signed and unsigned. These correspond to the algebraic intersection number and geometric intersection number, respectively.

Let $\alpha$ and $\beta$ be a pair of transverse, oriented, simple closed curves in an oriented surface $\Sigma.$ The \textit{algebraic intersection number} $\widehat{i}(\alpha,\beta)$ is defined as the sum of the indices of the intersection points of $\alpha$ and $\beta,$ where an intersection point is of index $+1$ when the orientation of the intersection agrees with the orientation of $S$ and is $-1$ otherwise. This definition only depends on the homology classes of $\alpha$ and $\beta.$

The \textit{geometric intersection number} between free homotoy classes $a$ and $b$ of simple closed curves in a surface $S$ is defined to be the minimal number of intersection points between a representative curve in the class $a$ and a representative curve in the class $b:$

$$i(a,b)= \text{min}\{ |\alpha \cap \beta|\; : \; \alpha\in a, \beta \in b \}.$$

\paragraph{Isotopy for simple closed curves}
Two simple closed curves $\alpha$ and $\beta$ are \textit{isotopic} if there is a smooth homotopy
$$H:S^1\times [0,1]\rightarrow S$$
from $\alpha$ to $\beta$ with the property that the closed curve $H(S^1\times \{t\})$ is simple for each $t\in [0,1].$

\paragraph{Extension of isotopies}
An isotopy of a surface $S$ is a smooth homotopy $H:S\times I\rightarrow S$ so that, for each $t\in [0,1],$ the map $H(S,t):S\times \{t\}\rightarrow S$ is a homeomorphism. Given an isotopy between two simple closed curves in $S,$ it will often be useful to promote this to an isotopy of $S,$ which we call an \textit{ambient isotopy} of $S.$

\begin{prop}
Let $S$ be any surface. If $F: S^1\times I\rightarrow S$ is a smooth isotopy of simple closed curves, then there is an isotopy $H:S\times I\rightarrow S$ so that $H|_{F(S^1\times 0)\times I}=F.$
\end{prop}

\subsection{Handlebodies}

\begin{defi}[Definition 1.7 in \cite{jes}]
Let $B_1,\ldots ,B_n$ be a collection of closed $3$-balls and let $D_1,\ldots , D_m,$ $D'_1,\ldots ,D'_m$ be a collection of pairwise disjoint disks in $\bigcup \partial B_i.$ For each $i\geq m,$ let $\phi_i:D_i\rightarrow D'_i$ be a homeomorphism. Let $H$ be the result of gluing along $\phi_1,$ then gluing along $\phi_2,$ and so on.
After the final gluing if $H$ is connected then $H$ is a \textit{handlebody.}
\end{defi}

\begin{defi}
The genus of a handlebody is the genus of its boundary $\partial H.$
\end{defi}

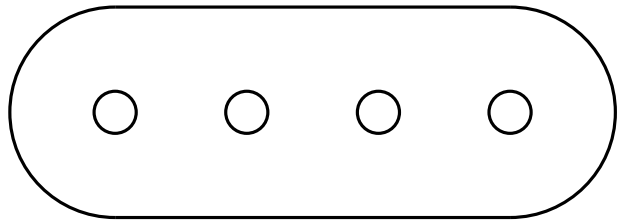
\begin{figure}[H]
\begin{center}
\begin{tikzpicture}[scale=.7]
\draw[very thick] (-2.5,-2) -- (5,-2);
\draw[very thick] (-2.5,2) -- (5,2);
\draw[very thick] (-2.5,2) arc [radius=2, start angle=90, end angle=270];
\draw[very thick] (5,2) arc [radius=2, start angle=90, end angle=-90];

\draw[very thick] (-2.5,0) circle [radius=.4];
\draw[very thick] (0,0) circle [radius=.4];
\draw[very thick] (2.5,0) circle [radius=.4];
\draw[very thick] (5,0) circle [radius=.4];
\end{tikzpicture}
\end{center}
\caption{A Handlebody of genus 4}
\end{figure}

\begin{defi}
A properly embedded disk $D\subset H$ is essential if its boundary does not bound a disk in $\partial H.$
\end{defi}

\begin{defi}[Definition 2.1 in \cite{jes}]
A collection $\{D_1,\ldots ,D_m \}$ of properly embedded, essential disks is called a \textit{system of disks} for $H$ if the complement of a regular neighbourhood of $\bigcup D_i$ is a collection of balls.
\end{defi}

\begin{figure}[H]
\begin{center}
\begin{tikzpicture}[scale=.7]
\draw[very thick] (-4.5,-2) -- (3,-2);
\draw[very thick] (-4.5,2) -- (3,2);
\draw[very thick] (-4.5,2) arc [radius=2, start angle=90, end angle=270];
\draw[very thick] (3,2) arc [radius=2, start angle=90, end angle=-90];

\draw[very thick] (-4.5,0) circle [radius=.4];
\draw[very thick] (-2,0) circle [radius=.4];
\draw[very thick] (0.5,0) circle [radius=.4];
\draw[very thick] (3,0) circle [radius=.4];

\draw[thick] (-4.1,0) to [out=30,in=180] (-3.25,0.3);
\draw[thick] (-3.25,0.3) to [out=0,in=150] (-2.4,0);
\draw[thick,dashed] (-4.1,0) to [out=-30,in=-150] (-2.4,0);

\draw[thick] (-1.6,0) to [out=30,in=180] (-0.75,0.3);
\draw[thick] (-0.75,0.3) to [out=0,in=150] (0.1,0);
\draw[thick, dashed] (-1.6,0) to [out=-30,in=-150] (0.1,0);

\draw[thick] (0.9,0) to [out=30,in=180] (1.75,0.3);
\draw[thick,dashed] (0.9,0) to [out=-30,in=180] (1.75,-0.3);

\draw[thick] (1.75,0.3) to [out=0,in=150] (2.6,0);
\draw[thick,dashed] (1.75,-0.3) to [out=0,in=210] (2.6,0);

\draw[thick] (-4.5,-0.4) to [out=180,in=90] (-5,-1.2);
\draw[thick] (-4.5,-2) to [out=180,in=-90] (-5,-1.2);
\draw[thick, dashed] (-4.5,-0.4) to [out=0,in=0] (-4.5,-2);
\draw[thick] (-2,-0.4) to [out=180,in=90] (-2.5,-1.2);
\draw[thick] (-2,-2) to [out=180,in=-90] (-2.5,-1.2);
\draw[thick, dashed] (-2,-0.4) to [out=0,in=0] (-2,-2);
\draw[thick] (0.5,-0.4) to [out=180,in=90] (0,-1.2);
\draw[thick] (0.5,-2) to [out=180,in=-90] (0,-1.2);
\draw[thick, dashed] (0.5,-0.4) to [out=0,in=0] (0.5,-2);
\draw[thick] (3,-0.4) to [out=180,in=90] (2.5,-1.2);
\draw[thick] (3,-2) to [out=180,in=-90] (2.5,-1.2);
\draw[thick, dashed] (3,-0.4) to [out=0,in=0] (3,-2);

\draw[thick] (-2,0.4) to [out=180,in=-90] (-2.5,1.2);
\draw[thick] (-2,2) to [out=180,in=90] (-2.5,1.2);
\draw[thick, dashed] (-2,0.4) to [out=0,in=0] (-2,2);
\draw[thick] (0.5,0.4) to [out=180,in=-90] (0,1.2);
\draw[thick] (0.5,2) to [out=180,in=90] (0,1.2);
\draw[thick, dashed] (0.5,0.4) to [out=0,in=0] (0.5,2);

\end{tikzpicture}
\end{center}
\caption{A system of disks of a handlebody}
\end{figure}
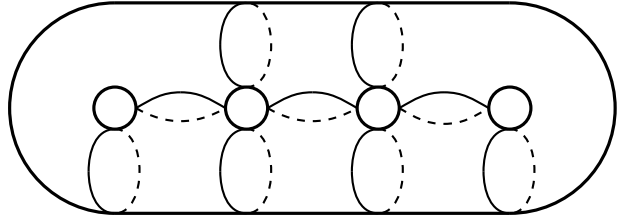

\begin{lema}[Lemma 2.2 in \cite{jes}]
Every handlebody has a system of disks.
\end{lema}

\begin{defi}[Definition 2.6 in \cite{jes}]
A collection of disks in a handlebody is \textit{minimal} if its complement is connected.
\end{defi}

\begin{teo}[Loop Theorem]
\label{cor_loop}
Let $M$ be an $3$-manifold with boundary, not necessarily compact or orientable,
and let $S\subset M$ be a $2$-sided surface. If the induced map $\pi_1 S\rightarrow \pi_1 M$ is not injective for some choice of base point in some component of $S,$ then there is a disk $D^2\subset M$ with $D^2\cap S=\partial D^2$ a nontrivial circle in $S.$
\end{teo}

\begin{prop}
\label{prop_minsys} Let $\Sigma_g$ be a standardly embedded surface in $\mathbf{S}^3,$ and $\mathcal{H}_g$ the inner handlebody.
Let $D_\gamma$ be an essential proper embedded disk in $\mathcal{H}_g$ with boundary a bounding simple closed curve $\gamma.$ Then there exists a minimal system of disks $\{D_{\beta_l}\}_l$ that does not intersect $D_\gamma.$
\end{prop}

\begin{proof} Observe that the disk $D_\gamma$ gives us, respectively, two compatible decompositions of $\Sigma_g$ and $\mathcal{H}_g$ as:
\begin{align}
& \Sigma_g= \Sigma_h \;\#_{\gamma} \;\Sigma_{g-h}, \label{eq_desc1}\\
& \mathcal{H}_g=V_r \;\#_{D_\gamma} \;V_{g-r}. \label{eq_desc2}
\end{align} 

We want to show that $V_r$ and $V_{g-r}$ are handlebodies, obviously with $\partial V_r= \Sigma_h \cup D_\gamma$ and $\partial V_{g-r}=\Sigma_{g-h}\cup D_\gamma.$ Take $S_h=\Sigma_h \cup D_\gamma$ and $S_{g-h}=\Sigma_{g-h}\cup D_\gamma.$

Observe that the decompositions \eqref{eq_desc1}, \eqref{eq_desc2} induce analog decompositions in homotopy.

Fixed a point $x_0$ of the curve $\gamma,$ using the Seifert van Kampen theorem we have that
$$\pi_1(\mathcal{H}_g,x_0)=\pi_1(V_r,x_0)\ast_{\pi_1(D_{\gamma},x_0)} \pi_1(V_{g-r},x_0)=\pi_1(V_r,x_0)\ast \pi_1(V_{g-r},x_0).$$
Since $\pi_1(\mathcal{H}_g,x_0)$ is free, by Corollary 2.9. in \cite{magnus1}, we have that $\pi_1(V_r,x_0)$ and $\pi_1(V_{g-r},x_0)$ are free subgroups of $\pi_1(\mathcal{H}_g,x_0).$
Computing the abelianization, we have that they are free groups of rank $r$ and $g-r$ respectively.

Assume that $r\geq h$ (if $r\leq h,$ an analogous proof holds switching $S_h$ and $S_{g-h}).$  

Now we prove that the inclusion $S_h\rightarrow V_r$ induces a surjective map $\pi_1(S_h)\rightarrow \pi_1(V_r)$ necessarily without trivial kernel.
Observe that we have the following commutative diagram:
$$
\xymatrix@C=10mm@R=13mm{\pi_1(\Sigma_g) \ar@{->}[r] \ar@{->}[d] & \pi_1(\mathcal{H}_g) \ar@{->}[d] \\
\pi_1(S_h) \ar@{->}[r] & \pi_1(V_r)
}
$$
where $\pi_1(\Sigma_g) \rightarrow \pi_1(\mathcal{H}_g)$ is clearly surjective, and $\pi_1(\mathcal{H}_g) \rightarrow \pi_1(V_r)$ is also surjective since every generator element of $\pi_1(V_r)$ is one of the generators of $\pi_1(\mathcal{H}_g).$
Then $\pi_1(S_h)\rightarrow \pi_1(V_r)$ is surjective.
Moreover this map can not be injective since $\pi_1(V_r)$ is free and $\pi_1(S_h)$ is not free and then it is not possible that $\pi_1(V_r) \cong \pi_1(S_h).$

By Theorem \eqref{cor_loop}, there exists a properly embedded disc $D_{\beta_1}$ in $V_r,$ with $\partial D_{\beta_1}=\beta_1$  and $\beta_1$ not nullhomotopic in $S_h$, so in particular $D_{\beta_1}$ is essential.

Notice that, without lose of generality, we can assume that  $\beta_1 \cap D_\gamma = \emptyset.$ Since if $\beta_1$ intersects $D_\gamma,$ then we can take an homotopy pushing $\beta_1$ out of $D_\gamma.$

Deleting a small open tubular neighbourhood $N(D_{\beta_1})$ of $D_{\beta_1}$ from $V_r,$ such that it does not intersect $D_\gamma,$ we obtain a $3$-manifold $V_{r-1}=V_r-N(D_{\beta_1})$ with boundary. The neighbourhood $N(D_{\beta_1})$ is an interval-bundle over $D_{\beta_1},$ and since $D_{\beta_1}$ is orientable, $N(D_{\beta_1})$ is a product $D_{\beta_1}\times (-\varepsilon,\varepsilon).$ Denote by $D^-_{\beta_1}, D^+_{\beta_1}$ the disks $D_{\beta_1}\times \{-\varepsilon\},D_{\beta_1}\times \{\varepsilon \}$ respectively. Then we have 3 embedded disks $D^-_{\beta_1}, D^+_{\beta_1},D_\gamma$ in
$$S_{h-1}=\partial(V_r-N(D_{\beta_1}))=(S_h\backslash Int(C))\cup D^+_{\beta_1} \cup D^-_{\beta_1},$$
where $C=\overline{N(D_{\beta_1})}.$

We prove that $S_{h-1}$ is a surface of genus $h-1.$

Recall that if $X$ is a topological space with $X=S'\cup S''$ where $S',S''$ are closed subsets then
$\chi(X)=\chi(S')+\chi(S'')-\chi(S'\cap S'')$ (where $\chi$ denotes the Euler's characteristic).

In our case we have that
\begin{align*}
\chi(S_h) & =\chi(S_h\backslash Int(C))+\chi(C)-\chi((S_h\backslash Int(C))\cap C) \\
 & =\chi(S_h\backslash Int(C))+0-\chi(\partial D^-_{\beta_1} \sqcup \partial D^+_{\beta_1}) \\
 & =\chi(S_h\backslash Int(C))-\chi(\partial D^-_{\beta_1})-\chi(D^+_{\beta_1}) \\
 & =\chi(S_h\backslash Int(C))-0-0=\chi(S_h\backslash Int(C)) \\
\end{align*}
and adding the disks $D^-_{\beta_1},$ $D^+_{\beta_1}$ to $S_h\backslash Int(C)$ we get that
\begin{align*}
\chi(S_{h-1}) & = \chi(S_h\backslash Int(C))+\chi(D^+_{\beta_1})+\chi(D^-_{\beta_1})-\chi(\partial D^+_{\beta_1})-\chi(\partial D^-_{\beta_1}) \\
 & = \chi(S_h\backslash Int(C))+1+1-0-0 \\
  & = 2-2h+2=2-2(h-1). 
\end{align*}
Then the new surface $S_{h-1}$ has genus $h-1.$ 

Moreover, let be $p^-,$ $p^+$  points of $D_{\beta_1}^-,$ $D_{\beta_1}^+$ respectively, and $\epsilon \in N(D_{\beta_1})$ an arc with end points $p^-,$ $p^+.$ We can do a deformation retract from $V_r$ to $(V_r-N(D_{\beta_1}))\cup \epsilon.$

Since $S_{h-1}$ is arc-connected there exists another arc $\epsilon'$ in $S_{h-1}$ with end points $p^-, p^+.$ So now we can do another deformation retract, from $(V_r-N(D_{\beta_1})) \cup \epsilon$ to 
$(V_r-N(D_{\beta_1})) \vee_{x_0} S^1$ sending $p^-$ to $p^+$ through $\epsilon'.$
Thus,
$$\pi_1(V_r-N(D_{\beta_1}))\cong \pi_1(((V_r-N(D_{\beta_1})) \bigvee_{p^+} S^1)-(S^1-\{p^+\})).$$
Now taking $V_{r-1}=(V_r-N(D_{\beta_1}))$
and using the Seifert van Kampen theorem we get that
$$\mathbb{F}_h=\pi_1(V_r,p^+)=\pi_1(V_{r-1} \vee_{p^+} S^1,p^+)=\pi_1(V_{r-1},p^+)\ast \pi_1(S^1,p^+)=\pi_1(V_{r-1},p^+)\ast \mathbb{F}_1.$$
Since $\pi_1(V_r,p^+)$ is free, by Corollary 2.9 in \cite{magnus1}, we have that $\pi_1(V_{r-1},p^+)$ is a free subgroup of $\pi_1(V_r,p^+).$
Computing the abelianization, we have that this subgroup is a free group of rank $h-1.$

Next, if we delete $N(D_{\beta_1})$ from $V_r$ we obtain an inclusion $S_{h-1}\hookrightarrow V_{r-1}.$

Repeating this argument $h$ times at the end we get a surface $S_0$ with genus zero, i.e. a smooth $2$-sphere $\mathbf{S}^2$ embedded in $\mathbf{S}^3.$
Applying Schöenflies theorem, we get that our embedded smooth $2$-sphere $\mathbf{S}^2$ is the boundary of two embedded smooth $3$-balls in $\mathbf{S}^3.$
Then $V_{r-h}$ is one of these $3$-balls, and so $\mathbb{F}_{r-h}=\pi_1(V_{r-h})=0,$ i.e. $h=r.$ Thus $V_h$ is a handlebody with system of disks given by $\{D_{\beta_i}\}_i=\{D_{\beta_1},\ldots, D_{\beta_h}\}$ that is minimal by construction.

Using the same argument for the subsurface $S_{g-h},$ we obtain a handlebody $V_{g-r}$ with boundary $S_{g-h}$, and a minimal system of discs $\{D_{{\beta}_j}\}_j$ for $V_{g-k}$, such that $\{D_{{\beta}_j}\}_j$ do not intersect $D_\gamma, D_\delta.$

Then we can view $\mathcal{H}_g$ as a boundary connected sum of $V_k$ and $V_{g-k}$ given by:

\begin{equation*}
\mathcal{H}_g= V_k\; \#_{D_\gamma} V_{g-k}:= \left(V_k\; \bigsqcup \;V_{g-k}\right)/D_\gamma
\end{equation*}

Hence the union of the two families $\{D_{{\beta}_i}\}_i$ and $\{D_{{\beta}_j}\}_j$ gives us a new system of disks $\{D_{{\beta''}_l}\}_l$ for $\mathcal{H}_g$ such that $\{D_{{\beta}_l}\}_l$ don't intersect $D_\gamma.$
And by construction $\{D_{{\beta}_l}\}_l,$ is a minimal disk system of $\mathcal{H}_g.$
\end{proof}

\section{Mapping class group}

Let $\Sigma_{g,b}$ denote an oriented, conneted surface of genus $g\geq 0$ with $b\geq 0$ disjoint open disks removed. Let $Homeo^+(\Sigma_{g,b},\partial \Sigma_{g,b})$ denote the group of orientation-preserving homeomorphisms of $\Sigma_{g,b}$ that restrict to the identity on $\partial \Sigma_{g,b}.$ We endow this group with the compact-open topology.

The \textit{mapping class group} of $\Sigma_{g,b},$ denoted $Mod(\Sigma_{g,b}),$ is the group
$$Mod(\Sigma_{g,b})=\pi_0(Homeo^+(\Sigma_{g,b},\partial \Sigma_{g,b})).$$
Equivalently, $Mod(\Sigma_{g,b})$ is the group of isotopy classes of elements of $Homeo^+(\Sigma_{g,b},\partial \Sigma_{g,b}),$ where isotopies are required to fix the boundary pointwise. If $Homeo_0(\Sigma_{g,b},\partial \Sigma_{g,b})$ denotes the connected component of the identity in $Homeo^+(\Sigma_{g,b},\partial \Sigma_{g,b}),$ then we can equivalently write
$$Mod(\Sigma_{g,b})=Homeo^+(\Sigma_{g,b},\partial \Sigma_{g,b})/Homeo_0(\Sigma_{g,b},\partial \Sigma_{g,b}).$$
The elements of $Mod(\Sigma_{g,b})$ are called \textit{mapping classes}.

There is not a unique definition of $Mod(\Sigma_{g,b}).$
For example, we could consider diffeomorphisms instead of homeomorphisms, or homotopy classes instead of isotopy classes. However, by Section 1.4 in \cite{farb}, these definitions would produce isomorphic groups, i.e. we have the following equivalent definitions of $Mod(\Sigma_{g,b}):$
\begin{align*}
Mod(\Sigma_{g,b})= & \pi_0(Homeo^+(\Sigma_{g,b},\partial \Sigma_{g,b})) \\
\approx & Homeo^+(\Sigma_{g,b},\partial \Sigma_{g,b})/\text{homotopy}\\
 \approx & \pi_0(\text{Diff}^+(\Sigma_{g,b},\partial \Sigma_{g,b})),
\end{align*}
where $\text{Diff}^+(\Sigma_{g,b},\partial \Sigma_{g,b})$ is the group of orientation-preserving diffeomorphisms of $\Sigma_{g,b}$ that are the identity on the boundary.

\subsection{Dehn twists}
There is a particular type of mapping class called a Dehn twist. Dehn twists are the simplest infinite-order mapping classes in the sense that the Dehn twists play the role for mapping class groups that elementary matrices play for linear groups.

\paragraph{Twist map in annulus.}
Consider the annulus $A=S^1\times [0,1].$ To orient $A$ we embed it in the $(\theta, r)$-plane via the map $(\theta, t)\mapsto (\theta, t+1)$ and take the orientation induced by the standard orientation of the plane.

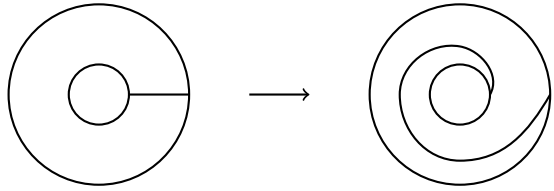
\begin{figure}[H]
\begin{center}
\begin{tikzpicture}[scale=.8]

\draw[thick] (0,0) circle [radius=.5];
\draw[thick] (0,0) circle [radius=1.5];
\draw[thick] (0.5,0) -- (1.5,0);

\draw[->,thick] (2.5,0) to [out=0,in=180] (3.5,0);

\draw[thick] (6,0) circle [radius=.5];
\draw[thick] (6,0) circle [radius=1.5];

\draw[thick] (6.5,0) to [out=60,in=-10] (6,0.8) to [out=170,in=90] (5,0) to [out=-90,in=180] (6,-1.1) to [out=0,in=-120] (7.5,0);

\end{tikzpicture}
\end{center}
\caption{Left twist}
\end{figure}

Let $T:A\rightarrow A$ be the twist map of $A$ given by the formula
$$T(\theta, t)=(\theta +2\pi t, t).$$
The map $T$ is an orientation-preserving homeomorphism that fixes $\partial A$ pointwise. Note that instead of using $\theta +2\pi t$ we could have used $\theta -2\pi t.$ Our choice is a left twist, while the other is a right twist.

\paragraph{Dehn twist in a general surface.}
Let $S$ be an arbitrary oriented surface and let $\alpha$ be a simple closed curve in $S.$ Let $N$ be a regular neighbourhood of $\alpha$ and choose an orientation preserving homeomorphism $\phi: A\rightarrow N.$ We obtain a homeomorphism $T_\alpha: S\rightarrow S,$ called a Dehn twist about $\alpha,$ as follows:
$$T_\alpha (x)=\left\{\begin{array}{ll}
\phi \circ T \circ \phi^{-1}(x) & \text{if } x\in N \\
x & \text{if } x\in S\backslash N.
\end{array}\right.$$
In other words, the instructions for $T_\alpha$ are ''perform the twist map $T$ on the annulus $N$ and fix every point outside of $N.$''

The Dehn twist $T_\alpha$ depends on the choice of $N$ and the homeomorphism $\phi.$ However, by the uniqueness of regular neighbourhoods, the isotopy class of $T_\alpha$ does not depend on either of these choices.

\paragraph{The action on simple closed curves via surgery}

We can understand $T_a$ by examining its action on the isotopy classes of simple closed curves on $S.$ If $b$ is an isotopy class with $i(a,b)=0,$ then $T_a(b)=b.$ In the case that $i(a,b)\neq 0$ the isotopy class $T_a(b)$ is determined by the following rule:
given particular representatives $\beta$ and $\alpha$ of $b$ and $a,$ respectiely, each segment of $\beta$ crossing $\alpha$ is replaced with a segment that turns left, follows $\alpha$ all the way around, and then turns right.
Analogously, we can understand $T_a^{-1}$ following the same rule described above switching ''left'' by ''right''.

The reason that we can distinguish left from right is that the map $\phi$ used in the definition of $T_a$ is taken to be orientation-preserving.

\begin{figure}[H]
\begin{center}
\begin{tikzpicture}[scale=.6]

\draw[thick] (-2,0) -- (2,0);
\draw[thick] (0,-2) -- (0,2);

\draw[->,thick] (3,0) to [out=0,in=180] (5,0);

\draw[thick] (7,0) arc [radius=1, start angle=-90, end angle=0];
\draw[thick] (8,-1) arc [radius=1, start angle=180, end angle=90];

\draw[thick] (6,0) -- (7,0);
\draw[thick] (10,0) -- (9,0);
\draw[thick] (8,2) -- (8,1);
\draw[thick] (8,-2) -- (8,-1);
\node [right] at (0.1,1.5) {$\alpha$};
\node [below] at (-2,-0.1) {$\beta$};
\end{tikzpicture}
\end{center}
\caption{Surgery}
\end{figure}

\paragraph{Properties of Dehn twists.}
Next we give some properties about Dehn twists. For the proofs of these properties see Section 3.3 in \cite{farb}.

Let $f\in Mod(S)$ and $a,$ $b$ isotopy classes of simple closed curves in $S$
\begin{enumerate}[i)]
\item $T_a=T_b \Longleftrightarrow a=b.$
\item $T_{f(a)}=fT_af^{-1}.$
\item $ f \text{ commutes with } T_a \Longleftrightarrow f(a)=a. $
\item $i(a,b)=0 \Longleftrightarrow T_a(b)=b \Longleftrightarrow T_aT_b=T_bT_a.$
\end{enumerate}

\paragraph{The Lantern relation}
The \textit{Lantern relation} is a relation in $Mod(S)$ between seven Dehn twists all lying on a subsurface of $S$ homeomorphic to $S_{0,4},$ a sphere with four boundary components.

\begin{prop}
Let $x,$ $y,$ $z,$ $b_1,$ $b_2,$ $b_3$ and $b_4$ be simple closed curves in a surface $S$ that are arranged as the curves shown in Figure \ref{lantern_rel}.

\begin{figure}[H]
\begin{center}
\begin{tikzpicture}[scale=.7]
\draw[very thick] (-2.8,0.8) to [out=20,in=-110] (-0.8,2.8);
\draw[very thick] (2.8,0.8) to [out=160,in=-70] (0.8,2.8);
\draw[very thick] (-2.8,-0.8) to [out=-20,in=110] (-0.8,-2.8);
\draw[very thick] (2.8,-0.8) to [out=-160,in=70] (0.8,-2.8);

\draw[thick] (-2.8,0.8) to [out=-60,in=60] (-2.8,-0.8) to [out=120,in=-120] (-2.8,0.8);
\draw[thick] (-0.8,2.8) to [out=30,in=150] (0.8,2.8) to [out=-150,in=-30] (-0.8,2.8);
\draw[thick] (2.8,0.8) to [out=-60,in=60] (2.8,-0.8);
\draw[thick, dashed] (2.8,-0.8) to [out=120,in=-120] (2.8,0.8);
\draw[thick, dashed] (-0.8,-2.8) to [out=30,in=150] (0.8,-2.8);
\draw[thick] (0.8,-2.8) to [out=-150,in=-30] (-0.8,-2.8);

\draw[thick, dashed] (-1.55,-1.55) to [out=60,in=210] (1.55,1.55);
\draw[thick] (-1.55,-1.55) to [out=30,in=240] (1.55,1.55);

\draw[thick] (1.55,-1.55) to [out=120,in=-30] (-1.55,1.55);
\draw[thick, dashed] (1.55,-1.55) to [out=150,in=-60] (-1.55,1.55);

\draw[thick] (-1.1,2.1) to [out=-80,in=80] (-1.1,-2.1);
\draw[thick] (1.1,2.1) to [out=-100,in=100] (1.1,-2.1);

\draw[thick, dashed] (-1.1,2.1) to [out=20,in=160] (1.1,2.1);
\draw[thick, dashed] (-1.1,-2.1) to [out=20,in=160] (1.1,-2.1);

\node [left] at (-3,0) {$b_3$};
\node [right] at (3,0) {$b_1$};
\node [above] at (0,3) {$b_2$};
\node [below] at (0,-3) {$b_4$};

\node [right] at (0.8,0) {$z$};
\node [below] at (-0.3,1.7) {$x$};
\node [above] at (-0.3,-1.7) {$y$};

\end{tikzpicture}
\end{center}
\caption{Lantern relation}
\label{lantern_rel}
\end{figure}
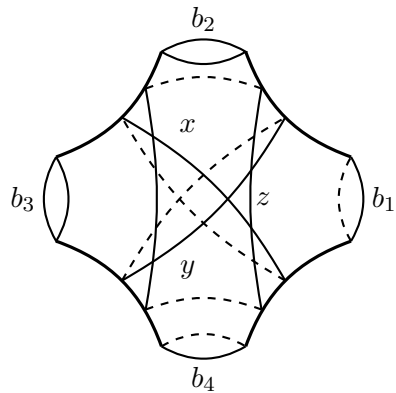

Precisely this means that there is an orientation-preserving embedding $S_{0,4}\hookrightarrow S$ and that each of the above seven curves is the image of the curve with the same name in Figure \ref{lantern_rel}.
In $Mod(S)$ we have the relation
$$T_xT_yT_z=T_{b_1}T_{b_2}T_{b_3}T_{b_4}.$$
\end{prop}

\subsection{Generators of the Mapping class group}

\begin{teo}[Theorem 4.1 in \cite{farb}]
For $g\geq 0,$ the mapping class group $Mod(S_g)$ is generated by finitely many Dehn twists about nonseparating simple closed curves.
\end{teo}
In particular we have the following result.

\begin{teo}[Theorem 4.14 in \cite{farb}]
Let $S$ be either a closed surface or a surface with one boundary component and genus $g\geq 3.$ Then the group $Mod(S)$ is generated by the Dehn twists about the $2g+1$ isotopy classes of nonseparating simple closed curves $c_0,\ldots c_{2g}$ of Figure \ref{Humphries_gen}.

\begin{figure}[H]
\begin{center}
\begin{tikzpicture}[scale=.7]
\draw[very thick] (-4.5,-2) -- (7,-2);
\draw[very thick] (-4.5,2) -- (7,2);
\draw[very thick] (-4.5,2) arc [radius=2, start angle=90, end angle=270];

\draw[very thick] (-4.5,0) circle [radius=.4];
\draw[very thick] (-2,0) circle [radius=.4];
\draw[very thick] (4.5,0) circle [radius=.4];
\draw[very thick] (0.5,0) circle [radius=.4];

\draw[->,thick] (-4.1,0) to [out=30,in=180] (-3.25,0.3);
\draw[thick] (-3.25,0.3) to [out=0,in=150] (-2.4,0);
\draw[thick,dashed] (-4.1,0) to [out=-30,in=-150] (-2.4,0);

\node [above] at (-3.25,0.3) {$c_3$};

\draw[->,thick] (-1.6,0) to [out=30,in=180] (-0.75,0.3);
\draw[thick] (-0.75,0.3) to [out=0,in=150] (0.1,0);
\draw[thick, dashed] (-1.6,0) to [out=-30,in=-150] (0.1,0);

\node [above] at (-0.75,0.3) {$c_5$};

\draw[->,thick] (0.9,0) to [out=30,in=180] (1.75,0.3);
\draw[thick,dashed] (0.9,0) to [out=-30,in=180] (1.75,-0.3);

\node [above] at (1.75,0.3) {$c_7$};

\draw[thick, dotted] (2,0) -- (3,0);

\draw[thick] (3.25,0.3) to [out=0,in=150] (4.1,0);
\draw[thick,dashed] (3.25,-0.3) to [out=0,in=210] (4.1,0);

\node [above] at (3.1,0.3) {$c_{2g-1}$};

\draw[->,thick] (-0.3,0) to [out=90,in=180] (0.5,0.8);
\draw[thick] (1.3,0) to [out=90,in=0] (0.5,0.8);
\draw[thick] (-0.3,0) to [out=-90,in=180] (0.5,-0.8) to [out=0,in=-90] (1.3,0);

\draw[->, thick] (-5.3,0) to [out=90,in=180] (-4.5,0.8);
\draw[thick] (-3.7,0) to [out=90,in=0] (-4.5,0.8);
\draw[thick] (-5.3,0) to [out=-90,in=180] (-4.5,-0.8) to [out=0,in=-90] (-3.7,0);

\draw[->,thick] (-2.8,0) to [out=90,in=180] (-2,0.8);
\draw[thick] (-1.2,0) to [out=90,in=0] (-2,0.8);
\draw[thick] (-2.8,0) to [out=-90,in=180] (-2,-0.8) to [out=0,in=-90] (-1.2,0);

\draw[thick] (5.3,0) to [out=90,in=0] (4.5,0.8);
\draw[->,thick] (3.7,0) to [out=90,in=180] (4.5,0.8);
\draw[thick] (5.3,0) to [out=-90,in=0] (4.5,-0.8) to [out=180,in=-90] (3.7,0);

\draw[thick] (-4.5,-0.4) to [out=180,in=90] (-5,-1.2);
\draw[->,thick] (-4.5,-2) to [out=180,in=-90] (-5,-1.2);
\draw[thick, dashed] (-4.5,-0.4) to [out=0,in=0] (-4.5,-2);
\draw[thick] (-2,-0.4) to [out=180,in=90] (-2.5,-1.2);
\draw[->,thick] (-2,-2) to [out=180,in=-90] (-2.5,-1.2);
\draw[thick, dashed] (-2,-0.4) to [out=0,in=0] (-2,-2);

\node [left] at (-5,-1.2) {$c_1$};
\node [above] at (-4.5,0.8) {$c_2$};
\node [left] at (-2.5,-1.2) {$c_0$};
\node [above] at (-2,0.8) {$c_4$};
\node [above] at (0.5,0.8) {$c_6$};
\node [above] at (4.5,0.8) {$c_{2g}$};

\draw[thick,pattern=north west lines] (7,-2) to [out=130,in=-130] (7,2) to [out=-50,in=50] (7,-2);

\end{tikzpicture}
\end{center}
\caption{Humphries generators}
\label{Humphries_gen}
\end{figure}
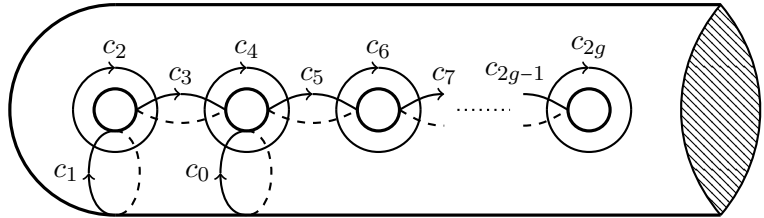

\end{teo}

\subsection{The Symplectic representation}
We first give the definition of symplectic matrix as well as the Symplectic group.
\begin{defi}[Section 1 in \cite{newman}]
Consider the matrix $\Omega\in M_{2g\times 2g}(\mathbb{Z})$ given by
$$\Omega=\left(\begin{matrix}
0 & Id_g \\
-Id_g & 0
\end{matrix}\right).$$
If $M$ is a matrix satisfying $M\Omega M^t = \Omega,$ then $M$ will be said to be \textit{symplectic.}
We define the Symplectic group as
$$Sp_{2g}(\mathbb{Z})=\{M\in M_{2g\times 2g}(\mathbb{Z})\mid M\Omega M^t \equiv \Omega \}.$$
From the definitions, it is easy to verify that a matrix
$$M=\left(\begin{matrix}
A & B \\
C & D
\end{matrix}\right)$$
is symplectic if an only if
$$AD^t-BC^t=Id_g,\quad AB^t=BA^t,\quad CD^t=DC^t.$$
\end{defi}
Next we define the Symplectic representation.
Let $S_g$ be a closed surface of genus $g$ and $Mod(S_g)$ its mapping class group. Recall that by definition, $Mod(S_g)=\pi_0(Homeo^+(S_g)).$ Therefore we can view every mapping class of $\phi\in Mod(S_g)$ as an element of $Aut(\Gamma),$ where $\Gamma=\pi_1(S_g)$ (the fundamental group of $S_g).$
Moreover it is well known that the commutator subgroup $[\Gamma,\Gamma]\subset \Gamma$ is a characteristic subgroup, in other words, every element of $Aut(\Gamma)$ preserves $[\Gamma,\Gamma].$

Then taking $\phi\in Aut(\Gamma)$ composed with the abelianization of $\Gamma$ we get an element of $Aut(H_1(S_g))$ and as a consequence every element $\phi\in Mod(S_g)$ induces an automorphism
$$\phi_*:H_1(S_g;\mathbb{Z})\rightarrow H_1(S_g;\mathbb{Z}).$$
More precisely, since $H_1(S_g;\mathbb{Z})\cong \mathbb{Z}^{2g},$ we get that $\phi_*\in Aut(\mathbb{Z}^{2g})=GL_{2g}(\mathbb{Z}).$
Therefore we have a linear representation
$$\Psi:Mod(S_g)\rightarrow GL_{2g}(\mathbb{Z}).$$
Moreover, since the algebraic intersection number
$\widehat{i}:H_1(S_g;\mathbb{Z})\times H_1(S_g;\mathbb{Z})\rightarrow \mathbb{Z}$
on $H_1(S_g;\mathbb{Z})$ endows the vector space $H_1(S_g;\mathbb{Z})$ with a symplectic structure and this symplectic structure is preserved by the image of $\Psi,$ then $\Psi$ is a representation
$$\Psi: Mod(S_g)\rightarrow Sp_{2g}(\mathbb{Z}).$$
The homomorphism $\Psi$ is called the \textit{symplectic representation} of $Mod(S_g).$

\paragraph{The action of a Dehn twist on homology}
To understand $\Psi,$ first of all we need to understand  what it does to Dehn twists.
We have the following formula.

\begin{prop}
\label{prop_act_twist}
Let $a$ and $b$ be isotopy classes of oriented simple closed curves in $S_g.$ For any $k\geq 0,$ we have
$$\Psi(T_b^k)([a])=[a]+ k\cdot \widehat{i}(a,b)[b].$$
\end{prop}

In particular from Proposition \eqref{prop_act_twist} we get that
$\Psi(T_a)=\Psi(T_{a'})\Longleftrightarrow [a]=[a'],$ and also that if $[a]=0,$ then $\Psi(T_a)$ is the identity.

\begin{teo}[Theorem 6.4 in \cite{farb}]
\label{teo_rep_symplec}
The symplectic representation is surjective.
\end{teo}

\subsection{Heegaard splittings of $3$-manifolds}

Let $\Sigma_g$ an oriented surface of genus $g$ standardly embedded in the $3$-sphere $\mathbf{S}^3$. Denote by $\Sigma_{g,1}$ the complement of the interior of a small disc embedded in $\Sigma_g.$ We fix a base point $x_0$ on the boundary of $\Sigma_{g,1}.$ 

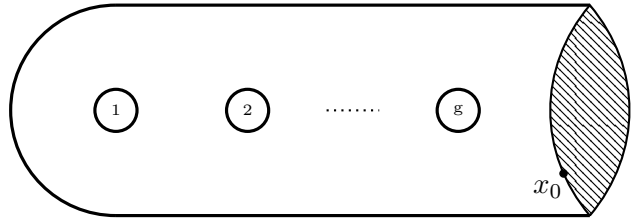
\begin{figure}[H]
\begin{center}
\begin{tikzpicture}[scale=.7]
\draw[very thick] (-2,-2) -- (7,-2);
\draw[very thick] (-2,2) -- (7,2);
\draw[very thick] (-2,2) arc [radius=2, start angle=90, end angle=270];

\draw[very thick] (-2,0) circle [radius=.4];
\draw[very thick] (4.5,0) circle [radius=.4];
\draw[very thick] (0.5,0) circle [radius=.4];

\draw[thick, dotted] (2,0) -- (3,0);

\node at (-2,0) {\tiny{1}};
\node at (0.5,0) {\tiny{2}};
\node at (4.5,0) {\tiny{g}};

\draw[thick,pattern=north west lines] (7,-2) to [out=130,in=-130] (7,2) to [out=-50,in=50] (7,-2);

\draw[fill] (6.5,-1.2) circle [radius=0.07];
\node at (6.2,-1.5) {$x_0$};

\end{tikzpicture}
\end{center}
\caption{Standardly embedded $\Sigma_{g,1}$ in $\mathbf{S}^3$}
\label{fig_stand_embded}
\end{figure}

Throughout this thesis we denote by $\mathcal{M}_{g,1}$ the following mapping class group
$$\mathcal{M}_{g,1}=\pi_0(\text{Diff}^+(\Sigma_{g,1}, \partial \Sigma_{g,1})).$$
\paragraph{Handlebody subgroups.} Recall that our surface is standarly embedded in the oriented $3$-sphere $\mathbf{S}^3.$ As such it determines two embedded handlebodies $\mathbf{S}^3=\mathcal{H}_g \cup -\mathcal{H}_g.$ By the inner handlebody $\mathcal{H}_g$ we will mean the one that is visible in Figure \ref{fig_stand_embded} and by the outer handelbody $-\mathcal{H}_{g}$ we will mean the complementary handlebody. They are naturally pointed by $x_0 \in \mathcal{H}_g \cap -\mathcal{H}_g.$
From these we get the following three natural subgroups of $\mathcal{M}_{g,1}:$

\begin{itemize}
\item $\mathcal{A}_{g,1}=$ subgroup of restrictions of diffeomorphisms of the outer handlebody $-\mathcal{H}_g,$
\item $\mathcal{B}_{g,1}=$ subgroup of restrictions of diffeomorphisms of the inner handlebody $\mathcal{H}_g,$
\item $\mathcal{AB}_{g,1}=$ subgroup of restriction of diffeomorphisms of the $3$-sphere $\mathbf{S}^3.$
\end{itemize}
Moreover, in \cite{wald}, F. Waldhausen proved that in fact, $\mathcal{AB}_{g,1}=\mathcal{A}_{g,1}\cap\mathcal{B}_{g,1}.$

\paragraph{The stabilization map on the mapping class group.}
Define the stabilization map on the mapping class group $\mathcal{M}_{g,1}$ as follows.
Glue one of the boundary components of a two-holed torus on the boundary of $\Sigma_{g,1}$ to get $\Sigma_{g+1,1}.$ Extending an element of $\mathcal{M}_{g,1}$ by the identity over the torus yields an injective homomorphism $\mathcal{M}_{g,1}\hookrightarrow \mathcal{M}_{g+1,1},$ this is the stabilization map. This map is compatible with the action on homology and is compatible with the definition of the subgroups $\mathcal{A}_{g,1}$ and $\mathcal{B}_{g,1}.$

\paragraph{Heggaard splittings of $3$-manifolds}
\begin{defi}
A \textit{Heegaard splitting} for a $3$-manifold $M$ is an ordered triple $(\Sigma, H_1,H_2)$ where $\Sigma$ is a closed surface embedded in $M$ and $H_1$ and $H_2$ are handlebodies embedded in $M$ such that $\partial H_1=\Sigma=\partial H_2=H_1\cap H_2$ and $H_1\cup H_2=M.$ The surface $\Sigma$ is called a \textit{Heegaard surface.}
\end{defi}

\begin{defi}
Let $H_1$ and $H_2$ be handlebodies of the same genus and $h:\partial H_1\rightarrow \partial H_2$ a homeomorphism. Then the quotient of the disjoint union of the two handlebodies by the relation $x\sim y$ if $x\in \partial H_1,$ $y\in \partial H_2$ and $h(x)=y,$ is a closed $3$-manifold. We denote by $M=H_1\cup_h H_2$ such manifold.
\end{defi}

Consider the following equivalence relation on $\mathcal{M}_{g,1}:$
\begin{equation}
\label{eq_rel1}
\phi \sim \psi \quad\Leftrightarrow \quad \exists \zeta_a \in \mathcal{A}_{g,1}\;\exists \zeta_b \in \mathcal{B}_{g,1} \quad \text{such that} \quad \zeta_a \phi \zeta_b=\psi
\end{equation}
The equivalence relation $\eqref{eq_rel1}$ is compatible with the stabilization map.

Choose a map $\iota_g\in \mathcal{M}_{g,1}$ such that $S^3=\mathcal{H}_g\cup_{\iota_g}-\mathcal{H}_g.$
It is also possible to choose the map $\iota_g$ to be compatible with the stabilization map ${\iota_{g+1}}_{\mid_{\sigma_g}}=\iota_g.$

Denote by $\mathcal{V}^3$ the set of oriented diffeomorphism classes of compact, closed and oriented smooth 3-manifolds.
\begin{teo}[J. Singer \cite{singer}]
\label{bij_MCG_3man}
The following map is well defined and is bijective:
\begin{align*}
\lim_{g\to \infty}\mathcal{A}_{g,1}\backslash \mathcal{M}_{g,1}/\mathcal{B}_{g,1} & \longrightarrow \mathcal{V}^3, \\
\phi & \longmapsto S^3_\phi=\mathcal{H}_g \cup_{\iota_g\phi} -\mathcal{H}_g.
\end{align*}
\end{teo}

\subsection{The Torelli group}
The Torelli group $\mathcal{T}_{g,1}$ is the normal subgroup of the mapping class group $\mathcal{M}_{g,1}$ of those elements of $\mathcal{M}_{g,1}$ that act trivially on $H_1(\Sigma_{g,1};\mathbb{Z}).$ In other words, $\mathcal{T}_{g,1}$ is characterized by the following short exact sequence:
$$
\xymatrix@C=10mm@R=13mm{1 \ar@{->}[r] & \mathcal{T}_{g,1} \ar@{->}[r] & \mathcal{M}_{g,1} \ar@{->}[r]^-{\Psi} & Sp_{2g}(\mathbb{Z}) \ar@{->}[r] & 1 }
$$

\paragraph{Example of elements of the Torelli group}
\begin{enumerate}[i)]
\item \textbf{Dehn twists about separating curves.}
Each Dehn twist about a separating simple closed curve $\gamma$ in $\Sigma_{g,1}$ is an element of $\mathcal{T}_{g,1}.$ This is because there  exists a basis for $H_1(\Sigma_{g,1};\mathbb{Z})$ where each element is represented by an oriented simple closed curve disjoint from $\gamma.$
Since $T_\gamma$ fixes each of these curves it fixes the corresponding homology classes and is hence an element of $\mathcal{T}_{g,1}.$

The group generated by Dehn twists about separating simple closed curves in $\Sigma_{g,1}$ is denoted by $\mathcal{K}_{g,1}$ and is known as the Johnson subgroup. In the following picture we give an example of a separating simple closed curve $\gamma.$

\begin{figure}[H]
\begin{center}
\begin{tikzpicture}[scale=.5]
\draw[very thick] (-4.5,-2) -- (0,-2);
\draw[very thick] (-4.5,2) -- (0,2);
\draw[very thick] (-4.5,2) arc [radius=2, start angle=90, end angle=270];
\draw[very thick] (-4.5,0) circle [radius=.4];
\draw[very thick] (-2,0) circle [radius=.4];

\draw[dashed,thick] (-3.25,2) to [out=-50,in=90] (-2.7,0) to [out=-90,in=50] (-3.25,-2);

\draw[thick] (-3.25,-2) to [out=130,in=-90] (-3.8,0);

\draw[->,thick] (-3.25,2) to [out=230,in=90] (-3.8,0);

\node [above] at (-4,-1.8) {$\gamma$};

\draw[thick,pattern=north west lines] (0,-2) to [out=130,in=-130] (0,2) to [out=-50,in=50] (0,-2);

\end{tikzpicture}
\end{center}
\caption{A separating simple closed curve}
\end{figure}
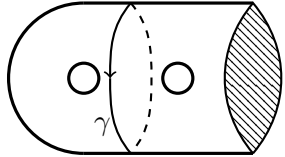

\item \textbf{Bounding pair maps.}
A \textit{bounding pair} in a surface is a pair of disjoint, homologous, nonseparating simple closed curves. A \textit{bounding pair map}, (abbreviated, BP-\textit{map}) is a mapping class of the form
$$T_aT_b^{-1}$$
where $a$ and $b$ form a bounding pair. Since $a$ and $b$ are homologous, by Proposition 6.3 in \cite{farb} the images of $T_a$ and $T_b$ in $Sp_{2g}(\mathbb{Z})$ are equal. Thus any bounding pair map is an element of $\mathcal{T}_{g,1}.$
In the following picture we give two examples of a bounding pairs $\{\eta,\eta'\}$ and $\{\beta,\beta'\}.$
\begin{figure}[H]
\begin{center}
\begin{tikzpicture}[scale=.5]
\draw[very thick] (-4.5,-2) -- (1,-2);
\draw[very thick] (-4.5,2) -- (1,2);
\draw[very thick] (-4.5,2) arc [radius=2, start angle=90, end angle=270];
\draw[very thick] (-4.5,0) circle [radius=.4];
\draw[very thick] (-2,0) circle [radius=.4];

\draw[->,thick] (-2.8,0) to [out=90,in=180] (-2,0.8);
\draw[thick] (-1.2,0) to [out=90,in=0] (-2,0.8);
\draw[thick] (-2.8,0) to [out=-90,in=180] (-2,-0.8) to [out=0,in=-90] (-1.2,0);

\draw[->,thick] (-2,-2) to [out=180,in=-90] (-3.5,0);
\draw[thick] (-3.5,0) to [out=90,in=180] (-2,2);
\draw[thick,dashed] (-2,2) to [out=0,in=0] (-2,-2);

\draw[thick,pattern=north west lines] (1,-2) to [out=130,in=-130] (1,2) to [out=-50,in=50] (1,-2);

\node [above] at (-2,0.8) {$\eta'$};
\node [left] at (-3.4,-0.6) {$\eta$};

\end{tikzpicture}
\qquad
\begin{tikzpicture}[scale=.5]
\draw[very thick] (-4.5,-2) -- (1,-2);
\draw[very thick] (-4.5,2) -- (1,2);
\draw[very thick] (-4.5,2) arc [radius=2, start angle=90, end angle=270];
\draw[very thick] (-4.5,0) circle [radius=.4];
\draw[very thick] (-2,0) circle [radius=.4];

\draw[thick, dashed] (-2,0.4) to [out=0,in=0] (-2,2);
\draw[thick, dashed] (-2,-0.4) to [out=0,in=0] (-2,-2);

\draw[thick] (-2,-0.4) to [out=180,in=90] (-2.5,-1.2);
\draw[<-,thick](-2.5,-1.2) to [out=-90,in=180] (-2,-2);
\draw[thick] (-2,0.4) to [out=180,in=-90] (-2.5,1.2);
\draw[<-,thick](-2.5,1.2) to [out=90,in=180] (-2,2);

\draw[thick,pattern=north west lines] (1,-2) to [out=130,in=-130] (1,2) to [out=-50,in=50] (1,-2);

\node [left] at (-2.5,1.2) {$\beta$};
\node [left] at (-2.5,-1.2) {$\beta'$};

\end{tikzpicture}
\end{center}
\caption{Boundig pairs}
\end{figure}

\end{enumerate}

\subsection{Heegaard splittings of homology $3$-spheres}

\begin{defi}
A $3$-manifold $X$ is an integral homology $3$-sphere if
$$H_*(X;\mathbb{Z}) \cong H_*(\mathbf{S}^3;\mathbb{Z}).$$
Throughout this thesis we will refer to such manifolds simply as homology $3$-spheres.
\end{defi}

\begin{exe}[Section 9.D. in \cite{rolf}]
The most important example of a homology $3$-sphere is the \textit{Poincaré sphere} which is the $3$-manifold given by a full dodecahedron with opposite faces identified with a twist of $\frac{\pi}{5}.$ The importance of this $3$-manifold comes from the fact that it was the first example of an homology $3$-sphere which is not an homotopy $3$-sphere, i.e. its homotopy groups do not coincide with the homotopy groups of the $3$-sphere.
\end{exe}

Denote by $\mathcal{S}^3\subset\mathcal{V}^3$ the subset of integral homology $3$-spheres.
\begin{teo}[S. Morita \cite{mor}]
The following map is well defined and is bijective:
\begin{align*}
\lim_{g\to \infty}\mathcal{A}_{g,1}\backslash\mathcal{T}_{g,1}/\mathcal{B}_{g,1} & \longrightarrow \mathcal{S}^3, \\
\phi & \longmapsto S^3_\phi=\mathcal{H}_g \cup_{\iota_g\phi} -\mathcal{H}_g.
\end{align*}
\end{teo}
Then we can get any element of  $\mathcal{S}^3$ as $\mathcal{H}_g\cup_{\iota_g\phi}-\mathcal{H}_g$ for a suitable $g$ and $\phi\in \mathcal{T}_{g,1}.$

From the group-theoretical point of view, the induced equivalence relation on $\mathcal{T}_{g,1},$ which is given by:
\begin{equation}
\phi \sim \psi \quad\Leftrightarrow \quad \exists \zeta_a \in \mathcal{A}_{g,1}\;\exists \zeta_b \in \mathcal{B}_{g,1} \quad \text{such that} \quad \zeta_a \phi \zeta_b=\psi,
\end{equation}
is a quite unsatisfactory, since it looks like, but is not, a double coset relation on $\mathcal{T}_{g,1}.$ However, by Lemma \eqref{lema_eqiv_coset_torelli}, due to W. Pitsch \cite{pitsch}, we know that this relation is the composite of a double coset relation in $\mathcal{T}_{g,1}$ and a conjugancy-induced equivalence relation.

\begin{defi}
We define the following subgroups of $\mathcal{T}_{g,1}:$
$$\mathcal{TA}_{g,1}=\mathcal{T}_{g,1}\cap \mathcal{A}_{g,1},\quad\mathcal{TB}_{g,1}=\mathcal{T}_{g,1}\cap \mathcal{B}_{g,1},\quad\mathcal{TAB}_{g,1}=\mathcal{T}_{g,1}\cap \mathcal{AB}_{g,1}.$$
\end{defi}

\begin{lema}[W. Pitsch \cite{pitsch}]
\label{lema_eqiv_coset_torelli}
Two maps $\phi, \psi \in \mathcal{T}_{g,1}$ are equivalent if and only if there exists a map $\mu \in \mathcal{AB}_{g,1}$ and two maps $\xi_a\in \mathcal{TA}_{g,1}$ and $\xi_b\in\mathcal{TB}_{g,1}$ such that $\phi=\mu \xi_a\psi \xi_b\mu^{-1}.$
\end{lema}

To summarize, we have the following bijection:
\begin{equation}
\label{bij_ZHS_torelli}
\begin{aligned}
\lim_{g\to \infty}\left(\mathcal{TA}_{g,1}\backslash\mathcal{T}_{g,1}/\mathcal{TB}_{g,1}\right)_{\mathcal{AB}_{g,1}} & \longrightarrow \mathcal{S}^3, \\
\phi & \longmapsto S^3_\phi=\mathcal{H}_g \cup_{\iota_g\phi} -\mathcal{H}_g.
\end{aligned}
\end{equation}

\subsection{Homology actions}
\label{homol_homoto_actions}

The isotopy class of the curves $\alpha_i,\beta_i,$ for $1\geq i \geq g,$ are free generators of the free group $\pi_1(\Sigma_{g,1},x_0),$ (see figure below). The first homology group of the surface $H:=H_1(\Sigma_g,\mathbb{Z})$ is endowed via Poincaré duality with a natural symplectic intersection form $\omega: \bigwedge^2 H \rightarrow \mathbb{Z}.$ The homology classes $\{A_i,B_i\}_i$ of the above curves freely generate the abelian group $H \simeq \mathbb{Z}^{2g}$ and define two transverse Lagrangians $A$ and $B$ in $H.$
Throughout this thesis we fix a basis of $H$ given by $\langle A_1,\ldots A_g,B_1,\ldots B_g \rangle$ and a symplectic intersection form $\omega$ given by $\omega(B_i,A_i)=1=-\omega(A_i,B_i)$ for every $i$ and zero otherwise, i.e. $Sp \;\omega\cong Sp_{2g}(\mathbb{Z}).$
\[
\begin{tikzpicture}[scale=.8]
\draw[very thick] (-6.5,-2) -- (2.5,-2);
\draw[very thick] (-6.5,2) -- (2.5,2);
\draw[very thick] (-6.5,2) arc [radius=2, start angle=90, end angle=270];
\draw[very thick] (2.5,2) arc [radius=2, start angle=90, end angle=-90];
\draw[very thick] (-6.5,0) circle [radius=.4];
\draw[very thick] (-2,0) circle [radius=.4];
\draw[very thick] (2.5,0) circle [radius=.4];

\draw[thick] (-2,-0.4) to [out=180,in=160] (-2,-0.8);
\draw[->,thick](0.5,-1.2) to [out=180,in=-20] (-2,-0.8);
\draw[thick, dashed] (-2,-0.4) to [out=0,in=90] (-1.5,-1.3);
\draw[thick, dashed] (-1.5,-1.3) to [out=-90,in=0] (-2,-2);
\draw[thick] (-2,-2) to [out=180,in=200] (-2,-1.6)to [out=20,in=180] (0.5,-1.2);

\draw[->,thick] (0.5,-1.2) to [out=180,in=-20] (-2,-0.6) to [out=160,in=-90] (-2.6,0) to [out=90,in=180] (-2,0.6);
\draw[thick] (-2,0.6) to [out=-10,in=170] (0.5,-1.2);

\draw[thick,pattern=north east lines] (0.5,-1.5) circle [radius=.3];
\draw[fill] (0.5,-1.2) circle [radius=0.07];

\node [above] at (0.5,-1.2) {$x_0$};
\node [above] at (-2.3,0.6) {$\alpha_i$};
\node [above] at (-2.3,-1.7) {$\beta_i$};

\node [below] at (-6.5,-0.5) {$1$};
\node [below] at (2.5,-0.5) {$g$};
\end{tikzpicture}
\]

According to Griffith \cite{grif}, the subgroup $\mathcal{B}_{g,1}$ (resp. $\mathcal{A}_{g,1})$ is characterised by the fact that its action on $\pi_1(\Sigma_g,x_0)$ preserves the normal subgroup generated by the curves $\beta_1,\ldots ,\beta_g$ (resp. $\alpha_1,\ldots ,\alpha_g).$ As a consequence the action on homology of $\mathcal{B}_{g,1}$ (resp. $\mathcal{A}_{g,1})$ preserves the Lagrangian $B$ (resp. $A).$

\begin{nota}
Throughout this thesis, given two groups $N,$ $H$ and an action of $H$ on $N,$ we will denote by $H\ltimes N$ the semidirect product of $H$ and $N.$
\end{nota}

If one writes the matrices of the symplectic group $Sp_{2g}(\mathbb{Z})$ as blocks according to the decomposition $H=A\oplus B,$ then the image of $\mathcal{B}_{g,1}\rightarrow Sp_{2g}(\mathbb{Z})$ is contained in the subgroup $Sp_{2g}^B(\mathbb{Z})$ of matrices of the form:
$$\left(\begin{matrix}
G_1 & 0 \\
M & G_2
\end{matrix}\right).$$

Such matrices are symplectic if and only if $G_2= {}^tG^{-1}_1$ and ${}^tG^{-1}_1M$ is symmetric. As a consequence we have an isomorphism:
\begin{equation}
\label{sym_mat_B}
\begin{aligned}
\phi^B: Sp_{2g}^B(\mathbb{Z}) & \longrightarrow GL_g(\mathbb{Z}) \ltimes_B S_g(\mathbb{Z}), \\
\left(\begin{matrix}
G & 0 \\
M & {}^tG^{-1}
\end{matrix}\right) & \longmapsto (G,{}^tGM).
\end{aligned}
\end{equation}
Here $S_g(\mathbb{Z})$ denotes the symmetric group of $g\times g$ matrices over the integers; the composition on the semi-direct product is given by the rule
$$(G,S)(H,T)=(GH,{}^tHSH+T).$$

Analogously, the image of $\mathcal{A}_{g,1}\rightarrow Sp_{2g}(\mathbb{Z})$ is contained in the subgroup $Sp_{2g}^A(\mathbb{Z})$ of matrices of the form:
$$\left(\begin{matrix}
H_1 & N \\
0 & H_2
\end{matrix}\right).$$

Such matrices are symplectic if and only if $H_2={}^tH_1^{-1}$ and ${}^tH_2N$ is symmetric. Similarly, we have an isomorphism:
\begin{equation}
\label{sym_mat_A}
\begin{aligned}
\phi^A: Sp_{2g}^A(\mathbb{Z}) & \longrightarrow GL_g(\mathbb{Z}) \ltimes_A S_g(\mathbb{Z}), \\
\left(\begin{matrix}
H & N \\
0 & {}^tH^{-1}
\end{matrix}\right) & \longmapsto (H,H^{-1}N).
\end{aligned}
\end{equation}
Here the composition on the semi-direct product is given by the rule
$$(G,S)(H,T)=(GH,{}^tH^{-1}SH^{-1}+T).$$

Finally, the image of $\mathcal{AB}_{g,1}\rightarrow Sp_{2g}(\mathbb{Z})$ is contained in the subgroup $Sp_{2g}^{AB}(\mathbb{Z})$ of matrices of the form:
$$\left(\begin{matrix}
G_1 & 0 \\
0 & G_2
\end{matrix}\right).$$

Such matrices are symplectic if and only if $G_2= {}^tG^{-1}_1.$ As a consequence have an isomorphism:
\begin{equation}
\label{sym_mat_AB}
\begin{aligned}
\phi^{AB}: Sp_{2g}^{AB}(\mathbb{Z}) & \longrightarrow GL_g(\mathbb{Z}), \\
\left(\begin{matrix}
G & 0 \\
0 & {}^tG^{-1}
\end{matrix}\right) & \longmapsto G.
\end{aligned}
\end{equation}
Checking on generators of $\mathcal{B}_{g,1}$ (see \cite{suzuki}) we get:

\begin{lema}
\label{lem_B}
There is a short exact sequence of groups:
$$
\xymatrix@C=10mm@R=13mm{1 \ar@{->}[r] & \mathcal{TB}_{g,1} \ar@{->}[r] & \mathcal{B}_{g,1} \ar@{->}[r]^-{\phi^B \circ \Psi} & GL_g(\mathbb{Z})\ltimes S_g(\mathbb{Z}) \ar@{->}[r] & 1 }
$$
\end{lema}
An analogous statement holds for $\mathcal{A}_{g,1}$ replacing the lagrangian $B$ by $A.$
Similarly, checking on generators of $\mathcal{AB}_{g,1}$ we have the following result
\begin{lema}
\label{lem_AB}
There is a short exact sequence of groups:
$$
\xymatrix@C=10mm@R=13mm{1 \ar@{->}[r] & \mathcal{TAB}_{g,1} \ar@{->}[r] & \mathcal{AB}_{g,1} \ar@{->}[r] & GL_g(\mathbb{Z}) \ar@{->}[r] & 1 }
$$
\end{lema}

\subsection{The Johnson homomorphism}
Computing the action of the Torelli group on the second nilpotent quotient of $\pi_1(\Sigma_{g,1},x_0)$ Johnson defines a morphism of groups known as the first Johnson homomorphism:
$$\tau:\mathcal{T}_{g,1}\longrightarrow \bigwedge^3 H.$$
Notice that the mapping class group $\mathcal{M}_{g,1}$ acts naturally by conjugation on $\mathcal{T}_{g,1}$ and acts also on $\bigwedge^3 H$ via its natural action on homology. In \cite{jon_1}, \cite{jon_2}, \cite{jon_3} Johnson proves that

\begin{prop}
\label{prop_johnson}
The map $\tau$ is $\mathcal{M}_{g,1}$-equivariant with respect to the above actions. Up to a finite dimensional $\mathbb{Z}/2$-vector space $\bigwedge^3 H$ is the abelianization of the Torelli group: any homomorphism $\mathcal{T}_{g,1}\rightarrow A$ where $A$ is an abelian group without $2$-torsion factors uniquely through $\tau.$
\end{prop}

\chapter{Trivial cocycles and invariants on the Torelli group}
\label{chapter: trivocicles torelli}
In \cite{pitsch}, W. Pitsch gave a tool to construct invariants of homology $3$-spheres, with values in any abelian group without $2$-torsion, from a family of trivial $2$-cocycles on $\mathcal{T}_{g,1}.$
In this chapter we generalize the results of \cite{pitsch}, to include any abelian group without restrictions.

The main difficulty to generalize such results comes from the fact that if we consider an abelian group with $2$-torsion,
then the abelianization of $\mathcal{T}_{g,1}$ is given by $\extp^3H\oplus T,$ where $T$ is a $2$-torsion group, and $Hom(\mathcal{T}_{g,1},A)^{\mathcal{AB}_{g,1}}$ is not zero because, unlikely the case of an abelian group without $2$-torsion, in this case $T$ has not to be sent to zero.

We show that the elements $Hom(\mathcal{T}_{g,1},A)^{\mathcal{AB}_{g,1}}$ are given by multiples of the Rohlin invariant.
As a consequence, any two invariants constructed from the same family of $2$-cocycles differ by an element of $Hom(\mathcal{T}_{g,1},A)^{\mathcal{AB}_{g,1}},$ and hence the Rohlin invariant explains the failure of unicity in this case.

In the first section we review some basic definitions and properties about the Boolean algebra and the Birman-Craggs-Johnson-homomoprhism.
In the second section we recall the definition of a contractible bounding pair twist, the Luft group and we exhibit some interesting results about the handlebody subgroup $\mathcal{B}_{g,1}$ and the Luft-Torelli group $\mathcal{LTB}_{g,1}$.
Finally, in the last two sections we give the aforementioned generalization.

\section{The Boolean algebra and the BCJ-homomorphism}
\begin{defi}[Boolean algebra]
The Boolean polynomial algebra $\mathfrak{B}=\mathfrak{B}_{g,1}$ is a $\mathbb{Z}/2$-algebra (with unit $1$) with a generator $\overline{x}$ for each $x\in H_1(\Sigma_{g,1},\mathbb{Z}/2)$ and subject to the relations:
\begin{enumerate}[(a)]
\item $\overline{x+y}=\overline{x}+\overline{y}+ x\cdot y,$ where $x\cdot y$ is the mod $2$ intersection number,
\item $\overline{x}^2=\overline{x}.$
\end{enumerate}
The relation (b) implies that $p^2=p$ for any $p\in \mathfrak{B}$ and also that if $\{e_i \mid i\in \{ 1,\ldots 2g\}\}$ is a basis for $H_1(\Sigma_{g,1},\mathbb{Z}/2)$ then the set of all monomials
$e_{i_1}e_{i_2}\ldots e_{i_r}$ with $0\leq r\leq 2g,$ $1\leq i_1< i_2< \ldots < i_r\leq 2g$ is a $\mathbb{Z}/2$-basis for $\mathfrak{B}.$ Denote by  $\mathfrak{B}_k=\mathfrak{B}^k_{g,1}$ the subspace generated by all monomials of ''degree'' $\leq k.$
\end{defi}

In \cite{jo1}, D. Johnson constructed a surjective homomorphism $\sigma: \mathcal{T}_{g,1}\rightarrow \mathfrak{B}_3,$ called the Birman-Craggs-Johnson homomorphism (abbreviated BCJ-homomorphism), which may be described as follows.

Consider the $\mathbb{Z}/2$-basis of $\mathfrak{B}_3$ given by
\begin{align*}
\left\{\overline{1},\overline{A}_i,\overline{B}_i,\overline{A}_i\overline{A}_j,\right. & \overline{B}_i\overline{B}_j,  \overline{A}_i\overline{A}_j\overline{A}_k,\overline{B}_i\overline{B}_j\overline{B}_k, \\
&  \left.\overline{A}_i\overline{B}_j, \overline{A}_i\overline{B}_i, \overline{A}_i\overline{B}_j\overline{B}_k,\overline{A}_i\overline{A}_j\overline{B}_k,\overline{A}_i\overline{B}_i\overline{B}_j, \overline{A}_i\overline{B}_i\overline{A}_j \right\},
\end{align*}
where $i,j,k\in \{1,\ldots, 2g\}$ are pairwise distinct,
and consider the curves depicted in the following figure:

\begin{figure}[H]
\begin{center}
\begin{tikzpicture}[scale=.7]
\draw[very thick] (-4.5,-2) -- (7,-2);
\draw[very thick] (-4.5,2) -- (7,2);
\draw[very thick] (-4.5,2) arc [radius=2, start angle=90, end angle=270];
\draw[very thick] (-4.5,0) circle [radius=.4];
\draw[thick, dotted] (-3.5,0) -- (-3,0);
\draw[very thick] (-2,0) circle [radius=.4];
\draw[very thick] (2,0) circle [radius=.4];
\draw[thick, dotted] (3.5,0) -- (3,0);
\draw[very thick] (4.5,0) circle [radius=.4];

\draw[thick, dashed] (0,2) to [out=-70,in=70] (0,-2);
\draw[thick ] (0,2) to [out=-110,in=110] (0,-2);

\draw[thick,dashed] (2,0.4) to [out=70,in=-70] (2,2);
\draw[thick] (2,0.4) to [out=110,in=-110] (2,2);

\draw[thick,dashed] (2,-0.4) to [out=-70,in=70] (2,-2);
\draw[thick] (2,-0.4) to [out=-110,in=110] (2,-2);

\node[above] at (2,2) {$\beta$};
\node[below] at (2,-2) {$\beta'$};
\node[below] at (0,-2) {$\gamma$};

\node at (-4.5,0) {\tiny{1}};
\node at (-2,0) {\tiny{k}};
\node at (2,0) {\tiny{k+1}};
\node at (4.5,0) {\tiny{g}};

\draw[thick,pattern=north west lines] (7,-2) to [out=130,in=-130] (7,2) to [out=-50,in=50] (7,-2);

\end{tikzpicture}
\end{center}
\end{figure}

The BCJ-homomorphism is given on a BP-map $T_\beta T_{\beta'}^{-1}$ by
\[\sigma(T_\beta T_{\beta'}^{-1})= \sum^k_{i=1}\overline{C_i}\overline{D_i}(\overline{E}+\overline{1}),\]
where $E$ is the homology class of $\beta,$ and $\{C_i,D_i\}_i$ is a symplectic basis of a subsurface $\Sigma_{k,1}$ of $\Sigma_{g,1}$ with boundary component $\gamma,$ such that $\gamma \cup \beta \cup \beta'$ is the boundary of a subsurface with genus zero in $\Sigma_{g,1}.$

The BCJ-homomorphism is given on a Dehn twist about a bounding simple closed curve $\gamma$ of genus $k,$ by:
$$\sigma(T_\gamma)= \sum^k_{i=1}\overline{C_i}\overline{D_i},$$
where $\{C_i,D_i\}_i$ is the symplectic basis of the subsurface of genus $k$ bounded by $\gamma.$

By Lemma 13 in \cite{jo1}, $\sigma$ is an $\mathcal{M}_{g,1}$-equivariant map, in other words, for $f\in \mathcal{M}_{g,1},$ $h\in \mathcal{T}_{g,1}$ we have that $\sigma(fhf^{-1})=\hat{f}\cdot \sigma(h),$ where $\hat{f}$ denotes the image of the map $f$ under the symplectic representation mod $2,$ and the action of $\hat{f}$ on $\mathfrak{B}_3$ is induced from the action of $\hat{f}$ on $H_1(\Sigma_{g,1},\mathbb{Z}/2),$ i.e.
$\hat{f}\cdot(\overline{Z_i}\overline{Z_j}\overline{Z_k})=\overline{\hat{f}Z_i}\overline{\hat{f}Z_j}\overline{\hat{f}Z_k},$ where $Z_i,Z_j,Z_k$ are any three elements of $H_1(\Sigma_{g,1},\mathbb{Z}/2).$

\section{The Luft group and CBP-twists}
\label{section_luft}
Denote by $\mathcal{L}_{g,1}$ the kernel of the map $\mathcal{B}_{g,1} \twoheadrightarrow Aut(\pi_1(\mathcal{H}_g))$. We call it the \textit{Luft group.} It was identified by Luft in \cite{luft} as the ''Twist group'' of the handlebody $\mathcal{H}_g$.
Denote $\mathcal{LTB}_{g,1}$ the intersection $\mathcal{L} \cap \mathcal{TB}_{g,1}$, and $IA$ the kernel of the natural map $Aut(\pi_1(\mathcal{H}_g))\rightarrow GL_g(\mathbb{Z})$.

\begin{prop}
\label{prop_surj_Luft}
There is a short exact sequence
\begin{equation}
\label{ses_L}
\xymatrix@C=7mm@R=10mm{1 \ar@{->}[r] & \mathcal{LTB}_{g,1} \ar@{->}[r] & \mathcal{L}_{g,1} \ar@{->}[r]^-{\phi^B\circ\Psi} & S_g(\mathbb{Z}) \ar@{->}[r] & 1 .}
\end{equation}
\end{prop}
 
\begin{proof}
We first prove that $\phi^B\circ\Psi:\mathcal{L}_{g,1} \rightarrow S_g(\mathbb{Z})$ is well defined, i.e. $\phi^B\circ\Psi(\mathcal{L}_{g,1})\subset S_g(\mathbb{Z}).$
Recall that $\mathcal{L}_{g,1}=Ker(\mathcal{B}_{g,1}\rightarrow Aut(\pi_1(\mathcal{H}_g))).$
As a consequence, if $x\in \mathcal{L}_{g,1},$ $\Psi(x)=\left(\begin{smallmatrix}
Id & 0 \\
M & Id
\end{smallmatrix}\right)$ and then $\phi^B\circ\Psi(\mathcal{L}_{g,1})\subset \{Id\}\ltimes S_g(\mathbb{Z})\cong S_g(\mathbb{Z}).$

Next we prove that $\phi^B\circ\Psi:\mathcal{L}_{g,1}\rightarrow S_g(\mathbb{Z})$ is surjective with kernel $\mathcal{LTB}_{g,1}.$ 

Recall that $S_g(\mathbb{Z})$ is generated by the following family of matrices:
$$\{E_{ii} \mid 1\leq i \leq g\}\cup\{E_{ij}+E_{ji} \mid 1\leq i<j \leq g\}.$$
where $E_{ij}$ denotes the matrix with $1$ at position $(i,j)$ and $0$'s elsewhere.
Thus it is enough to find a preimage for each $E_{ii}$ and $E_{ij}.$
Consider $\beta_i,$ $\beta_j$ and $\gamma_{ij}$ the curves depicted in the following figure:

\begin{figure}[H]
\begin{center}
\begin{tikzpicture}[scale=.7]
\draw[very thick] (-4.5,-2) -- (7,-2);
\draw[very thick] (-4.5,2) -- (7,2);
\draw[very thick] (-4.5,2) arc [radius=2, start angle=90, end angle=270];
\draw[very thick] (-4.5,0) circle [radius=.4];
\draw[thick, dotted] (-3.5,0) -- (-3,0);
\draw[very thick] (-2,0) circle [radius=.4];
\draw[thick, dotted] (-1.2,0) -- (-0.7,0);
\draw[thick, dotted] (0.7,0) -- (1.2,0);
\draw[very thick] (2,0) circle [radius=.4];
\draw[thick, dotted] (3.5,0) -- (3,0);
\draw[very thick] (4.5,0) circle [radius=.4];

\draw[very thick] (0,0) circle [radius=.4];

\draw[thick] (-1.7,-0.2) to [out=-30,in=210] (1.7,-0.2);
\draw[thick, dashed] (-1.7,-0.2) to [out=-40,in=220] (1.7,-0.2);
\node [below] at (0,-1) {$\gamma_{i,j}$};

\draw[thick,dashed] (-2,0.4) to [out=70,in=-70] (-2,2);
\draw[thick] (-2,0.4) to [out=110,in=-110] (-2,2);

\draw[thick,dashed] (2,0.4) to [out=70,in=-70] (2,2);
\draw[thick] (2,0.4) to [out=110,in=-110] (2,2);

\node[above] at (-2,2) {$\beta_i$};
\node[above] at (2,2) {$\beta_j$};

\node at (-4.5,0) {\tiny{1}};
\node at (-2,0) {\tiny{i}};
\node at (2,0) {\tiny{j}};
\node at (4.5,0) {\tiny{g}};

\draw[thick,pattern=north west lines] (7,-2) to [out=130,in=-130] (7,2) to [out=-50,in=50] (7,-2);

\end{tikzpicture}
\end{center}
\caption{Contractible simple closed curves}
\end{figure}
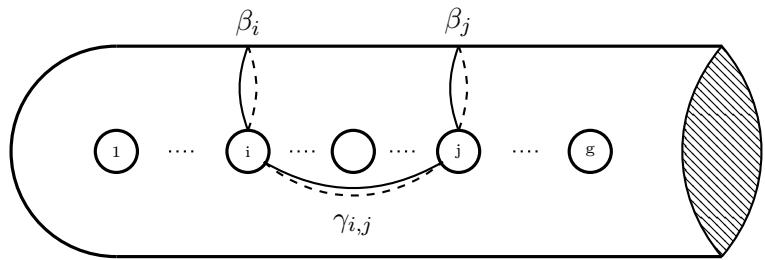

Notice that the curves $\beta_k,$ $\gamma_{ij}$ are contractible in $\mathcal{H}_g$ and as a consequence $T_{\beta_k},$ $T_{\gamma_{ij}}$ are elements of $\mathcal{L}_{g,1}.$
In addition, we have that
\begin{align*}
\Psi(T^{-1}_{\beta_k})(a_l)= & a_l- \widehat{i}(a_l,b_k)b_k=
\left\{\begin{array}{ll}
a_k+b_k & \text{if} \quad l=k \\
a_l & \text{if} \quad l\neq k,
\end{array}\right. \\
\Psi(T^{-1}_{\beta_k})(b_l)= & b_l,\\
\Psi(T_{\gamma_{ij}})(a_l)= & a_l+ \widehat{i}(a_l,[\gamma_{ij}])[\gamma_{ij}] = \\
= & a_l+ \widehat{i}(a_l,b_i-b_j)(b_i-b_j)=
\left\{\begin{array}{ll}
a_i-b_i+b_j & \text{if} \quad l=i \\
a_j+b_i-b_j & \text{if} \quad l=j \\
a_l & \text{if} \quad l\neq i,j,
\end{array}\right. \\
\Psi(T_{\gamma_{ij}})(b_l)= & b_l.
\end{align*}
Thus we get that
$$\Psi(T_{\beta_k}^{-1})=\left(\begin{matrix}
Id & 0 \\
E_{kk} & Id
\end{matrix}\right), \qquad
\Psi(T_{\gamma_{ij}})=\left(\begin{matrix}
Id & 0 \\
-E_{ii}-E_{jj}+E_{ij}+E_{ji} & Id
\end{matrix}\right).$$
As a consequence,
$$\Psi(T_{\beta_i}^{-1}T_{\gamma_{ij}}T_{\beta_j}^{-1})=\left(\begin{matrix}
Id & 0 \\
E_{ij}+E_{ji} & Id
\end{matrix}\right).$$
Therefore $\phi^B\circ \Psi:\mathcal{L}_{g,1}\rightarrow S_g(\mathbb{Z})$ is surjective.
In addition, notice that
\begin{align*}
Ker(\phi^B\circ \Psi:\mathcal{L}_{g,1}\rightarrow S_g(\mathbb{Z}))= & Ker(\mathcal{B}_{g,1}\rightarrow GL_g(\mathbb{Z}) \ltimes S_g(\mathbb{Z}))\cap \mathcal{L}_{g,1}= \\
= & \mathcal{TB}_{g,1}\cap \mathcal{L}_{g,1}=\mathcal{LTB}_{g,1}.
\end{align*}
Therefore we get the desired short exact sequence.
\end{proof}

\begin{prop}
\label{prop_prod_B}
For every $h\in \mathcal{B}_{g,1},$ there exist elements $l\in \mathcal{L}_{g,1},$ $f\in \mathcal{AB}_{g,1}$ and $\xi_b\in \mathcal{TB}_{g,1}$ such that
$h=\xi_b f l,$ i.e.
$$\mathcal{B}_{g,1}=\mathcal{TB}_{g,1}\cdot \mathcal{AB}_{g,1} \cdot\mathcal{L}_{g,1}.$$
\end{prop}

\begin{proof}
Recall that by Lemma \eqref{lem_B}, we have the following short exact sequence:
\begin{equation}
\label{ses_B}
\xymatrix@C=10mm@R=13mm{1 \ar@{->}[r] & \mathcal{TB}_{g,1} \ar@{->}[r] & \mathcal{B}_{g,1} \ar@{->}[r] & GL_g(\mathbb{Z})\ltimes S_g(\mathbb{Z}) \ar@{->}[r] & 1 }
.\end{equation}
Moreover, by Lemma \eqref{lem_AB} and Proposition \eqref{prop_surj_Luft} we know that $\phi^B\circ\Psi(\mathcal{AB}_{g,1})=GL_g(\mathbb{Z})$ and $\phi^B\circ\Psi(\mathcal{L}_{g,1})=S_g(\mathbb{Z}),$ respectively.

As a consequence, if $h\in \mathcal{B}_{g,1}$ then there exist $f\in \mathcal{AB}_{g,1},$ $l\in \mathcal{L}_{g,1}$ such that
$$\Psi(hl^{-1}f^{-1})=Id,$$
and by the short exact sequence \eqref{ses_B}, we get that there exists an element $\xi_b\in \mathcal{TB}_{g,1}$ such that
$\xi_b=hl^{-1}f^{-1},$ so $h=\xi_b f l.$

\end{proof}

\begin{defi}[Section 4 in \cite{pitsch}]
A contractible bounding pair map, abbreviated CBP-twist, is a map of the form $T_\beta T_{\beta'}^{-1},$ where $\beta$ and $\beta'$ are two homologous non-isotopic and disjoint simple closed curves on $\Sigma_{g,1}$ such that each one bounds a properly embedded disc in $\mathcal{H}_g.$
\end{defi}

\begin{prop}[Theorem 9 in \cite{pitsch}]
\label{prop_Luft_gen}
The Luft-Torelli group $\mathcal{LTB}_{g,1}$ is generated by CBP-twists.
\end{prop}

\begin{lema}
\label{lem_CBP_discs}
Let $T_\beta T_{\beta'}^{-1}$ be a CBP-twist of genus $k,$ and let $D_\beta,$ $D_{\beta'}$ be two essential proper embedded discs in $\mathcal{H}_g$ with boundaries $\beta,$ $\beta'$ respectively. Then there exist $g-1$ essential proper embedded discs $D_{\beta_1},\ldots ,D_{\beta_{g-1}}$ in $\mathcal{H}_g,$ with boundaries $\beta_1,\ldots ,\beta_{g-1}$ respectively, such that
$$Int(\mathcal{H}_g)-N\left( D_\beta \cup D_{\beta'} \cup D_{\beta_1} \cup \ldots \cup D_{\beta_{g-1}}\right)$$
is the disjoint union of two open $3$-balls.
\end{lema}

\begin{proof}
Let $D_\beta,$ $D_{\beta'}$ be essential proper embedded discs with boundaries $\beta,$ $\beta'$ respectively.
Since our embedding of $\Sigma_{g,1}$ into $\mathbf{S}^3$ is standard, there exists a simple closed curve $\alpha$ in $\Sigma_{g,1}$ which bounds a properly embedded disc $D_\alpha$ on $-\mathcal{H}_g$ (the outer handlebody) and which intersects each of the curves $\beta$ and $\beta'$ in exactly one point. Consider a regular neighbourhood of the union $D_\beta\cup D_\alpha \cup D_{\beta'}.$ It is a $3$-ball,
whose intersection with $\Sigma_{g,1}$ is the disjoint union of two bounding simple closed curves, which are in $\mathcal{LTB}_{g,1}$ by construction.
Applying Proposition \eqref{prop_minsys} to these curves we get the result.
\end{proof}

\begin{lema}[Lema 2.9 in \cite{jes}]
\label{lem_ext_homeo}
Let $H,$ $H'$ be handlebodies and let $\{ D_1,\ldots ,D_m \},$ $\{D'_1,\ldots ,D'_m\}$ be systems of disks for $H,$ $H',$ respectively. Assume that there is a homeomorphism $\phi:\partial H \rightarrow \partial H'$ such that for each $i,$
$\phi(\partial D_i)=\partial D'_i.$ Then there is a homeomorphism $\psi: H\rightarrow H'$ such that $\psi\mid_{\partial H}=\phi.$
\end{lema}

\begin{prop}
\label{prop_trans_B}
$\mathcal{B}_{g,1}$ acts transitively on CBP-twists of a given genus.
\end{prop}

\begin{proof}
Let $T_\zeta T^{-1}_{\zeta'},$ $T_\beta T_{\beta'}^{-1}$ be CBP-twists of genus $k$ on $\Sigma_{g,1}.$ We prove that there is an element $\psi\in\mathcal{B}_{g,1}$ such that $\psi(\beta)=\zeta,$ $\psi(\beta')=\zeta'$ and as a consequence we will get that the following equality: 
$\psi(T_\beta T_{\beta'}^{-1})\psi^{-1}=T_{\psi(\beta)} T_{\psi(\beta')}^{-1}=T_\zeta T_{\zeta'}^{-1}.$

Given $T_\beta T_{\beta'}^{-1},$ $T_\zeta T_{\zeta'}^{-1}$ CBP-twists of genus $k,$
by Lemma \eqref{lem_CBP_discs}, we know that there exist $g-1$ essential proper embedded discs $D_{\beta_1},\ldots ,D_{\beta_{g-1}}$ in $\mathcal{H}_g,$ with boundaries $\beta_1,\ldots ,\beta_{g-1}$ respectively, such that
$$Int(\mathcal{H}_g)-N\left( D_\beta \cup D_{\beta'} \cup D_{\beta_1} \cup \ldots \cup D_{\beta_{g-1}}\right)$$
is the disjoint union of two open $3$-balls,
and there exist $g-1$ essential proper embedded discs $D_{\zeta_1},\ldots ,D_{\zeta_{g-1}}$ in $\mathcal{H}_g,$ with boundaries $\zeta_1,\ldots ,\zeta_{g-1}$ respectively, such that
$$Int(\mathcal{H}_g)-N\left( D_\zeta \cup D_{\zeta'} \cup D_{\zeta_1} \cup \ldots \cup D_{\zeta_{g-1}}\right)$$
is the disjoint union of two open $3$-balls.

Observe that in particular, $\{D_\beta, D_{\beta'}, D_{\beta_1} , \ldots , D_{\beta_{g-1}}\}$ and $\{D_\zeta, D_{\zeta'}, D_{\zeta_1} , \ldots , D_{\zeta_{g-1}}\}$ are two system of disks for $\mathcal{H}_g.$
 
Since $T_\zeta T^{-1}_{\zeta'},$ $T_\beta T_{\beta'}^{-1}$ are BP-maps of the same genus, by the change of coordinates principle,
there exists a homeomorphism $\phi$ from $\Sigma_{g,1}$ to $\Sigma_{g,1}$ sending $\{\beta, \beta', \beta_1 , \ldots , \beta_{g-1}\}$ to $\{\zeta, \zeta', \zeta_1 , \ldots , \zeta_{g-1}\}$ respectively.

By Lemma \eqref{lem_ext_homeo} we get that there exists a homeomorphism $\psi$ from $\mathcal{H}_g$ to $\mathcal{H}_g$ such that $\psi\mid_{\partial \mathcal{H}_g}=\phi.$
Therefore $\phi$ extends to $\mathcal{H}_{g,1}$ i.e. $\phi\in \mathcal{B}_{g,1}.$
\end{proof}

\begin{prop}
\label{prop_CBP_1}
Every CBP-twist of genus $k$ is a product of $k$ CBP-twists of genus $1.$
\end{prop}

\begin{proof}
Let $T_\beta T_{\beta'}^{-1}$ be a CBP-twist of genus $k.$ Consider the following simple closed curves in the standarly embedded surface $\Sigma_{g,1}:$

\[
\begin{tikzpicture}[scale=.5]
\draw[very thick] (-12,-2) -- (1.5,-2);
\draw[very thick] (-12,2) -- (1.5,2);
\draw[thick, dotted] (-7.5,0) -- (-6.5,0);
\draw[very thick] (-12,2) arc [radius=2, start angle=90, end angle=270];

\draw[very thick] (-12,0) circle [radius=.4];
\draw[very thick] (-9.5,0) circle [radius=.4];
\draw[very thick] (-4.5,0) circle [radius=.4];
\draw[very thick] (-2,0) circle [radius=.4];

\draw[thick,dashed] (-2,0.4) to [out=70,in=-70] (-2,2);
\draw[thick] (-2,0.4) to [out=110,in=-110] (-2,2);

\draw[thick,dashed] (-2,-0.4) to [out=-70,in=70] (-2,-2);
\draw[thick] (-2,-0.4) to [out=-110,in=110] (-2,-2);

\draw[thick] (-2.4,0) to [out=90,in=0] (-4.5,1.6) to [out=180,in=0] (-9.5,1.6) to [out=180,in=90] (-11,0) to [out=-90,in=110] (-10.8,-2);
\draw[thick, dashed] (-2.4,0) to [out=90,in=0] (-4.5,1.2) to [out=180,in=0] (-9.5,1.2) to [out=180,in=90] (-10.6,0) to [out=-90,in=70] (-10.8,-2);

\draw[thick] (-2.4,0) to [out=90,in=0] (-4.5,1.3) to [out=180,in=0] (-7,1.3) to [out=180,in=90] (-8.5,0) to [out=-90,in=110] (-8.3,-2);
\draw[thick, dashed] (-2.4,0) to [out=90,in=0] (-4.5,0.9) to [out=180,in=0] (-7,0.9) to [out=180,in=90] (-8.1,0) to [out=-90,in=70] (-8.3,-2);

\draw[thick] (-2.4,0) to [out=90,in=0] (-4.5,1) to [out=180,in=90] (-6,0) to [out=-90,in=110] (-5.8,-2);
\draw[thick, dashed] (-2.4,0) to [out=90,in=0] (-4.5,0.6) to [out=180,in=90] (-5.6,0) to [out=-90,in=70] (-5.8,-2);

\node [above] at (-2,2) {$\zeta$};
\node [below] at (-2,-2) {$\zeta'$};
\node [below] at (-10.8,-2) {$\zeta_1$};
\node [below] at (-8.3,-2) {$\zeta_2$};
\node [below] at (-5.8,-2) {$\zeta_{k-1}$};

\draw[thick, dotted] (-7.5,-2.6) -- (-7,-2.6);

\draw[thick,pattern=north west lines] (1.5,-2) to [out=130,in=-130] (1.5,2) to [out=-50,in=50] (1.5,-2);
\end{tikzpicture}
\]
Observe that $T_\zeta T_{\zeta'}^{-1}$ is a CBP-twist of genus $k$ and for $i=0,\ldots k-1,$ $T_{\zeta_i} T_{\zeta_{i+1}}^{-1}$ are CBP-twists of genus $1,$  where $\zeta_0=\zeta,$ $\zeta_k=\zeta'.$
By Proposition \eqref{prop_trans_B}
we know that there is an element $h\in \mathcal{B}_{g,1}$ such that
$T_\beta T_{\beta'}^{-1}=hT_\zeta T_{\zeta'}^{-1}h^{-1}.$
Therefore,
\begin{align*}
T_\beta T_{\beta'}^{-1}=hT_\zeta T_{\zeta'}^{-1}h^{-1}= &
(hT_{\zeta_0} T_{\zeta_1}^{-1}h^{-1})(hT_{\zeta_1} T_{\zeta_2}^{-1}h^{-1})\cdots
(hT_{\zeta_{k-1}} T_{\zeta_k}^{-1}h^{-1})= \\
= &
(T_{h(\zeta_0)} T_{h(\zeta_1)}^{-1})(T_{h(\zeta_1)} T_{h(\zeta_2)}^{-1})\cdots
(T_{h(\zeta_{k-1})} T_{h(\zeta_k)}^{-1}).
\end{align*}
Since $\{T_{\zeta_i}T_{\zeta_{i+1}}^{-1}\}_i$ are CBP-twists of genus $1$ and $h\in \mathcal{B}_{g,1}$ then
$\{T_{h(\zeta_i)} T_{h(\zeta_{i+1})}^{-1}\}_i$ are also CBP-twists of genus $1,$ as desired.
\end{proof}

\begin{rem}
A posteriori we found that, in \cite{omori}, G. Omori had obtained independently Proposition \eqref{prop_CBP_1}.
\end{rem}

\section{From invariants to trivial cocycles}
Let $A$ be an abelian group. Denote by $A_2$ the subgroup of $2$-torsion elements.

Consider an $A$-valuated invariant of homology $3$-spheres
$$F: \mathcal{S}^3 \rightarrow A.$$
Precomposing with the canonical maps $\mathcal{T}_{g,1}\rightarrow \lim_{g\rightarrow \infty} \mathcal{T}_{g,1}/\sim \rightarrow  \mathcal{S}^3$ we get a family of maps $\{F_g\}_g$ with $F_g:\mathcal{T}_{g,1}\rightarrow A$ satisfying the following properties:
\begin{enumerate}[i)]
\item $F_{g+1}(x)=F_g(x) \quad \text{for every }x\in \mathcal{T}_{g,1},$
\item $F_g(\xi_a x\xi_b)=F_g(x) \quad \text{for every } x\in \mathcal{T}_{g,1},\;\xi_a\in \mathcal{TA}_{g,1},\;\xi_b\in \mathcal{TB}_{g,1},$
\item $F_g(\phi x \phi^{-1})=F_g(x)  \quad \text{for every }  x\in \mathcal{T}_{g,1}, \; \phi\in \mathcal{AB}_{g,1}.$
\end{enumerate}
Since the stabilization maps are injective, the map $F_g$ determines by restriction all maps $F_{g'}$ for $g'<g.$ We avoid the peculiarities of the first Torelli groups by restricting ourselves to $g\geq 3.$ We also consider the associated trivial $2$-cocycles $\{C_g\}_g,$ which measure the failure of the maps $\{F_g\}_g$ to be homomorphisms of groups, i.e.
\begin{align*}
C_g: \mathcal{T}_{g,1}\times \mathcal{T}_{g,1} & \longrightarrow A, \\
 (\phi,\psi) & \longmapsto F_g(\phi)+F_g(\psi)-F_g(\phi\psi).
\end{align*}
Since $F$ is an invariant, the family of maps $\{F_g\}_g$ satisfies the properties i), ii), iii) and as a consequence, the family of $2$-cocycles $\{C_g\}_g$ inherits the following properties:
\begin{enumerate}[(1)]
\item The $2$-cocycles $\{C_g\}_g$ are compatible with the stabilization map, i.e. the following diagram of maps commutes:
$$
\xymatrix@C=15mm@R=13mm{ \mathcal{T}_{g,1}\times \mathcal{T}_{g,1} \ar@{->}[rd]_{C_g}\ar@{^{(}->}[r] & \mathcal{T}_{g+1,1}\times \mathcal{T}_{g+1,1} \ar@{->}[d]^{C_{g+1}} \\
 & A}
$$
\item The $2$-cocycles $\{C_g\}_g$ are invariant under conjugation by elements in $\mathcal{AB}_{g,1},$ i.e. for every $\phi \in \mathcal{AB}_{g,1},$ $$C_g(\phi - \phi^{-1},\phi - \phi^{-1})=C_g(-,-),$$
\item If $\phi\in \mathcal{TA}_{g,1}$ or $\psi \in \mathcal{TB}_{g,1}$ then $C_g(\phi, \psi)=0.$
\end{enumerate}

In general there are many families of maps $\{F_g\}_g$ satisfying the properties i) - iii) that induce the same family of trivial $2$-cocycles $\{C_g\}_g.$

Recall that any two trivializations of a given trivial $2$-cocycle on $\mathcal{T}_{g,1}$ differ by an element of the group $Hom(\mathcal{T}_{g,1},A).$ Then, given two families of maps $\{F_g\}_g,$ $\{F_g'\}_g$ satisfying the properties i) - iii), we have that
$\{F_g-F_g'\}_g$ is a family of homomorphisms satisfying the same properties.
As a consequence, the number of families $\{F_g\}_g$ satisfying the properties i) - iii) that induce the same family of trivial $2$-cocycles $\{C_g\}_g.$, coincides with the number of homomorphisms in $Hom(\mathcal{T}_{g,1},A)^{\mathcal{AB}_{g,1}}$ compatible with the stabilization map, that are trivial over $\mathcal{TA}_{g,1}$ and $\mathcal{TB}_{g,1}.$
We devote the rest of this section to compute and study such homomorphisms. In order to achieve our target, we first give three algebraic lemmas.
Throughout this chapter, given an element $x\in A$ of order $2,$ set $\varepsilon^x:\mathbb{Z}/2\rightarrow A$ the homomorphism which sends $1$ to $x.$

\begin{lema}
\label{lem_B_3-inv}
For $g\geq 4,$ let $ \pi : \mathfrak{B}_3\rightarrow \mathfrak{B}_0=\mathbb{Z}/2$ be the canonical projection. Then we have the following isomorphism:
\begin{align*}
\Upsilon: A_2 &\longrightarrow Hom(\mathfrak{B}^3_{g,1},A)^{GL_g(\mathbb{Z})} \\
x & \longmapsto \varphi_g^x=\epsilon^x\circ \pi.
\end{align*}
\end{lema}

\begin{proof}
First of all notice that $\Upsilon$ is well defined because $\epsilon^x$ and the canonical projection $ \pi : \mathfrak{B}_3\rightarrow \mathfrak{B}_0$ are $GL_g(\mathbb{Z})$-invariant.
Moreover, it is clear that $\Upsilon$ is injective.
Next we show that $\Upsilon$ is surjective.

Let $\varphi_g\in Hom(\mathfrak{B}^3_{g,1},A)^{GL_g(\mathbb{Z})},$ we prove that $\varphi_g=\varphi_g^x$ for some $x\in A_2.$

Consider $f$ the matrix
$\left(\begin{smallmatrix}
G & 0 \\
0 & {}^tG^{-1}
\end{smallmatrix}\right)$
with $G=(1i)(2j)(3k)\in \mathfrak{S}_g\subset GL_g(\mathbb{Z})$ for $i,j,k$ pairwise different.
Then, for $Z_l=A_l$ or $B_l$ for each $l$, we have the following equalities
\begin{align}
 & \varphi_g(\overline{Z_1})=\varphi_g(\overline{f(Z_1)})=\varphi_g(\overline{Z_i}), \label{eq:I} \\
 & \varphi_g(\overline{Z_1}\overline{Z_2})=\varphi_g(\overline{f(Z_1)}\overline{f(Z_2)})=\varphi_g(\overline{Z_i}\overline{Z_j}) ,\label{eq:II}\\
 & \varphi_g(\overline{A_1}\overline{B_1})=\varphi_g(\overline{f(A_1)}\overline{f(B_1)})=\varphi_g(\overline{A_i}\overline{B_i}),\label{eq:III}\\
 & \varphi_g(\overline{Z_1}\overline{Z_2}\overline{Z_3})=\varphi_g(\overline{f(Z_1)}\overline{f(Z_2)}\overline{f(Z_3)})=\varphi_g(\overline{Z_i}\overline{Z_j}\overline{Z_k}) ,\label{eq:IV}\\
 & \varphi_g(\overline{A_1}\overline{B_1}\overline{Z_2})=\varphi_g(\overline{f(A_1)}\overline{f(B_1)}\overline{f(Z_2)})=\varphi_g(\overline{A_i}\overline{B_i}\overline{Z_j}).\label{eq:V}
\end{align}

Next, consider $f$ the matrix
$\left(\begin{smallmatrix}
G & 0 \\
0 & {}^tG^{-1}
\end{smallmatrix}\right)$
with $G\in GL_g(\mathbb{Z})$ the matrix with $1$'s at the diagonal and position $(2,1)$ and $0$'s elsewhere.
Then we have the following equalities
\begin{equation}
\label{eq:1}
\begin{aligned}
 & \varphi_g(\overline{A_1})=\varphi_g(\overline{f(A_1)})=
 \varphi_g(\overline{A_1+A_2})=\varphi_g(\overline{A_1})+\varphi_g(\overline{A_2}), \\
 & \text{hence, } \varphi_g(\overline{A_2})=0.
 \end{aligned}
\end{equation}

\begin{equation}
\label{eq:2}
 \begin{aligned}
 & \varphi_g(\overline{B_2})=\varphi_g(\overline{f(B_2)})=
 \varphi_g(\overline{B_1+B_2})=\varphi_g(\overline{B_1})+\varphi_g(\overline{B_2}), \\
 & \text{hence, } \varphi_g(\overline{B_1})=0.
 \end{aligned}
\end{equation}
As a consequence of relations \eqref{eq:I}, \eqref{eq:1}, \eqref{eq:2}, we get that
\begin{equation}
\label{eq:Z}
\varphi_g(\overline{Z_i})=0 \text{ for all } i.
\end{equation}
In addition,
\begin{equation}
\label{eq:3}
\begin{aligned}
 & \varphi_g(\overline{A_1}\overline{Z_3})=\varphi_g(\overline{f(A_1)}\overline{f(Z_3)})=
 \varphi_g((\overline{A_1+A_2})\overline{Z_3})= \\
 &=\varphi_g(\overline{A_1}\overline{Z_3})+\varphi_g(\overline{A_2}\overline{Z_3}), \text{ hence, } \varphi_g(\overline{A_2}\overline{Z_3})=0. \end{aligned}
\end{equation}

 \begin{equation}
 \label{eq:4}
\begin{aligned}
 & \varphi_g(\overline{B_2}\overline{Z_3})=\varphi_g(\overline{f(B_2)}\overline{f(Z_3)})=
 \varphi_g((\overline{B_1+B_2})\overline{Z_3})= \\
 &=\varphi_g(\overline{B_1}\overline{Z_3})+\varphi_g(\overline{B_2}\overline{Z_3}), \text{ hence, } \varphi_g(\overline{B_1}\overline{Z_3})=0. \end{aligned}
\end{equation}

As a consequence of relations \eqref{eq:II},\eqref{eq:3},\eqref{eq:4}, we get that
\begin{equation}
\label{eq:ZZ}
\varphi_g(\overline{Z_i}\overline{Z_j})=0 \text{ for all } i,j \text{ with } i\neq j.
\end{equation}
Besides,
\begin{equation}
\label{eq:6}
\begin{aligned}
 & \varphi_g(\overline{A_3}\overline{B_3}\overline{A_1})=\varphi_g(\overline{f(A_3)}\overline{f(B_3)}\overline{f(A_1)})=
 \varphi_g(\overline{A_3}\overline{B_3}(\overline{A_1+A_2}))= \\
 & =\varphi_g(\overline{A_3}\overline{B_3}\overline{A_1})+\varphi_g(\overline{A_3}\overline{B_3}\overline{A_2}), \text{hence, } \varphi_g(\overline{A_3}\overline{B_3}\overline{A_2})=0.
 \end{aligned}
\end{equation}

\begin{equation}
\label{eq:7}
\begin{aligned}
 & \varphi_g(\overline{A_3}\overline{B_3}\overline{B_2})=\varphi_g(\overline{f(A_3)}\overline{f(B_3)}\overline{f(B_2)})=
 \varphi_g(\overline{A_3}\overline{B_3}(\overline{B_1+B_2}))= \\
 & =\varphi_g(\overline{A_3}\overline{B_3}\overline{B_1})+\varphi_g(\overline{A_3}\overline{B_3}\overline{B_2}), \text{hence, } \varphi_g(\overline{A_3}\overline{B_3}\overline{B_1})=0.
 \end{aligned}
\end{equation}
As a consequence of relations \eqref{eq:V},\eqref{eq:6},\eqref{eq:7} we get that
\begin{equation}
\label{eq:ABZ}
\varphi_g(\overline{A_i}\overline{B_i}\overline{Z_j})=0 \text{ for all } i,j \text{ with } i\neq j.
\end{equation}
Moreover,
\begin{equation}
\label{eq:9}
\begin{aligned} 
 & \varphi_g(\overline{A_1}\overline{Z_3}\overline{Z_4})=\varphi_g(\overline{f(A_1)}\overline{f(Z_3)}\overline{f(Z_4)})=
 \varphi_g((\overline{A_1+A_2})\overline{Z_3}\overline{Z_4})= \\
 &=\varphi_g(\overline{A_1}\overline{Z_3}\overline{Z_4})+\varphi_g(\overline{A_2}\overline{Z_3}\overline{Z_4}), \text{ hence, } \varphi_g(\overline{A_2}\overline{Z_3}\overline{Z_4})=0.
  \end{aligned}
\end{equation}

\begin{equation}
\label{eq:10}
\begin{aligned}
  & \varphi_g(\overline{B_2}\overline{Z_3}\overline{Z_4})=\varphi_g(\overline{f(B_2)}\overline{f(Z_3)}\overline{f(Z_4)})=
 \varphi_g((\overline{B_1+B_2})\overline{Z_3}\overline{Z_4})= \\
 &=\varphi_g(\overline{B_1}\overline{Z_3}\overline{Z_4})+\varphi_g(\overline{B_2}\overline{Z_3}\overline{Z_4}), \text{ hence, } \varphi_g(\overline{B_1}\overline{Z_3}\overline{Z_4})=0.
 \end{aligned}
\end{equation}
As a consequence of relations \eqref{eq:IV}, \eqref{eq:9}, \eqref{eq:10} we have that
$$\varphi_g(\overline{Z_i}\overline{Z_j}\overline{Z_k})=0 \text{ for all } i,j,k \text{ with } i\neq j\neq k.$$
Furthermore,
\begin{equation}
\label{eq:11}
\begin{aligned}
 & \varphi_g(\overline{A_1}\overline{B_1}\overline{B_2})=\varphi_g(\overline{f(A_1)}\overline{f(B_1)}\overline{f(B_2)})=
 \varphi_g((\overline{A_1+A_2})\overline{B_1}(\overline{B_1+B_2}))= \\
 & =\varphi_g(\overline{A_1}\overline{B_1})+\varphi_g(\overline{A_1}\overline{B_1}\overline{B_2})+\varphi_g(\overline{A_2}\overline{B_1})+\varphi_g(\overline{A_2}\overline{B_1}\overline{B_2}) \\
 & \text{Thus, } \varphi_g(\overline{A_1}\overline{B_1})=-\varphi_g(\overline{A_2}\overline{B_1})-\varphi_g(\overline{A_2}\overline{B_1}\overline{B_2}).
 \end{aligned}
\end{equation}
As a consequence of relations \eqref{eq:ZZ}, \eqref{eq:ABZ}, \eqref{eq:11}, we get that
$$\varphi_g(\overline{A_i}\overline{B_i})=0 \text{ for all } i.$$
Thus $\varphi_g$ is zero on all basis elements of $\mathfrak{B}_3$ except possibly on $\overline{1},$ which can only be sent to an element of order $\leq 2$ in $A.$
\end{proof}

\begin{lema}
\label{lem_B_2-inv}
For $g\geq 4$ there is an isomorphism
\begin{align*}
\Xi: A_2\times A_2 & \longrightarrow Hom(\mathfrak{B}^2_{g,1},A)^{GL_g(\mathbb{Z})} \\
(x_1,x_2) & \longmapsto \varphi_g^{x_1,x_2}:
\left\lbrace\begin{aligned}
\overline{1}\longmapsto x_1 \\
\overline{A_iB_i}\longmapsto x_2.
\end{aligned}\right. 
\end{align*}
\end{lema}

\begin{proof}
First of all we show that $\Xi$ is well-defined, i.e. $\varphi_g^{x_1,x_2}$ is $GL_g(\mathbb{Z})$-invariant.

Let $x_1$ be a $2$-torsion element of $A.$ Consider $\varphi_g^{x_1,0}$ the homomorphism given by the composition of the projection map $\mathfrak{B}_2\rightarrow \mathfrak{B}_0=\mathbb{Z}/2$ and $\varepsilon^{x_1}:\mathbb{Z}/2\rightarrow A.$ Notice that $\varphi_g^{x_1,0}$ is a $GL_g(\mathbb{Z})$-invariant homomorphism.

Let $x_2$ be a $2$-torsion element of $A.$ Consider $\varphi_g^{0,x_2} \in Hom(\mathfrak{B}^2_{g,1},A)$ the homomorphism given by 
$\varphi_g^{0,x_2}(\overline{Z_iZ_j})=\omega(Z_i,Z_j)x_2,$ $\varphi_g^{0,x_2}(\overline{Z_i})=\varphi_g^{0,x_2}(\overline{Z_iZ_i})=\omega(Z_i,Z_i)x_2=0$ and $\varphi_g^{0,x_2}(\overline{1})=0.$
Notice that the action of $GL_g(\mathbb{Z})$ on $\overline{1}$ is trivial. On the other hand, since $\omega$ is $Sp_{2g}$-invariant, we have that for any matrix $f$ of the form $\left(\begin{smallmatrix}
G & 0 \\
0 & {}^tG^{-1}
\end{smallmatrix}\right)$, the following equality holds
$$
\varphi_g^{0,x_2}(f.\overline{Z_i}\overline{Z_j})=
\varphi_g^{0,x_2}(\overline{f(Z_i)}\;\overline{f(Z_j)})=
\omega(f(Z_i),f(Z_j))x_2= \omega(Z_i,Z_j)x_2=\varphi_g^{0,x_2}(\overline{Z_i}\overline{Z_j}).
$$
Hence, $\varphi_g^{0,x_2}$ is a $GL_g(\mathbb{Z})$-invariant homomorphism.
Therefore, $\varphi_g^{x_1,x_2}=\varphi_g^{x_1,0}+\varphi_g^{0,x_2}$
is $GL_g(\mathbb{Z})$-invariant.
Moreover, it is clear that $\Xi$ is injective. Next we prove that $\Xi$ is surjective.
Let $\varphi_g \in Hom(\mathfrak{B}^2_{g,1},A)^{GL_g(\mathbb{Z})},$
we show that $\varphi_g=\varphi_g^{x_1,x_2}$ for some $x_1,x_2\in A_2.$

Following the proof of Lemma \eqref{lem_B_3-inv}, in particular, by the equalities \eqref{eq:I}, \eqref{eq:II}, \eqref{eq:III}, \eqref{eq:1}, \eqref{eq:2}, \eqref{eq:3}, \eqref{eq:4} of Lemma \eqref{lem_B_3-inv}, we have that
$$\varphi_g(\overline{Z_i})=0 \quad \varphi_g(\overline{Z_i}\overline{Z_j})=0 \quad \varphi_g(\overline{A_1}\overline{B_1})=\varphi_g(\overline{A_i}\overline{B_i}) \text{ for all } i,j \text{ with } i\neq j.$$
Hence, the elements of $Hom(\mathfrak{B}^2_{g,1},A)^{GL_g(\mathbb{Z})}$ are completely determined by the images of $\overline{1}$ and $\overline{A_1}\overline{B_1}.$
\end{proof}

\begin{lema}
\label{lema_hom_extp_H_d}
For $g\geq 4,$ the group $Hom(\extp^3 H,A)^{GL_g(\mathbb{Z})}$ is zero.
\end{lema}
\begin{proof}
Let $f$ be an element of $Hom(\extp^3 H,A)^{GL_g(\mathbb{Z})}.$ 
\begin{itemize}
\item Consider the element
$\phi=\left(\begin{smallmatrix}
G & 0 \\
0 & {}^tG^{-1}
\end{smallmatrix}\right) \in Sp_{2g}(\mathbb{Z})$
with $G=(1,i)(2,j)(3,k)\in \mathfrak{S}_g.$ Then:
\begin{align*}
\phi \cdot f( c_1\wedge c_2\wedge c_3)= & f(\phi \cdot c_1\wedge c_2\wedge c_3)= f(c_i\wedge c_j\wedge c_k), \\
\phi \cdot f( c_1\wedge a_2\wedge b_2)= & f(\phi \cdot c_1\wedge a_2\wedge b_2)= f(c_i\wedge a_j\wedge b_j).
\end{align*}
Thus, every element of $Hom(\extp^3H,A)^{GL_{g}(\mathbb{Z})}$ is determined by the images of the elements $c_1\wedge c_2\wedge c_3,$ $c_1\wedge a_2\wedge b_2$ with $c_i=a_i \text{ or } b_i.$
\item Consider the element
$\phi=\left(\begin{smallmatrix}
G & 0 \\
0 & {}^tG^{-1}
\end{smallmatrix}\right) \in Sp_{2g}(\mathbb{Z})$
with $G\in GL_g(\mathbb{Z})$ the matrix with $1$'s at the diagonal and position $(3,4),$ and $0$'s at the other positions. Then we get that
\begin{align*}
\phi \cdot f( c_1\wedge c_2\wedge a_4)= & f(\phi \cdot c_1\wedge c_2\wedge a_4)=
f(c_1\wedge c_2\wedge a_3+c_1\wedge c_2\wedge a_4) \\
= & f(c_1\wedge c_2\wedge a_3)+f(c_1\wedge c_2\wedge a_4)
\end{align*}
Thus, $f(c_1\wedge c_2\wedge a_3)=0.$
Analogously, we have that $f(c_1\wedge c_2\wedge b_3)=0.$
\item Consider the element
$\phi=\left(\begin{smallmatrix}
G & 0 \\
0 & {}^tG^{-1}
\end{smallmatrix}\right) \in Sp_{2g}(\mathbb{Z})$
with $G\in GL_g(\mathbb{Z})$ the matrix with $1$'s at the diagonal and position $(1,3),$ and $0$'s at the other positions. Then we get that
\begin{align*}
\phi \cdot f( a_3\wedge a_2\wedge b_2)= & f(\phi \cdot a_3\wedge a_2\wedge b_2)=
f(a_1\wedge a_2\wedge b_2+a_3\wedge a_2\wedge b_2) \\
= & f(a_1\wedge a_2\wedge b_2)+f(a_3\wedge a_2\wedge b_2)
\end{align*}
Thus, $f(a_1\wedge a_2\wedge b_2)=0.$ Analogously, $f(b_1\wedge a_2\wedge b_2)=0.$
\end{itemize}
Therefore we get the desired result.
\end{proof}

\begin{rem}
Using the same arguments in proof of Lemma \eqref{lema_hom_extp_H_d}, we also have that given an integer $d\geq 2,$ for any $g\geq 4$ $Hom(\extp^3 H_d,A)^{GL_g(\mathbb{Z})}$ is zero.
\end{rem}

Next, we give a lemma which ensures us that every $\mathcal{AB}_{g,1}$-invariant homomorphism has to be zero on $\mathcal{TA}_{g,1}, \mathcal{TB}_{g,1}.$

\begin{lema}
\label{lem_TB,TA} For $g\geq 3,$ every $\mathcal{AB}_{g,1}$-invariant homomorphism
$$\varphi_g: \mathcal{TB}_{g,1}\rightarrow A \quad \text{and} \quad \varphi_g:\mathcal{TA}_{g,1}\rightarrow A$$
has to be zero.
\end{lema}

\begin{proof}
We only give the proof for $\mathcal{TB}_{g,1}$ the other case is similar.

Let $\varphi_g\in Hom(\mathcal{TB}_{g,1},A)^{\mathcal{AB}_{g,1}}.$
Consider the following short exact sequence:
\begin{equation}
\label{seq_LTB}
\xymatrix@C=10mm@R=10mm{ 1 \ar@{->}[r] & \mathcal{LTB}_{g,1} \ar@{->}[r] & \mathcal{TB}_{g,1} \ar@{->}[r]  & IA \ar@{->}[r] & 1 .}
\end{equation}
We prove that
\begin{enumerate}[I)]
\item $\varphi_g$ factors through $IA:=ker(Aut (\pi_1(\mathcal{H}_g))\rightarrow GL_g(\mathbb{Z})).$
\item the morphism $\varphi_g:IA \rightarrow A$ is trivial.
\end{enumerate}
\textbf{I)} By the short exact sequence \eqref{seq_LTB}, $\varphi_g:\mathcal{TB}_{g,1}\rightarrow A$ factors through $IA$ if and only if $\varphi_g$ is zero over $\mathcal{LTB}_{g,1}.$
Moreover, by Propositions \eqref{prop_Luft_gen} and \eqref{prop_CBP_1}, we have that 
$\mathcal{LTB}_{g,1}$ is generated by CBP-twists of genus $1.$
Thus it is enough to show that $\varphi_g$ is zero on all CBP-twists of genus $1.$

Next, we divide the proof into three steps:
\begin{enumerate}[1)]
\item $\varphi_g$ takes the same value on all CBP-twists of genus $1.$
\item $\varphi_g$ is zero on all CBP-twists of genus $2.$
\item $\varphi_g$ is zero on all CBP-twists of genus $1.$
\end{enumerate}

\textbf{1)} Consider the following simple closed curves in the standardly embedded surface $\Sigma_{g,1}:$
\[
\begin{tikzpicture}[scale=.7]
\draw[very thick] (-7,-2) -- (3,-2);
\draw[very thick] (-7,2) -- (3,2);
\draw[very thick] (-7,2) arc [radius=2, start angle=90, end angle=270];
\draw[very thick] (-7,0) circle [radius=.4];
\draw[very thick] (-4.5,0) circle [radius=.4];
\draw[very thick] (-2,0) circle [radius=.4];
\draw[very thick] (1,0) circle [radius=.4];
\draw[thick, dotted] (-1,0) -- (0,0);

\draw[thick,dashed] (-4.5,0.4) to [out=70,in=-70] (-4.5,2);
\draw[thick] (-4.5,0.4) to [out=110,in=-110] (-4.5,2);

\draw[thick,dashed] (-4.5,-0.4) to [out=-70,in=70] (-4.5,-2);
\draw[thick] (-4.5,-0.4) to [out=-110,in=110] (-4.5,-2);

\node [above] at (-4.5,2) {$\beta$};
\node [below] at (-4.5,-2) {$\beta'$};

\draw[thick,pattern=north west lines] (3,-2) to [out=130,in=-130] (3,2) to [out=-50,in=50] (3,-2);

\end{tikzpicture}
\]
Observe that $T_\beta T^{-1}_{\beta'}$ is a CBP-twist of genus $1.$

By Proposition \eqref{prop_trans_B} we know that for every CBP-twist of genus $1,$ $T_\nu T_{\nu'}^{-1},$ on $ \Sigma_{g,1}$ there exists an element $h\in\mathcal{B}_{g,1}$ such that
$T_\nu T_{\nu'}^{-1}=hT_{\beta} T_{\beta'}^{-1}h^{-1}.$

Besides, by Proposition \eqref{prop_prod_B}, we have that there exist elements $l\in \mathcal{L}_{g,1},$ $f\in \mathcal{AB}_{g,1}$ and $\xi_b\in \mathcal{TB}_{g,1}$ such that
$h=\xi_b f l.$

Thus, since $\varphi_g$ is a $\mathcal{AB}_{g,1}$-invariant homomorphism, we get that
\begin{align*}
\varphi_g(T_\nu T_{\nu'}^{-1})= & \varphi_g(hT_{\beta} T_{\beta'}^{-1}h^{-1})=\varphi_g(\xi_b f lT_{\beta} T_{\beta'}^{-1}l^{-1}f^{-1}\xi_b^{-1})= \\
= & \varphi_g(\xi_b)+ \varphi_g(f lT_{\beta} T_{\beta'}^{-1}l^{-1}f^{-1})+\varphi_g(\xi_b^{-1})= \\
= & \varphi_g(f lT_{\beta} T_{\beta'}^{-1}l^{-1}f^{-1})=\varphi_g( lT_{\beta} T_{\beta'}^{-1}l^{-1})
\end{align*}

Therefore there exists an element $l\in \mathcal{L}_{g,1}$ such that
\begin{equation}
\label{eq_CBP_l}
\varphi_g(T_\nu T_{\nu'}^{-1})= \varphi_g( lT_{\beta} T_{\beta'}^{-1}l^{-1}).
\end{equation}

Recall that $S_g(\mathbb{Z})$ is generated by the following family of matrices:
$$\{E_{ii} \mid 1\leq i \leq g\}\cup\{E_{ij}+E_{ji} \mid 1\leq i<j \leq g\}.$$
where $E_{ij}$ denotes the matrix with $1$ at position $(i,j)$ and $0$'s elsewhere.

Take the theoretic section $s:\; S_g(\mathbb{Z})\rightarrow \mathcal{L}_{g,1}$ of $\Psi: \mathcal{L}_{g,1}\rightarrow S_g(\mathbb{Z})$ given by
$$
s(E_{ii})= T_{\beta_i}, \qquad
s(E_{ij}+E_{ji})= \left\{\begin{array}{cc}
T_{\beta_1}^{-1}T_{\gamma_{1j}}T_{\beta_j}^{-1} & \text{for}\quad i=1 \\
T_{\beta_i}^{-1}T_{\gamma'_{ij}}T_{\beta_j}^{-1} & \text{for}\quad i\geq 2,
\end{array} \right. $$
where $\beta_i,$ $\gamma_{ij},$ $\gamma'_{ij}$ are given in the following picture:
\[
\begin{tikzpicture}[scale=.7]
\draw[very thick] (-4.5,-2) -- (7,-2);
\draw[very thick] (-4.5,2) -- (7,2);
\draw[very thick] (-4.5,2) arc [radius=2, start angle=90, end angle=270];
\draw[very thick] (-4.5,0) circle [radius=.4];
\draw[thick, dotted] (-3.5,0) -- (-3,0);
\draw[very thick] (-2,0) circle [radius=.4];
\draw[thick, dotted] (-1.2,0) -- (-0.7,0);
\draw[thick, dotted] (0.7,0) -- (1.2,0);
\draw[very thick] (2,0) circle [radius=.4];
\draw[thick, dotted] (3.5,0) -- (3,0);
\draw[very thick] (4.5,0) circle [radius=.4];
\draw[very thick] (0,0) circle [radius=.4];

\draw[thick] (-1.7,-0.2) to [out=-30,in=210] (1.7,-0.2);
\draw[thick, dashed] (-1.7,-0.2) to [out=-40,in=220] (1.7,-0.2);
\node [below] at (0,-1) {$\gamma_{i,j}$};

\draw[thick, dashed] (-1.7,0.2) to [out=30,in=-210] (1.7,0.2);
\draw[thick] (-1.7,0.2) to [out=40,in=-220] (1.7,0.2);
\node [above] at (0,0.8) {$\gamma'_{i,j}$};

\draw[thick,dashed] (-2,0.4) to [out=70,in=-70] (-2,2);
\draw[thick] (-2,0.4) to [out=110,in=-110] (-2,2);

\draw[thick,dashed] (2,0.4) to [out=70,in=-70] (2,2);
\draw[thick] (2,0.4) to [out=110,in=-110] (2,2);

\node[above] at (-2,2) {$\beta_i$};
\node[above] at (2,2) {$\beta_j$};

\node at (-4.5,0) {\tiny{1}};
\node at (-2,0) {\tiny{i}};
\node at (2,0) {\tiny{j}};
\node at (4.5,0) {\tiny{g}};

\draw[thick,pattern=north west lines] (7,-2) to [out=130,in=-130] (7,2) to [out=-50,in=50] (7,-2);

\end{tikzpicture}
\]
By the short exact sequence \eqref{ses_L}, we know that given an element $l\in \mathcal{L}_{g,1}$ there exists an element $\xi_b\in \mathcal{LTB}_{g,1}$ such that $l=\xi_b s(\Psi(l)).$ Thus, by the equality \eqref{eq_CBP_l},
\begin{equation}
\label{eq_CBP_s}
\begin{aligned}
\varphi_g(T_\nu T_{\nu'}^{-1})= & \varphi_g( lT_{\beta} T_{\beta'}^{-1}l^{-1})=\varphi_g( \xi_b s(\Psi(l))T_{\beta} T_{\beta'}^{-1}s(\Psi(l))^{-1}\xi_b^{-1})= \\
= & \varphi_g(s(\Psi(l))T_{\beta} T_{\beta'}^{-1}s(\Psi(l))^{-1}).
\end{aligned}
\end{equation}
Now observe that $s(\Psi(l))$ is a product of the following elements
$$\{T_{\gamma_{1j}}\mid \; j\geq 2\},\quad \{T_{\gamma'_{ij}}\mid \; 2\leq i<j \},\quad \{T_{\beta_i}\mid 1\leq i \leq g\}.$$
Recall that by the properties of Dehn twists we know that for every $a,b$ isotopy classes of simple closed curves on $\Sigma_{g,1},$
$$T_aT_b=T_bT_a \Longleftrightarrow i(a,b)=0.$$
Then, since the curves
\begin{equation}
\label{died_curves}
\gamma_{12},\quad \{\gamma'_{ij}\mid \; 2\leq i<j \},\quad \{\beta_i\mid 1\leq i \leq g\},
\end{equation}
are disjoint with the curves
\begin{equation}
\label{survived_curves}
\beta,\quad \beta', \quad\{\gamma_{1j}\mid \; j\geq 3\},
\end{equation}
the geometric intersection number between a curve of the family \eqref{died_curves} and a curve of the family \eqref{survived_curves}
is zero. Therefore, the elements of the family of Dehn twists
\begin{equation}
\label{died_twists}
T_{\gamma_{12}},\quad \{T_{\gamma'_{ij}}\mid \; 2\leq i<j \},\quad \{T_{\beta_i}\mid 1\leq i \leq g\},
\end{equation}
commutes with the elements of the family of Dehn twists
\begin{equation}
\label{survived_twists}
T_{\beta},\quad T_{\beta'}, \quad\{T_{\gamma_{1j}}\mid \; j\geq 3\}.
\end{equation}
Furthermore, the elements of the family $\{T_{\gamma_{1j}}\mid \; j\geq 3\}$ commute between them because the curves of the family
$\{\gamma_{1j}\mid \; j\geq 3\}$ are pairwise disjoint.
Therefore,
\begin{equation}
\label{eq_CBP_expand}
s(\Psi(l))T_{\beta}  T_{\beta'}^{-1}s(\Psi(l))^{-1}
=(T_{\gamma_{13}}^{x_3}\cdots T_{\gamma_{1g}}^{x_g})T_{\beta}  T_{\beta'}^{-1}(T_{\gamma_{13}}^{x_3}\cdots T_{\gamma_{1g}}^{x_g})^{-1},
\end{equation}
where $x_3, \ldots , x_g\in \mathbb{Z},$ and $\beta,$ $\beta',$
$\{\gamma_{1j}\mid \; 3\leq j\geq g\}$ are the curves given in the following picture:
\[
\begin{tikzpicture}[scale=.7]
\draw[very thick] (-2.5,-2) -- (11,-2);
\draw[very thick] (-2.5,2) -- (11,2);
\draw[very thick] (-2.5,2) arc [radius=2, start angle=90, end angle=270];

\draw[very thick] (0,0) circle [radius=.4];
\draw[very thick] (-2.5,0) circle [radius=.4];
\draw[very thick] (2.5,0) circle [radius=.4];
\draw[very thick] (5,0) circle [radius=.4];
\draw[very thick] (8.5,0) circle [radius=.4];

\draw[thick, dotted] (6,0) -- (7.5,0);
\draw[thick, dotted] (4.5,-1) -- (5.5,-1);

\draw[thick,dashed] (0,0.4) to [out=70,in=-70] (0,2);
\draw[thick] (0,0.4) to [out=110,in=-110] (0,2);

\draw[thick,dashed] (0,-0.4) to [out=-70,in=70] (0,-2);
\draw[thick] (0,-0.4) to [out=-110,in=110] (0,-2);

\draw[thick] (-2.2,-0.2) to [out=-20,in=200] (2.5,-0.4);
\draw[thick, dashed] (-2.2,-0.2) to [out=-22,in=202] (2.5,-0.4);

\draw[thick] (-2.3,-0.3) to [out=-20,in=200] (5,-0.4);
\draw[thick, dashed] (-2.3,-0.3) to [out=-22,in=202] (5,-0.4);

\draw[thick] (-2.5,-0.4) to [out=-20,in=200] (8.5,-0.4);
\draw[thick, dashed] (-2.5,-0.4) to [out=-22,in=202] (8.5,-0.4);

\node at (-2.5,0) {\tiny{1}};
\node at (0,0) {\tiny{2}};
\node at (2.5,0) {\tiny{3}};
\node at (5,0) {\tiny{4}};
\node at (8.5,0) {\tiny{g}};

\node[below] at (1.25,0) {$\gamma_{13}$};
\node[below] at (3.75,0) {$\gamma_{14}$};
\node[below] at (7,0) {$\gamma_{1g}$};
\node [above] at (0,2) {$\beta$};
\node [below] at (0,-2) {$\beta'$};

\draw[thick,pattern=north west lines] (11,-2) to [out=130,in=-130] (11,2) to [out=-50,in=50] (11,-2);

\end{tikzpicture}
\]
As a consequence of equalities \eqref{eq_CBP_s}, \eqref{eq_CBP_expand} we get that
\begin{equation}
\label{eq_CBP_product}
\varphi_g(T_\nu T_{\nu'}^{-1})=\varphi_g((T_{\gamma_{13}}^{x_3}\cdots T_{\gamma_{1g}}^{x_g})T_{\beta} T_{\beta'}^{-1}(T_{\gamma_{13}}^{x_3}\cdots T_{\gamma_{1g}}^{x_g})^{-1})=\varphi_g(T_{\beta} T_{(T_{\gamma_{13}}^{x_3}\cdots T_{\gamma_{1g}}^{x_g})(\beta')}^{-1}).
\end{equation}
Next, we prove that
$$\varphi_g(T_{\beta} T_{(T_{\gamma_{13}}^{x_3}\cdots T_{\gamma_{1g}}^{x_g})(\beta')}^{-1})= \varphi_g(T_{\beta} T_{\beta'}^{-1})+\sum^g_{i=1}x_i(\varphi_g(T_{\beta} T_{\beta'}^{-1})-\varphi_g(T_{\beta} T_{T^{-1}_{\gamma_{1i}}\beta'}^{-1})).$$

Consider the curves $\{\gamma_{1j}, \gamma'_{1j} \mid \; 3\leq j\geq g\}$ given in the following picture:

\[
\begin{tikzpicture}[scale=.7]
\draw[very thick] (-2.5,-2) -- (11,-2);
\draw[very thick] (-2.5,2) -- (11,2);
\draw[very thick] (-2.5,2) arc [radius=2, start angle=90, end angle=270];

\draw[very thick] (0,0) circle [radius=.4];
\draw[very thick] (-2.5,0) circle [radius=.4];
\draw[very thick] (2.5,0) circle [radius=.4];
\draw[very thick] (5,0) circle [radius=.4];
\draw[very thick] (8.5,0) circle [radius=.4];

\draw[thick, dotted] (6,0) -- (7.5,0);
\draw[thick, dotted] (4.5,-1) -- (5.5,-1);
\draw[thick, dotted] (4.5,1) -- (5.5,1);

\draw[thick,dashed] (0,0.4) to [out=70,in=-70] (0,2);
\draw[thick] (0,0.4) to [out=110,in=-110] (0,2);

\draw[thick,dashed] (0,-0.4) to [out=-70,in=70] (0,-2);
\draw[thick] (0,-0.4) to [out=-110,in=110] (0,-2);

\draw[thick] (-2.2,-0.2) to [out=-20,in=200] (2.5,-0.4);
\draw[thick, dashed] (-2.2,-0.2) to [out=-22,in=202] (2.5,-0.4);

\draw[thick] (-2.3,-0.3) to [out=-20,in=200] (5,-0.4);
\draw[thick, dashed] (-2.3,-0.3) to [out=-22,in=202] (5,-0.4);

\draw[thick] (-2.5,-0.4) to [out=-20,in=200] (8.5,-0.4);
\draw[thick, dashed] (-2.5,-0.4) to [out=-22,in=202] (8.5,-0.4);

\draw[thick, dashed] (-2.2,0.2) to [out=20,in=-200] (2.5,0.4);
\draw[thick] (-2.2,0.2) to [out=22,in=-202] (2.5,0.4);

\draw[thick, dashed] (-2.3,0.3) to [out=20,in=-200] (5,0.4);
\draw[thick] (-2.3,0.3) to [out=22,in=-202] (5,0.4);

\draw[thick, dashed] (-2.5,0.4) to [out=20,in=-200] (8.5,0.4);
\draw[thick] (-2.5,0.4) to [out=22,in=-202] (8.5,0.4);

\node at (-2.5,0) {\tiny{1}};
\node at (0,0) {\tiny{2}};
\node at (2.5,0) {\tiny{3}};
\node at (5,0) {\tiny{4}};
\node at (8.5,0) {\tiny{g}};

\node[above] at (1.25,-0.1) {$\gamma'_{13}$};
\node[above] at (3.75,-0.1) {$\gamma'_{14}$};
\node[above] at (7,-0.1) {$\gamma'_{1g}$};
\node[below] at (1.25,-0.1) {$\gamma_{13}$};
\node[below] at (3.75,-0.1) {$\gamma_{14}$};
\node[below] at (7,-0.1) {$\gamma_{1g}$};
\node [above] at (0,2) {$\beta$};
\node [below] at (0,-2) {$\beta'$};

\draw[thick,pattern=north west lines] (11,-2) to [out=130,in=-130] (11,2) to [out=-50,in=50] (11,-2);

\end{tikzpicture}
\]
Notice that $T_{\gamma'_{1i}}T_{\gamma_{1i}}^{-1}\in \mathcal{AB}_{g,1}$ for $3\leq i\leq g.$

Fix an integer j with $3\leq j\leq g.$ Define $\beta''_j=(T_{\gamma_{1(j+1)}}^{x_{(j+1)}}\cdots T_{\gamma_{1g}}^{x_g})(\beta')$ for $j\leq g-1$ and $\beta''_g=\beta'.$ Then the following equality holds
$$T_{\beta} T_{T_{\gamma_{1j}}^{k}(\beta''_j)}^{-1}=(T_\beta T_{\beta'}^{-1})(T_{\beta'}T_{T^{-1}_{\gamma_{1j}'}(\beta)}^{-1}) (T_{T^{-1}_{\gamma_{1j}'}(\beta)}T_{T^k_{\gamma_{1j}}(\beta''_j)}^{-1}).$$
Since $\varphi_g$ is an $\mathcal{AB}_{g,1}$-invariant homomorphism, we have that
\begin{equation}
\label{eq_red_twists}
\begin{aligned}
\varphi_g(T_\beta T_{T^k_{\gamma_{1j}}(\beta''_j)}^{-1})= & \varphi_g(T_\beta T_{\beta'}^{-1})+\varphi_g(T_{\beta'}T_{T^{-1}_{\gamma_{1j}'}(\beta)}^{-1})+ \varphi_g(T_{T^{-1}_{\gamma_{1j}'}(\beta)}T_{T^k_{\gamma_{1j}}(\beta''_j)}^{-1})= \\
= & \varphi_g(T_\beta T_{\beta'}^{-1})+\varphi_g((T_{\gamma_{1j}'} T_{\gamma_{1j}}^{-1})T_{\beta'}T_{T^{-1}_{\gamma_{1j}'}(\beta)}^{-1}(T_{\gamma_{1j}'} T_{\gamma_{1j}}^{-1})^{-1})+ \\
& + \varphi_g((T_{\gamma_{1j}'} T_{\gamma_{1j}}^{-1})T_{T^{-1}_{\gamma_{1j}'}(\beta)}T_{T^k_{\gamma_{1j}}(\beta''_j)}^{-1}(T_{\gamma_{1j}'} T_{\gamma_{1j}}^{-1})^{-1})= \\
= & \varphi_g(T_\beta T_{\beta'}^{-1})+\varphi_g(T_{T^{-1}_{\gamma_{1j}}(\beta')}T_{\beta}^{-1})+ \varphi_g(T_{\beta}T_{T^{k-1}_{\gamma_{1j}}(\beta''_j)}^{-1}).
\end{aligned}
\end{equation}
Applying equality \eqref{eq_red_twists} from $k=x_j$ to $k=1,$ we get that
\begin{equation}
\label{eq_dif_CBP}
\begin{aligned}
\varphi_g(T_\beta T_{\beta''_{j-1}}^{-1})= & \varphi_g(T_\beta T_{T^{x_j}_{\gamma_{1j}}(\beta''_j)}^{-1})= \\
= & x_j \varphi_g(T_\beta T_{\beta'}^{-1})+x_j \varphi_g(T_{T^{-1}_{\gamma_{1j}}(\beta')}T_{\beta}^{-1})+ \varphi_g(T_\beta T_{\beta''_j}^{-1})= \\
= & x_j(\varphi_g(T_\beta T_{\beta'}^{-1})-\varphi_g(T_{\beta}T^{-1}_{T^{-1}_{\gamma_{1j}}(\beta')}))+ \varphi_g(T_\beta T_{\beta''_j}^{-1}).
\end{aligned}
\end{equation}
Then, applying recursively the equality \eqref{eq_dif_CBP} from $j=3$ to $j=g$, we obtain the following formula:
\begin{equation}
\label{eq_dif_CBP_general}
\varphi_g(T_{\beta} T_{(T_{\gamma_{13}}^{x_3}\cdots T_{\gamma_{1g}}^{x_g})(\beta')}^{-1})= \varphi_g(T_{\beta} T_{\beta'}^{-1})+\sum^g_{i=1}x_i(\varphi_g(T_{\beta} T_{\beta'}^{-1})-\varphi_g(T_{\beta} T_{T^{-1}_{\gamma_{1i}}\beta'}^{-1})).
\end{equation}

Next we prove that for $3\leq k\leq g,$
$$\varphi_g(T_\beta T_{\beta'}^{-1})=\varphi_g(T_{\beta}T_{T^{-1}_{\gamma_{1k}}(\beta')}^{-1}).$$

Consider the element $f_k\in \mathcal{AB}_{g,1}$ given by the half twist of the shaded ball depicted in the following figure, that exchanges the holes $3$ and $k.$

\[
\begin{tikzpicture}[scale=.7]
\draw[very thick] (-2.5,-2) -- (11,-2);
\draw[very thick] (-2.5,2) -- (11,2);
\draw[very thick] (-2.5,2) arc [radius=2, start angle=90, end angle=270];

\draw[very thick] (0,0) circle [radius=.4];
\draw[very thick] (-2.5,0) circle [radius=.4];
\draw[very thick] (2.5,0) circle [radius=.4];
\draw[very thick] (5,0) circle [radius=.4];
\draw[very thick] (8,0) circle [radius=.4];

\draw[thick,dashed] (0,0.4) to [out=70,in=-70] (0,2);
\draw[thick] (0,0.4) to [out=110,in=-110] (0,2);

\draw[thick,dashed] (0,-0.4) to [out=-70,in=70] (0,-2);
\draw[thick] (0,-0.4) to [out=-110,in=110] (0,-2);

\draw[thick] (-2.2,-0.2) to [out=-20,in=200] (5,-0.4);
\draw[thick, dashed] (-2.2,-0.2) to [out=-25,in=205] (5,-0.4);

\draw[thick, dotted] (6,0) -- (7,0);
\draw[thick, dotted] (3.5,0) -- (4,0);

\node at (-2.5,0) {\tiny{1}};
\node at (0,0) {\tiny{2}};
\node at (2.5,0) {\tiny{3}};
\node at (5,0) {\tiny{k}};
\node at (8,0) {\tiny{g}};

\node [above] at (0,2) {$\beta$};
\node [below] at (0,-2) {$\beta'$};

\draw[thick, fill=gray!70, nearly transparent] (1.8,0) to [out=90,in=180] (3.75,1.2) to [out=0,in=90] (5.7,0) to [out=-90,in=0] (3.75,-1.2) to [out=180,in=-90] (1.8,0);

\draw[thick, gray, ->] (3.25,1.4) to [out=10,in=170] (4.25,1.4);

\node at (2,-1.7) {$\gamma_{1k}$};

\draw[thick,pattern=north west lines] (11,-2) to [out=130,in=-130] (11,2) to [out=-50,in=50] (11,-2);

\end{tikzpicture}
\]

Observe that $f_k$ leaves $\beta,$ $\beta'$ invariant and sends $\gamma_{1k}$ to
$\gamma_{13}.$ Since $\varphi_{g}$ is $\mathcal{AB}_{g,1}$-invariant,
for $3\leq k\leq g,$ we have that
\begin{equation}
\label{eq_CBP_igual_3}
\begin{aligned}
\varphi_g(T_{\beta}T_{T^{-1}_{\gamma_{1k}}(\beta')}^{-1})= &
\varphi_g(T^{-1}_{\gamma_{1k}}T_{\beta}T_{\beta'}^{-1}T_{\gamma_{1k}})=
\varphi_g(f_kT^{-1}_{\gamma_{1k}}f_k^{-1}T_{\beta}T_{\beta'}^{-1}f_kT_{\gamma_{1k}}f_k^{-1})= \\
= & \varphi_g(T^{-1}_{\gamma_{13}}T_{\beta}T_{\beta'}^{-1}T_{\gamma_{13}})= \varphi_g(T_{\beta}T_{T^{-1}_{\gamma_{13}}(\beta')}^{-1}).
\end{aligned}
\end{equation}
Therefore it is enough to show that
$$\varphi_g(T_\beta T_{\beta'}^{-1})=\varphi_g(T_{\beta}T_{T^{-1}_{\gamma_{13}}(\beta')}^{-1}).$$

Since $\beta_1,$ $\beta_3$ are disjoint with $\beta,$ $\beta',$ $\gamma_{13}$ then
$$T_{\beta}T^{-1}_{T_{\gamma_{13}}(\beta')}=T_{\gamma_{13}}T_{\beta}T_{\beta'}^{-1}T_{\gamma_{13}}^{-1}=(T_{\beta_1}^{-1}T_{\gamma_{13}}T_{\beta_3}^{-1})T_{\beta}T_{\beta'}^{-1}(T_{\beta_1}^{-1}T_{\gamma_{13}}T_{\beta_3}^{-1})^{-1}.$$
Now take $f\in \mathcal{AB}_{g,1}$ given by $f=(T_{\alpha_3}T_{\eta_{34}}T_{\beta_{4}}^{-1})^{-1},$
where $\alpha_3,$ $\eta_{34},$ $\beta_{4}$ are the curves on $\Sigma_{g,1}$ given in the following picture:

\[
\begin{tikzpicture}[scale=.7]
\draw[very thick] (-2.5,-2) -- (11,-2);
\draw[very thick] (-2.5,2) -- (11,2);
\draw[very thick] (-2.5,2) arc [radius=2, start angle=90, end angle=270];

\draw[very thick] (0,0) circle [radius=.4];
\draw[very thick] (-2.5,0) circle [radius=.4];
\draw[very thick] (2.5,0) circle [radius=.4];
\draw[very thick] (5,0) circle [radius=.4];
\draw[very thick] (8,0) circle [radius=.4];

\draw[thick, dotted] (6,0) -- (7,0);

\draw[thick] (2.5,0) circle [radius=.8];
\draw[thick] (2.5,1.2) arc [radius=1.2, start angle=90, end angle=360];
\draw[thick] (3.7,0) to [out=90,in=180] (4.3,0.8);

\draw[thick] (2.5,1.2) -- (4.3,1.2);

\draw[thick] (4.3,0.8) to [out=0,in=150] (4.8,0.3);
\draw[thick] (4.3,1.2) to [out=0,in=210] (4.8,2);
\draw[thick,dashed] (4.8,0.3) to [out=70,in=-70] (4.8,2);

\draw[thick,dashed] (5.2,0.3) to [out=70,in=-70] (5.2,2);
\draw[thick] (5.2,0.3) to [out=110,in=-110] (5.2,2);

\draw[thick,dashed] (0,0.4) to [out=70,in=-70] (0,2);
\draw[thick] (0,0.4) to [out=110,in=-110] (0,2);

\draw[thick,dashed] (0,-0.4) to [out=-70,in=70] (0,-2);
\draw[thick] (0,-0.4) to [out=-110,in=110] (0,-2);

\node [above] at (0,2) {$\beta$};
\node [below] at (0,-2) {$\beta'$};

\node at (3,1.5) {$\eta_{34}$};
\node at (3.4,0.7) {$\alpha_3$};
\node at (5.8,1) {$\beta_{4}$};

\node at (-2.5,0) {\tiny{1}};
\node at (0,0) {\tiny{2}};
\node at (2.5,0) {\tiny{3}};
\node at (5,0) {\tiny{4}};
\node at (8,0) {\tiny{g}};

\draw[thick,pattern=north west lines] (11,-2) to [out=130,in=-130] (11,2) to [out=-50,in=50] (11,-2);

\end{tikzpicture}\]
Since $\alpha_3,$ $\eta_{34},$ $\beta_4$ do not intersect either of $\beta,$ $\beta'$ then $f$ commutes with $T_\beta T^{-1}_{\beta'}.$ Thus,
$$fT_{\gamma_{13}}T_{\beta}T_{\beta'}^{-1}T_{\gamma_{13}}^{-1}f^{-1}=(f(T_{\beta_1}^{-1}T_{\gamma_{13}}T_{\beta_3}^{-1})f^{-1})T_{\beta}T_{\beta'}^{-1}(f(T_{\beta_1}^{-1}T_{\gamma_{13}}T_{\beta_3}^{-1})f^{-1})^{-1}.$$
Observe that
$$\Psi(f)= \left(\begin{matrix}
G & 0 \\
0 & {}^tG^{-1}
\end{matrix}\right)
\qquad
\Psi(T_{\beta_1}^{-1}T_{\gamma_{13}}T_{\beta_3}^{-1})= \left(\begin{matrix}
Id & 0 \\
M & Id
\end{matrix}\right),$$
with
$$G= \left(\begin{matrix}
1 & 0 & 0 & 0\\
0 & 1 & 0 & 0 \\
0 & 0 & 1 & -1 \\
0 & 0 & 0 & 1
\end{matrix}\right) \qquad M=\left(\begin{matrix}
0 & 0 & 1 & 0\\
0 & 0 & 0 & 0 \\
1 & 0 & 0 & 0 \\
0 & 0 & 0 & 0
\end{matrix}\right).$$
Thus
$$\Psi(f(T_{\beta_1}^{-1}T_{\gamma_{13}}T_{\beta_3}^{-1})f^{-1})=\left(\begin{matrix}
Id & 0 \\
{}^tG^{-1}MG^{-1} & Id
\end{matrix}\right)$$
with
$${}^tG^{-1}MG^{-1}= \left(\begin{matrix}
1 & 0 & 0 & 0\\
0 & 1 & 0 & 0 \\
0 & 0 & 1 & 0 \\
0 & 0 & 1 & 1
\end{matrix}\right)
\left(\begin{matrix}
0 & 0 & 1 & 0\\
0 & 0 & 0 & 0 \\
1 & 0 & 0 & 0 \\
0 & 0 & 0 & 0
\end{matrix}\right)
\left(\begin{matrix}
1 & 0 & 0 & 0\\
0 & 1 & 0 & 0 \\
0 & 0 & 1 & 1 \\
0 & 0 & 0 & 1
\end{matrix}\right)=$$
$$=\left(\begin{matrix}
0 & 0 & 1 & 0\\
0 & 0 & 0 & 0 \\
1 & 0 & 0 & 0 \\
1 & 0 & 0 & 0
\end{matrix}\right)
\left(\begin{matrix}
1 & 0 & 0 & 0\\
0 & 1 & 0 & 0 \\
0 & 0 & 1 & 1 \\
0 & 0 & 0 & 1
\end{matrix}\right)
=\left(\begin{matrix}
0 & 0 & 1 & 1\\
0 & 0 & 0 & 0 \\
1 & 0 & 0 & 0 \\
1 & 0 & 0 & 0
\end{matrix}\right).$$
Then
$$\Psi(f(T_{\beta_1}^{-1}T_{\gamma_{13}}T_{\beta_3}^{-1})f^{-1})= \Psi((T_{\beta_1}^{-1}T_{\gamma_{13}}T_{\beta_3}^{-1})(T_{\beta_1}^{-1}T_{\gamma_{14}}T_{\beta_4}^{-1})).$$
As a consequence, by the short exact sequence \eqref{ses_L}, there exists an element $\xi_b\in \mathcal{LTB}_{g,1}$ such that
\begin{equation}
\label{eq_gamma_13,14}
f(T_{\beta_1}^{-1}T_{\gamma_{13}}T_{\beta_3}^{-1})f^{-1}=\xi_b (T_{\beta_1}^{-1}T_{\gamma_{13}}T_{\beta_3}^{-1})(T_{\beta_1}^{-1}T_{\gamma_{14}}T_{\beta_4}^{-1}).
\end{equation}
Since $\varphi_g$ is an $\mathcal{AB}_{g,1}$-invariant homomorphism, $T_{\beta_1},$ $T_{\beta_3},$ $T_{\beta_4}$ commute with $T_{\gamma_{13}},$ $T_{\gamma_{14}},$ $T_\beta,$ $T_{\beta'},$ and $f$ commutes with $T_\beta,$ $T_{\beta'},$ by \eqref{eq_gamma_13,14} we get that
\begin{equation}
\label{eq_CBP_13_to_1314}
\begin{aligned}
\varphi_g(T_{\gamma_{13}}T_{\beta}T_{\beta'}^{-1}T_{\gamma_{13}}^{-1})
= & \varphi_g\left((T_{\beta_1}^{-1}T_{\gamma_{13}}T_{\beta_3}^{-1})T_{\beta}T_{\beta'}^{-1}(T_{\beta_1}^{-1}T_{\gamma_{13}}T_{\beta_3}^{-1})^{-1}\right)= \\
= & \varphi_g\left((f(T_{\beta_1}^{-1}T_{\gamma_{13}}T_{\beta_3}^{-1})f^{-1})T_{\beta}T_{\beta'}^{-1}(f(T_{\beta_1}^{-1}T_{\gamma_{13}}T_{\beta_3}^{-1})f^{-1})^{-1}\right)= \\
= & \varphi_g(\xi_b(T_{\gamma_{13}}T_{\gamma_{14}})T_{\beta}T_{\beta'}^{-1}(T_{\gamma_{13}}T_{\gamma_{14}})^{-1}\xi_b^{-1})= \\
= & \varphi_g(\xi_b)+\varphi_g((T_{\gamma_{13}}T_{\gamma_{14}})T_{\beta}T_{\beta'}^{-1}(T_{\gamma_{13}}T_{\gamma_{14}})^{-1})-\varphi_g(\xi_b)= \\
= & \varphi_g((T_{\gamma_{13}}T_{\gamma_{14}})T_{\beta}T_{\beta'}^{-1}(T_{\gamma_{13}}T_{\gamma_{14}})^{-1})= \\
= & \varphi_g(T_{\beta}T_{T_{\gamma_{13}}T_{\gamma_{14}}(\beta')}^{-1}).
\end{aligned}
\end{equation}
Notice that
$$T_{\beta}T_{T_{\gamma_{13}}T_{\gamma_{14}}(\beta')}^{-1}=(T_{\beta}T_{\beta'}^{-1})(T_{\beta'} T^{-1}_{T_{\gamma'_{13}}^{-1}(\beta)})(T_{T_{\gamma'_{13}}^{-1}(\beta)}T_{T_{\gamma_{13}}T_{\gamma_{14}}(\beta')}^{-1}).$$

Since $\varphi_g$ is a $\mathcal{AB}_{g,1}$-invariant homomorphism and $f_4,$ $T_{\gamma_{13}}T_{\gamma'_{13}}^{-1}\in \mathcal{AB}_{g,1},$ we have that
\begin{equation}
\label{eq_CBP_1314_beta}
\begin{aligned}
\varphi_g(T_{\beta}T_{T_{\gamma_{13}}T_{\gamma_{14}}(\beta')}^{-1})= & \varphi_g(T_{\beta}T_{\beta'}^{-1})+\varphi_g(T_{\beta'} T^{-1}_{T_{\gamma'_{13}}^{-1}(\beta)})+\varphi_g(T_{T_{\gamma'_{13}}^{-1}(\beta)}T_{T_{\gamma_{13}}T_{\gamma_{14}}(\beta')}^{-1})= \\
= & \varphi_g(T_{\beta}T_{\beta'}^{-1})+\varphi_g((T_{\gamma_{13}}^{-1}T_{\gamma'_{13}})T_{\beta'} T^{-1}_{T_{\gamma'_{13}}^{-1}(\beta)}(T_{\gamma_{13}}^{-1}T_{\gamma'_{13}})^{-1}) + \\
& +\varphi_g((T_{\gamma_{13}}^{-1}T_{\gamma'_{13}})T_{T_{\gamma'_{13}}^{-1}(\beta)}T_{T_{\gamma_{13}}T_{\gamma_{14}}(\beta')}^{-1}(T_{\gamma_{13}}^{-1}T_{\gamma'_{13}})^{-1})= \\
= & \varphi_g(T_{\beta}T_{\beta'}^{-1})+\varphi_g(T_{T_{\gamma_{13}}^{-1}(\beta')} T^{-1}_{\beta}) + \varphi_g(T_{\beta} T^{-1}_{T_{\gamma_{14}}(\beta')})= \\
= & \varphi_g(T_{\beta}T_{\beta'}^{-1})-\varphi_g( T_{\beta}T_{T_{\gamma_{13}}^{-1}(\beta')}^{-1}) + \varphi_g(f_4 T_{\beta} T^{-1}_{T_{\gamma_{14}}(\beta')}f_4^{-1})=\\
= & \varphi_g(T_{\beta}T_{\beta'}^{-1})-\varphi_g( T_{\beta}T_{T_{\gamma_{13}}^{-1}(\beta')}^{-1}) + \varphi_g(T_{\beta} T^{-1}_{T_{\gamma_{13}}(\beta')}).
\end{aligned}
\end{equation}
Hence, by equalities \eqref{eq_CBP_13_to_1314},\eqref{eq_CBP_1314_beta}, we get that
\begin{equation*}
\varphi_g(T_{\beta}T^{-1}_{T_{\gamma_{13}}(\beta')})=\varphi_g(T_{\beta}T_{\beta'}^{-1})-\varphi_g( T_{\beta}T_{T_{\gamma_{13}}^{-1}(\beta')}^{-1}) + \varphi_g(T_{\beta} T^{-1}_{T_{\gamma_{13}}(\beta')}).
\end{equation*}
Then $\varphi_g(T_{\beta}T_{\beta'}^{-1})=\varphi_g( T_{\beta}T_{T_{\gamma_{13}}^{-1}(\beta')}^{-1}),$ and by \eqref{eq_CBP_igual_3} we get that
\begin{equation}
\label{eq_CBP_iguals_general}
\varphi_g(T_{\beta}T_{\beta'}^{-1})=\varphi_g( T_{\beta}T_{T_{\gamma_{1k}}^{-1}(\beta')}^{-1}).
\end{equation}

Finally, by equalities \eqref{eq_dif_CBP_general}, \eqref{eq_CBP_iguals_general}, \eqref{eq_CBP_product}, we get that
$$\varphi_g(T_\nu T_{\nu'}^{-1})=\varphi_g( T_{\beta} T_{\beta'}^{-1}),$$
i.e. $\varphi_g$ takes the same value on all CBP-twists of genus 1.

\textbf{2)} We prove that $\varphi_g$ is zero on all CBP-twists of genus $2$ using 1).

By Proposition \eqref{prop_CBP_1} we know that every CBP-twist of genus $2$ is a product of two CBP-twists of genus $1,$ what is more, by the proof of Proposition \eqref{prop_CBP_1}, we know that given
$T_\varepsilon T_{\varepsilon'}^{-1}$ a CBP-twist of genus $2$, there exists a curve $\zeta$ such that $T_\varepsilon T_{\zeta}^{-1}$ and $T_\zeta T_{\varepsilon'}^{-1}$ are CBP-twists of genus $1.$

Thus we have that
\begin{align*}
\varphi_g(T_\varepsilon T_{\varepsilon'}^{-1})= & \varphi_g(T_\varepsilon T_{\zeta}^{-1})+\varphi_g(T_\zeta T_{\varepsilon'}^{-1}) \\
= & \varphi_g(T_\varepsilon T_{\zeta}^{-1})-\varphi_g(T_{\varepsilon'}T^{-1}_\zeta),
\end{align*}
and since, by 1), we know that $\varphi_g(T_\varepsilon T_{\zeta}^{-1})=\varphi_g(T_{\varepsilon'}T^{-1}_\zeta),$ then 
$$\varphi_g(T_\varepsilon T_{\varepsilon'}^{-1})=0,$$
i.e. $\varphi_g$ is zero on all CBP-twists of genus $2.$

\textbf{3)} We prove that $\varphi_g$ is zero on all CBP-twists of genus $1$ using 1), 2) and the lantern relation.

Consider the following curves in the standardly embedded surface $\Sigma_{g,1}:$

\[
\begin{tikzpicture}[scale=.7]
\draw[very thick] (-4.5,-2) -- (10,-2);
\draw[very thick] (-4.5,2) -- (10,2);
\draw[very thick] (-4.5,2) arc [radius=2, start angle=90, end angle=270];
\draw[very thick] (-4.5,0) circle [radius=.4];
\draw[very thick] (-2,0) circle [radius=.4];

\draw[very thick] (4.5,0) circle [radius=.4];
\draw[very thick] (7,0) circle [radius=.4];

\draw[thick, dashed] (-2,0.4) to [out=0,in=0] (-2,2);
\draw[thick, dashed] (-2,-0.4) to [out=0,in=0] (-2,-2);
\draw[thick, dashed] (-4.5,-0.4) to [out=0,in=0] (-4.5,-2);

\draw[thick] (-4.5,-0.4) to [out=180,in=90] (-5,-1.2);
\draw[<-,thick](-5,-1.2) to [out=-90,in=180] (-4.5,-2);
\draw[thick] (-2,-0.4) to [out=180,in=90] (-2.5,-1.2);
\draw[<-,thick](-2.5,-1.2) to [out=-90,in=180] (-2,-2);
\draw[thick] (-2,0.4) to [out=180,in=-90] (-2.5,1.2);
\draw[<-,thick](-2.5,1.2) to [out=90,in=180] (-2,2);

\draw[thick, dashed] (-4.5,0.4) to [out=30,in=180] (-2,1);
\draw[thick, dashed] (-2,1) to [out=0,in=90] (-0.7,0);
\draw[thick, dashed] (-0.7,0) to [out=-90,in=50] (-1,-2);

\draw[thick] (-4.5,0.4) to [out=0,in=180] (-2,0.7);
\draw[thick] (-2,0.7) to [out=0,in=90] (-1.3,0);
\draw[thick] (-1.3,0) to [out=-90,in=130] (-1,-2);

\draw[thick] (-4.15,-0.2) to [out=-20,in=200] (-2.35,-0.2);
\draw[thick, dashed] (-4.15,-0.2) to [out=20,in=160] (-2.35,-0.2);

\draw[thick, dashed] (-4.15,0.2) to [out=20,in=160] (-2.35,0.2);

\draw[thick] (-4.15,0.2) to [out=-20,in=-90] (-3.5,0.5);
\draw[thick] (-3.5,0.5) to [out=90,in=180] (-2,1.5);
\draw[thick] (-2,1.5) to [out=0,in=180] (-0.5,1.5);
\draw[thick] (-0.5,1.5) to [out=0,in=220] (0,2);
\draw[dashed,thick] (0,2) to [out=-50,in=90] (0.5,0) to [out=-90,in=50] (0,-2);
\draw[thick] (-2.35,0.2) to [out=200,in=-90] (-3,0.5);
\draw[thick] (-3,0.5) to [out=90,in=180] (-2,1.1);
\draw[thick] (-2,1.1) to [out=0,in=180] (-1,1.1);
\draw[thick] (-1,1.1) to [out=0,in=140] (0,-2);

\node [above] at (-2,2) {$\beta'_1$};
\node [below] at (-2,-2) {$\beta_1$};
\node [below] at (-1,-2) {$\beta'_2$};
\node [below] at (-4.5,-2) {$\beta_2$};
\node [above] at (0,2) {$\beta'_3$};
\node [below] at (-3.25,-0.3) {$\beta_3$};

\draw[dashed,thick] (1.5,2) to [out=-50,in=90] (2,0) to [out=-90,in=50] (1.5,-2);
\draw[thick] (1.5,-2) to [out=130,in=-90] (1,0);
\draw[->,thick] (1.5,2) to [out=230,in=90] (1,0);

\node [above] at (1.5,2) {$\gamma$};

\draw[thick] (7,0.4) to [out=180,in=-90] (6.5,1.2);
\draw[<-,thick](6.5,1.2) to [out=90,in=180] (7,2);
\draw[thick] (4.5,-0.4) to [out=180,in=90] (4,-1.2);
\draw[<-,thick](4,-1.2) to [out=-90,in=180] (4.5,-2);
\draw[thick] (4.5,0.4) to [out=180,in=-90] (4,1.2);
\draw[<-,thick](4,1.2) to [out=90,in=180] (4.5,2);

\draw[thick, dashed] (4.5,0.4) to [out=0,in=0] (4.5,2);
\draw[thick, dashed] (4.5,-0.4) to [out=0,in=0] (4.5,-2);
\draw[thick, dashed] (7,0.4) to [out=0,in=0] (7,2);

\draw[thick] (4.85,-0.2) to [out=-20,in=200] (6.65,-0.2);
\draw[thick, dashed] (4.85,-0.2) to [out=20,in=160] (6.65,-0.2);

\draw[thick, dashed] (4.85,0.2) to [out=20,in=160] (6.65,0.2);
\draw[thick] (4.85,0.2) to [out=-20,in=-90] (5.5,0.5);
\draw[thick] (5.5,0.5) to [out=90,in=0] (4.5,1.1);
\draw[thick] (4.5,1.1) to [out=180,in=0] (3,1.1);
\draw[thick] (3,1.1) to [out=180,in=130] (2.8,-2);

\draw[thick] (6.65,0.2) to [out=200,in=-90] (6,0.5);
\draw[thick] (6,0.5) to [out=90,in=0] (4.5,1.5);
\draw[thick] (4.5,1.5) to [out=180,in=0] (2.5,1.5);
\draw[thick] (2.5,1.5) to [out=180,in=230] (2.8,2);
\draw[dashed,thick] (2.8,2) to [out=-50,in=90] (3.3,0) to [out=-90,in=50] (2.8,-2);

\draw[thick] (7,0.4) to [out=150,in=0] (4.5,1);
\draw[thick] (4.5,1) to [out=180,in=90] (3,0);
\draw[thick] (3,0) to [out=-90,in=130] (3.5,-2);

\draw[thick, dashed] (7,0.4) to [out=180,in=0] (4.5,0.7);
\draw[thick, dashed] (4.5,0.7) to [out=180,in=90] (3.8,0);
\draw[thick, dashed] (3.8,0) to [out=-90,in=50] (3.5,-2);

\node [below] at (4.5,-2) {$\zeta_1$};
\node [above] at (7,2) {$\zeta_2$};
\node [above] at (3,2) {$\zeta'_3$};
\node [above] at (4.5,2) {$\zeta'_1$};
\node [below] at (3.5,-2) {$\zeta'_2$};
\node [below] at (5.75,-0.3) {$\zeta_3$};

\draw[thick,pattern=north west lines] (10,-2) to [out=130,in=-130] (10,2) to [out=-50,in=50] (10,-2);

\end{tikzpicture}
\]
Observe that for $i=1,2,3,$ $T_{\beta_i}T^{-1}_{\beta'_i}$ are CBP-twists of genus $1,$ $T_{\zeta_i}T^{-1}_{\zeta'_i}$ are CBP-twists of genus $2,$ and $T_\gamma\in \mathcal{LTB}_{g,1}.$
Using the following lantern relations
\[
\begin{tikzpicture}[scale=.7]
\draw[very thick] (-2.8,0.8) to [out=20,in=-110] (-0.8,2.8);
\draw[very thick] (2.8,0.8) to [out=160,in=-70] (0.8,2.8);
\draw[very thick] (-2.8,-0.8) to [out=-20,in=110] (-0.8,-2.8);
\draw[very thick] (2.8,-0.8) to [out=-160,in=70] (0.8,-2.8);

\draw[thick] (-2.8,0.8) to [out=-60,in=60] (-2.8,-0.8) to [out=120,in=-120] (-2.8,0.8);
\draw[thick] (-0.8,2.8) to [out=30,in=150] (0.8,2.8) to [out=-150,in=-30] (-0.8,2.8);
\draw[thick] (2.8,0.8) to [out=-60,in=60] (2.8,-0.8);
\draw[thick, dashed] (2.8,-0.8) to [out=120,in=-120] (2.8,0.8);
\draw[thick, dashed] (-0.8,-2.8) to [out=30,in=150] (0.8,-2.8);
\draw[thick] (0.8,-2.8) to [out=-150,in=-30] (-0.8,-2.8);

\draw[thick, dashed] (-1.55,-1.55) to [out=60,in=210] (1.55,1.55);
\draw[thick] (-1.55,-1.55) to [out=30,in=240] (1.55,1.55);

\draw[thick] (1.55,-1.55) to [out=120,in=-30] (-1.55,1.55);
\draw[thick, dashed] (1.55,-1.55) to [out=150,in=-60] (-1.55,1.55);

\draw[thick] (-1.1,2.1) to [out=-80,in=80] (-1.1,-2.1);
\draw[thick] (1.1,2.1) to [out=-100,in=100] (1.1,-2.1);

\draw[thick, dashed] (-1.1,2.1) to [out=20,in=160] (1.1,2.1);
\draw[thick, dashed] (-1.1,-2.1) to [out=20,in=160] (1.1,-2.1);

\node [left] at (-3,0) {$\beta_2$};
\node [right] at (3,0) {$\beta_1$};
\node [above] at (0,3) {$\gamma$};
\node [below] at (0,-3) {$\beta_3$};

\node [right] at (0.8,0) {$\beta'_3$};
\node [below] at (-0.3,1.7) {$\beta'_1$};
\node [above] at (-0.3,-1.7) {$\beta'_2$};

\end{tikzpicture}
\;\;
\begin{tikzpicture}[scale=.7]
\draw[very thick] (-2.8,0.8) to [out=20,in=-110] (-0.8,2.8);
\draw[very thick] (2.8,0.8) to [out=160,in=-70] (0.8,2.8);
\draw[very thick] (-2.8,-0.8) to [out=-20,in=110] (-0.8,-2.8);
\draw[very thick] (2.8,-0.8) to [out=-160,in=70] (0.8,-2.8);

\draw[thick] (-2.8,0.8) to [out=-60,in=60] (-2.8,-0.8) to [out=120,in=-120] (-2.8,0.8);
\draw[thick] (-0.8,2.8) to [out=30,in=150] (0.8,2.8) to [out=-150,in=-30] (-0.8,2.8);
\draw[thick] (2.8,0.8) to [out=-60,in=60] (2.8,-0.8);
\draw[thick, dashed] (2.8,-0.8) to [out=120,in=-120] (2.8,0.8);
\draw[thick, dashed] (-0.8,-2.8) to [out=30,in=150] (0.8,-2.8);
\draw[thick] (0.8,-2.8) to [out=-150,in=-30] (-0.8,-2.8);

\draw[thick, dashed] (-1.55,-1.55) to [out=60,in=210] (1.55,1.55);
\draw[thick] (-1.55,-1.55) to [out=30,in=240] (1.55,1.55);

\draw[thick] (1.55,-1.55) to [out=120,in=-30] (-1.55,1.55);
\draw[thick, dashed] (1.55,-1.55) to [out=150,in=-60] (-1.55,1.55);

\draw[thick] (-1.1,2.1) to [out=-80,in=80] (-1.1,-2.1);
\draw[thick] (1.1,2.1) to [out=-100,in=100] (1.1,-2.1);

\draw[thick, dashed] (-1.1,2.1) to [out=20,in=160] (1.1,2.1);
\draw[thick, dashed] (-1.1,-2.1) to [out=20,in=160] (1.1,-2.1);

\node [left] at (-3,0) {$\zeta_1$};
\node [right] at (3,0) {$\zeta_2$};
\node [above] at (0,3) {$\gamma$};
\node [below] at (0,-3) {$\zeta_3$};

\node [right] at (0.8,0) {$\zeta'_3$};
\node [below] at (-0.3,1.7) {$\zeta'_2$};
\node [above] at (-0.3,-1.7) {$\zeta'_1$};

\end{tikzpicture}
\]
we get the following equalities:
$$(T_{\beta'_2}T^{-1}_{\beta_2})(T_{\beta'_1}T^{-1}_{\beta_1})(T_{\beta'_3}T^{-1}_{\beta_3})=T_\gamma,$$
$$(T_{\zeta'_1}T^{-1}_{\zeta_1})(T_{\zeta'_2}T^{-1}_{\zeta_2})(T_{\zeta'_3}T^{-1}_{\zeta_3})=T_\gamma.$$
Thus,
$$T_{\beta'_3}T^{-1}_{\beta_3}=(T_{\beta_1}T^{-1}_{\beta'_1})(T_{\beta_2}T^{-1}_{\beta'_2})(T_{\zeta'_1}T^{-1}_{\zeta_1})(T_{\zeta'_2}T^{-1}_{\zeta_2})(T_{\zeta'_3}T^{-1}_{\zeta_3}).$$

In general, if $T_\nu T_{\nu'}^{-1}$ a CBP-twist of genus $1,$ by Proposition \eqref{prop_trans_B}, there exists an element $h\in\mathcal{B}_{g,1}$
such that $T_\nu T_{\nu'}^{-1}=hT_{\beta'_3} T_{\beta_3}^{-1}h^{-1}.$ Then
\begin{align*}
T_\nu T_{\nu'}^{-1}= & (hT_{\beta_1}T^{-1}_{\beta'_1}h^{-1})(hT_{\beta_2}T^{-1}_{\beta'_2}h^{-1})(hT_{\zeta'_1}T^{-1}_{\zeta_1}h^{-1})(hT_{\zeta'_2}T^{-1}_{\zeta_2}h^{-1})(hT_{\zeta'_3}T^{-1}_{\zeta_3}h^{-1}) \\
= & (T_{h(\beta_1)}T^{-1}_{h(\beta'_1)})(T_{h(\beta_2)}T^{-1}_{h(\beta'_2)})(T_{h(\zeta'_1)}T^{-1}_{h(\zeta_1)})(T_{h(\zeta'_2)}T^{-1}_{h(\zeta_2)})(T_{h(\zeta'_3)}T^{-1}_{h(\zeta_3)}).
\end{align*}
Since $T_{\zeta'_1}T^{-1}_{\zeta_1},$ $T_{\zeta'_2}T^{-1}_{\zeta_2},$ $T_{\zeta'_3}T^{-1}_{\zeta_3}$ are CBP-twists of genus $2$ and $h\in \mathcal{B}_{g,1}$ then
$T_{h(\zeta'_1)}T^{-1}_{h(\zeta_1)},$ $T_{h(\zeta'_2)}T^{-1}_{h(\zeta_2)},$ $T_{h(\zeta'_3)}T^{-1}_{h(\zeta_3)}$ also are CBP-twists of genus $2.$ Thus, by 2), we get that
\begin{align*}
\varphi_g(T_\nu T_{\nu'}^{-1})= & \varphi_g(T_{h(\beta_1)}T^{-1}_{h(\beta'_1)})+\varphi_g(T_{h(\beta_2)}T^{-1}_{h(\beta'_2)})= \\
= & \varphi_g(T_{h(\beta_1)}T^{-1}_{h(\beta'_1)})-\varphi_g(T_{h(\beta'_2)}T^{-1}_{h(\beta_2)}),
\end{align*}
and since, by 1), we know that $\varphi_g(T_{h(\beta_1)}T^{-1}_{h(\beta'_1)})=\varphi_g(T_{h(\beta'_2)}T^{-1}_{h(\beta_2)}),$ then
$$\varphi_g(T_\nu T_{\nu'}^{-1})=0,$$
i.e. $\varphi_g$ is zero on all CBP-twists.

\textbf{II)} As a consequence of I), $\varphi_g$ factors through $IA$. As the action on the fundamental group of the inner handlebody $\mathcal{H}_g$ induces
a surjective map $\mathcal{AB}_{g,1} \rightarrow Aut (\pi_1(\mathcal{H}_g))$, we can view $\varphi_g$ as an $Aut (\pi_1 (\mathcal{H}_g))$-invariant map $\varphi_g: IA \rightarrow A$. Let $\alpha_1, \cdots ,\alpha_g$ denote the generators of $\pi_1(\mathcal{H}_g)$. According to Magnus \cite{magnus1}, the group $IA$ is normally generated as a subgroup of $Aut (\pi_1(\mathcal{H}_g))$ by the automorphism $K_{12}$ given by $K_{12}(\alpha_1)=\alpha_2\alpha_1\alpha_2^{-1}$ and $K_{12}(\alpha_i)=\alpha_i$ for $i\geq 2$. By invariance, $\varphi_g$
is determinated by its value on $K_{12}$. So it is enough to show that $\varphi_g(K_{12})=0$.

\noindent Consider the automorphism $f\in Aut(\pi_1(\mathcal{H}_g))$ given by
$f(\alpha_3)=\alpha_3\alpha_2$ and $f(\alpha_i)=\alpha_i$ for $i\neq 3,$
with inverse $f^{-1}\in Aut(\pi_1(\mathcal{H}_g))$ given by
$f^{-1}(\alpha_3)=\alpha_3\alpha_2^{-1}$ and $f^{-1}(\alpha_i)=\alpha_i$ for $i\neq 3.$

Take the element $K_{13}\in Aut (\pi_1(\mathcal{H}_g))$ given by
$$K_{13}(\alpha_1)=\alpha_3\alpha_1\alpha_3^{-1}\quad \text{and} \quad K_{13}(\alpha_i)=\alpha_i\quad \text{for }i\geq 2.$$
Clearly $K_{13}$ is an element of $IA.$ Observe that
$$fK_{13}f^{-1}(\alpha_1)=\alpha_3\alpha_2\alpha_1\alpha_2^{-1}\alpha_3^{-1}\quad \text{and} \quad fK_{13}f^{-1}(\alpha_i)=\alpha_i\quad \text{for }i\geq 2.$$
Then $fK_{13}f^{-1}=K_{12}K_{13},$
and since $\varphi_g$ is an $Aut (\pi_1 (\mathcal{H}_g))$-invariant map, we have that
$$\varphi_g(K_{13})=\varphi_g(fK_{13}f^{-1})=\varphi_g(K_{12}K_{13})=\varphi_g(K_{12})+\varphi_g(K_{13}).$$
Therefore $\varphi_g(K_{12})=0,$ as desired.
\end{proof}

Now we are ready to compute the $\mathcal{AB}_{g,1}$-invariant homomorphisms on the Torelli group.

\begin{lema}
\label{lem_abinv}
For $g\geq 4,$ there is an isomorphism
\begin{align*}
\Lambda: A_2 & \longrightarrow Hom(\mathcal{T}_{g,1},A)^{\mathcal{AB}_{g,1}} \\
x & \longmapsto \mu^x_g= \varphi_g^{x}\circ \sigma
\end{align*}
where $\varphi_g^x$ is the map defined in Lemma \eqref{lem_B_3-inv}.
\end{lema}

\begin{proof}
First of all, notice that $\Lambda$ is well defined because the BCJ-homomorphism $\sigma$ is $Sp_{2g}(Z)$-equivariant, $\varphi^x$ is $\mathcal{AB}_{g,1}$-invariant and then $\mu^x_g$ is also $\mathcal{AB}_{g,1}$-invariant.
Moreover, it is clear that $\Lambda$ is injective.
Next we will prove that $\Lambda$ is surjective.

From the work of D. Johnson (see \cite{jon_3}), we know that there is a short exact sequence
\begin{equation*}
\xymatrix@C=7mm@R=10mm{1 \ar@{->}[r] & \mathfrak{B}_2 \ar@{->}[r]^-{i} & H_1(\mathcal{T}_{g,1};\mathbb{Z}) \ar@{->}[r]^-{\tau} & \extp^3H \ar@{->}[r] & 1 .}
\end{equation*}
Taking the 5-term exact sequence associated to the above short exact sequence we get the exact sequence
\begin{equation*}
\xymatrix@C=7mm@R=10mm{0 \ar@{->}[r] & Hom(\extp^3H,A)  \ar@{->}[r]^-{inf} & Hom(H_1(\mathcal{T}_{g,1};\mathbb{Z}),A) \ar@{->}[r]^-{res} & Hom \left(\mathfrak{B}_2 ,A\right).}
\end{equation*}
Taking $GL_g(\mathbb{Z})$-invariants we get another exact sequence
\begin{equation*}
\xymatrix@C=4mm@R=10mm{0 \ar@{->}[r] & Hom(\extp^3H,A)^{GL_g(\mathbb{Z})}  \ar@{->}[r] & Hom(H_1(\mathcal{T}_{g,1};\mathbb{Z}),A)^{GL_g(\mathbb{Z})} \ar@{->}[r] & Hom\left(\mathfrak{B}_2 ,A\right)^{GL_g(\mathbb{Z})}.}
\end{equation*}
By Lemma \eqref{lema_hom_extp_H_d}, $Hom(\extp^3 H, A)^{GL_g(\mathbb{Z})}=0,$ for $g\geq 4.$ As a consequence, we have an injection 
$$
 Hom(H_1(\mathcal{T}_{g,1};\mathbb{Z}),A)^{GL_g(\mathbb{Z})} \hookrightarrow  Hom\left(\mathfrak{B}_2 ,A\right)^{GL_g(\mathbb{Z})}.
$$
By Lemma \eqref{lem_B_2-inv}, all elements of $Hom(\mathfrak{B}_2,A)^{GL_g(\mathbb{Z})}$ are $\varphi^{x_1,x_2}$ with $x_1,x_2\in A_2.$
Next we check which elements $\varphi^{x_1,x_2}$ can be extended to
$Hom(H_1(\mathcal{T}_{g,1};\mathbb{Z}),A)^{GL_g(\mathbb{Z})}.$

Suppose that for some $g\geq 4$ there exists an element $h_g\in Hom(H_1(\mathcal{T}_{g,1};\mathbb{Z}),A)^{GL_g(\mathbb{Z})}$ which restricted on $\mathfrak{B}_2$ coincides with $\varphi_g^{x_1,x_2},$ for some $x_1,x_2\in A_2.$
By Lemma \eqref{lem_TB,TA}, $h_g$ has to be zero on $\mathcal{TB}_{g,1},$ in particular on $\mathcal{KB}_{g,1}=\mathcal{K}_{g,1}\cap \mathcal{B}_{g,1}.$

Then, if we consider the element $T_\gamma \in \mathcal{K}_{g,1}$ with $\gamma$ depicted in the following figure
\begin{figure}[H]
\begin{center}
\begin{tikzpicture}[scale=.5]
\draw[very thick] (-4.5,-2) -- (0,-2);
\draw[very thick] (-4.5,2) -- (0,2);
\draw[very thick] (-4.5,2) arc [radius=2, start angle=90, end angle=270];
\draw[very thick] (-4.5,0) circle [radius=.4];
\draw[very thick] (-2,0) circle [radius=.4];

\draw[dashed,thick] (-3.25,2) to [out=-50,in=90] (-2.7,0) to [out=-90,in=50] (-3.25,-2);

\draw[thick] (-3.25,-2) to [out=130,in=-90] (-3.8,0);

\draw[->,thick] (-3.25,2) to [out=230,in=90] (-3.8,0);

\node [above] at (-4,-1.8) {$\gamma$};

\draw[thick,pattern=north west lines] (0,-2) to [out=130,in=-130] (0,2) to [out=-50,in=50] (0,-2);

\end{tikzpicture}
\end{center}
\end{figure}
we have that $T_\gamma \in \mathcal{TB}_{g,1},$ and by definition of the BCJ-homomorphism, $\sigma(T_\gamma)=\overline{A_1}\overline{B_1}.$ As a consequence, if $\varphi_g^{x_1,x_2}$ can be extended to $Hom(H_1(\mathcal{T}_{g,1};\mathbb{Z}),A)^{GL_g(\mathbb{Z})},$ then $\varphi_g^{x_1,x_2}$ has to be zero on $\overline{A_1}\overline{B_1},$ and so $x_2$ has to be zero.
Finally, notice that $\varphi_g^{x,0}\in Hom(\mathfrak{B}_2,A)^{GL_g(\mathbb{Z})}$ is the restriction of $\mu_g^x \in Hom(H_1(\mathcal{T}_{g,1};\mathbb{Z}),A)^{GL_g(\mathbb{Z})}$ to $\mathfrak{B}_2.$

\end{proof}

Next we show that the family of homomorphisms $\{\mu_g^x\}_g$ that we have given in Lemma \eqref{lem_abinv} reassemble into an invariant of homology spheres and we identify such invariant.

\begin{lema}
\label{lem_stab}
The homomorphisms $\{\mu_g^x\}_g,$ defined in Lemma \eqref{lem_abinv}, are compatible with the stabilization map.
\end{lema}
\begin{proof}
By definition, $\sigma: \mathcal{T}_{g,1}\rightarrow \mathfrak{B}_3$ and $\varphi_g^x:\mathfrak{B}_3 \rightarrow A,$ are compatible with the stabilization map. As a consequence, the compositions of this maps, $\{\mu_g^x\}_g,$ are compatible with the stabilization map.
\end{proof}

By Lemmas \eqref{lem_TB,TA}, \eqref{lem_stab}, the $\mathcal{AB}_{g,1}$-invariant homomorphisms $\{\mu_g^x\}_g,$ in Lemma \eqref{lem_abinv}, are compatible with the stabilization map and zero on $\mathcal{TA}_{g,1},$ $\mathcal{TB}_{g,1}.$
Then, by bijection \eqref{bij_ZHS_torelli}, the family of homomorphism $\{\mu^x_g\}_g$ reassemble into an invariant of homology spheres.

\begin{lema}
The homomorphisms $\{\mu_g^x\}_g,$ defined in Lemma \eqref{lem_abinv}, take the value $x$ on the Poincaré sphere.
\end{lema}

\begin{proof}
By Example 9.G.2. in \cite{rolf}, we know that the Poicaré sphere is obtained as a $1$ Dehn surgery along the right hand trefoil knot.
By Lemmas V.1.1, V.1.2, V.1.3 in \cite{akbu},
we have that this Dehn surgery is equivalent to a Heegaard splitting given by the Dehn twist about a right hand trefoil knot.

Next, we follow the construction of the Heegaard splitting given in the proofs of Lemmas V.1.1, V.1.2, V.1.3 in \cite{akbu}.
Take a Seifert surface $S$ of the right hand trefoil knot and thicken this surface obtaining a handlebody $S\times [0,1].$ Consider the boundary of $S\times [0,1]$ and the knot $K=\partial S\times \{\frac{1}{2}\},$  which is the right hand trefoil knot $K$ in $\partial S\times [0,1].$
Taking this Heegaard surface with the Dehn twist about $K$ we get the Heegaard model
\[
\begin{tikzpicture}[scale=.8]
\draw[very thick] (-4.5,-2) -- (-2,-2);
\draw[very thick] (-4.5,2) -- (-2,2);
\draw[very thick] (-4.5,2) arc [radius=2, start angle=90, end angle=270];
\draw[very thick] (-2,2) arc [radius=2, start angle=90, end angle=-90];
\draw[very thick] (-4.5,0) circle [radius=.4];
\draw[very thick] (-2,0) circle [radius=.4];

\draw[thick] (-4.5,0.4) to [out=0,in=180] (-2,-0.4) ;
\draw[dashed, thick] (-4.5,-0.4) to [out=0,in=180] (-2,0.4);
\draw[dashed, thick] (-2,-0.4) to [out=0,in=-60] (-1.2,1) to [out=120,in=0] (-3.25,2);
\draw[dashed, thick] (-4.5,0.4) to [out=180,in=120] (-5.3,-1) to [out=-60,in=180] (-3.25,-2);
\draw[->,thick] (-2,0.4) to [out=0,in=60] (-1.2,-1);
\draw[thick] (-1.2,-1) to [out=240,in=0] (-3.25,-2);
\draw[thick] (-4.5,-0.4) to [out=180,in=240] (-5.3,1);
\draw[->,thick](-3.25,2) to [out=180,in=60] (-5.3,1);
\end{tikzpicture}
\]
which corresponds to the $1$ Dehn surgery of right hand trefoil knot.

We show that the image of $T_K$ by $\mu_g^x$ is $x.$ We first compute the image of $T_K$ by the BCJ-morphism $\sigma$ and then we apply $\varphi_g^x.$
To compute $\sigma(T_K),$ we write $T_K$ as a product of Dehn twists on unknotted simple closed curves.

Taking the product of left Dehn twists about the simple closed curves $a_1$ and $a_2$,
we have that $T_{a_1}^{-1}T_{a_2}^{-1}(K)=K'$ as in the following figure.
\[
\begin{tikzpicture}[scale=.8]
\draw[very thick] (-4.5,-2) -- (-2,-2);
\draw[very thick] (-4.5,2) -- (-2,2);
\draw[very thick] (-4.5,2) arc [radius=2, start angle=90, end angle=270];
\draw[very thick] (-2,2) arc [radius=2, start angle=90, end angle=-90];
\draw[very thick] (-4.5,0) circle [radius=.4];
\draw[very thick] (-2,0) circle [radius=.4];

\draw[thick] (-4.5,0.4) to [out=0,in=180] (-2,-0.4) ;
\draw[dashed, thick] (-4.5,-0.4) to [out=0,in=180] (-2,0.4);
\draw[dashed, thick] (-2,-0.4) to [out=0,in=-60] (-1.2,1) to [out=120,in=0] (-3.25,2);
\draw[dashed, thick] (-4.5,0.4) to [out=180,in=120] (-5.3,-1) to [out=-60,in=180] (-3.25,-2);
\draw[->,thick] (-2,0.4) to [out=0,in=60] (-1.2,-1);
\draw[thick] (-1.2,-1) to [out=240,in=0] (-3.25,-2);
\draw[thick] (-4.5,-0.4) to [out=180,in=240] (-5.3,1);
\draw[->,thick](-3.25,2) to [out=180,in=60] (-5.3,1);

\draw[->,blue, thick] (-5.3,0) to [out=90,in=180] (-4.5,0.8);
\draw[blue, thick] (-3.7,0) to [out=90,in=0] (-4.5,0.8);
\draw[blue, thick] (-5.3,0) to [out=-90,in=180] (-4.5,-0.8) to [out=0,in=-90] (-3.7,0);
\node [above,blue] at (-4.5,0.8) {$a_1$};

\draw[->,blue, thick] (-2.8,0) to [out=90,in=180] (-2,0.8);
\draw[blue, thick] (-1.2,0) to [out=90,in=0] (-2,0.8);
\draw[blue, thick] (-2.8,0) to [out=-90,in=180] (-2,-0.8) to [out=0,in=-90] (-1.2,0);
\node [above,blue] at (-2,0.8) {$a_2$};

\draw[thick ,->] (1.5,0) -- (3.5,0);
\node [above] at (2.5,0) {$T_{a_1}^{-1}T_{a_2}^{-1}$};
\end{tikzpicture}
\qquad
\begin{tikzpicture}[scale=.8]
\draw[very thick] (-4.5,-2) -- (-2,-2);
\draw[very thick] (-4.5,2) -- (-2,2);
\draw[very thick] (-4.5,2) arc [radius=2, start angle=90, end angle=270];
\draw[very thick] (-2,2) arc [radius=2, start angle=90, end angle=-90];
\draw[very thick] (-4.5,0) circle [radius=.4];
\draw[very thick] (-2,0) circle [radius=.4];

\draw[dashed,thick] (-4.5,0.4) to [out=0,in=180] (-2,-0.4) ;
\draw[thick] (-4.5,-0.4) to [out=0,in=180] (-2,0.4);

\draw[dashed, thick] (-2,0.4) to [out=0,in=-60] (-1.6,1.2) to [out=120,in=0] (-3.25,2);
\draw[dashed, thick] (-4.5,-0.4) to [out=180,in=120] (-4.9,-1.2) to [out=-60,in=180](-3.25,-2);
\draw[->,thick] (-2,-0.4) to [out=0,in=60] (-1.6,-1.2);
\draw[thick] (-3.25,-2) to [out=0,in=240] (-1.6,-1.2);
\draw[thick] (-4.5,0.4) to [out=180,in=240] (-4.9,1.2);
\draw[->,thick] (-3.25,2) to [out=180,in=60] (-4.9,1.2);
\end{tikzpicture}
\]

Taking the left Dehn twist about the simple closed curve $c$ depicted in the following figure, we have that $T_c^{-1}(K')=\gamma.$

\[
\begin{tikzpicture}[scale=.8]
\draw[very thick] (-4.5,-2) -- (-2,-2);
\draw[very thick] (-4.5,2) -- (-2,2);
\draw[very thick] (-4.5,2) arc [radius=2, start angle=90, end angle=270];
\draw[very thick] (-2,2) arc [radius=2, start angle=90, end angle=-90];
\draw[very thick] (-4.5,0) circle [radius=.4];
\draw[very thick] (-2,0) circle [radius=.4];

\draw[dashed,thick] (-4.5,0.4) to [out=0,in=180] (-2,-0.4) ;
\draw[thick] (-4.5,-0.4) to [out=0,in=180] (-2,0.4);

\draw[dashed, thick] (-2,0.4) to [out=0,in=-60] (-1.6,1.2) to [out=120,in=0] (-3.25,2);
\draw[dashed, thick] (-4.5,-0.4) to [out=180,in=120] (-4.9,-1.2) to [out=-60,in=180](-3.25,-2);
\draw[->,thick] (-2,-0.4) to [out=0,in=60] (-1.6,-1.2);
\draw[thick] (-3.25,-2) to [out=0,in=240] (-1.6,-1.2);
\draw[thick] (-4.5,0.4) to [out=180,in=240] (-4.9,1.2);
\draw[->,thick] (-3.25,2) to [out=180,in=60] (-4.9,1.2);

\draw[->,blue, thick] (-4.1,0) to [out=90,in=180] (-3.25,0.5);
\draw[blue, thick] (-2.4,0) to [out=90,in=0] (-3.25,0.5);
\draw[dashed, blue, thick] (-4.1,0) to [out=-90,in=180] (-3.25,-0.5) to [out=0,in=-90] (-2.4,0);

\node [above,blue] at (-3.25,0.5) {$c$};

\draw[thick ,->] (1.5,0) -- (3.5,0);
\node [above] at (2.5,0) {$T_{c}^{-1}$};
\end{tikzpicture}
\qquad
\begin{tikzpicture}[scale=.8]
\draw[very thick] (-4.5,-2) -- (-2,-2);
\draw[very thick] (-4.5,2) -- (-2,2);
\draw[very thick] (-4.5,2) arc [radius=2, start angle=90, end angle=270];
\draw[very thick] (-2,2) arc [radius=2, start angle=90, end angle=-90];
\draw[very thick] (-4.5,0) circle [radius=.4];
\draw[very thick] (-2,0) circle [radius=.4];

\draw[dashed,thick] (-3.25,2) to [out=-50,in=90] (-2.7,0) to [out=-90,in=50] (-3.25,-2);

\draw[thick] (-3.25,-2) to [out=130,in=-90] (-3.8,0);

\draw[->,thick] (-3.25,2) to [out=230,in=90] (-3.8,0);

\node [above] at (-4,-1.8) {$\gamma$};
\end{tikzpicture}
\]
Then we have that the product of Dhen twists $T:=T_{c}^{-1}T_{a_1}^{-1}T_{a_2}^{-1}$ takes $K$ to $\gamma.$
Thus, $T^{-1}$ takes $\gamma$ to $K.$ Hence
$$\sigma(T_K)=\sigma(T^{-1}T_\gamma T)= \Psi(T^{-1})\sigma(T_\gamma).$$

Consider the subsurface on the left of our curve $\gamma$ (if we take the other subsurface the same argument works), Observe that $A_1,B_1$ is a symplectic basis of the homology of this subsurface. So $\sigma(T_\gamma)= \overline{A_1}\overline{B_1}.$

Next, we compute $\Psi(T^{-1}).$ Since $\Psi$ is a homomorphism we have that
$$\Psi(T^{-1})=\Psi(T_{a_2}T_{a_1}T_{c})=\Psi (T_{a_2})\Psi(T_{a_1})\Psi(T_{c}).$$
By definition the images of $T_{c},$ $T_{a_1}$ and $T_{a_2}$ by $\Psi: \mathcal{M}_g \rightarrow Sp_4(\mathbb{Z}),$ are given by

$$\Psi(T_{a_1})=\left(\begin{matrix}
1 & 0 & 1 & 0 \\
0 & 1 & 0 & 0 \\
0 & 0 & 1 & 0 \\
0 & 0 & 0 & 1 \\
\end{matrix}\right),\quad
\Psi(T_{a_2})=\left(\begin{matrix}
1 & 0 & 0 & 0 \\
0 & 1 & 0 & 1 \\
0 & 0 & 1 & 0 \\
0 & 0 & 0 & 1 \\
\end{matrix}\right),
\quad
\Psi(T_{c})=\left(\begin{matrix}
1 & 0 & 0 & 0 \\
0 & 1 & 0 & 0 \\
-1 & 1 & 1 & 0 \\
1 & -1 & 0 & 1 \\
\end{matrix}\right).$$

Hence

$$\Psi(T^{-1})=
\left(\begin{matrix}
1 & 0 & 1 & 0 \\
0 & 1 & 0 & 1 \\
0 & 0 & 1 & 0 \\
0 & 0 & 0 & 1 \\
\end{matrix}\right)
\left(\begin{matrix}
1 & 0 & 0 & 0 \\
0 & 1 & 0 & 0 \\
-1 & 1 & 1 & 0 \\
1 & -1 & 0 & 1 \\
\end{matrix}\right)=
\left(\begin{matrix}
0 & 1 & 1 & 0 \\
1 & 0 & 0 & 1 \\
-1 & 1 & 1 & 0 \\
1 & -1 & 0 & 1 \\
\end{matrix}\right).
$$
Thus,
\begin{align*}
\sigma(T_K)= & \Psi(T^{-1})\sigma(T_\gamma)=\overline{\Psi(T^{-1})(A_1)}\overline{\Psi(T^{-1})(B_1)} \\
= &(\overline{A_2+B_1+B_2})(\overline{A_1+B_1}) \\
= & (\overline{A_2}+\overline{B_1}+\overline{B_2}+\overline{1})(\overline{A_1}+\overline{B_1}+\overline{1}) \\
= & 
\overline{A_2}\overline{A_1}+\overline{B_1}\overline{A_1}+\overline{B_2}\overline{A_1}+\overline{A_1}+\overline{A_2}\overline{B_1}+\overline{B_1}\overline{B_1}+\overline{B_2}\overline{B_1}+\overline{B_1}+\overline{A_2}+\overline{B_1}+\overline{B_2}+\overline{1} \\
= & 
\overline{A_2}\overline{A_1}+\overline{B_1}\overline{A_1}+\overline{B_2}\overline{A_1}+\overline{A_1}+\overline{A_2}\overline{B_1}+\overline{B_1}+\overline{B_2}\overline{B_1}+\overline{A_2}+\overline{B_2}+\overline{1}. 
\end{align*}
Therefore, by definition of $\varphi_g^x,$ we have that
$\mu_g^x(T_K)=\varphi_g^x\circ\sigma(T_K)=x,$ as desired.
\end{proof}

\begin{prop}
\label{prop_inv_coincide}
The invariant $\mu_g^x$ coincides with the Rohlin invariant $R_g$ composed with the homomorphism $\varepsilon^x.$
\end{prop}
\begin{proof}
In \cite{jon_4}, D. Johnson proved that the Rohlin invarinat induces a family of homomorphisms $\{R_g\}_g$ with $R_g\in Hom(\mathcal{T}_{g,1};\mathbb{Z}/2)^{\mathcal{AB}_{g,1}}.$
Since, by Lemma \eqref{lem_abinv}, there is only one non-zero element in $Hom(\mathcal{T}_{g,1};\mathbb{Z}/2)^{\mathcal{AB}_{g,1}},$ we have that $\mu_g^1$ and $R_g$ must coincide. Therefore, $\varepsilon^x\circ\mu_g^1$ and $\varepsilon^x \circ R_g$ must coincide too.
\end{proof}

\section{From trivial cocycles to invariants}
\label{section_cocycles to invariants}

Conversely, what are the conditions for a family of trivial $2$-cocycles $C_g$ on $\mathcal{T}_{g,1}$ satisfying properties (1)-(3) to actually provide an invariant?

Firstly we need check the existence of an $\mathcal{AB}_{g,1}$-invariant trivialization of each $C_g.$ This is a cohomological problem.
Denote by $\mathcal{Q}_{C_g}$ the set of all trivializations of the cocycle $C_g:$
$$\mathcal{Q}_{C_g}=\{q:\mathcal{T}_{g,1}\rightarrow A\mid q(\phi)+q(\psi)-q(\phi\psi)=C_g(\phi,\psi)\}.$$
Recall that any two trivializations of a given $2$-cocycle differ by an element of $Hom(\mathcal{T}_{g,1},A).$
As the cocycle $C_g$ is invariant under conjugation by $\mathcal{AB}_{g,1}$ this latter group acts on $\mathcal{Q}_{C_g}$ via its conjugation action on the Torelli group. Explicitly, if $\phi\in \mathcal{AB}_{g,1}$ and $q\in \mathcal{Q}_{C_g}$ then $\phi\cdot q(\eta)=q(\phi^{-1}\eta \phi).$ This action confers the set $\mathcal{Q}_{C_g}$ the structure of an affine set over the abelian group $Hom(\mathcal{T}_{g,1},A).$ Choose an arbitrary element $q\in \mathcal{Q}_{C_g}$ and define a map as follows
\begin{align*}
\rho_q:\mathcal{AB}_{g,1} & \longrightarrow Hom(\mathcal{T}_{g,1},A)\\ \phi & \longmapsto \phi \cdot q-q.
\end{align*}
A direct computation shows that $\rho_q$ is a derivation, i.e. $\rho_q(\phi\psi)=\phi \cdot \rho_q(\psi)+ \rho_q(\psi),$ and the difference $\rho_q-\rho_{q'}$ for two elements in $\mathcal{Q}_{C_g}$ is a principal derivation. Therefore we have a well-defined cohomology class
$$\rho(C_g)\in H^1(\mathcal{AB}_{g,1};Hom(\mathcal{T}_{g,1},A))$$
called the torsor of the cocycle $C_g.$

By construction, if the action of $\mathcal{AB}_{g,1}$ on $\mathcal{Q}_{C_g}$ has a fixed point, the class $\rho(C_g)$ is trivial. Conversely, if $\rho(C_g)$ is trivial, then for any $q\in \mathcal{Q}_{C_q}$ the map $\rho_q$ is a principal derivation, i.e.
there exists $m_q\in Hom(\mathcal{T}_{g,1}, A)$ such that
$$\forall \phi \in \mathcal{AB}_{g,1}\qquad \rho_q(\phi)=\phi \cdot m_q-m_q.$$
In particular the element $q-m_q \in \mathcal{Q}_{C_g}$ is fixed under the action of $\mathcal{AB}_{g,1},$ since
$$\phi \cdot (q-m_q)=\phi\cdot q-\phi \cdot m_q=(\rho_q(\phi)+q)-(\rho_q(\phi)+m_q)=q-m_q$$
So we have proved:

\begin{prop}
The natural action of $\mathcal{AB}_{g,1}$ on $\mathcal{Q}_{C_g}$ admits a fixed point if and only if the associated torsor $\rho(C_g)$ is trivial.
\end{prop}

Suppose that for every $g\geq 4$ there is a fixed point $q_g$ of $\mathcal{Q}_{C_g}$ for the action of $\mathcal{AB}_{g,1}$ on $\mathcal{Q}_{C_g}.$
Since every pair of $\mathcal{AB}_{g,1}$-invariant trivialization differ by a $\mathcal{AB}_{g,1}$-invariant homomorphism, by Lemma \eqref{lem_abinv}, for every $g\geq 3$ the fixed points are $q_g+\mu_g^x$ with $x\in A_2.$

By Lemma \eqref{lem_stab}, all elements of $Hom(\mathcal{T}_{g,1},A)^{\mathcal{AB}_{g,1}}$ are compatible with the stabilization map. Then, given two different fixed points $q_g,$ $q'_g$ of $\mathcal{Q}_{C_g}$ for the action of $\mathcal{AB}_{g,1},$ we have that
$${q_g}_{\mid\mathcal{T}_{g-1,1}}-{q'_g}_{\mid\mathcal{T}_{g-1,1}}=(q_g-q'_g)_{\mid\mathcal{T}_{g-1,1}} ={\mu_g^x}_{\mid\mathcal{T}_{g-1,1}}=\mu_{g-1}^x.$$
Therefore the restriction of the trivializations of $\mathcal{Q}_{C_g}$ to $\mathcal{T}_{g-1,1},$ give us a bijection between the fixed points of $\mathcal{Q}_{C_g}$ for the action of $\mathcal{AB}_{g,1}$ and the fixed points of $\mathcal{Q}_{C_{g-1}}$ for the action of $\mathcal{AB}_{g-1,1}.$

Therefore, given an $\mathcal{AB}_{g,1}$-invariant trivialization $q_g,$ for each $x\in A_2,$ we get a well-defined map
$$
q+\mu^x= \lim_{g\to \infty}q_g+\mu^x_g: \lim_{g\to \infty}\mathcal{T}_{g,1}\longrightarrow A.
$$
These are the only candidates to be $A$-valued invariants of homology spheres with associated family of $2$-cocycles $\{C_g\}_g.$ For these maps to be invariants, since they are already $\mathcal{AB}_{g,1}$-invariant, we only have to prove that they are constant on the double cosets $\mathcal{TA}_{g,1}\backslash \mathcal{T}_{g,1}/\mathcal{TB}_{g,1}.$
From property (3) of our cocycle we have that $\forall \phi\in \mathcal{T}_{g,1},$
$\forall \psi_a\in \mathcal{TA}_{g,1}$ and $\forall \psi_b\in \mathcal{TB}_{g,1},$
\begin{equation}
\label{eq_TA,TB_constant}
\begin{aligned}
(q_g+\mu^x_g)(\phi)-(q_g+\mu^x_g)(\phi \psi_b)= & -(q_g+\mu^x_g)(\psi_b) , \\
(q_g+\mu^x_g)(\phi)-(q_g+\mu^x_g)(\psi_a\phi )= & -(q_g+\mu^x_g)(\psi_a).
\end{aligned}
\end{equation}
Thus in particular, taking $\phi=\psi_a,\psi_b$ in above equations, we get that $q_g+\mu^x_g$ with $x\in A_2,$ are homomorphisms on $\mathcal{TA}_{g,1},$ $\mathcal{TB}_{g,1}.$ Then, by Lemma \eqref{lem_TB,TA}, we get that $q_g+\mu^x_g$ is trivial on $\mathcal{TA}_{g,1}$ and $\mathcal{TB}_{g,1}.$

Therefore, by equalities \eqref{eq_TA,TB_constant}, we obtain that $q_g+\mu_g^x$ with $x\in A_2$ are constant on the double cosets $\mathcal{TA}_{g,1}\backslash \mathcal{T}_{g,1}/\mathcal{TB}_{g,1}.$

Summarizing, we get the following result:

\begin{teo}
\label{teo_constr_torelli}
Let $A$ be an abelian group and $A_2$ the subgroup of $2$-torsion elements.
For each $x\in A_2,$ a family of cocycles $\{C_g\}_{g\geq 3}$ on the Torelli groups $\mathcal{T}_{g,1}, g\geq 3,$ satisfying conditions (1)-(3) provides a compatible familiy of trivializations $F_g+\mu_g^x: \mathcal{T}_{g,1}\rightarrow A$ that reassemble into an invariant of homology spheres
$$\lim_{g\to \infty}F_g+\mu_g^x:\mathcal{S}^3\rightarrow A$$
if and only if the following two conditions hold:
\begin{enumerate}[(i)]
\item The associated cohomology classes $[C_g]\in H^2(\mathcal{T}_{g,1};A)$ are trivial.
\item The associated torsors $\rho(C_g)\in H^1(\mathcal{AB}_{g,1},Hom(\mathcal{T}_{g,1},A))$ are trivial.
\end{enumerate}
\end{teo}

Notice that Proposition \eqref{prop_inv_coincide} tell us that the homomorphisms $\mu_g^x$ factor through the Rohlin invariant $R_g.$
Therefore the existence of the Rohlin invariant makes to fail the unicity in the construction of invariants from a family of $2$-cocycles in Theorem \eqref{teo_constr_torelli}.

\chapter{The mod $d$ Torelli group and homology $3$-spheres}
\label{chapter: RHS}

In this chapter, we study the relation between the mod $d$ Torelli group, the $\mathbb{Z}/d$-homology $3$-spheres and the $\mathbb{Q}$-homology $3$-spheres.
To be more precise, our aim is to show that given an integer $d\geq 2,$ a Heegaard splitting with gluing map an element of the (mod $d$) Torelli group is a 
$\mathbb{Z}/d$-homology $3$-sphere, and so a $\mathbb{Q}$-homology $3$-sphere, and that every $\mathbb{Q}$-homology $3$-sphere can be obtained as a Heegaard splitting with a gluing map an element of the (mod $p$) Torelli group. This last result is a direct consequence of Proposition 6 in \cite{per}, in which B. Perron stated that given a $\mathbb{Q}$-homology $3$-sphere $M^3$ with $n=|H_1(M^3;\mathbb{Z})|,$ if $d\mid n-1$, then $M^3$ can be obtained as a Heegaard splitting with gluing map an element of the (mod $d$) Torelli group.
Unfortunately, the proof of such result is not available in the literature.
In this chapter, we prove a generalization of this result in which we consider $d\mid n\pm 1$ instead of $d\mid n-1.$
We point out that, as we will see, the divisibility condition is indispensable for such result.

As a starting point, in the first two sections, we give some basic definitions and properties about the symplectic representation modulo $d$ as well as the (mod $d$) Torelli group.
In the third section, we study the action of the subgroups $\mathcal{A}_{g,1},$ $\mathcal{B}_{g,1},$ $\mathcal{AB}_{g,1}$ on $H_d:=H_1(\Sigma_{g,1};\mathbb{Z}/d).$ The main difference between the action on $H_d$ and the action on $H:=H_1(\Sigma_{g,1};\mathbb{Z}),$ is the fact that, if $\Psi_d$ denotes the symplectic representation modulo $d,$ then $\Psi_d(\mathcal{B}_{g,1})$ does not correspond to the subgroup of $Sp_{2g}(\mathbb{Z}/d)$ formed by matrices of the form
$\left(\begin{smallmatrix}
* & 0 \\
* & *
\end{smallmatrix} \right),$
which is isomorphic to $GL_g(\mathbb{Z}/d)\ltimes S_g(\mathbb{Z}/d).$
But for $d$ a prime number $\Psi_d(\mathcal{B}_{g,1})$ is isomorphic to $SL_g^{\pm}(\mathbb{Z}/d)\ltimes S_g(\mathbb{Z}/d).$ For $d
$ not prime, such isomorphism does not hold. However, there is an stable version that give us an isomorphism $\Psi_d(\mathcal{B}_{\infty,1})\cong SL^{\pm}(\mathbb{Z}/d)\ltimes S(\mathbb{Z}/d).$

Finally, in the last section, we give a criterion to know whenever a $\mathbb{Q}$-homology $3$-shpere is homeomorphic to a Heegaard splitting with gluing map an element of (mod $d$) Torelli group.
As an application of this criterion, if we denote by $\mathcal{S}^3[d]$ the set of $\mathbb{Q}$-homology $3$-spheres which are homeomorphic to $\mathcal{H}_g \cup_{\iota_g\phi} -\mathcal{H}_g$ for some $\phi\in \mathcal{M}_{g,1}[d],$ unlike the case of the Torelli group and homology $3$-spheres, we show that, in general, the set of $\mathbb{Z}/d$-homology $3$-spheres do not coincide with $\mathcal{S}^3[d].$
In addition, analogously to the case of the integral homology $3$-spheres and the Torelli group, we show that there is the following bijection:
$
\lim_{g\to \infty}\left(\mathcal{A}_{g,1}[d]\backslash\mathcal{M}_{g,1}[d]/\mathcal{B}_{g,1}[d]\right)_{\mathcal{AB}_{g,1}} \simeq \mathcal{S}^3[d].$

\section{The Symplectic representation modulo $d$}

\begin{defi}[Section 1 in \cite{newman}]
Let $d\geq 2$ be an integer and $M\in M_n(\mathbb{Z})$ be a matrix satisfying $M\Omega M^t \equiv \Omega\;(\text{mod}\;d),$ then $M$ will be said to be \textit{symplectic modulo $d$.}
We define the \textit{Symplectic group modulo $d$} as
$$Sp_{2g}(\mathbb{Z}/d)=\{M\in M_{2g\times 2g}(\mathbb{Z}/d)\mid M\Omega M^t = \Omega \}.$$
As in the case of the symplectic matrices, from the definitions, it is easy to verify that a matrix
$$M=\left(\begin{matrix}
A & B \\
C & D
\end{matrix}\right)$$
is symplectic modulo $n$ if an only if
$$AD^t-BC^t= Id_g,\quad AB^t= BA^t,\quad CD^t= DC^t.$$
\end{defi}
We define the Symplectic representation modulo $d$ as the composition 
$$
\Psi_d:\xymatrix@C=7mm@R=10mm{ \mathcal{M}_{g,1} \ar@{->}[r]^-{\Psi} & Sp_{2g}(\mathbb{Z}) \ar@{->}[r]^-{r_d} & Sp_{2g}(\mathbb{Z}/d).}
$$
Notice that Symplectic representation modulo $d$ is surjective because the Symplectic representation $\Psi$ is surjective and, by Theorem 1 in \cite{newman}, $r_d$ is also surjective.

\section{The mod $d$ Torelli group}
The (mod $d$) Torelli group $\mathcal{M}_{g,1}[d]$ is the normal subgroup of the mapping class group $\mathcal{M}_{g,1}$ of those
elements of $\mathcal{M}_{g,1}$ that act trivially on $H_1(\Sigma_{g,1};\mathbb{Z}/d).$ In other words, $\mathcal{M}_{g,1}[d]$ is characterized by the following short exact sequence:
$$
\xymatrix@C=10mm@R=13mm{1 \ar@{->}[r] & \mathcal{M}_{g,1}[d] \ar@{->}[r] & \mathcal{M}_{g,1} \ar@{->}[r]^-{\Psi_d} & Sp_{2g}(\mathbb{Z}/d) \ar@{->}[r] & 1 }
$$
Let $D_{g,1}[d]$ be the subgroup of $\mathcal{M}_{g,1}$ generated by the $d^{th}$-powers of Dehn twists.
The following proposition and corollary announced in \cite{per} by B. Perron and proved by J. Cooper in \cite{coop} enlightens the structure of $\mathcal{M}_{g,1}[d].$

\begin{prop}
\label{prop_gen_M[p]}
The mod $d$ Torelli group $\mathcal{M}_{g,1}[d]$ is the normal subgroup of $\mathcal{M}_{g,1}$ generated by the Torelli group $\mathcal{T}_{g,1}$ and the $d^{th}$-powers of Dehn twists. 
\end{prop}

Since a conjugate of a Dehn twist is a Dehn twist, one gets:

\begin{cor}
\label{cor_gen_M[p]}
Every element $\varphi\in \mathcal{M}_{g,1}[d]$ can be written as:
$\varphi=f\circ m$ with $f\in \mathcal{T}_{g,1}$ and $m\in D_{g,1}[d].$
\end{cor}

\begin{rem}
In \cite{coop}, J. Cooper give a proof of Proposition \eqref{prop_gen_M[p]} and Corollary \eqref{cor_gen_M[p]} for the case of $d$ prime.
Nevertheless, following the same proof one gets the same result for the case of any integer $d.$
\end{rem}

\section{Homology actions modulo d}

If one writes the matrices of the symplectic group modulo $d$ $Sp_{2g}(\mathbb{Z}/d)$ as blocks according to the decomposition $H_d=A_d\oplus B_d,$ then the image of $\mathcal{B}_{g,1}\rightarrow Sp_{2g}(\mathbb{Z}/d)$ is contained in the subgroup $Sp_{2g}^B(\mathbb{Z}/d)$ of matrices of the form:
$\left(\begin{matrix}
G_1 & 0 \\
M & G_2
\end{matrix}\right).$

Such matrices are symplectic modulo $d$ if and only if $G_2= {}^tG_1^{-1}$ and ${}^tG_1^{-1}M$ is symmetric. As a consequence, we have an isomorphism:
\begin{align*}
\phi_d^B: Sp_{2g}^B(\mathbb{Z}/d) & \longrightarrow GL_g(\mathbb{Z}/d) \ltimes_B S_g(\mathbb{Z}/d), \\
\left(\begin{matrix}
G & 0 \\
M & {}^tG^{-1}
\end{matrix}\right) & \longmapsto (G,{}^tGM).
\end{align*}
Here $S_g(\mathbb{Z}/d)$ denotes the symmetric group of $g\times g$ matrices over the $\mathbb{Z}/d.$ The composition on the semi-direct product is given by the rule $(G,S)(H,T)=(GH,{}^tHSH+T).$

Analogously, the image of $\mathcal{A}_{g,1}\rightarrow Sp_{2g}(\mathbb{Z}/d)$ is contained in the subgroup $Sp_{2g}^A(\mathbb{Z}/d)$ of matrices of the form:
$\left(\begin{matrix}
H_1 & N \\
0 & H_2
\end{matrix}\right).$

Such matrices are symplectic modulo $d$ if and only if $H_2= {}^tH_1^{-1}$ and ${}^tH_2N$ is symmetric. As a consequence, we have an isomorphism:
\begin{align*}
\phi_d^A: Sp_{2g}^A(\mathbb{Z}/d) & \longrightarrow GL_g(\mathbb{Z}/d) \ltimes_A S_g(\mathbb{Z}/d), \\
\left(\begin{matrix}
H & N \\
0 & {}^tH^{-1}
\end{matrix}\right) & \longmapsto (G,{}^tGM)
\end{align*}
where the composition on the semi-direct product is given by the rule
$$(G,S)(H,T)=(GH,{}^tH^{-1}SH^{-1}+T).$$
Similarly, the image of $\mathcal{AB}_{g,1}\rightarrow Sp_{2g}(\mathbb{Z}/d)$ is contained in the subgroup $Sp_{2g}^{AB}(\mathbb{Z}/d)$ of matrices of the form:
$\left(\begin{matrix}
G_1 & 0 \\
0 & G_2
\end{matrix}\right).$

Such matrices are symplectic modulo $d$ if and only if $G_2= {}^tG_1^{-1}.$ As a consequence, we have an isomorphism:
\begin{align*}
\phi_d^{AB}: Sp_{2g}^{AB}(\mathbb{Z}/d) & \longrightarrow GL_g(\mathbb{Z}/d), \\
\left(\begin{matrix}
G & 0 \\
0 & {}^tG^{-1}
\end{matrix}\right) & \longmapsto G.
\end{align*}

\begin{nota}
Denote by $r'_d:GL_g(\mathbb{Z})\rightarrow GL_g(\mathbb{Z}/d),$ $r''_d:S_g(\mathbb{Z})\rightarrow S_g(\mathbb{Z}/d)$ the respective reductions modulo $d.$
\end{nota}

\begin{defi}
Let $d$ a positive integer, we define the following groups:
\begin{align*}
SL^\pm_g(\mathbb{Z}/d)= & \{ A\in M_{g\times g}(\mathbb{Z}/d) \mid det(A)=\pm 1 \}, \\
SL_g(\mathbb{Z},d)= & Ker(r'_d:SL_g(\mathbb{Z})\rightarrow SL_g(\mathbb{Z}/d)), \\
S_g(d\mathbb{Z})= & Ker(r''_d:S_g(\mathbb{Z})\rightarrow S_g(\mathbb{Z}/d)).
\end{align*}
\end{defi}

\begin{nota}
Let $R$ be a division ring, $i,$ $j$ distinct integers between $1$ and $g$ and $u\in R.$ Denote by $e_{ij}(u)\in SL_g(R)$ the matrix with entry $u$ in the $(i,j)$ position, ones on the diagonal and zeros elsewhere. Such matrices are called elementary matrices of rank $g.$
We denote by $E_g(R)$ the subgroup of $SL_g(R)$ generated by elementary matrices of rank $g.$
We denote by $D_g$ the diagonal matrix of rank $g,$ with a $-1$ at position $(1,1)$ and $1$'s at positions $(i,i)$ with $2\leq i \leq g.$
\end{nota}

\begin{teo}[Theorem in \cite{green}]
\label{teo_SL_elem}
Let $R$ be a division ring. The group $SL_g(R)$ coincides with $E_g(R).$
\end{teo}

\begin{lema}
\label{lem_Sp[p]_2}
Let $p$ be a prime number,
there is a short exact sequence of groups
$$
\xymatrix@C=7mm@R=13mm{1 \ar@{->}[r] & SL_g(\mathbb{Z},p) \ar@{->}[r] & GL_g(\mathbb{Z}) \ar@{->}[r]^-{r'_p} & SL_g^\pm(\mathbb{Z}/p) \ar@{->}[r] & 1. }
$$
\end{lema}

\begin{proof}
First of all, we show that $r'_p: GL_g(\mathbb{Z}) \rightarrow SL_g^\pm(\mathbb{Z}/p)$ is surjective.
By Theorem \eqref{teo_SL_elem} $SL_g(\mathbb{Z}/p)=E_g(\mathbb{Z}/p).$
Observe that if $A\in SL^\pm_g(\mathbb{Z}/p)$ with $det(A)=-1$ then $det(D_gA)=1,$ i.e. $D_gA\in SL_g(\mathbb{Z}/p).$
Therefore $SL_g^\pm(\mathbb{Z}/p)$ is generated by $\{D_g,e_{ij}(u)\mid u\in \mathbb{Z}/p\}.$
Since $r'_p(D_g)=D_g$ and $r'_p(E_g(\mathbb{Z}))=E_g(\mathbb{Z}/p),$
we have that $r'_p$ is surjective.

Next, notice that we have a pull-back diagram
$$
\xymatrix@C=7mm@R=10mm{1 \ar@{->}[r] & Ker(r'_p) \ar@{=}[d] \ar@{->}[r] & SL_g(\mathbb{Z}) \ar@{^{(}->}[d] \ar@{->}[r]^-{r'_p} & SL_g(\mathbb{Z}/p) \ar@{^{(}->}[d] \ar@{->}[r] & 1 \\
1 \ar@{->}[r] & Ker(r'_p) \ar@{->}[r] & GL_g(\mathbb{Z}) \ar@{->}[r]^-{r'_p} & SL_g^\pm(\mathbb{Z}/p) \ar@{->}[r] & 1.}
$$
As a consequence $Ker(r'_p)\cong SL_g(\mathbb{Z},p),$
as desired.
\end{proof}

\begin{lema}[Lemma 1 in \cite{newman}]
\label{lem_S(p)}
Let $d$ be an integer, there is a short exact sequence of groups
$$
\xymatrix@C=7mm@R=13mm{1 \ar@{->}[r] & S_g(d\mathbb{Z}) \ar@{->}[r] & S_g(\mathbb{Z}) \ar@{->}[r]^-{r''_d} & S_g(\mathbb{Z}/d) \ar@{->}[r] & 1 .}
$$
\end{lema}

\begin{lema}
\label{lem_Sp[p]_3}
Let $p$ be a prime number,
there is a short exact sequence of groups
$$
\xymatrix@C=8mm@R=8mm{1 \ar@{->}[r] & SL_g(\mathbb{Z},p) \ltimes S_g(p\mathbb{Z}) \ar@{->}[r] & GL_g(\mathbb{Z})\ltimes S_g(\mathbb{Z}) \ar@{->}[r]^-{r'_p\times r''_p} & SL^\pm_g(\mathbb{Z}/p)\ltimes S_g(\mathbb{Z}/p) \ar@{->}[r] & 1.}
$$
\end{lema}

\begin{proof}
First of all we show that $r'_d \times r''_d$ is a homomrophism.

Let $G,H\in GL_g(\mathbb{Z}),$ $S,T\in S_g(\mathbb{Z}),$ then we have that the following equality holds
\begin{align*}
(r'_p \times r''_p)((G,S)(H,T))= & (r'_d \times r''_d)((GH, HSH^t+T)) \\
= & (r'_p(GH), r''_p(HSH^t+T)) \\
= & (\overline{G}\overline{H},\overline{H}\overline{S}\overline{H^t}+\overline{T}) \\
= & (\overline{G},\overline{S})(\overline{H},\overline{T}) \\
= & (r'_p \times r''_p)(G,S)(r'_p \times r''_p)(H,T).
\end{align*}
Notice that, by Lemmas \eqref{lem_Sp[p]_2}, \eqref{lem_S(p)},
$ker(r'_p)=SL_g(\mathbb{Z},p)$, $Ker(r''_p)=S_g(p\mathbb{Z}).$
Then
$$Ker(r'_p \times r''_p)=Ker(r'_p)\ltimes Ker(r''_p)=SL_g(\mathbb{Z},p)\ltimes S_g(p\mathbb{Z}).$$
Finally, we have that $r'_p \times r''_p$ is surjective because $r'_p,r''_p$ are surjective too.
\end{proof}

\begin{defi}
We define the following subgroups of $\mathcal{M}_{g,1}[d]:$
$$\mathcal{A}_{g,1}[d]=\mathcal{M}_{g,1}[d]\cap \mathcal{A}_{g,1},\quad\mathcal{B}_{g,1}[d]=\mathcal{M}_{g,1}[d]\cap \mathcal{B}_{g,1},\quad\mathcal{AB}_{g,1}[d]=\mathcal{M}_{g,1}[d]\cap \mathcal{AB}_{g,1}.$$
\end{defi}

\begin{lema}
\label{lem_B[p]}
Let $p$ be a prime number,
there are short exact sequence of groups
\begin{align*}
& \xymatrix@C=10mm@R=13mm{1 \ar@{->}[r] & \mathcal{B}_{g,1}[p] \ar@{->}[r] & \mathcal{B}_{g,1} \ar@{->}[r]^-{\phi^B_p\circ\Psi_p} & SL_g^\pm(\mathbb{Z}/p)\ltimes_B S_g(\mathbb{Z}/p) \ar@{->}[r] & 1,} \\
& \xymatrix@C=10mm@R=13mm{1 \ar@{->}[r] & \mathcal{A}_{g,1}[p] \ar@{->}[r] & \mathcal{A}_{g,1} \ar@{->}[r]^-{\phi^A_p\circ\Psi_p} & SL_g^\pm(\mathbb{Z}/p)\ltimes_A S_g(\mathbb{Z}/p) \ar@{->}[r] & 1,} \\
& \xymatrix@C=10mm@R=13mm{1 \ar@{->}[r] & \mathcal{AB}_{g,1}[p] \ar@{->}[r] & \mathcal{AB}_{g,1} \ar@{->}[r]^-{\phi^{AB}_p\circ \Psi_p} & SL_g^\pm(\mathbb{Z}/p) \ar@{->}[r] & 1. }
\end{align*}
\end{lema}
\begin{proof}
We only give the proof for the first short exact sequence (the proof for the other two short exact sequences is analogous).
Observe that by definition of $\mathcal{B}_{g,1}[p],$ we have that
$$Ker( \phi^B_p\circ\Psi_p:\mathcal{B}_{g,1} \rightarrow SL_g^\pm(\mathbb{Z}/p)\ltimes S_g(\mathbb{Z}/p))=\mathcal{B}_{g,1}\cap \mathcal{M}_{g,1}[p]=\mathcal{B}_{g,1}[p].$$
On the other hand, by Lemmas \eqref{lem_B}, \eqref{lem_Sp[p]_3} we have that
\begin{align*}
(\phi^B_p\circ\Psi_p)(\mathcal{B}_{g,1})= & (\phi^B_p\circ r_p\circ \Psi)(\mathcal{B}_{g,1})=((r'_p\times r''_p)\circ \phi^B \circ\Psi)(\mathcal{B}_{g,1})= \\
= & (r'_p\times r''_p)(GL_g(\mathbb{Z})\ltimes S_g(\mathbb{Z}))=SL_g^\pm(\mathbb{Z}/p)\ltimes S_g(\mathbb{Z}/p),
\end{align*}
as desired.
\end{proof}

\begin{nota}
Throughout this thesis we set:
\begin{align*}
Sp_{2g}^{A\pm}(\mathbb{Z}/d)= & (\phi^A)^{-1} (SL_g^\pm(\mathbb{Z}/d)\ltimes_A S_g(\mathbb{Z}/d)),\\
Sp_{2g}^{B\pm}(\mathbb{Z}/d)= & (\phi^B)^{-1} (SL_g^\pm(\mathbb{Z}/d)\ltimes_B S_g(\mathbb{Z}/d)).
\end{align*}
\end{nota}

We point out that
Lemmas \eqref{lem_Sp[p]_2}, \eqref{lem_Sp[p]_3}, \eqref{lem_B[p]} are not true switching the prime $p$ for a no prime integer $d.$
This comes from the fact that $SL_g(\mathbb{Z})\rightarrow SL_g(\mathbb{Z}/d)$ is surjective if and only if $d$ is prime, which is based on the fact that $E_g(\mathbb{Z}/d)=SL_g(\mathbb{Z}/d)$ if and only if $d$ is prime.
Nevertheless, as we will show next, taking the direct limit of above families of groups one gets the analogous results for the stable case with an integer $d.$

\paragraph{Direct systems and direct limits}
Let $\langle I,\leq \rangle$ be a directed set. Let $\{A_i \mid i\in I\}$ be a family of objects indexed by $I$ and $f_{ij}:A_i\rightarrow A_j$ be a homomorphism for all $i\leq j$ with the following properties:
\begin{enumerate}[i)]
\item $f_{ii}$ is the identity of $A_i,$
\item $f_{ik}= f_{jk}\circ f_{ij}$ for all $i\leq j\leq k.$
\end{enumerate}
Then the pair $\langle A_i, f_{ij} \rangle$ is called a direct system over $I.$

The direct limit of the direct system $\langle A_i, f_{ij} \rangle$ is denoted by $\varinjlim A_i$ and is defined as follows. Its underlying set is the disjoint union of the $A_i$'s modulo a certain equivalence relation $\sim :$
$$\varinjlim A_i= \bigsqcup_i A_i \Big/ \sim.$$
Here, if $x_i\in A_i$ and $x_j\in A_j,$ then $x_i\sim x_j$ if and only if there is some $k\in I$ with $i\leq k$ and $j\leq k$ and such that $f_{ik}(x_i)=f_{jk}(x_j).$

\paragraph{The stable General linear group and symmetric group.}
Let $R$ be a commutative ring with unite. Consider the General linear groups $\{GL_n(R)\}_n$ and the homomorphism:
$$i_{n,m}:GL_n(R)\longrightarrow GL_m(R),$$
which sends $A\in GL_n(R)$ to
$$\left(\begin{matrix}
A & 0 \\
0 & Id_{m-n}
\end{matrix}\right).$$
Then $\langle GL_n(R), i_{n,m}\rangle$ is a direct system

Similarly, we can consider the symmetric groups $\{S_n(R)\}_n$ and the homomorphism:
$$j_{n,m}:S_n(R)\longrightarrow S_m(R),$$
which sends $A\in S_n(R)$ to
$$\left(\begin{matrix}
A & 0 \\
0 & 0
\end{matrix}\right).$$
Then $\langle S_n(R), j_{n,m}\rangle$ is a direct system.

Therefore, taking the direct limit of the above direct systems, we obtain the following groups:
$$GL(R)=\varinjlim GL_g(R),\qquad SL(R)=\varinjlim SL_g(R), \qquad E(R)=\varinjlim E_g(R),$$
$$SL^\pm(R)=\varinjlim SL_g^\pm(R), \qquad S(R)=\varinjlim S_g(R).$$
Throughout this section we denote by $R^\times$ the group of units of $R,$ i.e. the set of invertible elements of $R,$ and we set:

Observe that we have the following short exact sequence of groups:
$$
\xymatrix@C=7mm@R=13mm{1 \ar@{->}[r] & SL(R) \ar@{->}[r] & GL(R) \ar@{->}[r] & R^\times \ar@{->}[r] & 1.}
$$
Taking the quotient by $E(R),$ the above short exact sequence becomes
$$
\xymatrix@C=7mm@R=13mm{1 \ar@{->}[r] & SK_1(R) \ar@{->}[r] & K_1(R) \ar@{->}[r] & R^\times \ar@{->}[r] & 1,}
$$
where
$$SK_1(R)=\frac{SL(R)}{E(R)}\qquad \text{and} \qquad K_1(R)= \frac{GL(R)}{E(R)}.$$
These groups are well known in $K$-theory (see for instance \cite{Weib}, \cite{Magurn}).

\begin{defi}[Sections 19, 20 in \cite{Lam}]
A commutative ring $R$ is a local ring if it has a unique maximal ideal. A commutative ring $R$ is a semilocal ring if it only has a finite number of maximal ideals.
In particular every commutative local ring is a commutative semilocal ring.
\end{defi}

\begin{prop}[Lemma III.1.4 in \cite{Weib}]
\label{prop_k1_locring}
Let $R$ be a commutative semilocal ring, then
$$K_1(R)\cong R^\times\qquad\text{and}\qquad SK_1(R)=1.$$
\end{prop}

Notice that, for every $d\geq 2,$ $\mathbb{Z}/d$ is a semilocal ring because $\mathbb{Z}/d$ is finite.
Therefore, by Proposition \eqref{prop_k1_locring}, we get the following result:
\begin{cor}
\label{cor_SK1_d}
Let $d\geq 2$ be an integer, $SK_1(\mathbb{Z}/d)=1,$ i.e. $SL(\mathbb{Z}/d)=E(\mathbb{Z}/d).$
\end{cor}

\begin{rem}[Remark III.1.2.3 in \cite{Weib}]
\label{rem_E_onto}
If $I$ is an ideal of $R,$ each homomorphism $E_n(R)\rightarrow E_n(R/I)$ is onto, because the generators $e_{ij}(r)$ of $E_n(R)$ map onto the generators $e_{ij}(\overline{r})$ of $E_n(R/I).$
\end{rem}

\begin{lema}
\label{lem_SL[d]_1_stab}
Let $d\geq 2$ be an integer,
we have a short exact sequence of groups
$$
\xymatrix@C=7mm@R=13mm{1 \ar@{->}[r] & SL(\mathbb{Z},d) \ar@{->}[r] & GL(\mathbb{Z}) \ar@{->}[r]^-{r'_d} & SL^\pm(\mathbb{Z}/d) \ar@{->}[r] & 1,}
$$
where $r'_d$ is given by modulo $d$ reduction.
\end{lema}

\begin{proof}
Fix an integer $g.$ Consider $A\in SL^\pm_g(\mathbb{Z}/d).$
Then $D_gA\in SL_g(\mathbb{Z}/d).$ By Corollary \eqref{cor_SK1_d} we have that there exists an integer $h$ and elementary matrices $e_k\in E_h(\mathbb{Z}/d)$ such that the following equality holds.
$$D_gA=i_{h,g}\left( \prod^r_{k=1} e_k \right).$$
Since $i_{h,g}(D_h)=D_g$ and $D_h=D_h^{-1},$ we have that 
$$A=i_{h,g}\left( D_h \prod^r_{k=1} e_k \right).$$
Notice that $r'_d(D_h)=D_h.$ Moreover, by Remark \eqref{rem_E_onto},
we have that $r'_d(E_h(\mathbb{Z}))=E_h(\mathbb{Z}/d).$
Therefore $r'_p$ is surjective.
\end{proof}

\begin{cor}
\label{cor_S_dZ_stab}
Let $d\geq 2$ be an integer,
there is a short exact sequence of groups
$$
\xymatrix@C=7mm@R=13mm{1 \ar@{->}[r] & S(d\mathbb{Z}) \ar@{->}[r] & S(\mathbb{Z}) \ar@{->}[r]^-{r''_d} & S(\mathbb{Z}/d) \ar@{->}[r] & 1,}
$$
where $r''_d$ is given by the reduction modulo $d.$
\end{cor}

\begin{lema}
\label{lem_SL[d]_2_stab}
Let $d\geq 2$ an integer,
there is a short exact sequence of groups:
$$
\xymatrix@C=8mm@R=13mm{1 \ar@{->}[r] & SL(\mathbb{Z},d) \ltimes S(d\mathbb{Z}) \ar@{->}[r] & GL(\mathbb{Z})\ltimes S(\mathbb{Z}) \ar@{->}[r]^-{r'_d\times r''_d} & SL^\pm(\mathbb{Z}/d)\ltimes S(\mathbb{Z}/d) \ar@{->}[r] & 1,}
$$
where $r'_d:GL(\mathbb{Z})\rightarrow SL^\pm(\mathbb{Z}/d)$, $r''_d:S(\mathbb{Z})\rightarrow S(\mathbb{Z}/d)$ are the reduction modulo $d.$
\end{lema}

\begin{proof}
This result is a direct consequence of Lemma \eqref{lem_SL[d]_1_stab} and Corollary \eqref{cor_S_dZ_stab}.
\end{proof}

\paragraph{The stable Mapping class group.}

Consider the Mapping class groups $\{\mathcal{M}_{g,1}\}_g$ and the stabilitzation map $i_{g,h}:\mathcal{M}_{g,1}\rightarrow \mathcal{M}_{h,1}.$ Then we have a direct system $\langle \mathcal{M}_{g,1},i_{g,h} \rangle.$ Define
$$\mathcal{M}_{\infty ,1}=\varinjlim \mathcal{M}_{g,1}.$$
Since $\mathcal{A}_{g,1},$ $\mathcal{B}_{g,1},$ $\mathcal{AB}_{g,1}$ are subgroups of $\mathcal{M}_{g,1},$ then we have the following natural subgroups of $\mathcal{M}_{\infty,1}:$
\[
\begin{array}{ccc}
\mathcal{A}_{\infty, 1}=\varinjlim\mathcal{A}_{g, 1}, &
\mathcal{B}_{\infty, 1}=\varinjlim\mathcal{B}_{g, 1}, &
\mathcal{AB}_{\infty, 1}=\varinjlim\mathcal{AB}_{g, 1}, \\
\mathcal{A}_{\infty, 1}[d]=\varinjlim\mathcal{A}_{g, 1}[d], &
\mathcal{B}_{\infty, 1}[d]=\varinjlim\mathcal{B}_{g, 1}[d], &
\mathcal{AB}_{\infty, 1}[d]=\varinjlim\mathcal{AB}_{g, 1}[d].
\end{array}
\]
Moreover, if we consider the direct limit of the short exact sequences in Lemmas \eqref{lem_B} and \eqref{lem_AB} we get the following result.
\begin{lema}
\label{lem_A,B,AB_infty}
There are short exact sequences of groups
\begin{equation*}
\begin{aligned}
& \xymatrix@C=7mm@R=7mm{1 \ar@{->}[r] & \mathcal{TB}_{\infty ,1} \ar@{->}[r] & \mathcal{B}_{\infty ,1} \ar@{->}[r]^-{\Psi} & GL(\mathbb{Z})\ltimes S(\mathbb{Z}) \ar@{->}[r] & 1,}
\\
& \xymatrix@C=7mm@R=7mm{1 \ar@{->}[r] & \mathcal{TA}_{\infty ,1} \ar@{->}[r] & \mathcal{A}_{\infty ,1} \ar@{->}[r]^-{\Psi} & GL(\mathbb{Z})\ltimes S(\mathbb{Z}) \ar@{->}[r] & 1,}
\\
& \xymatrix@C=7mm@R=7mm{1 \ar@{->}[r] & \mathcal{TAB}_{\infty ,1} \ar@{->}[r] & \mathcal{AB}_{\infty ,1} \ar@{->}[r]^-{\Psi} & GL(\mathbb{Z}) \ar@{->}[r] & 1.}
\end{aligned}
\end{equation*}

\end{lema}

\begin{lema}
\label{lem_B[d]_3_stab}
Let $d\geq 2$ an integer,
there are short exact sequences of groups
\begin{equation*}
\begin{aligned}
& \xymatrix@C=7mm@R=7mm{1 \ar@{->}[r] & \mathcal{B}_{\infty ,1}[d] \ar@{->}[r] & \mathcal{B}_{\infty ,1} \ar@{->}[r]^-{\Psi_d} & SL^\pm(\mathbb{Z}/d)\ltimes S(\mathbb{Z}/d) \ar@{->}[r] & 1,} \\
& \xymatrix@C=7mm@R=7mm{1 \ar@{->}[r] & \mathcal{A}_{\infty ,1}[d] \ar@{->}[r] & \mathcal{A}_{\infty ,1} \ar@{->}[r]^-{\Psi_d} & SL^\pm(\mathbb{Z}/d)\ltimes S(\mathbb{Z}/d) \ar@{->}[r] & 1,} \\
& \xymatrix@C=7mm@R=7mm{1 \ar@{->}[r] & \mathcal{AB}_{\infty ,1}[d] \ar@{->}[r] & \mathcal{AB}_{\infty,1} \ar@{->}[r]^-{\Psi_d} & SL^\pm(\mathbb{Z}/d)  \ar@{->}[r] & 1.}
\end{aligned}
\end{equation*}

\end{lema}

\begin{proof}
We only give the proof for the first short exact sequence (the proof for the other two short exact sequences is analogous). By definition of $\mathcal{B}_{g,1}[d],$ we have an exact sequence
$$
\xymatrix@C=7mm@R=7mm{1 \ar@{->}[r] & \mathcal{B}_{g ,1}[d] \ar@{->}[r] & \mathcal{B}_{g ,1} \ar@{->}[r]^-{\Psi_d} & SL^\pm_g(\mathbb{Z}/d)\ltimes S_g(\mathbb{Z}/d).}
$$
Then taking the direct limit, the above exact sequences becomes
$$
\xymatrix@C=7mm@R=7mm{1 \ar@{->}[r] & \mathcal{B}_{\infty ,1}[d] \ar@{->}[r] & \mathcal{B}_{\infty ,1} \ar@{->}[r]^-{\Psi_d} & SL^\pm(\mathbb{Z}/d)\ltimes S(\mathbb{Z}/d),}
$$
and by Lemmas \eqref{lem_A,B,AB_infty}, \eqref{lem_SL[d]_2_stab} we have that $\Psi_d:\;\mathcal{B}_{\infty ,1} \longrightarrow SL^\pm(\mathbb{Z}/d)\ltimes S(\mathbb{Z}/d)$ is surjective, as desired.
\end{proof}

\section{Heegaard splittings of homology $3$-spheres}

In this section we show what are the relation between (mod $d$)-Torelli groups, $\mathbb{Q}$-homology $3$-spheres and $\mathbb{Z}/d$-homology $3$-spheres.

\begin{defi}
Let $R=\mathbb{Q} \text{ or } \mathbb{Z}/d.$ A $3$-manifold $X$ is a $R$-homology $3$-sphere if
$$H_*(X;R) \cong H_*(\mathbf{S}^3;R).$$
\end{defi}

\begin{rem}
By Universal coefficients Theorem and the Poincaré duality, an orientable compact connected $3$-manifold $M$ is a $\mathbb{Q}$-homology $3$-sphere if and only if $H_1(M;\mathbb{Z})$ is finite.
As consequence, every $\mathbb{Z}/d$-homology $3$-sphere $M$ is also a $\mathbb{Q}$-homology $3$-sphere, because $H_1(M;\mathbb{Z}/d)=0$ implies that $H_1(M;\mathbb{Z})$ has no free part, i.e. $H_1(M;\mathbb{Z})$ is finite.
\end{rem}

\begin{exe}[Section 9.B in \cite{rolf}] Examples of $\mathbb{Q}$-homology $3$-sphere are the \textit{lens spaces}, which are defined as follows.
Consider two solid tori $V_1$ and $V_2.$
Fix longitude and meridian generators $a_1$ and $b_1$ for $\pi_1(\partial V_1),$ for instance

\[
\begin{tikzpicture}[scale=.7]
\draw[very thick] (-4.5,0) circle [radius=2];
\draw[very thick] (-4.5,0) circle [radius=.4];

\draw[->, thick] (-5.3,0) to [out=90,in=180] (-4.5,0.8);
\draw[thick] (-3.7,0) to [out=90,in=0] (-4.5,0.8);
\draw[thick] (-5.3,0) to [out=-90,in=180] (-4.5,-0.8) to [out=0,in=-90] (-3.7,0);

\draw[thick] (-4.5,-0.4) to [out=180,in=90] (-5,-1.2);
\draw[->,thick] (-4.5,-2) to [out=180,in=-90] (-5,-1.2);
\draw[thick, dashed] (-4.5,-0.4) to [out=0,in=0] (-4.5,-2);

\node [left] at (-5,-1.2) {$b_1$};
\node [above] at (-4.5,0.8) {$a_1$};

\end{tikzpicture}
\]
Consider an element $h\in Homeo^+(\partial V_1)$ such that $h_*(b_2)=pa_1+qb_1,$
where $p$ and $q$ are coprime integers.

The $3$-manifold obtained gluing the boundaries of $V_1$ and $V_2$ by the homeomorphism $h$ is called the lens space of type $(p,q)$ and denoted traditionally by $L(p,q).$
In other words, a $3$-manifold is a lens space if and only if it contains a solid torus, the closure of whose complement is also a solid torus.
Moreover we have the following classification result:
\begin{teo}[Remark 9.B.7 in \cite{rolf}] 
\label{teo_lens_classification}
Two lens spaces $L(p_1,q_1),$ $L(p_2,q_2)$ are homeomorphic if and only if
$$p_1=\pm p_2 \quad\text{and}\quad q_1\equiv \pm q_2^{\pm 1} (\text{mod } p_1).$$
\end{teo}
\end{exe}

Next we show that a Heegaard splitting with gluing map an element of (mod $d$) Torelli group is 
$\mathbb{Z}/d$-homology $3$-sphere and so a $\mathbb{Q}$-homology $3$-sphere.

\begin{prop}
\label{prop_homology_modp}
Let $S^3= \mathcal{H}_g \cup_{\iota_g} -\mathcal{H}_g.$
If we twist this glueing by an arbitrary map $\phi\in\mathcal{M}_{g,1}[d]$ we get a $\mathbb{Z}/d$-homology sphere $S_\phi^3= \mathcal{H}_g \cup_{\iota_g \phi} -\mathcal{H}_g.$
\end{prop}

\begin{proof}
In order to prove this statement we compute the homology groups of $\mathcal{H}_g \cup_{\iota_g \phi} -\mathcal{H}_g.$
using a Mayer-Vietoris sequence.

Consider the following Mayer-Vietoris sequence:
\begin{equation}
\label{les-MV1_Q}
\begin{aligned}
 H_3(\mathcal{H}_g;\mathbb{Z}/d)\oplus  H_3(-\mathcal{H}_g;\mathbb{Z}/d) \longrightarrow & H_3(M;\mathbb{Z}/d) \longrightarrow \\
 \longrightarrow & H_2(\Sigma_g;\mathbb{Z}/d) \longrightarrow H_2(\mathcal{H}_g;\mathbb{Z}/d)\oplus  H_2(-\mathcal{H}_g;\mathbb{Z}/d). 
\end{aligned}
\end{equation}
Observe that a handlebody $\mathcal{H}_g$ is homotopy equivalent to a wedge of circles $\bigvee_g S^1,$ so
$$H_3(\mathcal{H}_g;\mathbb{Z}/d)=H_3(-\mathcal{H}_g;\mathbb{Z}/d)=0 \quad \text{and} \quad H_2(\mathcal{H}_g;\mathbb{Z}/d)=H_2(-\mathcal{H}_g;\mathbb{Z}/d)=0.$$
Thus the exact sequence \eqref{les-MV1_Q} becomes
$$
\xymatrix@C=7mm@R=7mm{\cdots \ar@{->}[r] & 0 \ar@{->}[r] & H_3(M;\mathbb{Z}/d) \ar@{->}[r] & H_2(\Sigma_g;\mathbb{Z}/d) \ar@{->}[r] & 0 \ar@{->}[r] & \cdots .}
$$
Because $\Sigma_g$ is orientable $H_2(\Sigma_g;\mathbb{Z}/d)=\mathbb{Z}/d.$
Hence $H_3(M;\mathbb{Z}/d) \cong H_2(\Sigma_g;\mathbb{Z}/d)=\mathbb{Z}/d.$

Writing more terms of the Mayer-Vietoris sequence we get
\begin{equation}
\label{les-MV2_Q}
\begin{aligned}
 H_2(\mathcal{H}_g;\mathbb{Z}/d)\oplus  H_2(-\mathcal{H}_g;\mathbb{Z}/d) \longrightarrow & H_2(M;\mathbb{Z}/d) \longrightarrow \\
 \longrightarrow & H_1(\Sigma_g;\mathbb{Z}/d) \longrightarrow H_1(\mathcal{H}_g;\mathbb{Z}/d)\oplus  H_1(-\mathcal{H}_g;\mathbb{Z}/d). 
\end{aligned}
\end{equation}
Since $H_2(\mathcal{H}_g;\mathbb{Z}/d)=H_2(-\mathcal{H}_g;\mathbb{Z}/d)=0,$ the exact sequence \eqref{les-MV2_Q} becomes
$$
\xymatrix@C=7mm@R=7mm{ 0 \ar@{->}[r] & H_2(M;\mathbb{Z}/d) \ar@{->}[r] & H_1(\Sigma_g;\mathbb{Z}/d) \ar@{->}[r] & H_1(\mathcal{H}_g;\mathbb{Z}/d)\oplus  H_1(-\mathcal{H}_g;\mathbb{Z}/d) .}
$$
Thus $H_2(M;\mathbb{Z}/d)=ker\left(\varphi: H_1(\Sigma_g;\mathbb{Z}/d) \longrightarrow H_1(\mathcal{H}_g;\mathbb{Z}/d)\oplus  H_1(-\mathcal{H}_g;\mathbb{Z}/d)\right).$

Recall that the basis of $H_1(\Sigma_g)$ is given by the homology classes $\{a_1,\ldots, a_g,b_1, \ldots , b_g\},$ the basis of $H_1(\mathcal{H}_g)$ is given by the homology classes $\{a_1,\ldots, a_g\}$ and the basis of $H_1(-\mathcal{H}_g)$ is given by the homology classes $\{-a_1,\ldots, -a_g\}.$

The map $\varphi$ is given by $\varphi=(incl.,incl. \circ \iota_g \circ \phi),$ where $\iota_g$ denotes the map that switch $a_i$ and $-b_i$ for all i, and $incl.$ denotes the map induced by the natural inclusion of $\Sigma_g$ in $\mathcal{H}_g,$ i.e. the map that sends all $b_i$'s to zero.
Recall that $\phi\in \mathcal{M}_{g,1}[d],$ i.e. $\Psi_d(\phi)=Id.$
So the induced map $\varphi$ in homology with coefficients $\mathbb{Z}/d$ has the following matrix form:
$$\varphi=\left(\begin{matrix}
Id_g & 0 \\
0 & Id_g
\end{matrix}\right)\in M_{2g\times 2g}(\mathbb{Z}/d).$$
Thus $ker(\varphi)=0,$ i.e. $H_2(M;\mathbb{Z}/d)=0.$

Now, using the Mayer-Mietoris for reduced homology, since $\varphi= Id_{2g}$ i.e. $\varphi$ is an isomorphism, we get the following exact sequence:
$$
\xymatrix@C=7mm@R=7mm{ 0 \ar@{->}[r] & H_1(M;\mathbb{Z}/d) \ar@{->}[r] & \tilde{H}_0(\Sigma_g;\mathbb{Z}/d)  \ar@{->}[r] & }$$
$$\xymatrix@C=7mm@R=7mm{\ar@{->}[r] &
\tilde{H}_0(\mathcal{H}_g;\mathbb{Z}/d)\oplus  \tilde{H}_0(-\mathcal{H}_g;\mathbb{Z}/d) \ar@{->}[r] & \tilde{H}_0(M;\mathbb{Z}/d) \ar@{->}[r] & 0}
.$$
Since $H_0(\Sigma_g;\mathbb{Z}/d)=H_0(\mathcal{H}_g;\mathbb{Z}/d)=  H_0(-\mathcal{H}_g;\mathbb{Z}/d)=\mathbb{Z}/d,$ then
$$\tilde{H}_0(\Sigma_g;\mathbb{Z}/d)=\tilde{H}_0(\mathcal{H}_g;\mathbb{Z}/d)=  \tilde{H}_0(-\mathcal{H}_g;\mathbb{Z}/d)=0.$$
Thus, by above exact sequence,
$$H_1(M;\mathbb{Z}/d)=0 \quad \text{and} \quad H_0(M;\mathbb{Z}/d)= \tilde{H}_0(M;\mathbb{Z}/d)\oplus \mathbb{Z}/d=\mathbb{Z}/d.$$

Therefore $M$ is a $\mathbb{Z}/d$-homology $3$-sphere.
\end{proof}

\begin{rem}
As a direct consequence of Proposition \eqref{prop_homology_modp} we get that the Heegaard splitting $S_\phi^3= \mathcal{H}_g \cup_{\iota_g \phi} -\mathcal{H}_g$ with $\phi\in\mathcal{M}_{g,1}[d]$ is a $\mathbb{Q}$-homology $3$-sphere
\end{rem}

We now look for the converse:
when is a $\mathbb{Q}$-homology $3$-sphere constructed from a Heegaard splitting with gluing map an element of $\mathcal{M}_{g,1}[d]$ ?

Denote by $\mathcal{S}^3(d)\subset\mathcal{V}^3$ the subset of $\mathbb{Z}/d$-homology $3$-spheres
and by $\mathcal{S}^3[d]\subset \mathcal{S}^3(d)$ the set of $\mathbb{Z}/d$-homology $3$-spheres which are homeomorphic to $\mathcal{H}_g \cup_{\iota_g\phi} -\mathcal{H}_g$ for some $\phi\in \mathcal{M}_{g,1}[d].$

\begin{teo}[Theorem 368 in \cite{Gock}]
\label{teo_smith}
Let $A\in  M_{m\times n}(\mathbb{Z})$ be given. There exist matrices $U\in GL_m(\mathbb{Z}),$ $V\in GL_n(\mathbb{Z})$ and a diagonal matrix $S\in M_{m\times n}(\mathbb{Z})$ such that
$A=USV,$
the diagonal entries of $S$ are $d_1,d_2,\ldots , d_r,0,\ldots , 0,$ each $d_i$ is a positive integer, and $d_i|d_{i+1}$ for $i=1,2,\ldots , r-1.$

The diagonal matrix $S$ is called the \textit{Smith normal form} of $A.$
The diagonal entries $d_1,d_2,\ldots, d_r,$ of the Smith normal form of $A$ are called the elementary divisors (or sometimes the \textit{invariant factors}) of $A.$
\end{teo}

\begin{lema}
\label{lema_n_detH}
Let $M=\mathcal{H}_g\cup_{\iota_gf} -\mathcal{H}_g$ with $f\in \mathcal{M}_{g,1}$ be a $\mathbb{Q}$-homology $3$-sphere, $n=|H_1(M;\mathbb{Z})|$ and
$\Psi(f)=\left(\begin{smallmatrix}
E & F \\
G & H
\end{smallmatrix}\right).$ Then $n=|det(H)|.$
\end{lema}
\begin{proof}
Consider the Mayer-Vietoris sequence for the reduced homology of $\mathcal{H}_g\cup_{\iota_gf} -\mathcal{H}_g:$
\begin{equation}
\label{ses_MV_rat}
\xymatrix@C=7mm@R=7mm{ H_1(\Sigma_g;\mathbb{Z}) \ar@{->}[r]^-{\varphi} & H_1(\mathcal{H}_g;\mathbb{Z})\oplus  H_1(-\mathcal{H}_g;\mathbb{Z}) \ar@{->}[r] & H_1(M;\mathbb{Z}) \ar@{->}[r] & 0.}
\end{equation}
Then $H_1(M;\mathbb{Z})\cong Coker(\varphi).$
A direct inspection shows that if $\Psi(f)=\left(\begin{smallmatrix}
E & F \\
G & H
\end{smallmatrix}\right),$ we have that $\varphi=\left(\begin{smallmatrix}
Id & 0 \\
G & H
\end{smallmatrix}\right).$
Hence,
\begin{equation*}
H_1(M;\mathbb{Z})\cong Coker(H).
\end{equation*}
Next, by the Theory of invariant factors (see Theorem \eqref{teo_smith}), we know that there exist matrices $U,V\in GL_{2g}(\mathbb{Z})$ such that
$$UHV=\left(\begin{smallmatrix}
d_1 & & & & & \\
 & \ddots & & & & \\
 & & d_k & & & \\
 & & & 0 & & \\
 & & & & \ddots & \\
 & & & & & 0
 \end{smallmatrix}\right).$$
As a consequence,
\begin{equation*}
Coker(H)\cong \mathbb{Z}/d_1\times \mathbb{Z}/d_2\times \cdots \mathbb{Z}/d_k\times \mathbb{Z}^{g-k}.
\end{equation*}
Since $M$ is a $\mathbb{Q}$-homology sphere, by Universal coefficients Theorem, we have that
$0=H_1(M;\mathbb{Q})\cong H_1(M;\mathbb{Z})\otimes \mathbb{Q}\cong \mathbb{Q}^{g-k}.$ Then $k=g.$

Hence,
\begin{equation*}
det(H)=\pm det(UHV)=\pm \prod_{i=1}^{2g}d_i\neq 0.
\end{equation*}
Moreover, we have that
\begin{equation*}
n=|H_1(M;\mathbb{Z})|=\prod_{i=1}^{2g}d_i\neq 0.
\end{equation*}
Therefore we get that
$n=|det(H)|,$ as desired.
\end{proof}

Now we are ready to study whenever a $\mathbb{Q}$-homology $3$-sphere is homeomorphic to a Heegaard splitting with gluing map an element of (mod $d$) Torelli group.
The following result is inspired in Proposition 6 in \cite{per} due to B. Perron.

\begin{teo}
\label{teo_rat_homology_gen}
Let $M$ be a $\mathbb{Q}$-homology $3$-sphere and $n=|H_1(M;\mathbb{Z})|.$
Then $M\in \mathcal{S}^3[d],$ i.e. $M$ has a Heegaard splitting $\mathcal{H}_g\cup_{\iota_gf} -\mathcal{H}_g$ of some genus $g$ with gluing map $f\in \mathcal{M}_{g,1}[d]$ with $d\geq 2,$ if and only if $d$ divides $n-1$ or $n+1.$
\end{teo}
\begin{proof}
We first prove that if $M$ is a $\mathbb{Q}$-homology $3$-sphere with $n=|H_1(M;\mathbb{Z})|$ and $d\geq 2$ divides $n-1$ or $n+1,$ then $M\in \mathcal{S}^3[d].$

By Theorem \eqref{bij_MCG_3man}, there exists an element $f\in\mathcal{M}_{g,1}$ such that $M$ is homeomorphic to $\mathcal{H}_g\cup_{\iota_gf} -\mathcal{H}_g.$
By definition of $\mathcal{M}_{g,1}[d]$ we have the short exact sequence
\begin{equation}
\label{ses_def_MCGmodp}
\xymatrix@C=10mm@R=13mm{1 \ar@{->}[r] & \mathcal{M}_{g,1}[d] \ar@{->}[r] & \mathcal{M}_{g,1} \ar@{->}[r]^-{\Psi_d} & Sp_{2g}(\mathbb{Z}/d) \ar@{->}[r] & 1 .}
\end{equation}
Consider the image of $f$ by $\Psi :$
\begin{equation*}
\Psi(f)=\left(\begin{matrix}
E & F \\
G & H
\end{matrix}\right).
\end{equation*}
We show that there exist matrices $X\in Sp_{2g}^{A\pm}(\mathbb{Z}/d),$ $Y\in Sp_{2g}^{B\pm}(\mathbb{Z}/d),$
such that
\begin{equation}
\label{eq_XY}
X\Psi_d(f)Y=Id.
\end{equation}
Then, by Lemma \eqref{lem_B[d]_3_stab}, we will get that there exist an integer $h\geq g$ and elements $\tilde{\xi_a}\in \mathcal{A}_{h,1},$ $\tilde{\xi_b}\in\mathcal{B}_{h,1}$, $\tilde{f}\in \mathcal{M}_{h,1}$ such that $\Psi_d(\tilde{\xi_a})=i_{g,h}(X),$ $\Psi_d(\tilde{\xi_b})=i_{g,h}(Y),$ $\tilde{f}=i_{g,h}(f)$ and
$$\Psi_d(\tilde{\xi_a} \tilde{f}\tilde{\xi_b})=Id.$$
As a consequence, by the short exact sequence \eqref{ses_def_MCGmodp} and Theorem \eqref{bij_MCG_3man}, we will get that $f$ is equivalent to an element of $\mathcal{M}_{h,1}[d].$

In order to construct such matrices $X,Y,$ we first show that $H\in SL_g^\pm(\mathbb{Z}/d).$

By hypothesis, $d\mid n-1$ or $d\mid n+1,$
i.e. $n \equiv \pm 1 \; (\text{mod }d).$
Moreover, by Lemma \eqref{lema_n_detH}, we know that $n=|det(H)|.$
Therefore, $det(H)\equiv \pm 1 \; (\text{mod }d),$ i.e. $H\in SL_g^\pm(\mathbb{Z}/d).$

Next we proceed to construct the aforementioned matrices $X,Y.$

Consider the matrix
$$X=\left(\begin{matrix}
Id & A \\
0 & Id
\end{matrix}\right),$$
with $A=-FH^{-1}\in M_{g\times g}(\mathbb{Z}/d).$

We prove that $X$ is an element of $Sp_{2g}^{A\pm}(\mathbb{Z}/d).$

Since $\Psi_d(f)\in Sp_{2g}(\mathbb{Z}/d),$ by definition, $H^tF$ is symmetric and as a consequence the following equality holds
$$
A= -FH^{-1}=-(H^t)^{-1}H^tFH^{-1}
= -(H^t)^{-1}F^t=-(H^{-1})^tF^t=A^t.
$$
Then $A\in S_g(\mathbb{Z}/d),$ and therefore $X\in Sp_{2g}^{A\pm}(\mathbb{Z}/d).$

Next, observe that we have that
$$\left(\begin{matrix}
Id & A \\
0 & Id
\end{matrix}\right)
\left(\begin{matrix}
E & F \\
G & H
\end{matrix}\right)=
\left(\begin{matrix}
E+AG & F+AH \\
G & H
\end{matrix}\right)=
\left(\begin{matrix}
E+AG & 0 \\
G & H
\end{matrix}\right),
$$
which is an element of $Sp_{2g}^B(\mathbb{Z}/d).$

In fact, since $H\in SL^\pm_g(\mathbb{Z}/d),$ we know that
$\left(\begin{smallmatrix}
E+AG & 0 \\
G & H
\end{smallmatrix}\right)\in Sp_{2g}^{B\pm}(\mathbb{Z}/d).$

Therefore, setting
$$Y=\left(\begin{matrix}
E+AG & 0 \\
G & H
\end{matrix}\right)^{-1}\in Sp_{2g}^{B\pm}(\mathbb{Z}/d),$$
we get the result.

Next we show that the condition $d$ divides $n+1$ or $n-1$ is a necessary condition for $M$ to be in $\mathcal{S}^3[d].$
Let $M\in \mathcal{S}^3[d].$ By Theorem \eqref{bij_MCG_3man}, there exists an element $f\in\mathcal{M}_{g,1}$ such that $M$ is homeomorphic to $\mathcal{H}_g\cup_{\iota_gf} -\mathcal{H}_g.$
In addition, since $M\in \mathcal{S}^3[d],$ there exists $h\geq g$ for which the class $\tilde{f}=i_{g,h}(f)$ in $\mathcal{A}_{h,1}\backslash \mathcal{M}_{h,1}/\mathcal{B}_{h,1}$ contains an element in $\mathcal{M}_{h,1}[d].$ Then there exist $\xi_a\in \mathcal{A}_{h,1}, \xi_b\in \mathcal{B}_{h,1}$ and $\varphi\in \mathcal{M}_{h,1}[d]$ such that
$$\xi_a\varphi \xi_b=\tilde{f}.$$
Thus,
$$\Psi_d(\tilde{f})=\Psi_d(\xi_a)\Psi_d(\varphi)\Psi_d( \xi_b)=\Psi_d(\xi_a)\Psi_d( \xi_b).$$
Set 
$$\Psi_d(\xi_a)=\left(\begin{matrix}
E & A \\
0 & {}^tE^{-1}
\end{matrix}\right)\in Sp_{2h}^{A\pm}(\mathbb{Z}/d),
\quad
\Psi_d( \xi_b)=\left(\begin{matrix}
F & 0 \\
B & {}^tF^{-1}
\end{matrix}\right)\in Sp_{2h}^{B\pm}(\mathbb{Z}/d).$$
Then
$$\Psi_d(\tilde{f})=\Psi_d(\xi_a)\Psi_d( \xi_b)=
\left(\begin{matrix}
E & A \\
0 & {}^tE^{-1}
\end{matrix}\right)
\left(\begin{matrix}
F & 0 \\
B & {}^tF^{-1}
\end{matrix}\right)=
\left(\begin{matrix}
EF+AB & A{}^tF^{-1} \\
{}^tE^{-1}B & {}^tE^{-1}{}^tF^{-1}
\end{matrix}\right).$$
Since $E,F\in SL_h^\pm(\mathbb{Z}/d)$ and by Lemma \eqref{lema_n_detH} $n=|det({}^tE^{-1}{}^tF^{-1})|,$ then we have that
$$n\equiv \pm det({}^tE^{-1}{}^tF^{-1})\equiv \pm 1 \;(\text{mod }d),$$
as desired.
\end{proof}

As a consequence of Theorem \eqref{teo_rat_homology_gen}, we get that the family of groups $\{\mathcal{M}_{g,1}[d]\}_{g,d}$ is enough to produce all $\mathbb{Q}$-homology spheres.

Notice that Theorem \eqref{teo_rat_homology_gen} give us a criterion to know whenever a $\mathbb{Q}$-homology sphere can be formed as a Heegaard splitting with gluing map an element of the (mod $d$) Torelli group.
Next, using this criterion, we study whenever the set of $\mathbb{Z}/d$-homology $3$-spheres, $\mathcal{S}^3(d),$ and the set of $\mathbb{Z}/d$-homology $3$-spheres which are constructed as a Heegaard splitting with gluing map an element of the mod $d$ Torelli group, $\mathcal{S}^3[d],$ coincide.
As we will see, unlike the integral case, these two sets do not coincide in general.

\begin{lema}
\label{lem_inv_mod_d}
Let $d$ be a positive integer, the invertible elements of $\mathbb{Z}/d$ are contained in $\{1,-1\}$ if and only if $d=2,3,4,6.$
\end{lema}

\begin{proof}
Let $d\geq 2$ be an integer. Recall that an element of $\mathbb{Z}/d$ is invertible if and only if $gcd(x,d)=1.$
As a consequence, given a prime number $p$ and $n\in \mathbb{N},$
$|(\mathbb{Z}/p^n)^\times |=(p-1)p^{n-1}.$
Let $d=p_1^{n_1}\cdots p_r^{n_r},$ we know that
$(\mathbb{Z}/d)^\times =(\mathbb{Z}/p_1^{n_1})^\times \times \cdots \times (\mathbb{Z}/p_r^{n_r})^\times,$ so
$$|(\mathbb{Z}/d)^\times|=\prod_{i=1}^{r}(p_i-1)p_i^{n_i-1}.$$
Observe that if for some $i,j$ with $j\neq i,$
$p_i\geq 5,$ or $p_i=3$ and $n_i\geq 2,$ or $p_i=2$ and $n_i\geq 3,$
or $p_i=2,$ $n_i=2$ and $p_j=3,$ then $|(\mathbb{Z}/d)^\times |\geq 3.$ On the other hand, we directly check that $|(\mathbb{Z}/d)^\times |\leq 2$ for $d=2,3,4,6.$
\end{proof}

\begin{prop}
\label{prop_counterexample_M[d]}
Given an integer $d=5$ or $d\geq 7.$ The set of $\mathbb{Z}/d$-homology $3$-spheres, $\mathcal{S}^3(d),$ and the set of $\mathbb{Z}/d$-homology $3$-spheres which are constructed as a Heegaard splitting with gluing map an element of the mod $d$ Torelli group, $\mathcal{S}^3[d],$ do not coincide.
\end{prop}

\begin{proof}
We construct a $\mathbb{Q}$-homology $3$-sphere $M\in \mathcal{S}^3(d),$ which does not belong to $\mathcal{S}^3[d].$
By Lemma \eqref{lem_inv_mod_d} we know that, for $d=5$ and for every $d\geq 7,$ there exists an element $u\in (\mathbb{Z}/d)^\times$ with $u$ different from $\pm 1.$
Fix such an element $u\in (\mathbb{Z}/d)^\times.$

Consider the matrix
$\left(\begin{smallmatrix}
G & 0 \\
0 & {}^tG^{-1}
\end{smallmatrix}\right)\in Sp_{2g}(\mathbb{Z}/d),$
with $G\in GL_g(\mathbb{Z}/d)$ given by
$$G=\left(\begin{matrix}
u & 0 \\
0 & Id
\end{matrix}\right).$$
Since $\Psi_d$ is surjective, we know that there is an element $f\in \mathcal{M}_{g,1}$ satisfying that
$$\Psi_d(f)=\left(\begin{matrix}
G & 0 \\
0 & {}^tG^{-1}
\end{matrix}\right).$$
Consider the $3$-manifold $M$ given by the Heegaard splitting $\mathcal{H}_g\cup_{\iota_gf} -\mathcal{H}_g.$

Following the argument of the proof of Proposition \eqref{prop_homology_modp} with $f$ defined as before, one gets that $M$ is a $\mathbb{Z}/d$-homology $3$-sphere, i.e. $M\in \mathcal{S}^3(d).$

On the other hand, by Lemma \eqref{lema_n_detH}, we have that $\pm n=\pm |H_1(M;\mathbb{Z})|=det({}^tG^{-1})=u^{-1}\neq \pm 1.$
Then, by Theorem \eqref{teo_rat_homology_gen}, $M$ is not in $\mathcal{S}^3[d].$
\end{proof}

\begin{prop}
Let $d=2,3,4,6.$ If $M$ is a $\mathbb{Z}/d$-homology $3$-sphere, then $M$ has a Heegaard splitting $\mathcal{H}_g\cup_{\iota_gf} -\mathcal{H}_g$ of some genus $g$ with gluing map $f\in \mathcal{M}_{g,1}[d].$
\end{prop}

\begin{proof}
Let $d=2,3,4,6,$ $M$ a $\mathbb{Z}/d$-homology $3$-sphere and $n=|H_1(M;\mathbb{Z})|.$ We show that $n\equiv \pm 1 \;(\text{mod }d)$ and by Theorem \eqref{teo_rat_homology_gen} we will get the result.

Since $M$ is a $\mathbb{Z}/d$-homology $3$-sphere, we have that $gcd(d,n)=1.$
Then, by Bezout's identity, there exist integers $\lambda_1,\lambda_2$ such that $\lambda_1d+\lambda_2n=1.$
As a consequence, $\lambda_2 n\equiv 1 \;(\text{mod }d).$
Hence, $n\in (\mathbb{Z}/d)^\times.$
Finally, by Lemma \eqref{lem_inv_mod_d}, one gets that $|(\mathbb{Z}/d)^\times|\leq 2$ for $d=2,3,4,6,$ as desired.
\end{proof}

As a consequence of Proposition \eqref{prop_homology_modp} and Theorem
\eqref{bij_MCG_3man} we get the following result:

\begin{teo}
\label{teo_coresp_Zd-HS}
The following map is well defined and bijective:
\begin{align*}
\lim_{g\to \infty}\mathcal{A}_{g,1}\backslash\mathcal{M}_{g,1}[d]/\mathcal{B}_{g,1} & \longrightarrow \mathcal{S}^3[d], \\
\phi & \longmapsto S^3_\phi=\mathcal{H}_g \cup_{\iota_g\phi} -\mathcal{H}_g.
\end{align*}
\end{teo}

From the group-theoretical point of view, the induced equivalence relation on $\mathcal{M}_{g,1}[d],$ which is given by:
\begin{equation}
\phi \sim \psi \quad\Leftrightarrow \quad \exists \zeta_a \in \mathcal{A}_{g,1}\;\exists \zeta_b \in \mathcal{B}_{g,1} \quad \text{such that} \quad \zeta_a \phi \zeta_b=\psi,
\end{equation}
is quite unsatisfactory. However,
as in the integral case, Lemma 4 in \cite{pitsch}, we can rewrite this equivalence relation as follows:
\begin{lema}
\label{lema_eqiv_coset[d]}
Two maps $\phi, \psi \in \mathcal{M}_{g,1}[d]$ are equivalent if and only if there exists a map $\mu \in \mathcal{AB}_{g,1}$ and two maps $\xi_a\in \mathcal{A}_{g,1}[d]$ and $\xi_b\in\mathcal{B}_{g,1}[d]$ such that $\phi=\mu \xi_a\psi \xi_b\mu^{-1}.$
\end{lema}
\begin{proof}
The proof of this Lemma is analogous to the proof of Lemma $4$ in \cite{pitsch} due to W. Pitsch.
The ''if'' part of the Lemma is trivial. Conversely, assume that $\psi=\xi_a \phi \xi_b,$ where $\psi, \phi \in \mathcal{M}_{g,1}[d].$ Applying the symplectic representation modulo $d$ $\Psi_d$ to this equality we get
$$Id=\Psi_d(\xi_a)\Psi_d(\xi_b) \; (\text{mod } d).$$
By section \eqref{homol_homoto_actions} , $\Psi(\xi_b)$ is of the form
\begin{equation*}
\left(\begin{matrix}
G & 0 \\
M & {^t}G^{-1}
\end{matrix}\right)\qquad \text{with}\; G\in GL_g(\mathbb{Z})\; \text{and}\; {}^tGM\in S_g(\mathbb{Z}).
\end{equation*}
Similarly, $\Psi(\xi_a)$ is of the form
\begin{equation*}
\left(\begin{matrix}
H & N \\
0 & {^t}H^{-1}
\end{matrix}\right) \qquad \text{with}\; H\in GL_g(\mathbb{Z})\; \text{and}\; H^{-1}N\in S_g(\mathbb{Z}).
\end{equation*}
Therefore,
\begin{equation*}
\left(\begin{matrix}
Id & 0 \\
0 & Id
\end{matrix}\right)\equiv \left(\begin{matrix}
H & N \\
0 & {^t}H^{-1}
\end{matrix}\right)
\left(\begin{matrix}
G & 0 \\
M & {^t}G^{-1}
\end{matrix}\right)=
\left(\begin{matrix}
HG+NM & N^tG^{-1} \\
{}^tH^{-1}M & {^t}(HG)^{-1}
\end{matrix}\right) \; (\text{mod } d).
\end{equation*}
Thus $N\equiv 0 \equiv M\; (\text{mod } d)$ and $G\equiv H^{-1}\; (\text{mod } d).$

By Lemma \eqref{lem_AB}, we can choose a map $\mu \in \mathcal{AB}_{g,1}$ such that
$$\Psi(\mu)=\left(\begin{matrix}
H & 0 \\
0 & {^t}H^{-1}
\end{matrix}\right)$$

Since $N\equiv 0 \equiv M\; (\text{mod } d)$ and $G\equiv H^{-1}\; (\text{mod } d),$ we have that
\begin{align*}
\Psi_d(\mu^{-1}\xi_a)= & \left(\begin{matrix}
H^{-1} & 0 \\
0 & {^t}H
\end{matrix}\right)
\left(\begin{matrix}
H & N \\
0 & {^t}H^{-1}
\end{matrix}\right)=\left(\begin{matrix}
Id & H^{-1}N \\
0 & Id
\end{matrix}\right)\equiv Id \; (\text{mod } d). \\
\Psi_d(\xi_b\mu)= &
\left(\begin{matrix}
G & 0 \\
M & {^t}G^{-1}
\end{matrix}\right)
\left(\begin{matrix}
H & 0 \\
0 & {^t}H^{-1}
\end{matrix}\right)=\left(\begin{matrix}
GH & 0 \\
MH & {^t}(GH)^{-1}
\end{matrix}\right)\equiv Id \; (\text{mod } d).
\end{align*}
Therefore we get
$$\psi=\mu \circ (\mu^{-1}\xi_a)\phi(\xi_b \mu)\circ \mu^{-1},$$
where $(\mu^{-1}\xi_a)\in \mathcal{A}_{g,1}[d],$ $(\xi_b\mu)\in \mathcal{B}_{g,1}[d]$ and $\mu\in \mathcal{AB}_{g,1}$ as desired.
\end{proof}

To summarize, we get the following bijective map:
\begin{equation}
\label{bij_M[p]}
\begin{aligned}
\lim_{g\to \infty}\left(\mathcal{A}_{g,1}[d]\backslash\mathcal{M}_{g,1}[d]/\mathcal{B}_{g,1}[d]\right)_{\mathcal{AB}_{g,1}} & \longrightarrow \mathcal{S}^3[d], \\
\phi & \longmapsto S^3_\phi=\mathcal{H}_g \cup_{\iota_g\phi} -\mathcal{H}_g.
\end{aligned}
\end{equation}
The bijection \eqref{bij_M[p]} will play an important role in this thesis.
As we will see in the following sections, it allows us to relate the structure of invariants of $\mathcal{S}^3[d],$ with the algebraic structure of the group $\mathcal{M}_{g,1}[d].$

As we have seen in this chapter, for $d\neq 2,3,4,6$ the set of $\mathbb{Z}/d$-homology $3$-spheres, $\mathcal{S}^3(d),$ and the set of $\mathbb{Z}/d$-homology $3$-spheres which are constructed as a Heegaard splitting with gluing map an element of the mod $d$ Torelli group, $\mathcal{S}^3[d],$ do not coincide.
Then a natural question which arise from this chapter is:

What is the difference between $\mathcal{S}^3(d)$ and $\mathcal{S}^3[d]$ for $d\neq 2,3,4,6$ ?


\chapter{Automorphisms of descending mod-p central series}
\label{chapter: versal ext}

Let $\Gamma$ be a free group of finite rank and $\{ \Gamma_k^\bullet\}_k$ be the Stallings or Zassenhaus $p$-central series and
$\mathcal{N}_{k}^\bullet=\Gamma/\Gamma_{k+1}^\bullet.$ In this chapter, we compare the automorphisms of the group $\mathcal{N}^\bullet_k.$ To be more precise, we show that there is a well-defined homomorphism
$\psi_k^\bullet:Aut(\mathcal{N}_{k+1}^\bullet)\rightarrow Aut(\mathcal{N}_{k}^\bullet)$
that, using the theory of versal extensions mod $p,$ which is a mod $p$ version of the versal extensions given in \cite{pitsch2}, fits into a non central extension
\begin{equation}
\label{vers_ext_modp_intro}
\xymatrix@C=7mm@R=10mm{0 \ar@{->}[r] & Hom(\mathcal{N}^\bullet_1, \mathcal{L}^\bullet_{k+1}) \ar@{->}[r]^-{i} & Aut\;\mathcal{N}^\bullet_{k+1} \ar@{->}[r]^-{\psi_k^\bullet} & Aut \;\mathcal{N}^\bullet_k \ar@{->}[r] & 1.}
\end{equation}
In addition, we show that such extension does not split for $k\geq 2$ and we study its splitability for $k=1$ using the computations of $H^2(SL_{n}(\mathbb{Z}/p);\mathfrak{sl}_n(\mathbb{Z}/p)).$

As a starting point, in the first section, we recall some definitions about central series and mod $p$ central series with some examples. Moreover we give the notion of $p$-nilpotent group.
In the second section, we give the notion of a versal extension modulo $p$ and we develop some theory about it.
In the third section, we give some technical results that we will need in the last section.
Finally, in the last section, we study the extensions \eqref{vers_ext_modp_intro}.

\section{On $p$-central series and $p$-nilpotent groups}
\label{section: mod-p series}

\begin{defi}
A central series is a sequence of subgroups
$$ \{1\} \triangleleft A_n \triangleleft \ldots \triangleleft A_2
\triangleleft A_1 =G$$
such that $[A_j,G]\subseteq  A_{j+1}.$ Equivalently, $A_j/A_{j+1}$ is central in $G/A_{j+1}.$
\end{defi}

\begin{defi}
We say that a central series $\{G_k\}_k$ is the fastest descending series respect to a property $\mathcal{P}$ if for every central series $\{H_k\}_k$ satisfying $\mathcal{P},$ we have that $G_k\subseteq H_k$ for every $k.$
\end{defi}

\begin{defi}
Two series $\{G_i\}$ and $\{H_i\}$ are cofinal if for every natural number $n$ there exists a natural number $m$ such that $G_m<H_n,$ and for every natural number $l$ there is some natural number $k$ such that $H_k< G_l.$
\end{defi}

\begin{defi}[Lower central series]
The lower central series $\{G_k\}_k$ of a group $G$ is defined recursively by
$G_1=G$ and $G_k=[G,G_{k-1}].$
\end{defi}

Next we give the ''mod $p$ version'' of central series.

\begin{defi}
A $p$-central series is a central series whose successive quotients $A_i /A_{i+1}$ are $p$-elementary abelian groups, i.e. abelian groups in which every nontrivial element has order $p.$
\end{defi}

There are many $p$-central series, but the most important are the Zassenhaus mod-$p$ central series and the Sallings mod-$p$ central series which are defined as follows:

\begin{defi}[Zassenhaus]
The Zassenhaus mod-$p$ central series $\{G_k^Z\}$ for $G$ is defined by the rule:
$$G_k^Z=\prod_{ip^j\geq k} (G_i)^{p^j},$$
where $G_i$ denotes the $i^{th}$ term of the lower central series of $G.$

This series is the fastest descending series with respect to the properties:
\begin{itemize}
\item $[G^Z_k,G^Z_l]< G^Z_{k+l},$
\item $(G_k^Z)^p < G^Z_{pk}.$
\end{itemize}
\end{defi}

\begin{defi}[Stallings]
The Stallings mod-$p$ central series (also known as the lower $p$-central series) $\{G^S_k\}$ for $G$ is defined recursively by setting
$$G^S_1=G \quad \text{ and }\quad G^S_k=[G,G_{k-1}^S](G_{k-1}^S)^p.$$
This series is the fastest descending series with respect to the properties:
\begin{itemize}
\item $[G_k^S, G^S_l]< G^S_{k+l},$
\item $(G^S_k)^p < G^S_{k+1}.$
\end{itemize}
\end{defi}

Even though the aforementioned mod-$p$ central series are distinct, we have the following result, which will allow us to compare them.

\begin{prop}[Proposition 2.6 in \cite{coop}]
\label{prop_cofinal}
Given a prime number $p.$ The Stallings and Zassenhaus mod-$p$ central series are cofinal.
More precisely, for every positive integer $l,$
$$G^Z_{p^l}<G^S_l<G^Z_l.$$
\end{prop}

\begin{rem}
In \cite{coop} J. Cooper proved Proposition \eqref{prop_cofinal} for $G$ a free group. The same proof gives the result for any group $G.$
\end{rem}

Throughout this chapter we denote by $\Gamma$ a free group of finite rank and we set:
$$\mathcal{N}^\bullet_k=\Gamma/\Gamma^\bullet_{k+1},\qquad \mathcal{L}^\bullet_k=\Gamma^\bullet_k/\Gamma^\bullet_{k+1},$$
$$\widetilde{\mathcal{N}_{k}^\bullet}=\frac{\Gamma}{[\Gamma,\Gamma^\bullet_{k}](\Gamma^\bullet_{k})^p}, \qquad \widetilde{\mathcal{L}_{k}^\bullet}=\frac{\Gamma^\bullet_{k}}{[\Gamma,\Gamma^\bullet_{k}](\Gamma^\bullet_{k})^p},$$
where $\bullet=S \text{ or } Z.$

\begin{defi}
A group $G$ is \textit{$p$-nilpotent} if the corresponding
lower $p$-central series has finite length.
Notice that the groups $\mathcal{N}^\bullet_k$ with $\bullet=Z,S$ and $k$ a positive integer, are $p$-nilpotent groups and, in particular, $p$-groups.
\end{defi}

\begin{defi}
Let $G$ be a group and $d(G)$ the minimal cardinality of a generating set of $G.$ \textit{The Frattini subgroup $\Phi(G)$} is defined as the intersection of the set of all maximal subgroups of $G.$
\end{defi}

\begin{prop}[Proposition 1.2.4 in \cite{leed}]
Let $G$ be a finite $p$-group. Then the Frattini subgroup
$\Phi(G)=[G,G]G^p$ and $d(G)$ coincides with the rank of the abelian $p$-elementary group $G/\Phi(G).$
\end{prop}

\section{Versal extensions modulo $p$}
Fix a prime $p.$ Let $G$ be a group such that $H_1(G;\mathbb{Z}/p)$ is a free $\mathbb{Z}/p$-module.
Applying the Universal coefficients Theorem to the group $G$ and the $\mathbb{Z}/p$-module $H_2(G;\mathbb{Z}/p)$ with trivial $G$-action, since $H_1(G;\mathbb{Z}/p)$ is a free $\mathbb{Z}/p$-module, we have an induced natural isomorphism:
\begin{equation}
\label{iso_uct_h2}
\eta: \xymatrix@C=7mm@R=10mm{ H^2(G;H_2(G;\mathbb{Z}/p))
\ar@{->}[r]^-{\sim} & Hom_{\mathbb{Z}/p}(H_2(G;\mathbb{Z}/p),H_2(G;\mathbb{Z}/p)).}
\end{equation}

\begin{defi}[Section 9.5 in \cite{leed}]
Let $G$ be a group. A \textit{$p$-covering of $G$} is an extension of the form
$$\xymatrix@C=7mm@R=10mm{ 0 \ar@{->}[r] &  H_2(G;\mathbb{Z}/p) \ar@{->}[r] & E \ar@{->}[r] & G \ar@{->}[r] & 0}.$$
In this case we say that $E$ is a $p$-covering group of $G.$
\end{defi}

Denote by $v_p\in H^2(G;H_2(G;\mathbb{Z}/p))$ the preimage of the identity by the isomorphism \eqref{iso_uct_h2} and by $$V_p \quad \xymatrix@C=7mm@R=10mm{ 0 \ar@{->}[r] &  H_2(G;\mathbb{Z}/p) \ar@{->}[r] & E \ar@{->}[r] & G \ar@{->}[r] & 0,}$$
a central extension with associated cohomology class $v_p.$

\begin{defi}
A central extension $\xymatrix@C=7mm@R=10mm{ 0 \ar@{->}[r] &  H_2(G;\mathbb{Z}/p) \ar@{->}[r] & E' \ar@{->}[r] & G \ar@{->}[r] & 0}$ is called a \textit{versal extension modulo $p$} if it is equivalent to $V_p.$
\end{defi}

\begin{lema}
\label{lema_eq_cocy}
Let $v_p\in H^2(G;H_2(G;\mathbb{Z}/p))$ denote the cohomology class given by the preimage of the identity by the isomorphism $\eta$ of \eqref{iso_uct_h2}. Then for each $\phi\in Aut(G),$ the following equality holds $(H_2(\phi; \mathbb{Z}/p))_*(v_p)=\phi^*(v_p).$

Here $H_2(\phi; \mathbb{Z}/p)$ denotes the element of $Aut(H_2(\phi; \mathbb{Z}/p))$ induced by $\phi.$
\end{lema}

\begin{proof}
Fix an element $\phi\in Aut(G),$ by naturality of the isomorphism $\eta,$
we have that
\begin{align*}
\eta(\phi^*(v_p))= & (H_2(\phi; \mathbb{Z}/p))^*(id)= H_2(\phi; \mathbb{Z}/p), \\
\eta((H_2(\phi; \mathbb{Z}/p))_*(v_p))= & (H_2(\phi; \mathbb{Z}/p))_*(id)=H_2(\phi; \mathbb{Z}/p).
\end{align*}
\end{proof}

As a consequence we get the following result:

\begin{cor}
\label{cor_ver_ext_mod_p}
Let $G$ be a group and $0  \rightarrow H_2(G;\mathbb{Z}/p) \rightarrow E \rightarrow G \rightarrow 1$ a versal extension modulo $p.$
For every element $\phi\in Aut(G),$ there exists an element $\Phi\in Aut(E)$ such that the following diagram commutes
\begin{equation}
\xymatrix@C=10mm@R=10mm{
 0 \ar@{->}[r] & H_2(G;\mathbb{Z}/p) \ar@{->}[r] \ar@{->}[d]^-{H_2(\phi; \mathbb{Z}/p)} & E \ar@{->}[r] \ar@{->}[d]^-{\Phi} & G \ar@{->}[r] \ar@{->}[d]^-{\phi} & 1\\
0  \ar@{->}[r] & H_2(G;\mathbb{Z}/p) \ar@{->}[r]
 & E \ar@{->}[r] & G \ar@{->}[r] & 1. }
\end{equation}
Therefore every element $\phi\in Aut(G)$ is induced by an element $\Phi\in Aut(E).$
\end{cor}

\begin{proof}
Consider a versal extension modulo $p$
$$\xymatrix@C=10mm@R=10mm{
 0 \ar@{->}[r] & H_2(G;\mathbb{Z}/p) \ar@{->}[r] & E \ar@{->}[r] & G \ar@{->}[r] & 1. }$$
with associated cohomology class $v_p.$
By Lemma \eqref{lema_eq_cocy} we know that for every $\phi\in Aut(G),$ the following equality holds $(H_2(\phi; \mathbb{Z}/p))_*(v_p)=\phi^*(v_p).$
Then, by Corollary \eqref{cor_push-pull}, there exists an element $\Phi\in Hom(E,E)$ making the following diagram commutative
\begin{equation*}
\xymatrix@C=10mm@R=10mm{
 0 \ar@{->}[r] & H_2(G;\mathbb{Z}/p) \ar@{->}[r] \ar@{->}[d]^-{H_2(\phi; \mathbb{Z}/p)} & E \ar@{->}[r] \ar@{->}[d]^-{\Phi} & G \ar@{->}[r] \ar@{->}[d]^-{\phi} & 1\\
0  \ar@{->}[r] & H_2(G;\mathbb{Z}/p) \ar@{->}[r]
 & E \ar@{->}[r] & G \ar@{->}[r] & 1. }
\end{equation*}
Finally, the $5$-lemma implies that $\Phi$ is an automorphism of $E$ that lifts $\phi$ as desired.
\end{proof}
\subsection{Versal extensions modulo $p$ for $p$-groups}
If $G$ is a $p$-group, the modulo $p$ Hopf formula becomes simpler.
In this case, let $1\rightarrow R \rightarrow F \rightarrow G\rightarrow 1$ be a presentation of $G$ with $d(F)=d(G).$
Then $R\leq [F,F]F^p,$ so $H_2(G;\mathbb{Z}/p)\cong R/[R,F]R^p$ and $F/[R,F]R^p$ is a $p$-covering group of $G.$

\begin{defi}[Definition 9.5.12 in \cite{leed}]
We say that a $p$-covering of $G$
$$\xymatrix@C=7mm@R=10mm{ 0 \ar@{->}[r] &  H_2(G;\mathbb{Z}/p) \ar@{->}[r] & P \ar@{->}[r] & G \ar@{->}[r] & 0}$$
is \textit{the universal $p$-covering of $G$} and $P$ is \textit{the universal $p$-covering group of $G$} (up to canonical isomorphisms)
if for any other central extension
$$\xymatrix@C=7mm@R=10mm{ 0 \ar@{->}[r] &  A \ar@{->}[r] & E \ar@{->}[r] & G \ar@{->}[r] & 0}$$
with $A$ a $p$-elementary abelian group satisfying that $A\leq \Phi(E),$ there is a surjection of $P$ onto $E$ that induces a surjection of $H_2(G;\mathbb{Z}/p)$ onto $A,$ and the natural isomorphism from $P/H_2(G;\mathbb{Z}/p)$ onto $E/A.$
\end{defi}

\begin{prop}[Proposition 9.5.13 in \cite{leed}]
\label{prop_univ_p-cov}
The group $P=F/[R,F]R^p$ is \textit{the universal $p$-covering group of $G.$}
\end{prop}

\begin{prop}
\label{prop_versal_ext_general}
The central extension
\begin{equation}
\label{versal_ext_general}
\xymatrix@C=7mm@R=7mm{
0  \ar@{->}[r] & H_2(G;\mathbb{Z}/p) \ar@{->}[r]
 & P \ar@{->}[r] & G \ar@{->}[r] & 1, }
\end{equation}
where $P$ is the universal $p$-covering group of $G$,
is a versal extension mod $p.$
\end{prop}

\begin{proof}
Let $c \in H^2(G ;R/[F,R]R^p)$ the cohomology class associated to the central extension
\begin{equation}
\label{prop_ses_tilde-surj}
\xymatrix@C=7mm@R=10mm{
0 \ar@{->}[r] & \frac{R}{[R,F]R^p} \ar@{->}[r] & \frac{F}{[F,R]R^p} \ar@{->}[r] & G \ar@{->}[r] & 1. }
\end{equation}
Using the Universal coefficients Theorem we get the natural isomorphism
\begin{equation}
\label{uct_iso_eta_tilde}
\tilde{\eta}:\; H^2\left(G; \frac{R}{[R,F]R^p}\right)\longrightarrow Hom\left(H_2(G;\mathbb{Z}/p);\frac{R}{[R,F]R^p}\right),
\end{equation}
that takes $c$ to the homomorphism $h$ given by the evaluation on
$H_2(G;\mathbb{Z}/p)$ of the $2$-cocycle $r([g|h])=s(g)s(h)s(gh)^{-1}$  where $s$ is a theoretic section of \eqref{prop_ses_tilde-surj}.
Notice that, in fact, the homomorphism $h$ is the modulo $p$ Hopf isomorphism.

Since $h$ is an isomorphism and $\tilde{\eta}$ is natural, applying $h^{-1}_*$ to $\eqref{uct_iso_eta_tilde}$ we get a natural isomorphism 
$$\eta: H^2(G; H_2(G;\mathbb{Z}/p))\longrightarrow Hom(H_2(G;\mathbb{Z}/p); H_2(G;\mathbb{Z}/p)),$$
which takes $h^{-1}_*(c)$ to the identity as desired.
\end{proof}

\section{On central extensions and characteristic subgroups}

Consider the following group extensions
\begin{equation}
\label{p-central_exts}
\begin{aligned}
& \xymatrix@C=7mm@R=10mm{
0 \ar@{->}[r] &\mathcal{L}_{k+1}^\bullet \ar@{->}[r]^-{i} &\mathcal{N}_{k+1}^\bullet \ar@{->}[r]^-{\pi^\bullet_k} & \mathcal{N}^\bullet_k \ar@{->}[r] & 1 },
\\
& \xymatrix@C=7mm@R=10mm{
0 \ar@{->}[r] & \widetilde{\mathcal{L}_{k+1}^\bullet} \ar@{->}[r]^-{i} &\widetilde{\mathcal{N}_{k+1}^\bullet} \ar@{->}[r]^-{\widetilde{\pi^\bullet_k}} & \mathcal{N}^\bullet_k \ar@{->}[r] & 1 },
\end{aligned}
\end{equation}

where $i$ is the natural inclusion and $\pi^\bullet_k$, $\widetilde{\pi^\bullet_k}$ are the quotient map corresponding to $\Gamma^\bullet_{k+1}.$

\begin{prop}
The above group extensions \eqref{p-central_exts} are central.
\end{prop}

\begin{proof}
Let $N$ be either $\Gamma^\bullet_{k+2}$ or $[\Gamma,\Gamma^\bullet_{k+1}](\Gamma^\bullet_{k+1})^p.$
By construction, $N\triangleleft \Gamma^\bullet_{k+1} \triangleleft \Gamma,$ moreover $[\Gamma,\Gamma^\bullet_{k+1}]\subset N.$
Then a direct computation, using Hall identities shows that
$$\left[\frac{\Gamma}{N},\frac{\Gamma^\bullet_{k+1}}{N}\right]\subset [\Gamma, \Gamma_{k+1}] \;\text{mod } N.$$
\end{proof}

\begin{defi}
A subgroup $H$ of a group $G$ is a characteristic subgroup of $G$ if every automorphism of $G$ maps the subgroup $H$ to itself.
\end{defi}

\begin{prop}
\label{prop_charSZ}
Let $G$ be a group. The groups $G^\bullet_k$ are characteristic subgroups of $G.$
\end{prop}

\begin{proof}
We begin with the case of the Stallings series. We proceed by induction on $k.$ The base case $k=1$ holds because by definition, $G^S_1=G.$
Now assume that the Proposition holds for $k$ and we prove it for $k+1.$

Let $f\in Aut(G)$ and $x\in G^S_{k+1}.$ We prove that $f(x)\in G^S_{k+1}.$ Recall that by definition, $G^S_{k+1}=[G,G^S_k](G_k^S)^p.$
Thus $x$ is a product of elements of $[G,G^S_k](G_k^S)^p.$ Since $f$ is a homomorphism, it is enough to prove that $[G,G^S_k]$ and $(G_k^S)^p$ are characteristic subgroups of $G.$ 

Let $y\in G,$ $z\in G_k^S.$ By induction hypothesis we have that $f(y)\in G, f(z)\in G_k^S.$ Therefore
$$
f([y,z])=[f(y),f(z)]\in [G,G^S_k]  \quad \text{and} \quad f(z^p)=f(z)^p\in(G_k^S)^p .$$

For the case of the Zassenhaus series, recall that by definition $G^Z_k=\prod_{ip^j\geq k}(G_i)^{p^j}.$ Thus every element of $G^Z_k$ is a product of elements of the form $x_i^{p^j}\in G^Z_k$ with $x_i\in G_i$ and $ip^j\geq k.$
Since $f$ is a homomorphism, it is enough to prove that $f(x_i^{p^j})\in G^Z_k$ with $x_i\in G_i$ and $ip^j\geq k.$

It is well known that $G_i$ is a characteristic subgroup of $G.$
Moreover, by the properties of the Zassenhauss series we know that $G_i\leq G_i^Z$ and $(G_i^Z)^{p^j} \leq G_{ip^j}^Z \leq G_k^Z.$
Therefore we have that
$$f(x_i^{p^j})=f(x_i)^{p^j}\in (G_i^Z)^{p^j} \leq G_k^Z.$$
\end{proof}

\begin{prop}
\label{prop_ident_L_versal}
For every $k\geq 1,$ there are isomorphisms
\begin{equation}
\label{eq_ident_L_versal}
\mathcal{L}_k^S\cong(\mathcal{N}_k^S)^S_k, \quad \mathcal{L}_k^Z\cong(\mathcal{N}_k^Z)^Z_k, \quad \widetilde{\mathcal{L}_k^Z}\cong(\widetilde{\mathcal{N}_k^Z})^Z_k, \quad\frac{\Gamma^Z_{k+1}}{[\Gamma, \Gamma^Z_{k}](\Gamma^Z_{k})^p}\cong(\widetilde{\mathcal{N}_{k}^Z})^Z_{k+1}.
\end{equation}
\end{prop}

\begin{proof}
First we prove the last three isomorphisms of \eqref{eq_ident_L_versal}.

Let $N$ be either $\Gamma^Z_{k+1}$ or $[\Gamma, \Gamma^Z_{k}](\Gamma^Z_{k})^p$ and $l=k \text{ or } k+1.$ We have the following epimorphism:
$$f:\prod_{ip^j\geq l}\left(\Gamma_i\right)^{p^j}\rightarrow \prod_{ip^j\geq l}\left(\left(\frac{\Gamma}{N}\right)_i\right)^{p^j}.$$
The kernel of $f$ is $N \cap \prod_{ip^j\geq l}(\Gamma_i)^{p^j}=N\cap \Gamma^Z_{l}.$
Moreover, by the properties of Zassenhauss and Stallings series, we know that $N\subset \Gamma^Z_{l}.$
Therefore, the first isomorphism Theorem of groups give us an isomorphism
$$\frac{\Gamma^Z_{l}}{N} \cong \prod_{ip^j\geq l}\left(\left(\frac{\Gamma}{N}\right)_i\right)^{p^j}.$$
Then we get the last three isomorphism of \eqref{eq_ident_L_versal}.

Next we prove that $\mathcal{L}_k^S\cong(\mathcal{N}_k^S)^S_k.$
For every $l\leq k,$ consider the following epimorphism:
$$f_l:\Gamma^S_l\rightarrow \left( \frac{\Gamma}{\Gamma^S_{k+1}}\right)^S_{l}.$$
We show that $f_l$ induces an isomorphism
\begin{equation}
\label{eq_iso_stallings_l}
\frac{\Gamma^S_{l}}{\Gamma^S_{k+1}}\cong \left( \frac{\Gamma}{\Gamma^S_{k+1}}\right)^S_{l},
\end{equation}
for $l\leq k.$ We proceed by induction on $l.$

For $l=1,$ the result is clear by definition of the Stallings series.
Assume that $f_{l-1}$ induces an isomorphism
$$\frac{\Gamma^S_{l-1}}{\Gamma^S_{k+1}}\cong \left( \frac{\Gamma}{\Gamma^S_{k+1}}\right)^S_{l-1}.$$
We prove that $f_l$ induces an isomorphism \eqref{eq_iso_stallings_l}.
Observe that $f_l$ is given by an epimorphism
$$f_l:[\Gamma,\Gamma^S_{l-1}](\Gamma^S_{l-1})^p\rightarrow \left[\frac{\Gamma}{\Gamma^S_{k+1}},\frac{\Gamma^S_{l-1}}{\Gamma^S_{k+1}}\right]\left(\frac{\Gamma^S_{l-1}}{\Gamma^S_{k+1}}\right)^p$$
with kernel $\Gamma^S_{k+1}\cap [\Gamma,\Gamma^S_{l-1}](\Gamma^S_{l-1})^p=\Gamma^S_{k+1}\cap \Gamma^S_{l}= \Gamma^S_{k+1},$ because $l\leq k+1.$
Therefore, the first isomorphism Theorem of groups give us the desired isomorphism.
\end{proof}
As a consequence of Propositions \eqref{prop_ident_L_versal}, applying Proposition \eqref{prop_charSZ} to $\mathcal{N}_{k}^\bullet,$ $\widetilde{\mathcal{N}_{k}^Z}$ we get the following result.
\begin{cor}
\label{cor_charsub}
The groups $\mathcal{L}_{k}^\bullet,$ $\widetilde{\mathcal{L}_{k}^Z},$ $\frac{\Gamma^Z_{k+1}}{[\Gamma, \Gamma^Z_{k}](\Gamma^Z_{k})^p}$ are respectively characteristic subgroups of $\mathcal{N}_{k}^\bullet,$ $\widetilde{\mathcal{N}_{k}^Z},$ $\widetilde{\mathcal{N}_{k}^Z}.$
\end{cor}

\begin{prop}
\label{prop_Hopf_iso_modp}
For any prime number $p,$
$$H_2(\mathcal{N}^\bullet_{k};\mathbb{Z}/p)\cong \frac{\Gamma^\bullet_{k+1}}{[\Gamma,\Gamma^\bullet_{k+1}](\Gamma^\bullet_{k+1})^p}=\widetilde{\mathcal{L}_{k+1}^\bullet}.$$
\end{prop}

\begin{proof}

Applying the Hopf formula modulo $p$ to the presentation of $\mathcal{N}_k^\bullet$ given by
$$1\rightarrow \Gamma^\bullet_{k+1}\rightarrow \Gamma \rightarrow \mathcal{N}^\bullet_{k}\rightarrow 1,$$
since $\Gamma^\bullet_{k+1} \subset [\Gamma,\Gamma]\Gamma^p$ we get that
$H_2(\mathcal{N}^\bullet_{k};\mathbb{Z}/p)\cong \frac{\Gamma^\bullet_{k+1}}{[\Gamma,\Gamma^\bullet_{k+1}](\Gamma^\bullet_{k+1})^p}.$
\end{proof}

\section{Extensions of Automorphims}

Let $\Gamma$ be a free group of finite rank and $\{ \Gamma_k^\bullet\}_k$ be the Stallings or Zassenhaus $p$-central series.
Since $\mathcal{L}_{k+1}^\bullet$ is a characteristic subgroup of $\mathcal{N}_{k+1}^\bullet,$ there is a well defined homomorphism
$\psi_k^\bullet:Aut(\mathcal{N}_{k+1}^\bullet)\rightarrow Aut(\mathcal{N}_{k}^\bullet).$
The main goal of this section is to show that there is a short exact sequence
\begin{equation}
\label{ses_aut_p}
\xymatrix@C=7mm@R=10mm{0 \ar@{->}[r] & Hom(\mathcal{N}^\bullet_1, \mathcal{L}^\bullet_{k+1}) \ar@{->}[r]^-{i} & Aut\;\mathcal{N}^\bullet_{k+1} \ar@{->}[r]^-{\psi_k^\bullet} & Aut \;\mathcal{N}^\bullet_k \ar@{->}[r] & 1 .}
\end{equation}
To achieve our goal, we first prove that there is a short exact sequence
\begin{equation}
\label{ses_aut_p_tilde}
\xymatrix@C=7mm@R=10mm{0 \ar@{->}[r] & Hom(\mathcal{N}^\bullet_1,H_2(\mathcal{N}^\bullet_k;\mathbb{Z}/p)) \ar@{->}[r]^-{i} & Aut\;\widetilde{\mathcal{N}^\bullet_{k+1}} \ar@{->}[r]^-{\widetilde{\psi^\bullet_k}} & Aut \;\mathcal{N}^\bullet_k \ar@{->}[r] & 1. }
\end{equation}
Then we get the short exact sequence \eqref{ses_aut_p} as a push-out of this short exact sequence.

\begin{rem}
Notice that the short exact sequences \eqref{ses_aut_p}, \eqref{ses_aut_p_tilde} coincide for $\bullet=S$ and differ for $\bullet=Z.$
\end{rem}

Consider the group $\widetilde{\mathcal{N}_{k+1}^\bullet},$ which, by Proposition \eqref{prop_univ_p-cov}, is the universal p-covering group of $\mathcal{N}^\bullet_k.$
Since $\widetilde{\mathcal{L}_{k+1}^\bullet}$ is a characteristic subgroup of $\widetilde{\mathcal{N}_{k+1}^\bullet},$ there is a well defined homomorphism
$\widetilde{\psi^\bullet_k}: Aut (\widetilde{\mathcal{N}_{k+1}^\bullet}) \rightarrow Aut(\mathcal{N}^\bullet_k).$
Next we prove that $\widetilde{\psi^\bullet_k}$ is an epimorphism and
that $Ker(\widetilde{\psi^\bullet_k})$ is isomorphic to $Hom(\mathcal{N}^\bullet_1,H_2(\mathcal{N}^\bullet_k;\mathbb{Z}/p)).$

\paragraph{The map $\widetilde{\psi_k^\bullet}$ is an epimorphism.}
Applying Corollary \eqref{cor_ver_ext_mod_p} to the universal $p$-covering group $\widetilde{\mathcal{N}_{k+1}^\bullet}$ of $\mathcal{N}_{k}^\bullet,$ we get that $\widetilde{\psi^\bullet_k}: Aut (\widetilde{\mathcal{N}_{k+1}^\bullet}) \rightarrow Aut(\mathcal{N}^\bullet_k)$ is an epimorphism.

\paragraph{The Kernel of $\widetilde{\psi_k^\bullet}$ is isomorphic to
$Hom(\mathcal{N}^\bullet_1,H_2(\mathcal{N}^\bullet_k;\mathbb{Z}/p)).$}

\begin{prop} The kernel of the homomorphism
$\widetilde{\psi_k^\bullet}: Aut(\widetilde{\mathcal{N}_{k+1}^\bullet})
\rightarrow Aut(\mathcal{N}_{k}^\bullet)$
is isomorphic to the group $Stab(\mathcal{N}_{k}^\bullet,H_2(\mathcal{N}_{k}^\bullet;\mathbb{Z}/p)),$ i.e. the group of stabilizing automorphisms of the extension
$0 \rightarrow H_2(\mathcal{N}_{k}^\bullet;\mathbb{Z}/p) \rightarrow \widetilde{\mathcal{N}_{k+1}^\bullet} \rightarrow  \mathcal{N}^\bullet_k \rightarrow 1.$
\end{prop}

\begin{proof}

Consider the versal extension modulo $p$
\begin{equation}
\label{ext_tilde_stab}
\xymatrix@C=7mm@R=10mm{
0 \ar@{->}[r] &H_2(\mathcal{N}_{k}^\bullet;\mathbb{Z}/p) \ar@{->}[r] &\widetilde{\mathcal{N}_{k+1}^\bullet} \ar@{->}[r] & \mathcal{N}^\bullet_k \ar@{->}[r] & 1 .}
\end{equation}
Denote $v_p\in H^2(\mathcal{N}_{k}^\bullet;H_2(\mathcal{N}_{k}^\bullet;\mathbb{Z}/p))$ its associated cohomology class.
Let $\Phi\in Ker(\widetilde{\psi_k^\bullet}).$ Then $\Phi$ induces a commutative diagram
\begin{equation}
\label{com_diag_push-out_tilde}
\xymatrix@C=7mm@R=10mm{0 \ar@{->}[r] &H_2(\mathcal{N}_{k}^\bullet;\mathbb{Z}/p) \ar@{->}[r] \ar@{->}[d]^{\psi}&  \widetilde{\mathcal{N}_{k+1}^\bullet} \ar@{->}[r] \ar@{->}[d]^{\Phi}& \mathcal{N}^\bullet_k \ar@{=}[d] \ar@{->}[r]& 1 \\
0 \ar@{->}[r] & H_2(\mathcal{N}_{k}^\bullet;\mathbb{Z}/p) \ar@{->}[r] &\widetilde{\mathcal{N}_{k+1}^\bullet} \ar@{->}[r] & \mathcal{N}^\bullet_k \ar@{->}[r] & 1. }
\end{equation}
and this implies that
$\psi_*(v_p)=v_p.$
Applying the natural isomorphism $\eta$ of \eqref{iso_uct_h2} to the above equality, we obtain that
$id=\eta(v_p)=\eta \psi_*(v_p)=\psi_*\eta(v_p)=\psi_*(id)=\psi.$
Therefore $\Phi$ stabilize the extension \eqref{ext_tilde_stab}.
Hence, $Ker(\widetilde{\psi^\bullet_k})\cong Stab(\mathcal{N}_{k}^\bullet,H_2(\mathcal{N}_{k}^\bullet;\mathbb{Z}/p)).$
\end{proof}

Moreover, we have the following isomorphisms:
\begin{align*}
Stab(\mathcal{N}^\bullet_{k},H_2(\mathcal{N}^\bullet_k;\mathbb{Z}/p))\cong  & Der(\mathcal{N}^\bullet_{k},H_2(\mathcal{N}^\bullet_k;\mathbb{Z}/p))=Hom(\mathcal{N}^\bullet_{k},H_2(\mathcal{N}^\bullet_k;\mathbb{Z}/p))= \\
= & Hom(\mathcal{N}^\bullet_1,H_2(\mathcal{N}^\bullet_k;\mathbb{Z}/p)).
\end{align*}
Therefore we get the following result:
\begin{prop}
\label{prop_exact_seq_modp}
We have an exact sequence of groups
\begin{equation*}
\xymatrix@C=7mm@R=10mm{0 \ar@{->}[r] & Hom(\mathcal{N}^\bullet_1,H_2(\mathcal{N}^\bullet_k;\mathbb{Z}/p)) \ar@{->}[r] & Aut\;\widetilde{\mathcal{N}^\bullet_{k+1}} \ar@{->}[r]^-{\widetilde{\psi^\bullet_k}} & Aut \;\mathcal{N}^\bullet_k \ar@{->}[r] & 1. }
\end{equation*}
\end{prop}
As a consequence of Proposition \eqref{prop_Hopf_iso_modp}, using the modulo $p$ Hopf isomorphism, we get the following result:

\begin{cor}
\label{cor_exact_seq_modp}
We have an exact sequence of groups
\begin{equation*}
\xymatrix@C=7mm@R=10mm{0 \ar@{->}[r] & Hom(\mathcal{N}^\bullet_1, \widetilde{\mathcal{L}^\bullet_{k+1}}) \ar@{->}[r]^-{i} & Aut\;\widetilde{\mathcal{N}^\bullet_{k+1}} \ar@{->}[r]^-{\widetilde{\psi^\bullet_k}} & Aut \;\mathcal{N}^\bullet_k \ar@{->}[r] & 1, }
\end{equation*}
where $i$ is defined as $i(f)= (\gamma\mapsto f([\gamma])\gamma).$
\end{cor}

As a consequence of Corollary \eqref{cor_exact_seq_modp}, we get the following result:
\begin{prop}
\label{prop_exact_seq_modp_L}
We have an exact sequence of groups
\begin{equation*}
\xymatrix@C=7mm@R=10mm{0 \ar@{->}[r] & Hom(\mathcal{N}^\bullet_1,\mathcal{L}^\bullet_{k+1}) \ar@{->}[r]^-{i} & Aut\;\mathcal{N}^\bullet_{k+1} \ar@{->}[r]^-{\psi^\bullet_k} & Aut \;\mathcal{N}^\bullet_k \ar@{->}[r] & 1, }
\end{equation*}
where $i$ is defined as $i(f)= (\gamma\mapsto f([\gamma])\gamma).$
\end{prop}

\begin{proof}
Notice that we have a push-out diagram
\begin{equation*}
\xymatrix@C=7mm@R=10mm{0 \ar@{->}[r] & Hom(\mathcal{N}^\bullet_1, \widetilde{\mathcal{L}^\bullet_{k+1}}) \ar@{->}[d]^-{q} \ar@{->}[r]^-{i} & Aut\;\widetilde{\mathcal{N}^\bullet_{k+1}} \ar@{->}[d] \ar@{->}[r]^-{\widetilde{\psi^\bullet_k}} & Aut \;\mathcal{N}^\bullet_k \ar@{=}[d] \ar@{->}[r] & 1 \\
0 \ar@{->}[r] & Hom(\mathcal{N}^\bullet_1,\mathcal{L}^\bullet_{k+1}) \ar@{->}[r]^-{i} & Aut\;\mathcal{N}^\bullet_{k+1} \ar@{->}[r]^-{\psi^\bullet_k} & Aut \;\mathcal{N}^\bullet_k \ar@{->}[r] & 1, }
\end{equation*}
where $q$ is the quotient respect $\frac{\Gamma^\bullet_{k+2}}{[\Gamma,\Gamma^\bullet_{k+1}](\Gamma^\bullet_{k+1})^p}.$
\end{proof}

\begin{rem}
Despite the fact that the extensions of Proposition \eqref{prop_exact_seq_modp_L} for $\bullet=S$ and $\bullet=Z$ are diferent, notice that $\Gamma^Z_{p^l}<\Gamma^S_l<\Gamma^Z_l.$ 
Then, one can show that there is a well defined epimorphism $Aut (\mathcal{N}_{k}^S) \rightarrow Aut(\mathcal{N}^Z_k).$
\end{rem}

\subsection{On centrality}

In this section we prove that the extension of Proposition \eqref{prop_exact_seq_modp_L} is not central.
In particular, we prove that the action of $Aut(\mathcal{N}^\bullet_{k+1})$ on $Hom(\mathcal{N}_1^\bullet, \mathcal{L}^\bullet_{k+1})$ factors through $Aut(\mathcal{N}^\bullet_{1}).$

Denote by $IA^p(\mathcal{N}^\bullet_k)$ formed by the elements which act trivially on the first homology (with coefficients in $\mathbb{Z}/p$) of $\mathcal{N}^\bullet_{k}.$ In other words, $IA^p(\mathcal{N}^\bullet_k)$ is the kernel of the surjection $Aut\;\mathcal{N}^\bullet_k \rightarrow Aut\;\mathcal{N}^\bullet_1.$ i.e. we have the following short exact sequence
\begin{equation}
\label{ses_IAmodp}
\xymatrix@C=7mm@R=10mm{1 \ar@{->}[r] & IA^p(\mathcal{N}^\bullet_k) \ar@{->}[r] & Aut\;\mathcal{N}^\bullet_{k} \ar@{->}[r] & Aut \;\mathcal{N}^\bullet_1 \ar@{->}[r] & 1 .}
\end{equation}
If we restrict the short exact sequence of Proposition \eqref{prop_exact_seq_modp} to $IA^p(\mathcal{N}^\bullet_{k+1})$ we get the following short exact sequence
\begin{equation}
\label{ses_ZS}
\xymatrix@C=7mm@R=10mm{0 \ar@{->}[r] & Hom(H_p, \mathcal{L}^\bullet_{k+1}) \ar@{->}[r]^{i} & IA^p(\mathcal{N}^\bullet_{k+1}) \ar@{->}[r]^{\psi^\bullet_k} & IA^p(\mathcal{N}^\bullet_k) \ar@{->}[r] & 1 .}
\end{equation}

Next we prove that the short exact sequence \eqref{ses_ZS} is a central extension for $\bullet=S,Z.$
We first give some preliminary results.

\begin{defi}[complex commutator \cite{hall2}]
Let $P_1,P_2,\ldots, P_r$ be any $r$ elements of a group $G.$ We shall define by induction what we mean by a \textit{complex commutator of weight $w$ in the components $P_1,P_2,\ldots, P_r.$} The complex commutators of weight $1$ are the elements $P_1,P_2,\ldots, P_r$ themselves.
Supposing the complex commutators of all weights less than $w$ have already been defined, then those of weight $w$ consist of all the expressions of the form $[S,T],$ where $S$ and $T$ are any complex commutators of weights $w_1$ and $w_2$ in the components $P_1,P_2,\ldots P_r$ respectively, such that $w_1+w_2=w.$

The weight of a complex commutator is, of course, always relative to the choice of components; these must be specified before the weight can be determined.
For example, $[[P,Q],R]$ is of weight $3$ in the three components $P,$ $Q$ and $R;$ but of weight $2$ in the two components $[P,Q]$ and $R.$

\end{defi}

\begin{teo}[P. Hall, Theorem 3.2. in \cite{hall2}]
\label{teo_hall}
If $p$ is a prime, $\alpha$ is a positive integer, $P$ and $Q$ are any two elements of a group $G,$ and
$$R_1,R_2,\ldots ,R_i,\ldots \qquad (R_1=P,\quad R_2=Q)$$
are the various formally distinct complex commutators of $P$ and $Q$ arranged in order of increasing weights, then integers $n_1,n_2,\ldots ,n_i,\ldots$ can be found $(n_1=n_2=p^\alpha)$ such that
$$(PQ)^{p^{\alpha}}=R_1^{n_1}R_2^{n_2}\cdots R_i^{n_i}\cdots$$
and if the weight $w_i$ of $R_i$ in $P$ and $Q$ satisfies the inequality $w_i<p^\beta \leq p^\alpha,$ then $n_i$ is divisible by $p^{\alpha-\beta +1}.$
\end{teo}

\begin{lema}[\cite{hall} ,Three Subgroups Lemma]
Let $A,$ $B$ and $C$ be subgroups of a group $G.$ If $N\lhd G$ is normal subgroup such that $[A,[B,C]]$ and $[B,[C,A]]$ are contained in $N$ then $[C,[A,B]]$ is also contained in $N.$
\end{lema}

\begin{defi}(Holomorph group)
Let $G$ be a group. The Holomorph group of $G$ is defined as the semidirect product
$$Hol(G)=G\rtimes Aut(G),$$
where the multiplication is given by $(g_1,f_1)(g_2,f_2)=(g_1f_1(g_2),f_1f_2).$
\end{defi}

\begin{rem}
Consider the natural inclusion homomorphisms
\begin{equation*}
\begin{aligned}
\iota_1: \; G & \rightarrow G\rtimes Aut(G) ,\\
 g & \mapsto (g,id)
\end{aligned}
\qquad
\begin{aligned}
\iota_2: \;Aut(G) & \rightarrow G\rtimes Aut(G). \\
 f & \mapsto (1,f)
\end{aligned}
\end{equation*}
Using these homomorphisms, we can view any subgroup of $G$ and any subgroup of $Aut(G)$ as subgroups of $Hol(G),$
and any normal subgroup of $G$ as a normal subgroup of $Hol(G).$
\end{rem}

\begin{nota}
\label{notacio_com}
Let $x\in G,$ $f\in Aut(G),$ $H\triangleleft G$ and $K\triangleleft Aut(G).$ Throughout this chapter we denote by $[f,x]\in G$ the element $[f,x]=f(x)x^{-1},$ and by $[K,H]$ the subgroup given by
$[K,H]=\{[f,x]\in G \; ; \; f\in K, x\in H \}.$
\end{nota}

\begin{rem}
Given $x\in G$ and $f\in Aut(G).$ The notation $[f,x]=f(x)x^{-1}$ makes sense because computing the commutator $[\iota_2(f),\iota_1(x)]\in Hol(G)$ one gets that
\begin{align*}
[\iota_2(f),\iota_1(x)]= & [(1,f),(x,id)]= (1,f)(x,id)(1,f^{-1})(x^{-1},id)= \\
= & (f(x),f)(1,f^{-1})(x^{-1},id)= (f(x),id)(x^{-1},id)= \\
= & (f(x)x^{-1},id)=\iota_1(f(x)x^{-1}),
\end{align*}
which corresponds, by the natural inclusion homomorphism $\iota_1,$ to $f(x)x^{-1}\in G.$
\end{rem}

\begin{lema}[\cite{coop}, J. Cooper]
\label{lema_coop_SZ}
Given a prime $p.$ If $f\in \mathcal{I}^\bullet_{g,1}(k)$ and $x\in \Gamma^\bullet_l$ then $f(x)x^{-1}\in \Gamma^\bullet_{k+l}.$ Equivalently, $[\mathcal{I}^\bullet_{g,1}(k),\Gamma^\bullet_l]<\Gamma^\bullet_{k+l}.$
\end{lema} 

\begin{rem}
In the proof of Lemma \eqref{lema_coop_SZ}, J. Cooper asserts that the result for the case of the Zassenhaus filtration follows using the same argument that he used for the Stallings filtration.
Reviewing his proof we found that it is not clear how to use the same argument to prove the statement for the Zassenhaus filtration.
However, using the Hall-Petresco identity, we will give an argument to get the result.
Indeed, we realised that Lemma \eqref{lema_coop_SZ} is more general.
\end{rem}

Next we will state a generalization of Lemma \eqref{lema_coop_SZ},
whose proof follows the original one by J.Cooper.

\begin{defi}
Denote by $IA_k^\bullet(G)$ the elements of $Aut(G)$ that act trivially on $\frac{G}{G^\bullet_{k+1}},$ that is,
$$IA_k^\bullet(G)=\{f\in Aut(G)\;\mid \; f(x)x^{-1}\in G^\bullet_{k+1} \text{ for all } x\in G\}.$$
\end{defi}
Throughout this section we denote $IA_1^\bullet(G)$ by $IA^p(G),$ with $\bullet= S \text{ or }Z.$

\begin{lema}
\label{lema_coop_general}
Given a prime $p.$ If $f\in IA_k^\bullet(G)$ and $x\in G^\bullet_l,$ then $f(x)x^{-1}\in G^\bullet_{k+l}.$ Equivalently, $[IA_k^\bullet(G),G^\bullet_l]<G^\bullet_{k+l}.$
\end{lema}
\begin{proof}
We begin with the case of the Stallings series. Recall that by definition, $G^S_{l+1}=[G,G^S_l](G_l^S)^p.$ Thus every element of $G^S_{l+1}$ is a product of elements of the form $[x,y]\in G^S_{l+1},$ $z^p\in G^S_{l+1},$ where $x\in G$ and $y,z\in G^S_l.$
So first of all we will prove the statement for such elements and later for any product of them.

We proceed by induction on $l.$ The base case $l=1$ follows from the definition of $IA_k^S(G).$ Now assume that the lemma holds for $l.$ We will prove that lemma holds for $l+1.$

Consider elements of the form $[x,y]\in G^S_{l+1}$ and $z^p\in G^S_{l+1},$ where $x\in G$ and $y,z\in G^S_l.$

We first show that $f([x,y])[x,y]^{-1}\in G^S_{k+l+1}$ for $f\in IA_k^S(G).$ The main idea of this proof originally comes from \cite{andrea}. First note that
$$f([x,y])[x,y]^{-1}=[f,[x,y]]\in [IA_k^S(G),[G,G^S_l]].$$
The idea is to apply the Three Subgroup Lemma for the subgroups $IA_k^S(G),$ $G,$ $G^S_l$ of $Hol(G).$
Observe that by induction and the definition of the Stallings series,
\begin{align*}
[[IA_k^S(G),G^S_l],G] & <[G^S_{k+l},G]< G^S_{k+l+1}, \\
[[IA_k^S(G),G],G^S_l] & <[G^S_{k+l},G^S_l]< G^S_{k+l+1}.
\end{align*}
Moreover, since $G^S_{k+l+1}$ is a normal subgroup of $G,$ we can view $G^S_{k+l+1}$ as a normal subgroup of $Hol(G).$
Therefore the Three Subgroup Lemma implies that
$$[IA_k^S(G),[G,G^S_l]]<G^S_{k+l+1}.$$

Next we show that $f(z^p)z^{-p}\in G^S_{k+l+1}.$ First of all we prove that
$$f(z^p)z^{-p}\equiv (f(z)z^{-1})^p \text{ mod } G^S_{k+l+1}.$$
Observe that the following equality holds
$$f(z^p)z^{-p}=[f,z^p]=[f,z][f,z]^z\ldots [f,z]^{z^{p-1}}.$$
Note that by induction, $[f,z]\in G^S_{k+l}$ and by normality, $[f,z]^{z^i}\in G^S_{k+l}$ for each $i=1,\ldots, p-1.$ Furthermore,
$[[f,z],z^i]\in G^S_{k+l+1}$ so that $[f,z] \equiv [f,z]^{z^i} \text{ mod } G^S_{k+l+1}$ and finally $f(z^p)z^{-p}\equiv (f(z)z^{-1})^p \text{ mod } G^S_{k+l+1}.$

Now we note that by induction, $(f(z)z^{-1})^p\in (G^S_{k+l})^p,$ and by the properties of the Stallings series, $(G^S_{k+l})^p \in G^S_{k+l+1}.$
Therefore, $f(z^p)z^{-p}\in G^S_{k+l+1}.$

Finally, we prove the statement for products of elements of $G_{l+1}^S.$
Let $f\in IA_k^S(G)$ and $\omega_i\in G^S_{l+1}.$ Using the fact that $f(\omega_i)\omega_i^{-1}\in G^S_{k+l+1}$ for all $i,$ we have:
\begin{align*}
f\left(\prod_{i=1}^n \omega_i\right)\left(\prod_{i=1}^n \omega_i\right)^{-1}= & f(\omega_1)\cdots f(\omega_{n-1})f(\omega_n)\omega_n^{-1}\omega_{n-1}^{-1}\cdots \omega_1 \\
\equiv & f(\omega_1)\cdots f(\omega_{n-1})\omega_{n-1}^{-1}\cdots \omega_1^{-1} \quad(\text{mod }G^S_{k+l+1})\\
\equiv & f(\omega_1)\cdots f(\omega_{n-2})\omega_{n-2}^{-1}\cdots \omega_1^{-1} \quad(\text{mod }G^S_{k+l+1})\\
\equiv & \cdots \equiv f(\omega_{1})f(\omega_{2})\omega_{2}^{-1}\omega_{1}^{-1}\equiv f(\omega_{1})\omega_1 \equiv 1\quad(\text{mod }G^S_{k+l+1}).
\end{align*}
Therefore,
$$f\left(\prod_{i=1}^n \omega_i\right)\left(\prod_{i=1}^n \omega_i\right)^{-1}\in G^S_{k+l+1}.$$

For the case of the Zassenhaus series, recall that by definition $G^Z_l=\prod_{ip^j\geq l}(G_i)^{p^j}.$ Thus every element of $G^Z_l$ is a product of elements of the form $x_i^{p^j}\in G^Z_l$ with $x_i\in G_i$ and $ip^j\geq l.$
So first of all we will prove the statement for such elements and later for any product of them.

Let $f\in I_k^Z(G)$ and $x_i^{p^j}\in G^Z_l$ with $x_i\in G_i,$ i.e. $ip^j\geq l.$ We want to prove that $f(x_i^{p^j})x_i^{-p^j}\in G^Z_{l+1}.$
Observe that if $ip^j\geq l+1$ then $x_i^{p^j}\in G^Z_{l+1}$
and as a consequence $f(x_i^{p^j})x_i^{-p^j}\in G^Z_{l+1}.$
Thus we can assume that $ip^j=l.$
Next, the idea is to prove the statement for the elements $x_l\in G_l,$ and use an inductive argument on $j$ to prove the general result.

We first prove that $f(x_l)x_l^{-1} \in G^Z_{l+k}.$
By induction on $l.$ The base case $x_1$ follows from the definition of $IA^Z_k(G).$ Now assume that the lemma holds for $x_l,$ and  prove that the Lemma holds for $x_{l+1}.$
As in the case of Stallings, applying the Three Subgroup Lemma for the subgroups $IA^Z_k(G),G,G^Z_l< Hol(G),$ we have that
$[IA^Z_k(G),[G,G^Z_l]]<G^Z_{k+l+1}.$

Observe that
$$f(x_i^{p^j})x_i^{-p^j}=f(x_i^{p^{j-1}})^p x_i^{-p^j}.$$
Let $q=ip^{j-1},$ take $x_q=x_i^{p^{j-1}}.$ Observe that since $x_q\in G_{q}^Z,$ by induction hypothesis, $f(x_q)x^{-1}_q\in G^Z_{q+k}.$
Then there exists an element $y_{q+k}\in G^Z_{q+k}$ such that $f(x_q)=x_qy_{q+k}.$
Next we show that
$$(x_qy_{q+k})^p\equiv x_q^p \quad (\text{mod } G^Z_{l+k}).$$
By Theorem \eqref{teo_hall} we have that
$$(x_qy_{q+k})^{p}=R_1^{n_1}R_2^{n_2}\cdots R_i^{n_i}\cdots$$
where
$$R_1,R_2,\ldots ,R_i,\ldots \qquad (R_1=x_q,\quad R_2=y_{q+k})$$
are the various formally distinct complex commutators of $x_q$ and $y_{q+k}$ arranged in increasing weights order, and $n_1,n_2,\ldots ,n_i,\ldots$ positive integers such that $n_1=n_2=p$ and if the weight $w_i$ of $R_i$ in $x_q$ and $y_{q+k}$ satisfies that $w_i<p$ then $n_i$ is divisible by $p.$

Next we prove that $R_i^{n_i}\in G^Z_{l+k}$ for $i\geq 2.$
\begin{itemize}
\item \textbf{For $i=2,$} we know that $R_2=y_{q+k}$ and $n_2=p.$
Since $y_{q+k}\in G^Z_{q+k},$ we have that
$$y_{q+k}^p\in (G^Z_{q+k})^p\leq G^Z_{p(q+k)}\leq G^Z_{pq+k}=G^Z_{l+k}.$$
\item \textbf{For $i\geq 3,$} as $R_i$ are complex commutators of weight $w_i$ in the two components $x_k,$ $y_{q+k},$ we have that $w_i\geq 2.$
As a consequence,
at least one component of $R_i$ has to be $y_{q+k},$ because if it is not the case then $R_i$ has to be $1.$

If the weight $w_i$ of $R_i$ in $x_k$ and $y_{q+k}$ satisfies that $2\leq w_i<p,$ since at least one component of $R_i$ has to be $y_{q+k},$ we get that $R_i\in G^Z_{(\omega_i-1)q+(q+k)}=G^Z_{\omega_iq+k}.$
Moreover, in the case $2\leq w_i<p,$ we have that $p\mid n_i$ and then
$$R_i^{n_i}\in (G^Z_{\omega_iq+k})^{n_i}\leq (G^Z_{\omega_iq+k})^p\leq G^Z_{p(\omega_iq+k)}\leq G^Z_{pq+k}=G^Z_{l+k}.$$

On the other hand, if the weight $w_i$ of $R_i$ in $x_q$ and $y_{q+k}$ satisfies that $w_i\geq p,$ since at least one component of $R_i$ has to be $y_{q+k},$ then
$R_i\in G^Z_{(\omega_i-1)q+(q+k)}=G^Z_{\omega_iq+k}\leq G^Z_{pq+k}.$ As a consequence, $$R_i^{n_i}\in G^Z_{pq+k}.$$
\end{itemize}

Therefore,
$(x_qy_{q+k})^{p}\equiv x_q^p \quad (\text{ mod } G^Z_{l+k} ),$
i.e. $f(x_i^{p^j})x_i^{-p^j}\in G^Z_{l+k}.$

Finally, we prove the statement for products. The same argument of the case of Stallings works here, giving that if 
$f\in IA^Z_k(G)$ and $\omega_i\in G^Z_{l+k},$ then
$$f\left(\prod_{i=1}^n \omega_i\right)\left(\prod_{i=1}^n \omega_i\right)^{-1}\in G^S_{l+k}.$$
\end{proof}

As a direct consequence of Lemma \eqref{lema_coop_general}, using ideas of S. Andreadakis (see Theorem 1.1 in \cite{andrea}),
we get the following result.
\begin{cor}
\label{cor_IA_bullet}
For any two elements $\varphi\in IA^\bullet_k(G)$ and $\psi\in IA^\bullet_l(G),$ the commutator $[\varphi, \psi]$ is contained in $IA^\bullet_{k+l}(G).$ Equivalently, $[IA^\bullet_k(G),IA^\bullet_l(G)] < IA^\bullet_{k+l}(G).$
\end{cor}

\begin{proof}
Notice that $G,$ $IA^\bullet_k(G),$ $IA^\bullet_l(G)$ are subgroups of $Hol(G).$ By Lemma \eqref{lema_coop_general},
\begin{align*}
[IA^\bullet_l(G),[IA^\bullet_k(G),G]] & < [IA^\bullet_l(G),G^\bullet_{k+1}]< G^\bullet_{k+l+1}, \\
[IA^\bullet_k(G),[G, IA^\bullet_l(G)]] & < [IA^\bullet_k(G),G^\bullet_{l+1}]< G^\bullet_{k+l+1}.
\end{align*}
Then, by the Three Subgroups Lemma, $[G,[IA^\bullet_k(G),IA^\bullet_l(G)]]< G^\bullet_{k+l+1}.$
Hence,
$$[[IA^\bullet_k(G),IA^\bullet_l(G)],G]< G^\bullet_{k+l+1}.$$
As a consequence, for every $\gamma \in G,$ $\varphi\in IA^\bullet_k(G)$ and $\psi\in IA^\bullet_l(G),$ we have that
$$[\varphi, \psi](\gamma)\gamma^{-1}\in G^\bullet_{k+l+1},$$
which, by definition, means that $[\varphi, \psi]\in IA^\bullet_{k+l}(G),$ as desired.
\end{proof}

Next, we show that, for a free group of finite rank, Lemma \eqref{lema_coop_general} and Corollary \eqref{cor_IA_bullet} give us the most efficient bound, in the sense that
\begin{align*}
[IA_k^\bullet(\Gamma),\Gamma^\bullet_l]<\Gamma^\bullet_{k+l}, & \qquad [IA^\bullet_k(\Gamma),IA^\bullet_l(\Gamma)] < IA^\bullet_{k+l}(\Gamma), \\
[IA_k^\bullet(\Gamma),\Gamma^\bullet_l]\nless\Gamma^\bullet_{k+l+1}, &
\qquad [IA^\bullet_k(\Gamma),IA^\bullet_l(\Gamma)] \nless IA^\bullet_{k+l+1}(\Gamma).
\end{align*}

\begin{prop}
\label{prop_cotas_IA_bullet,N_bullet}
Let $\Gamma$ be a free group of finite rank $n > 1.$ Then
\begin{enumerate}[i)]
\item $[IA_k^\bullet(\Gamma),\Gamma^\bullet_l]\nless\Gamma^\bullet_{k+l+1},$
\item $[IA^\bullet_k(\Gamma),IA^\bullet_l(\Gamma)] \nless IA^\bullet_{k+l+1}(\Gamma).$
\end{enumerate}
\end{prop}

\begin{proof}
Consider the free group $\Gamma=\langle x_1,\ldots, x_n \rangle.$ Let $\gamma_1=[x_1,[x_2,[x_1,[x_2,\ldots]]]],$ a commutator of length $k$ alternating $x_1,$ $x_2,$ and $\gamma_2=[x_1,[x_2,[x_1,[x_2,\ldots]]]]$ a commutator of length $l$ alternating $x_1,$ $x_2.$

\textbf{i)} Take $f\in IA_k^\bullet(\Gamma)$ the inner automorphism given by the conjugation by $\gamma_1.$ Notice that $f\in IA_k^\bullet(\Gamma)$ because for every $x\in \Gamma,$
$$f(x)=\gamma_1x\gamma_1^{-1}= [\gamma_1,x]x=x \;(\text{mod } \Gamma_{k+1}).$$
Then we have that $[f,\gamma_2]\in [IA^\bullet_k(\Gamma),\Gamma^\bullet_l]$ and
$$[f,\gamma_2]=f(\gamma_2)\gamma_2^{-1}=\gamma_1\gamma_2\gamma_1^{-1}\gamma_2^{-1}=[\gamma_1,\gamma_2] \; \in \Gamma_{k+l},$$
which does not belong to $\Gamma_{k+l+1}^\bullet<\Gamma_{k+l+1} \Gamma^p$ because $[\gamma_1,\gamma_2]$ is neither a power $p$ element nor an element of $\Gamma_{k+l+1}$ because it can not be written as a product of at least $2(k+l)+2$ generators.

\textbf{ii)} Take $\varphi\in IA_k^\bullet(\Gamma),$ $\psi \in IA_l^\bullet(\Gamma)$ the inner automorphisms respectively given by the conjugation by $\gamma_1$ and $\gamma_2.$
Then we have that $[\varphi, \psi]\in [IA_k^\bullet(\Gamma),IA_l^\bullet(\Gamma)]$ and
$$
[\varphi, \psi](x_1)x_1^{-1}=[\gamma_1,\gamma_2]x_1[\gamma_1,\gamma_2]^{-1}x_1^{-1}=[[\gamma_1,\gamma_2],x_1]\;\in \Gamma_{k+l+1},
$$
which, as in $i),$ does not belong to $\Gamma_{k+l+2}^\bullet.$
Therefore, $[\phi,\psi]\notin IA^\bullet_{k+l+1}(\Gamma).$
\end{proof}

As a direct consequence we have the following result:

\begin{cor}
\label{coro_cota_IA_bullet,N_bullet}
The group $[IA^\bullet_k(\mathcal{N}_n^\bullet),(\mathcal{N}_n^\bullet)^\bullet_l]$ with $l+k=n$ and $[IA_k^\bullet(\mathcal{N}^\bullet_{n}),IA_l^\bullet(\mathcal{N}^\bullet_{n})]$ with $l+k=n-1,$ are not trivial.
\end{cor}

\begin{proof}
We use the same notation of proof of \eqref{prop_cotas_IA_bullet,N_bullet}.
Take the $f\in IA_k^\bullet(\Gamma)$ and $x\in \Gamma^\bullet_l$ as in the proof of Proposition \eqref{prop_cotas_IA_bullet,N_bullet}. These elements induce elements $\overline{f}\in IA_k^\bullet(\Gamma)$ and $\overline{x}\in \Gamma^\bullet_l.$ Moreover, by Proposition \eqref{prop_cotas_IA_bullet,N_bullet}, we have that
$$[\overline{f},\overline{x}]=\overline{[f,x]}=\overline{[\gamma_1,\gamma_2]}\neq 1.$$
Therefore there is an element in $[IA^\bullet_k(\mathcal{N}_n^\bullet),(\mathcal{N}_n^\bullet)^\bullet_l]$ which is not trivial.

Take $\varphi\in IA_k^\bullet(\Gamma),$ $\psi\in IA_l^\bullet(\Gamma),$ and $x_1\in \Gamma.$ These elements induce elements $\overline{\varphi}\in IA_k^\bullet(\mathcal{N}_n^\bullet),$ $\overline{\psi}\in IA_l^\bullet(\mathcal{N}_n^\bullet),$ and $\overline{x_1}\in \mathcal{N}_n^\bullet.$
Moreover, by Proposition \eqref{prop_cotas_IA_bullet,N_bullet}, we have that
$$[[\overline{\varphi},\overline{\psi}],\overline{x_1}]=\overline{[[\varphi,\psi],x_1]}=\overline{[[\gamma_1,\gamma_2],x_1]}\neq 1.$$
Therefore there is an element in $[IA_k^\bullet(\mathcal{N}^\bullet_{n}),IA_l^\bullet(\mathcal{N}^\bullet_{n})]$ which is not trivial.
\end{proof}

\begin{prop}
\label{prop_act_aut_bullet}
The natural action of $Aut(\mathcal{N}^\bullet_{k+1})$ on $Hom(\mathcal{N}_1^\bullet, \mathcal{L}^\bullet_{k+1})$
factors through $Aut(\mathcal{N}^\bullet_{1}).$
\end{prop}

\begin{proof}
In virtue of Proposition \eqref{prop_exact_seq_modp_L}, the natural action of $Aut(\mathcal{N}^\bullet_{k+1})$ on $Hom(\mathcal{N}_1^\bullet, \mathcal{L}^\bullet_{k+1})$ is given by
\begin{align*}
Aut(\mathcal{N}^\bullet_{k+1})\times Hom(\mathcal{N}_1^\bullet, \mathcal{L}^\bullet_{k+1})&\longrightarrow Hom(\mathcal{N}_1^\bullet, \mathcal{L}^\bullet_{k+1}) \\
(h,f) & \longmapsto (x\mapsto h(f(h^{-1}x)))),
\end{align*}
where $h^{-1}x$ is the action of $h^{-1}\in Aut(\mathcal{N}^\bullet_{k+1})$ on $x\in \mathcal{N}^\bullet_1$ via the surjection
$Aut\;\mathcal{N}^\bullet_{k} \rightarrow Aut \;\mathcal{N}^\bullet_1.$
Moreover, by Proposition \eqref{prop_ident_L_versal} and Lemma \eqref{lema_coop_general}, we know that if $h\in IA^p(\mathcal{N}^\bullet_{k+1})$ and $y\in \mathcal{L}_{k+1}^\bullet,$ then
$h(y)=y.$ Therefore the action of $Aut\;\mathcal{N}^\bullet_{k}$ on $\mathcal{L}_{k+1}^\bullet$ factors through $Aut(\mathcal{N}^\bullet_{1})$
via the surjection $Aut\;\mathcal{N}^\bullet_{k} \rightarrow Aut \;\mathcal{N}^\bullet_1$ and as a consequence we get the result.
\end{proof}

As a direct consequence of Proposition \eqref{prop_act_aut_bullet}, we have the following result:
\begin{cor}
\label{cor_cent_ext_IA}
Let $\Gamma$ be a free group of finite rank $n > 1.$ The extension
\begin{equation*}
\xymatrix@C=7mm@R=10mm{0 \ar@{->}[r] & Hom(\mathcal{N}^\bullet_1, \mathcal{L}^\bullet_{k+1}) \ar@{->}[r]^-{i} & IA^p(\mathcal{N}^\bullet_{k+1}) \ar@{->}[r]^-{\pi} & IA^p(\mathcal{N}^\bullet_k) \ar@{->}[r] & 1 }
\end{equation*}
is central.
\end{cor}

\begin{prop}
Let $\Gamma$ be a free group of finite rank $n > 1.$ The extension
\begin{equation*}
\xymatrix@C=7mm@R=10mm{0 \ar@{->}[r] & Hom(\mathcal{N}^\bullet_1,\mathcal{L}^\bullet_{k+1}) \ar@{->}[r]^-{i} & Aut\;\mathcal{N}^\bullet_{k+1} \ar@{->}[r]^-{\psi^\bullet_k} & Aut \;\mathcal{N}^\bullet_k \ar@{->}[r] & 1, }
\end{equation*}
is not central.
\end{prop}

\begin{proof}
We provide a counterexample.

Consider $\Gamma=\langle x_1,\ldots, x_n \rangle$ and $x=[x_1,[x_2,[x_1,[x_2,\ldots]]]]$ a commutator of length $k+1.$
Denote $[x_i]$ the class of $x_i$ in $\mathcal{N}^\bullet_1,$ and $\overline{x}$ the class of $x$ in $\mathcal{L}^\bullet_{k+1}.$
Observe that $\mathcal{N}_1^\bullet$ is a $\mathbb{F}_p$-vector space with basis $\{[x_1],\ldots ,[x_n]\}.$
Let $f\in Hom(\mathcal{N}_1^\bullet,\mathcal{L}^\bullet_{k+1})$ be the homomorphism defined on the basis of $\mathcal{N}_1^\bullet$ by
$f([x_1])=\overline{x}$ and $f([x_i])=0$ for $1<i\leq n.$

Consider $h=(x_1x_2)\in \mathfrak{S}_n\subset Aut \; \Gamma.$
Since $\Gamma^\bullet_{k+2}$ is a characteristic subgroup of $\Gamma,$
$h$ induces an element $\overline{h}\in Aut\;\mathcal{N}^\bullet_{k+1}.$ 
Then we have that
$$f([x_1])=\overline{x}\quad \text{and} \quad \overline{h}(f((\overline{h})^{-1}[x_1]))=\overline{h}(f([x_2]))=\overline{h}(1)=1.$$
Therefore the extension is not central.
\end{proof}

\subsection{Splitting the extensions}
In this section, our goal is to study the splitability of the extension
\begin{equation*}
\xymatrix@C=7mm@R=10mm{0 \ar@{->}[r] & Hom(\mathcal{N}^\bullet_1,\mathcal{L}^\bullet_{k+1}) \ar@{->}[r]^-{i} & Aut\;\mathcal{N}^\bullet_{k+1} \ar@{->}[r]^-{\psi^\bullet_k} & Aut \;\mathcal{N}^\bullet_k \ar@{->}[r] & 1. }
\end{equation*}
We first consider the case $k=1.$
To make notations lighter, we will denote by $H_p$ the group
$\mathcal{N}_1^S=\mathcal{N}_1^Z=\Gamma/\Gamma^p\Gamma_2.$

\begin{nota}
Let $p$ be a prime number and $n\geq 2$ an integer, we
denote by $\mathfrak{gl}_{n}(\mathbb{Z}/p)$ the additive group of matrices $n\times n$ with coefficients in $\mathbb{Z}/p,$
and by $\mathfrak{sl}_{n}(\mathbb{Z}/p)$ the subgroup of $\mathfrak{gl}_{n}(\mathbb{Z}/p)$ formed by the matrices of trace zero. 
\end{nota}

\begin{teo}[Theorem 7 in \cite{chih}]
\label{teo_SL_split}
Let $p$ be a prime number and $n\geq 2$ an integer. The extension
$$\xymatrix@C=7mm@R=10mm{ 0 \ar@{->}[r] & \mathfrak{sl}_{n}(\mathbb{Z}/p) \ar@{->}[r] &  SL_{n}(\mathbb{Z}/p^2) \ar@{->}[r]^{r_p} &  SL_{n}(\mathbb{Z}/p) \ar@{->}[r] & 1,}$$
only splits for $(p,n)=(3,2)\text{ and }(2,3).$
\end{teo}
\begin{cor}
\label{cor_put_GL_split}
Let $p$ be a prime number and $n\geq 2$ an integer. The extension
\begin{equation}
\label{ext_GL_split}
\xymatrix@C=7mm@R=10mm{ 0 \ar@{->}[r] & \mathfrak{gl}_{n}(\mathbb{Z}/p) \ar@{->}[r] &  GL_{n}(\mathbb{Z}/p^2) \ar@{->}[r]^{r_p} &  GL_{n}(\mathbb{Z}/p) \ar@{->}[r] & 1,}
\end{equation}
only splits for $(p,n)=(3,2),\; (2,2) \text{ and } (2,3).$
\end{cor}

\begin{proof}
We first prove that the extension does not split for $(p,n)\neq(3,2),\;(2,2) \text{ and } (2,3).$
Set
$SL_{n}^{(p)}(\mathbb{Z}/p^2)=\{ A\in GL_{n}(\mathbb{Z}/p^2) \mid det(A)\equiv 1 \;(\text{mod }p) \}.$

Notice that we have a pull-back diagram
\begin{equation*}
\xymatrix@C=7mm@R=10mm{ 0 \ar@{->}[r] & \mathfrak{gl}_{n}(\mathbb{Z}/p) \ar@{=}[d] \ar@{->}[r] &  SL^{(p)}_{n}(\mathbb{Z}/p^2)\ar@{^{(}->}[d] \ar@{->}[r]^{r_p} &  SL_{n}(\mathbb{Z}/p) \ar@{^{(}->}[d]^{\phi} \ar@{->}[r] & 1 \\
0 \ar@{->}[r] & \mathfrak{gl}_{n}(\mathbb{Z}/p) \ar@{->}[r] &  GL_{n}(\mathbb{Z}/p^2) \ar@{->}[r]^{r_p} &  GL_{n}(\mathbb{Z}/p) \ar@{->}[r] & 1.}
\end{equation*}
Then it is enough to show that the top extension in above diagram does not split.

Next notice that we have a push-out diagram
\begin{equation}
\label{diag_push-out_M}
\xymatrix@C=7mm@R=10mm{ 0 \ar@{->}[r] & \mathfrak{sl}_{n}(\mathbb{Z}/p) \ar@{^{(}->}[d]^{i} \ar@{->}[r] &  SL_{n}(\mathbb{Z}/p^2)\ar@{^{(}->}[d] \ar@{->}[r]^{r_p} &  SL_{n}(\mathbb{Z}/p) \ar@{=}[d] \ar@{->}[r] & 1 \\
0 \ar@{->}[r] & \mathfrak{gl}_{n}(\mathbb{Z}/p) \ar@{->}[r] &  SL^{(p)}_{n}(\mathbb{Z}/p^2) \ar@{->}[r]^{r_p} &  SL_{n}(\mathbb{Z}/p) \ar@{->}[r] & 1.}
\end{equation}
which induce a map
$i_*: H^2(SL_{n}(\mathbb{Z}/p);\mathfrak{sl}_{n}(\mathbb{Z}/p))\longrightarrow H^2(SL_{n}(\mathbb{Z}/p);\mathfrak{gl}_{n}(\mathbb{Z}/p)).$

By Theorem \eqref{teo_SL_split} the top extension in above diagram does not split for $(p,n)\neq (3,2),(2,3).$
Next we show that $i_*$ is injective for $(p,n)\neq (2,2),(3,2),$ and as a consequence, we will get that the extension \eqref{ext_GL_split} does not split for $(p,n)\neq (3,2),(2,2),(2,3).$

Consider the short exact sequence
$$\xymatrix@C=7mm@R=10mm{ 0 \ar@{->}[r] & \mathfrak{sl}_{n}(\mathbb{Z}/p) \ar@{->}[r] &  \mathfrak{gl}_{n}(\mathbb{Z}/p) \ar@{->}[r]^-{tr} & \mathbb{Z}/p \ar@{->}[r] & 1,}$$
where $tr$ is given by the matrix trace.
The long chomology sequence for $SL_{n}(\mathbb{Z}/p)$ with values in above short exact sequence, give us an exact sequence
$$\xymatrix@C=7mm@R=10mm{ H^1(SL_{n}(\mathbb{Z}/p);\mathbb{Z}/p) \ar@{->}[r] &  H^2(SL_{n}(\mathbb{Z}/p);\mathfrak{sl}_{n}(\mathbb{Z}/p)) \ar@{->}[r]^-{i_*} & H^2(SL_{n}(\mathbb{Z}/p);\mathfrak{gl}_{n}(\mathbb{Z}/p)).}$$
By proof of Theorem 7 in \cite{chih},
$H^1(SL_{n}(\mathbb{Z}/p);\mathbb{Z}/p)=0$ for $(p,n)\neq (2,2),(3,2).$
As a consequence, $i_*$ is injective for $(p,n)\neq (2,2), (3,2).$

Next we prove that the extension \eqref{ext_GL_split} splits for $(p,n)=(3,2),\; (2,2) \text{ and } (2,3).$

By Proposition 4.5 in \cite{chih2}, we know that $H^2(GL_2(\mathbb{Z}/p);\mathfrak{gl}_2(\mathbb{Z}/p))=0$ for $p=2,3.$ Therefore the extension \eqref{ext_GL_split} splits for $(p,n)=(3,2) \text{ and } (2,2).$

For the case $(p,n)=(2,3),$ consider the push-out diagram
\begin{equation}
\label{diag_push-out_M_(2,3)}
\xymatrix@C=7mm@R=10mm{ 0 \ar@{->}[r] & \mathfrak{sl}_{3}(\mathbb{Z}/2) \ar@{^{(}->}[d]^{i} \ar@{->}[r] &  SL_{3}(\mathbb{Z}/4)\ar@{^{(}->}[d] \ar@{->}[r]^{r_2} &  SL_{3}(\mathbb{Z}/2) \ar@{=}[d] \ar@{->}[r] & 1 \\
0 \ar@{->}[r] & \mathfrak{gl}_{3}(\mathbb{Z}/2) \ar@{->}[r] &  SL^{(2)}_{3}(\mathbb{Z}/4) \ar@{->}[r]^{r_2} &  SL_{3}(\mathbb{Z}/2) \ar@{->}[r] & 1.}
\end{equation}
By Theorem \eqref{teo_SL_split} we know that the top extension of this commutative diagram splits. Therefore, the bottom extension in above diagram splits too.
Notice that $SL^{(2)}_{3}(\mathbb{Z}/4)=GL_3(\mathbb{Z}/4)$ and $SL_3(\mathbb{Z}/2)=SL_3(\mathbb{Z}/2).$ Hence, the extension \eqref{ext_GL_split} splits for $(p,n)=(2,3).$
\end{proof}

\begin{defi}
Let $k$ be an arbitrary commutative ring, let $V$ be a $k$-module, and let $T^q(V)=V\otimes \cdots \otimes V$ ($q$ copies of $V$), where $\otimes=\otimes_k.$ We denote by $\extp^q(V)$ the quotient of $T^q(V)$ by the $k$-submodule generated by the elements $v_1\otimes \cdots \otimes v_q$ such that $v_i=v_{i+1}$ for some $i.$
\end{defi}

\begin{prop}
Let $p$ be a prime number and $\Gamma$ a free group of rank $n.$ The extension
\begin{equation}
\label{ext_split_1}
\xymatrix@C=7mm@R=10mm{0 \ar@{->}[r] & Hom(\mathcal{N}^\bullet_1,\mathcal{L}^\bullet_{2}) \ar@{->}[r]^-{i} & Aut\;\mathcal{N}^\bullet_{2} \ar@{->}[r]^-{\psi^\bullet_1} & Aut \;\mathcal{N}^\bullet_1 \ar@{->}[r] & 1, }
\end{equation}
only splits for $\bullet=Z$ with $p$ odd and for $\bullet=S$ with $(p,n)= (3,2), (2,2), (2,3).$
\end{prop}

\begin{rem}
Notice that for $p=2,$ the extensions \eqref{ext_split_1} with $\bullet=Z,$ $\bullet=S$ coincide, because in this case,
$\Gamma_i^S=\Gamma_i^Z$ for $i=1,2,3.$
\end{rem}

\begin{proof}
\textbf{For $\bullet=Z$ and $p$ an odd prime.} In this case, we have that
$$\mathcal{L}^Z_{2}\cong\extp^2 \mathcal{N}^Z_1=\extp^2 H_p, \qquad Aut\; \mathcal{N}^Z_1 =Aut\; H_p\cong GL_{n}(\mathbb{Z}/p).$$
Then, the extension \eqref{ext_split_1} becomes
\begin{equation*}
\xymatrix@C=7mm@R=10mm{0 \ar@{->}[r] & Hom(H_p,\extp^2 H_p) \ar@{->}[r]^-{i} & Aut\;\mathcal{N}^Z_{2} \ar@{->}[r]^-{\psi^Z_1} & GL_{n}(\mathbb{Z}/p) \ar@{->}[r] & 1. }
\end{equation*}
Notice that $-Id$ is an element of the center of $GL_{n}(\mathbb{Z}/p),$ which acts on $Hom(H_p,\extp^2 H_p)$ by the multiplication of $-1.$
Then, by the Center kills Lemma,
$$H_2(GL_{n}(\mathbb{Z}/p);Hom(H_p,\extp^2 H_p))=0.$$
As a consequence, by Lemma \eqref{lema_duality_homology}, we have that
\begin{align*}
H^2(GL_{n}(\mathbb{Z}/p);Hom(H_p,\extp^2 H_p))\cong\; & H^2(GL_{n}(\mathbb{Z}/p);Hom(H_p,\extp^2 H_p)^*)\cong \\
\cong\; & (H^2(GL_{n}(\mathbb{Z}/p);Hom(H_p,\extp^2 H_p)))^*=0.
\end{align*}
Therefore the extension \eqref{ext_split_1} splits.

\textbf{For $\bullet=S$ and $p$ prime with $(p,n)\neq (3,2), (2,2), (2,3).$} 
In this case, we have that
$$ Aut\; \mathcal{N}^S_1 =Aut\; H_p\cong GL_{n}(\mathbb{Z}/p).$$
Then, the extension \eqref{ext_split_1} becomes
\begin{equation}
\label{ses_nosplit_S}
\xymatrix@C=7mm@R=10mm{0 \ar@{->}[r] & Hom(H_p,\mathcal{L}^S_{2}) \ar@{->}[r]^-{i} & Aut\;\mathcal{N}^S_{2} \ar@{->}[r]^-{\psi^S_1} & GL_{n}(\mathbb{Z}/p) \ar@{->}[r]  & 1.}
\end{equation}
Set
$\overline{\mathcal{N}^S_{2}}=\frac{\Gamma }{\Gamma_2 \Gamma^{p^2}}, \; \overline{\mathcal{L}^S_{2}}=\frac{\Gamma^p}{\Gamma_2 \Gamma^{p^2}}.$ Observe that $\frac{\Gamma_2}{\Gamma_3\Gamma_2^p \Gamma^{p^2}}$ is a characteristic group of $\mathcal{N}_2^S.$ Then there is a well defined homomorphism $\widetilde{q}:Aut\;\mathcal{N}^S_{2} \rightarrow Aut\;\overline{\mathcal{N}^S_{2}}$ and there is a push-out diagram
\begin{equation}
\label{com_diag_S2}
\xymatrix@C=7mm@R=10mm{0 \ar@{->}[r] & Hom(H_p,\mathcal{L}^S_{2})\ar@{->}[d]^q \ar@{->}[r]^-{i} & Aut\;\mathcal{N}^S_{2} \ar@{->}[d]^{\widetilde{q}} \ar@{->}[r]^-{\psi^S_1} & GL_{n}(\mathbb{Z}/p) \ar@{->}[r] \ar@{=}[d] & 1 \\
0 \ar@{->}[r] & Hom(H_p,\overline{\mathcal{L}^S_{2}}) \ar@{->}[r]^-{i} & Aut\;\overline{\mathcal{N}^S_{2}} \ar@{->}[r]^-{\overline{\psi^S_1}} & GL_{n}(\mathbb{Z}/p) \ar@{->}[r] & 1, }
\end{equation}
where $q,$ $\widetilde{q}$ are induced by the quotient map respect to $\frac{\Gamma_2}{\Gamma_3\Gamma_2^p \Gamma^{p^2}}.$
Next, notice that $$Hom(H_p,\overline{\mathcal{L}^S_{2}})\cong \mathfrak{gl}_{n}(\mathbb{Z}/p), \qquad Aut\;\overline{\mathcal{N}^S_{2}}\cong GL_{n}(\mathbb{Z}/p^2).$$
Thus, the bottom row of diagram \eqref{com_diag_S2} becomes
\begin{equation*}
\xymatrix@C=7mm@R=10mm{ 0 \ar@{->}[r] & \mathfrak{gl}_{n}(\mathbb{Z}/p) \ar@{->}[r] &  GL_{n}(\mathbb{Z}/p^2) \ar@{->}[r]^-{r_p} &  GL_{n}(\mathbb{Z}/p) \ar@{->}[r] & 1.}
\end{equation*}
By Corollary \eqref{cor_put_GL_split}, the above extension does not split for $(p,n)\neq (3,2), (2,2), (3,2).$ Then, by push-out diagram \eqref{com_diag_S2}, we have that the extension \eqref{ses_nosplit_S} does not split for $(p,n)\neq (3,2),$ $(2,2), (2,3).$

\textbf{For $\bullet=S$ with $(p,n)=(3,2),(2,2).$}
In this case the extension \eqref{ext_split_1} becomes
\begin{equation}
\label{ses_split_S2}
\xymatrix@C=7mm@R=10mm{0 \ar@{->}[r] & Hom(H_p,\mathcal{L}^S_{2}) \ar@{->}[r]^-{i} & Aut\;\mathcal{N}^S_{2} \ar@{->}[r]^-{\psi^S_1} & GL_2(\mathbb{Z}/p) \ar@{->}[r]  & 1.}
\end{equation}
By Universal coefficients Theorem and Hopf formula we have that
$$\mathcal{L}^S_2=H_2(\mathcal{N}_1^S,\mathbb{Z}/p)\cong H_2(\mathcal{N}_1^S)\otimes \mathbb{Z}/p \oplus Tor(H_1(N_1^S); \mathbb{Z}/p)\cong \extp^2 H_p \oplus H_p,$$
as $GL_{2}(\mathbb{Z}/p)$-modules.
As a consequence, we have an isomorphism
$$Hom(H_p,\mathcal{L}^S_2)\cong Hom(H_p,\extp^2 H_p)\oplus Hom(H_p,H_p),$$
also as $GL_{2}(\mathbb{Z}/p)$-modules.
Then we have the following isomorphism in cohomology:
\begin{align*}
H^2(GL_{2}(\mathbb{Z}/p); & Hom(H_p,\mathcal{L}^S_2)) \cong \\
\cong H^2(GL_{2}(\mathbb{Z}/p);Hom(H_p, H_p)) & \oplus H^2(GL_{2}(\mathbb{Z}/p);Hom(H_p,\extp^2 H_p)).
\end{align*}
By Proposition 4.5 in \cite{chih2}, we know that $H^2(GL_2(\mathbb{Z}/p);Hom(H_p, H_p))=0$ for $p=2,3.$

For $(p,n)=(3,2),$ by the Center kills lemma, we have that
$$H^2(GL_{2}(\mathbb{Z}/3);Hom(H_3,\extp^2 H_3))=0.$$

For $(p,n)=(2,2),$ the groups $Sp_2(\mathbb{Z}/2),$ $SL_2(\mathbb{Z}/2)$ coincide and so we have an isomorphism $\extp^2H_2\cong \mathbb{Z}/2$ of $SL_2(\mathbb{Z}/2)$-modules given by the symplectic intersection form $\omega,$ which induces the following isomorphisms of $SL_2(\mathbb{Z}/2)$-modules
$$Hom(H_2,\extp^2 H_2)\cong Hom(H_2,\mathbb{Z}/2)\cong H_2.$$
Moreover, by Proposition 4.4 in \cite{chih2}, we know that $H^2(SL_2(\mathbb{Z}/2);H_2)=0.$

Therefore $H^2(GL_{2}(\mathbb{Z}/p);Hom(H_p,\mathcal{L}^S_2))=0$ for $p=2,3$ and so the extension \eqref{ses_split_S2} splits for $(p,n)=(3,2),(2,2).$

\textbf{For $\bullet=S$ with $(p,n)=(2,3).$}
Analogously to the above case, the isomorphism 
$$q_1\oplus q_2\; :Hom(H_2,\mathcal{L}_2^S)\rightarrow Hom(H_2,\extp^2 H_2) \oplus Hom(H_2,H_2)$$
induces the following isomorphism in cohomology:
\begin{equation}
\label{iso_suma_L2S}
\begin{aligned}
H^2(GL_{3}(\mathbb{Z}/2); & Hom(H_2,\mathcal{L}^S_2)) \cong \\
\cong H^2(GL_{3}(\mathbb{Z}/2);Hom(H_2, H_2)) & \oplus H^2(GL_{3}(\mathbb{Z}/2);Hom(H_2,\extp^2 H_2)).
\end{aligned}
\end{equation}
Notice that we have the following commutative diagram:
\begin{equation}
\label{diag_suma_L2S}
\xymatrix@C=7mm@R=10mm{0 \ar@{->}[r] & Hom(H_2,\mathfrak{gl}_3(\mathbb{Z}/2)) \ar@{->}[r]^-{i} & GL_{3}(\mathbb{Z}/4) \ar@{->}[r] & GL_{3}(\mathbb{Z}/2) \ar@{->}[r] & 1\\
0 \ar@{->}[r] & Hom(H_2,\mathcal{L}^S_{2})\ar@{->}[d]^{q_1}\ar@{->}[u]_{q_2} \ar@{->}[r]^-{i} & Aut\;\mathcal{N}^S_{2} \ar@{->}[d]^{\widetilde{q_1}}\ar@{->}[u]_{\widetilde{q_2}} \ar@{->}[r] & GL_{3}(\mathbb{Z}/2) \ar@{->}[r] \ar@{=}[d]\ar@{=}[u] & 1 \\
0 \ar@{->}[r] & Hom(H_2,\extp^2 H_2) \ar@{->}[r]^-{i} & Aut\;\frac{\Gamma}{\Gamma_3\Gamma^2} \ar@{->}[r] & GL_{3}(\mathbb{Z}/2) \ar@{->}[r] & 1. }
\end{equation}
By Corollary \eqref{cor_put_GL_split} the bottom extension in diagram \eqref{diag_suma_L2S} splits. Then, by the isomorphism \eqref{iso_suma_L2S}, the middle extension in diagram \eqref{diag_suma_L2S} splits if and only if the bottom extension in diagram \eqref{diag_suma_L2S} splits.

Next, observe that we have the following commutative diagram:
\begin{equation}
\xymatrix@C=7mm@R=10mm{0 \ar@{->}[r] & Hom(H_4,\extp^2 H_4)\ar@{->}[d]^q \ar@{->}[r]^-{i} & Aut\;\frac{\Gamma}{\Gamma_3\Gamma^4} \ar@{->}[d]^{\widetilde{q}} \ar@{->}[r] & GL_{3}(\mathbb{Z}/4) \ar@{->}[r] \ar@{->}[d] & 1 \\
0 \ar@{->}[r] & Hom(H_2,\extp^2 H_2) \ar@{->}[r]^-{i} & Aut\;\frac{\Gamma}{\Gamma_3\Gamma^2} \ar@{->}[r] & GL_{3}(\mathbb{Z}/2) \ar@{->}[r] & 1, }
\end{equation}
By the Center kills lemma, the top extension in the above diagram splits.
and, by Corollary \eqref{cor_put_GL_split}, the homomorphism
$GL_{3}(\mathbb{Z}/4)\rightarrow GL_{3}(\mathbb{Z}/2)$ also splits.
As a consequence, the bottom extension in the above commutative diagram splits too, which implies the result.
\end{proof}

\begin{prop}
\label{prop_non-split_versal}
Let $\Gamma$ be a free group of finite rank $n > 1.$ The extension
\begin{equation*}
\xymatrix@C=7mm@R=10mm{0 \ar@{->}[r] & Hom(\mathcal{N}^\bullet_1,\mathcal{L}^\bullet_{k+1}) \ar@{->}[r]^-{i} & IA^p(\mathcal{N}^\bullet_{k+1}) \ar@{->}[r]^-{\psi^\bullet_k} & IA^p(\mathcal{N}^\bullet_k ) \ar@{->}[r] & 1, }
\end{equation*}
does not split for $k\geq 2.$
\end{prop}

\begin{proof}
Applying Theorem \eqref{teo_5-term} to the central extension \begin{equation}
\label{ext_split_k}
\xymatrix@C=7mm@R=10mm{0 \ar@{->}[r] & Hom(\mathcal{N}^\bullet_1,\mathcal{L}^\bullet_{k+1}) \ar@{->}[r]^-{i} & IA^p(\mathcal{N}^\bullet_{k+1}) \ar@{->}[r]^-{\psi^\bullet_k} & IA^p(\mathcal{N}^\bullet_k ) \ar@{->}[r] & 1, }
\end{equation}
one gets the exact sequence
\begin{equation}
\label{les_5-term}
\begin{gathered}
\xymatrix@C=7mm@R=10mm{ Hom(IA^p(\mathcal{N}^\bullet_{k+1});Hom(\mathcal{N}^\bullet_1,\mathcal{L}^\bullet_{k+1})) \ar@{->}[r]^-{res} & Hom(Hom(\mathcal{N}^\bullet_1,\mathcal{L}^\bullet_{k+1}),Hom(\mathcal{N}^\bullet_1,\mathcal{L}^\bullet_{k+1}))\ar@{->}[r] &}
\\
\xymatrix@C=7mm@R=10mm{ \ar@{->}[r]^-{\delta} & H^2(IA^p(\mathcal{N}^\bullet_k ); Hom(\mathcal{N}^\bullet_1,\mathcal{L}^\bullet_{k+1})).}
\end{gathered}
\end{equation}
Then the cohomology class associated to the extension \eqref{ext_split_k} is given by $\delta(id).$

Suppose that the central extension \eqref{ext_split_k} splits, i.e. $\delta(id)=0.$ Since \eqref{les_5-term} is exact, there would be an element $f\in Hom(IA^p(\mathcal{N}^\bullet_{k+1});Hom(\mathcal{N}^\bullet_1,\mathcal{L}^\bullet_{k+1}))$ such that $res(f)=id.$

Notice that for every element $x\in [IA^p(\mathcal{N}^\bullet_{k+1}),IA^p(\mathcal{N}^\bullet_{k+1})],$ one has that $f(x)=0$ because $Hom(\mathcal{N}^\bullet_1,\mathcal{L}^\bullet_{k+1})$ is abelian.
On the other hand, by Corollary \eqref{coro_cota_IA_bullet,N_bullet},
$$[IA^p(\mathcal{N}^\bullet_{k+1}),IA^\bullet_{k-1}(\mathcal{N}^\bullet_{k+1})]\neq 1.$$
In addition, by Corollary \eqref{cor_IA_bullet},
$$\psi^\bullet_k([IA^p(\mathcal{N}^\bullet_{k+1}),IA^\bullet_{k-1}(\mathcal{N}^\bullet_{k+1})])=[IA^p(\mathcal{N}^\bullet_{k}),IA^\bullet_{k-1}(\mathcal{N}^\bullet_{k})]\leq IA^\bullet_{k}(\mathcal{N}^\bullet_{k})=1.$$
Then, by short exact sequence \eqref{ses_ZS},
$[IA^p(\mathcal{N}^\bullet_{k+1}),IA^\bullet_{k-1}(\mathcal{N}^\bullet_{k+1})] \hookrightarrow Hom(\mathcal{N}^\bullet_1,\mathcal{L}^\bullet_{k+1}).$

As a consequence, $res(f)$ can not be the identity.
\end{proof}

\begin{cor}
Let $\Gamma$ be a free group of finite rank $n > 1.$ The extension
\begin{equation*}
\xymatrix@C=7mm@R=10mm{0 \ar@{->}[r] & Hom(\mathcal{N}^\bullet_1,\mathcal{L}^\bullet_{k+1}) \ar@{->}[r]^-{i} & Aut\;\mathcal{N}^\bullet_{k+1} \ar@{->}[r]^-{\psi^\bullet_k} & Aut \;\mathcal{N}^\bullet_k \ar@{->}[r] & 1 }
\end{equation*}
does not split for $k\geq 2.$
\end{cor}

\begin{proof}
Observe that we have the following pull-back diagram:
\begin{equation*}
\xymatrix@C=7mm@R=10mm{0 \ar@{->}[r] & Hom(\mathcal{N}^\bullet_1,\mathcal{L}^\bullet_{k+1})\ar@{=}[d] \ar@{->}[r]^-{i} & IA^p(\mathcal{N}^\bullet_{k+1})\ar@{->}[r]^-{\psi^\bullet_k} \ar@{^{(}->}[d] & IA^p(\mathcal{N}^\bullet_k) \ar@{^{(}->}[d] \ar@{->}[r] & 1\\
0 \ar@{->}[r] & Hom(\mathcal{N}^\bullet_1,\mathcal{L}^\bullet_{k+1}) \ar@{->}[r]^-{i} & Aut\;\mathcal{N}^\bullet_{k+1} \ar@{->}[r]^-{\psi^\bullet_k} & Aut \;\mathcal{N}^\bullet_k \ar@{->}[r] & 1
.}
\end{equation*}
By Proposition \eqref{prop_non-split_versal}, the top extension of this diagram does not split. Therefore the bottom extension this diagram does not split.
\end{proof}

\chapter{Trivial cocycles and invariants on the $(\text{mod }d)$ Torelli group}
\label{chapter: trivocicles torelli mod p}

In this chapter, following the same strategy of Chapter \ref{chapter: trivocicles torelli}, we give a tool to construct invariants of $\mathbb{Q}$-homology spheres that have a Heegaard splitting with gluing map an element of the $(\text{mod }d)$ Torelli group, from a family of trivial $2$-cocycles on the $(\text{mod }d)$ Torelli group.
As in Chapter \ref{chapter: trivocicles torelli}, such a construction does not give a unique invariant. This is because there exists an invariant, which is not an extension of the Rohlin invariant, that makes to fail the unicity of the construction.

As a starting point, in the first section we give some results about a mod $p$ version of the Johnson homomorphism.
In the second and third sections we exhibit the construction of invariants form a family of trivial $2$-cocycles on the $(\text{mod }d)$ Torelli group.
Finally, in the last section we give a way to construct a family of $2$-cocycles satisfying the hypothesis of our construction as a pull-back of bilinear forms along the mod $p$ Johnson homomorphisms.

\section{Johnson homomorphisms mod p}
\label{sec_johnson mod p}
Let $\Gamma$ be the fundamental group of $\Sigma_{g,1}$ and $\Gamma^\bullet_k$ the $k$-term of one of the mod-$p$ central series defined in Section \ref{section: mod-p series}.
The natural action of the mapping class group $\mathcal{M}_{g,1}$ on $\Gamma$ induces one on the $p$-group $\mathcal{N}^\bullet_{k}=\Gamma/\Gamma^\bullet_{k+1}.$ Therefore we have a representation
$$\rho^\bullet_k:\mathcal{M}_{g,1}\rightarrow Aut\; \mathcal{N}^\bullet_k.$$
Denote  by $\mathcal{I}^\bullet_{g,1}(k)\subset \mathcal{M}_{g,1}$ the kernel of $\rho_{k}^\bullet:\mathcal{M}_{g,1} \rightarrow Aut\;\mathcal{N}^\bullet_{k}.$ For instance, $$\mathcal{I}^S_{g,1}(1)=\mathcal{I}^Z_{g,1}(1)=\mathcal{M}_{g,1}[p].$$
In \cite{coop} J. Cooper defined the following maps
\begin{align*}
\tau_k^\bullet:\mathcal{I}^\bullet_{g,1}(k) & \rightarrow Hom (H_p, \mathcal{L}^\bullet_{k+1}) \\
f &\mapsto (x\mapsto f(x)x^{-1})
\end{align*}
and proved that these are well defined homomorphisms.
These homomorphisms are called the \textit{Zassenhaus mod-$p$ Johnson homomorphisms} for $\bullet=Z,$ and \textit{Stallings mod-$p$ Johnson homomorphisms} for $\bullet=S.$
In general, for $\bullet=Z \text{ or } S,$ we call these maps \textit{the mod-$p$ Johnson homomorphisms.}

In fact, we show that these homomorphisms arise in a more natural way.
\begin{prop}
The mod-$p$ Johnson homomorphisms $\tau_k^\bullet$ are the restriction of $\rho_{k+1}^\bullet$ to $\mathcal{I}_{g,1}^\bullet.$
\end{prop}

\begin{proof}
Consider the following diagram:
\begin{equation*}
\xymatrix@C=7mm@R=10mm{0 \ar@{->}[r] & \mathcal{I}^\bullet_{g,1}(k) \ar@{->}[r] \ar@{->}[d]^-{\tau^\bullet_k}& \mathcal{M}_{g,1} \ar@{->}[r]^{\rho_{k}^\bullet} \ar@{->}[d]^-{\rho_{k+1}^\bullet}& Aut \;\mathcal{N}^\bullet_k \ar@{=}[d] &  \\
0 \ar@{->}[r] & Hom(H_p, \mathcal{L}^\bullet_{k+1}) \ar@{->}[r]^-{i} & Aut\;\mathcal{N}^\bullet_{k+1} \ar@{->}[r]^-{\pi} & Aut \;\mathcal{N}^\bullet_k \ar@{->}[r] & 1 .}
\end{equation*}
If $f\in \mathcal{I}^\bullet_{g,1}(k),$ a trivial computation show that show that $i(\tau_k^\bullet (f))=\rho_{k+1}^\bullet(f).$
\end{proof}

\begin{lema}[Lemma 5.10 in \cite{coop}]
Let $\gamma\subset \Sigma_{g,1}$ be a simple closed curve and $y\in \Gamma$ be a based curve isotopic to $\gamma.$ Let $x\in \Gamma$ with $|x\cap y|=n.$ Then
$$T^p_\gamma(x)x^{-1}=\prod^n_{i=1}z_iy^{\epsilon(p_i)p}z_i^{-1}$$
for some $z_i\in \Gamma$ (where $\epsilon(p_i)=\pm 1$ depending on the sign of the intersection between $x$ and $y$ at $p_i).$
\end{lema}

As a consequence of this result, following the proof of Lemma 5.11 in \cite{coop}, we have the following proposition:

\begin{prop}
\label{prop_tau2_Dp}
Let $p$ be an odd prime number and $D_{g,1}[p]$ the subgroup of $\mathcal{M}_{g,1}[p]$ generated by $p$-powers of Dehn twists. Then $D_{g,1}[p]\subset \mathcal{I}^Z_{g,1}(p-1),$ i.e. $D_{g,1}[p]$ is in the kernel of $\tau_k^Z$ for every $k\leq p-1.$
\end{prop}

\begin{proof}
We proceed by induction.
We know that $D_{g,1}[p]\subset \mathcal{I}^Z_{g,1}(1)=\mathcal{M}_{g,1}[p].$ 
Now suppose that $D_{g,1}[p]\subset \mathcal{I}^Z_{g,1}(k)$ for $k\leq p-1$ and we will prove that $D_{g,1}[p]\subset \mathcal{I}^Z_{g,1}(k+1).$
Recall that by the properties of Zassenhaus filtration, we know that $[\Gamma^Z_k,\Gamma_l]<\Gamma^Z_{k+l}$ and observe that by definition of $\Gamma^Z_l,$ for $l\leq p$ we have that $\Gamma^p<\Gamma^Z_l,$ so we obtain that
$[\Gamma^p,\Gamma]<\Gamma^Z_l$ for every $l\leq p.$
Then by definition of $\tau_k^Z,$ we know that given an element $f\in \mathcal{I}^Z_{g,1}(k),$
$$\tau_k^Z(f)= (x \mapsto f(x)x^{-1} \text{ mod } \Gamma^Z_{k+1})$$
Taking $f=T^p_\gamma,$ since $[\Gamma^p,\Gamma]<\Gamma^Z_{k+1}$ and $\Gamma^p<\Gamma^Z_{k+1},$ we have that
$$f(x)x^{-1}=\prod^n_{i=1}z_iy^{\epsilon(p_i) p}z_i^{-1} \equiv y^{pk}\equiv 0 \text{ mod }\Gamma^Z_{k+1},$$ where $k=\sum_{i=1}^n \epsilon(p_i).$
Thus $f\in Ker(\tau_k^Z)\subset Ker(\rho_{k+1}^Z),$ i.e. $f\in \mathcal{I}^Z_{g,1}(k+1).$
\end{proof}

\begin{rem}
Observe that $\mathcal{I}^\bullet_{g,1}(k)=Ker(\rho^\bullet_k)$ and thus $\mathcal{I}^\bullet_{g,1}(k+1)\subset \mathcal{I}^\bullet_{g,1}(k).$ So $Ker(\rho^\bullet_{k+1})$ is equal to the kernel of $\rho^\bullet_{k+1}$ restricted to $\mathcal{I}^\bullet_{g,1}(k),$ but by definition of $\tau_k^Z,$ we know that this is the kernel of $\tau_k^Z.$
Therefore $Ker(\rho^\bullet_{k+1})=Ker(\tau^\bullet_k).$
\end{rem}

\section{From invariants to trivial cocycles}
\label{section_From invariants to trivial cocycles}
Let $A$ be an abelian group. Denote by $A_d$ the subgroup formed by $d$-torsion elements.

Consider an $A$-valuated invariant of $\mathbb{Z}/d$-homology spheres
$$F: \mathcal{S}^3[d] \rightarrow A.$$
Precomposing with the canonical maps $\mathcal{M}_{g,1}[d]\rightarrow \lim_{g\rightarrow \infty} \mathcal{M}_{g,1}[d]/\sim \rightarrow  \mathcal{S}^3[d]$ we get a family of maps $\{F_g\}_g$ with $F_g:\mathcal{M}_{g,1}[d]\rightarrow A$ satisfying the following properties:
\begin{enumerate}[i)]
\item $F_{g+1}(x)=F_g(x) \quad \text{for every }x\in \mathcal{M}_{g,1}[d],$
\item $F_g(\xi_a x\xi_b)=F_g(x) \quad \text{for every } x\in \mathcal{M}_{g,1}[d],\;\xi_a\in \mathcal{A}_{g,1}[d],\;\xi_b\in \mathcal{B}_{g,1}[d],$
\item $F_g(\phi x \phi^{-1})=F_g(x)  \quad \text{for every }  x\in \mathcal{M}_{g,1}[d], \; \phi\in \mathcal{AB}_{g,1}.$
\end{enumerate}
We avoid the peculiarities of the first Torelli groups by restricting ourselves to $g\geq 3.$ If we consider the associated trivial $2$-cocycles $\{C_g\}_g,$ i.e.
\begin{align*}
C_g: \mathcal{M}_{g,1}[d]\times \mathcal{M}_{g,1}[d] & \longrightarrow A, \\
 (\phi,\psi) & \longmapsto F_g(\phi)+F_g(\psi)-F_g(\phi\psi),
\end{align*}
then the family of $2$-cocycles $\{C_g\}_g$ inherits the following properties:
\begin{enumerate}[(1)]
\item The $2$-cocycles $\{C_g\}_g$ are compatible with the stabilization map,
\item The $2$-cocycles $\{C_g\}_g$ are invariant under conjugation by elements in $\mathcal{AB}_{g,1},$
\item If $\phi\in \mathcal{A}_{g,1}[d]$ or $\psi \in \mathcal{B}_{g,1}[d]$ then $C_g(\phi, \psi)=0.$
\end{enumerate}

In general there are many families of maps $\{F_g\}_g$ satisfying the properties i) - iii) that induce the same family of trivial $2$-cocycles $\{C_g\}_g.$
Next, we concern about the number of such families.

Notice that given two families of maps $\{F_g\}_g,$ $\{F_g'\}_g$ satisfying the properties i) - iii), we have that
$\{F_g-F_g'\}_g$ is a family of homomorphisms satisfying the same conditions.
As a consequence, the number of families $\{F_g\}_g$ satisfying the properties i) - iii) that induce the same family of trivial $2$-cocycles $\{C_g\}_g,$ coincides with the number of homomorphisms in $Hom(\mathcal{M}_{g,1}[p],A)^{\mathcal{AB}_{g,1}}$ compatible with the stabilization map, and that are trivial over $\mathcal{A}_{g,1}[p]$ and $\mathcal{B}_{g,1}[p].$
We devote the rest of this section to compute and study such homomorphisms.

\begin{prop}
\label{prop_iso_M[d],Sp[d]}
Let $A$ be an abelian group. For $g\geq 3,$ $d\geq 3$ an odd integer and for $g\geq 5,$ $d\geq 2$ an even integer such that $4 \nmid d,$ the pull-back along $\Psi:\mathcal{M}_{g,1}[d]\rightarrow Sp_{2g}(\mathbb{Z},d)$ induces an isomorphism
$$Hom(\mathcal{M}_{g,1}[d],A)^{\mathcal{AB}_{g,1}}\cong Hom(Sp_{2g}(\mathbb{Z},d),A)^{GL_g(\mathbb{Z})}.$$
\end{prop}

\begin{proof}
Restricting the symplectic representation of the mapping class group to $\mathcal{M}_{g,1}[d]$ we have a short exact sequence
\begin{equation}
\label{ses_MCGp}
\xymatrix@C=10mm@R=13mm{1 \ar@{->}[r] & \mathcal{T}_{g,1} \ar@{->}[r] & \mathcal{M}_{g,1}[d] \ar@{->}[r]^-{\Psi} & Sp_{2g}(\mathbb{Z},d) \ar@{->}[r] & 1 }.
\end{equation}
Applying the 5-term sequence associated to the above short exact sequence, we get an exact sequence
\begin{equation}
\label{es_MCGpH_1}
\xymatrix@C=7mm@R=10mm{ H_1(\mathcal{T}_{g,1};\mathbb{Z})_{Sp_{2g}(\mathbb{Z},d)} \ar@{->}[r]^-{j} & H_1(\mathcal{M}_{g,1}[d];\mathbb{Z}) \ar@{->}[r] & H_1(Sp_{2g}(\mathbb{Z},d);\mathbb{Z}) \ar@{->}[r] & 1 .}
\end{equation}
In Lemma 6.4. in \cite{Putman}, A. Putman proved that the homomorphism $j$ in \eqref{es_MCGpH_1} is injective for $g\geq 3,$ $d\geq 3$ an odd integer and that for $g\geq 5,$ $d\geq 2$ an even integer with $4\nmid d,$ the kernel of $j$ in \eqref{es_MCGpH_1} is $\mathbb{Z}/2.$
Next we distinguish both cases.

\textbf{ For $g\geq 3$ and $d\geq 3$ an odd integer,}
we have the following commutative diagram of exact sequences:
\begin{equation*}
\xymatrix@C=7mm@R=10mm{1 \ar@{->}[r] & \mathcal{T}_{g,1} \ar@{->}[r]^-{i} \ar@{->>}[d] & \mathcal{M}_{g,1}[d] \ar@{->}[r]^-{\Psi} \ar@{->>}[d] & Sp_{2g}(\mathbb{Z},d) \ar@{->}[r] \ar@{->>}[d] & 1 \\
 1 \ar@{->}[r] & H_1(\mathcal{T}_{g,1};\mathbb{Z})_{Sp_{2g}(\mathbb{Z},d)} \ar@{->}[r]^-{j} & H_1(\mathcal{M}_{g,1}[d];\mathbb{Z}) \ar@{->}[r] & H_1(Sp_{2g}(\mathbb{Z},d);\mathbb{Z}) \ar@{->}[r] & 1. }
\end{equation*}
By Proposition 6.6 in \cite{Putman}, we have an isomorphism $H_1(\mathcal{T}_{g,1};\mathbb{Z})_{Sp_{2g}(\mathbb{Z},d)}\cong \extp^3H_d,$
induced by the first Zassenhaus mod-$d$ Johnson homomorphism $\tau_1^Z.$ 
Then we have a commutative diagram
\begin{equation}
\label{com_diag_M[p]}
\xymatrix@C=10mm@R=13mm{\mathcal{T}_{g,1} \ar@{->}[r]^-{i} \ar@{->>}[d]^{\tau_1^Z} & \mathcal{M}_{g,1}[d]  \ar@{->}[d]^{H_1} \\
 \extp^3H_d \ar@{->}[r]^-{j} & H_1(\mathcal{M}_{g,1}[d];\mathbb{Z})  .}
\end{equation}
Consider an element $f\in Hom(\mathcal{M}_{g,1}[d],A)^{\mathcal{AB}_{g,1}}.$
Since $A$ is an abelian group, by commutativity of diagram \eqref{com_diag_M[p]},
the restriction of $f$ to $\mathcal{T}_{g,1}$ factors through $\tau_1^Z.$
Since $\tau_1^Z$ is an $\mathcal{M}_{g,1}$-equivariant epimorphism, we have that the restriction of $f$ to $\mathcal{T}_{g,1}$ is a pull-back of an element of $Hom(\extp^3H_d,A)^{GL_g(\mathbb{Z})}$ along $\tau_1^Z.$
But, by Lemma \eqref{lema_hom_extp_H_d}, $Hom(\extp^3H_d,A)^{GL_g(\mathbb{Z})}=0.$

As a consequence, the restriction of $f$ to $\mathcal{T}_{g,1},$ is zero and therefore $f$ factors through the symplectic representation,
i.e. we have an isomorphism
$$Hom(\mathcal{M}_{g,1}[d],A)^{\mathcal{AB}_{g,1}}\cong Hom(Sp_{2g}(\mathbb{Z},d),A)^{GL_g(\mathbb{Z})}.$$

\textbf{For $g\geq 5$ and $d\geq 2$ an even integer such that $4 \nmid d,$}
we have the following commutative diagrams of exact sequences:
\begin{equation*}
\xymatrix@C=7mm@R=10mm{ & 1 \ar@{->}[r] & \mathcal{T}_{g,1} \ar@{->}[r] \ar@{->>}[d] & \mathcal{M}_{g,1}[d] \ar@{->}[r]^-{\Psi} \ar@{->>}[d] & Sp_{2g}(\mathbb{Z},d) \ar@{->}[r] \ar@{->>}[d] & 1 \\
 1 \ar@{->}[r] & \mathbb{Z}/2 \ar@{->}[r] & H_1(\mathcal{T}_{g,1};\mathbb{Z})_{Sp_{2g}(\mathbb{Z},d)} \ar@{->}[r]^-{j} & H_1(\mathcal{M}_{g,1}[d];\mathbb{Z}) \ar@{->}[r] & H_1(Sp_{2g}(\mathbb{Z},d);\mathbb{Z}) \ar@{->}[r] & 1 ,}
\end{equation*}
\begin{equation*}
\xymatrix@C=7mm@R=10mm{ 1 \ar@{->}[r] & \mathcal{K}_{g,1} \ar@{->}[r] \ar@{->>}[d]^{\sigma} & \mathcal{T}_{g,1} \ar@{->}[r]^-{\tau_1} \ar@{->>}[d] & \extp^3H \ar@{->}[r] \ar@{->>}[d] & 1 \\
 1 \ar@{->}[r] & \mathfrak{B}_2 \ar@{->}[r]^-{i} & H_1(\mathcal{T}_{g,1};\mathbb{Z})_{Sp_{2g}(\mathbb{Z},d)} \ar@{->}[r] & \extp^3H_d \ar@{->}[r] & 1 .}
\end{equation*}
Reassembling these diagrams we get another commutative diagram
\begin{equation}
\label{diag_res_K_par1}
\xymatrix@C=7mm@R=10mm{ \mathcal{K}_{g,1} \ar@{->}[r] \ar@{->>}[d]^{\sigma} & \mathcal{T}_{g,1} \ar@{->}[r] \ar@{->>}[d] & \mathcal{M}_{g,1}[d] \ar@{->>}[d] \\
 \mathfrak{B}_2 \ar@{->}[r]^-{i} & H_1(\mathcal{T}_{g,1};\mathbb{Z})_{Sp_{2g}(\mathbb{Z},d)} \ar@{->}[r]^-{j} & H_1(\mathcal{M}_{g,1}[d];\mathbb{Z}).}
\end{equation}
Let $f\in Hom(\mathcal{M}_{g,1}[d],A)^{\mathcal{AB}_{g,1}}.$ We show that $f$ restricted to $\mathcal{T}_{g,1}$ is zero.

Let $\overline{f}\in Hom(\mathcal{T}_{g,1},A)^{\mathcal{AB}_{g,1}}$ be the restriction of $f$ to $\mathcal{T}_{g,1}.$
By Lemma \eqref{lem_abinv}, we have that $\overline{f}= \varphi^x_g\circ \sigma,$ where $\varphi^x_g$ is zero on all $\mathfrak{B}_3$ except possibly on $\overline{1}.$
Since $A$ is abelian, $\overline{f}$ restricted on $\mathcal{K}_{g,1}$ factors through $j\circ i \circ \sigma.$

On the other hand, by Theorem 7.8. in \cite{Putman}, we know that $(j\circ i)(\overline{1})=0.$ As a consequence, $\varphi^x_g(\overline{1})=0$ and so $f$ restricted to $\mathcal{T}_{g,1}$ is zero. Therefore every element of $Hom(\mathcal{M}_{g,1}[d],A)^{\mathcal{AB}_{g,1}}$ factors through $Sp_{2g}(\mathbb{Z},d),$ i.e. we have an isomorphism
$$Hom(\mathcal{M}_{g,1}[d],A)^{\mathcal{AB}_{g,1}}\cong Hom(Sp_{2g}(\mathbb{Z},d),A)^{GL_g(\mathbb{Z})}.$$

\end{proof}

\begin{prop}
\label{prop_iso_M[d],sp(Z/d)}
Let $A$ be an abelian group and $g\geq 3,$ for every integer $d\geq 2$ we have the following isomorphism
$$Hom(Sp_{2g}(\mathbb{Z},d),A)^{GL_g(\mathbb{Z})}\cong Hom(\mathfrak{sp}_{2g}(\mathbb{Z}/d),A)^{GL_g(\mathbb{Z})}.$$
\end{prop}

\begin{proof}
Denote by $\mathfrak{sp}_{2g}(\mathbb{Z}/d)$ the additive group of $2g\times 2g$ matrices M with entries in $\mathbb{Z}/d$ satisfying $M^T\Omega +\Omega M=0.$
In section 1 in \cite{Lee}, R. Lee and R.H. Szczarba showed that the homomorphism
\begin{align*}
abel: Sp_{2g}(\mathbb{Z},d) & \longrightarrow \mathfrak{sp}_{2g}(\mathbb{Z}/d) \\
Id_{2g} +dA & \longmapsto A \;(\text{mod }d)
\end{align*}
is the abelianization.
Moreover, a direct computation shows that this homomorphism is $Sp_{2g}(\mathbb{Z})$-equivariant.
In Corollary of Lemmas 2.1, 2.2, 2.3 in \cite{Lee}, R. Lee and R.H. Szczarba proved that, for any prime $d,$ there is the following short exact sequence
\begin{equation}
\label{ses_abel_d}
\xymatrix@C=10mm@R=13mm{1 \ar@{->}[r] & Sp_{2g}(\mathbb{Z},d^2) \ar@{->}[r] & Sp_{2g}(\mathbb{Z},d) \ar@{->}[r]^{abel} & \mathfrak{sp}_{2g}(\mathbb{Z}/d) \ar@{->}[r] & 1 .}
\end{equation}
In fact, following their proof, the same short exact sequence holds for any integer $d\geq 2.$
Next, we distinguish the case of $d$ odd and the case of $d$ even.

\textbf{The case $d$ odd.}
In \cite{per}, \cite{sato_abel}, \cite{Putman_abel}, B. Perron , M. Sato  and A. Putman respectively proved independently that
for any $g\geq 3$ and an odd integer $d\geq 3,$
$$[Sp_{2g}(\mathbb{Z},d),Sp_{2g}(\mathbb{Z},d)]=Sp_{2g}(\mathbb{Z},d^2).$$
Then, by the short exact sequence \eqref{ses_abel_d}, since $A$ is an abelian group, every homomorphism of $Hom(Sp_{2g}(\mathbb{Z},d),A)^{GL_g(\mathbb{Z})}$ factors through the homomorphism $abel,$
i.e. we have the following isomorphism
$$Hom(Sp_{2g}(\mathbb{Z},d),A)^{GL_g(\mathbb{Z})}\cong Hom(\mathfrak{sp}_{2g}(\mathbb{Z}/d),A)^{GL_g(\mathbb{Z})},$$
given by the pull-back of the map $abel$ and where $GL_g(\mathbb{Z})$ acts on $\mathfrak{sp}_{2g}(\mathbb{Z}/d)$ by conjugation.

\textbf{The case $d$ even.}
Consider the following normal subgroup of $Sp_{2g}(\mathbb{Z}):$
$$Sp_{2g}(\mathbb{Z},d,2d)=\{A\in Sp_{2g}(\mathbb{Z},d) \mid A=Id_{2g}+dA',\; A'_{g+i,i}\equiv A'_{i,g+i} \equiv 0 \;(\text{mod } 2) \;\forall i   \}.$$
In Proposition 10.1 in \cite{sato_abel}, M. Sato proved that for every even integer $d\geq 2,$
$$Sp_{2g}(\mathbb{Z},d^2,2d^2)=[Sp_{2g}(\mathbb{Z},d),Sp_{2g}(\mathbb{Z},d)].$$
As a consequence, in (10.3) in \cite{sato_abel}, he obtained the following short exact sequence
\begin{equation*}
\xymatrix@C=10mm@R=13mm{1 \ar@{->}[r] & \frac{Sp_{2g}(\mathbb{Z},d^2)}{Sp_{2g}(\mathbb{Z},d^2,2d^2)} \ar@{->}[r] & \frac{Sp_{2g}(\mathbb{Z},d)}{Sp_{2g}(\mathbb{Z},d^2,2d^2)} \ar@{->}[r]^{abel} & \mathfrak{sp}_{2g}(\mathbb{Z}/d) \ar@{->}[r] & 1 .}
\end{equation*}
Since $A$ is abelian, we have that
$$Hom(Sp_{2g}(\mathbb{Z},d),A)^{GL_g(\mathbb{Z})}\cong Hom\left(\frac{Sp_{2g}(\mathbb{Z},d)}{Sp_{2g}(\mathbb{Z},d^2,2d^2)},A\right)^{GL_g(\mathbb{Z})}.$$
We prove that every element $\varphi_g\in Hom\left(\frac{Sp_{2g}(\mathbb{Z},d)}{Sp_{2g}(\mathbb{Z},d^2,2d^2)},A\right)^{GL_g(\mathbb{Z})}$ factors through $\mathfrak{sp}_{2g}(\mathbb{Z}/d),$ i.e. $\varphi_g$ is zero on $\frac{Sp_{2g}(\mathbb{Z},d^2)}{Sp_{2g}(\mathbb{Z},d^2,2d^2)}.$

In section 10 in \cite{sato_abel}, M. Sato proved that
$\frac{Sp_{2g}(\mathbb{Z},d^2)}{Sp_{2g}(\mathbb{Z},d^2,2d^2)}$ is generated by
\begin{equation}
\label{elem_Sp[2,4]}
\{(Id_{2g}+d^2E_{i,g+i}),(Id_{2g}+d^2E_{g+i,i})\}_{i=1}^g.
\end{equation}
Then it is enough to check that $\varphi_g$ is zero on all elements of \eqref{elem_Sp[2,4]}.

Take $f$ the matrix $\left(\begin{smallmatrix}
G & 0 \\
0 & {}^tG^{-1}
\end{smallmatrix}\right)$ with $G=(1,i)\in \mathfrak{S}_g.$ Since $\varphi_g$ is $GL_g(\mathbb{Z})$-invariant, we have that
\begin{equation}
\label{eq_Sp[2]1i}
\begin{aligned}
\varphi_g(Id_{2g}+d^2E_{i,g+i})= & \varphi_g(f(Id_{2g}+d^2E_{i,g+i})f^{-1})= \\
= & \varphi_g(Id_{2g}+d^2GE_{i,g+i}G^t)=\varphi_g(Id_{2g}+d^2E_{1,g+1}).
\end{aligned}
\end{equation}
Analogously, $\varphi_g(Id_{2g}+d^2E_{g+i,i})=\varphi_g(Id_{2g}+d^2E_{g+1,1}).$

Now, take $f$ the matrix $\left(\begin{smallmatrix}
G & 0 \\
0 & {}^tG^{-1}
\end{smallmatrix}\right)$ with $G\in GL_g(\mathbb{Z})$ the matrix with a $1$ at position $(2,1),$ $1$'s at the diagonal and $0$'s at the other positions. Then we have that
\begin{align*}
\varphi_g(Id_{2g}+d^2E_{1,g+1})= & \varphi_g(f(Id_{2g}+d^2E_{1,g+1})f^{-1})= \\
= & \varphi_g(Id_{2g}+d^2GE_{1,g+1}G^t)= \\
= & \varphi_g(Id_{2g}+d^2(E_{1,g+1}+E_{2,g+2}+(E_{1,g+2}+E_{2,g+1})))=\\
= & \varphi_g(Id_{2g}+d^2E_{1,g+1})+\varphi_g(Id_{2g}+d^2E_{2,g+2})+ \\
& +\varphi_g(Id_{2g}+d^2(E_{1,g+2}+E_{2,g+1})))
\end{align*}
But $Id_{2g}+d^2(E_{1,g+2}+E_{2,g+1}) \in Sp_{2g}(\mathbb{Z},d^2,2d^2),$ i.e. $\varphi_g(Id_{2g}+d^2(E_{1,g+2}+E_{2,g+1}))=0.$ Therefore, by relation \eqref{eq_Sp[2]1i}, we get that
$$\varphi_g(Id_{2g}+d^2E_{1,g+1})=2\varphi_g(Id_{2g}+d^2E_{1,g+1})$$
i.e.
$$\varphi_g(Id_{2g}+d^2E_{1,g+1})=0.$$
Analogously $\varphi_g(Id_{2g}+d^2E_{g+1,1})=0.$
Therefore we get the desired result.
\end{proof}

\begin{nota}
Throughout this chapter, we denote by $\pi_g\in Hom(\mathfrak{sp}_{2g}(\mathbb{Z}/d),\mathfrak{gl}_{g}(\mathbb{Z}/d))$ the homomorphism given by
\begin{align*}
\pi_g:\; \mathfrak{sp}_{2g}(\mathbb{Z}/d) & \rightarrow\mathfrak{gl}_{g}(\mathbb{Z}/d) \\
\left(\begin{smallmatrix}
\alpha & \beta \\
\gamma & -\alpha^t
\end{smallmatrix}\right)&\mapsto\alpha.
\end{align*}

\end{nota}

\begin{lema}
\label{lema_elem_sp(Z/d)}
Let $g\geq 3,$ $d\geq 2$ be integers. There is an isomorphism
\begin{align*}
\Theta: A_d & \longrightarrow Hom(\mathfrak{sp}_{2g}(\mathbb{Z}/d),A)^{GL_g(\mathbb{Z})} \\
x & \longmapsto \left(\overline{\varphi}_g^x:X\mapsto tr(\pi_g(X))x\right)
\end{align*}
\end{lema}

\begin{proof}
First of all we show that $\Theta$ is well-defined, i.e. $\overline{\varphi}_g^x\in Hom(\mathfrak{sp}_{2g}(\mathbb{Z}/d),A)^{GL_g(\mathbb{Z})}.$

Let $X'=\left(\begin{smallmatrix}
\alpha' & \beta' \\
\gamma' & -{}^{t}\alpha'
\end{smallmatrix}\right)$ and $X=\left(\begin{smallmatrix}
\alpha & \beta \\
\gamma & -\alpha^t
\end{smallmatrix}\right),$ we have that
$$\overline{\varphi}^x_g(X+X')=tr(\alpha+\alpha')x=tr(\alpha)x+tr(\alpha')x=\overline{\varphi}^x_g(X)+\overline{\varphi}^x_g(X').$$
In addition, it is well known that in general $tr(AB)=tr(BA)$ for any matrices $A,B\in M_{g\times g}(\mathbb{Z}).$
Then,
for every matrix $f\in Sp_{2g}(\mathbb{Z})$ of the form $\left(\begin{smallmatrix}
G & 0 \\
0 & {}^tG^{-1}
\end{smallmatrix}\right)$ with $G\in GL_g(\mathbb{Z}),$ we have that
$$\overline{\varphi}_g^x(fXf^{-1})=tr(G\alpha G^{-1})x=tr(\alpha)x=\overline{\varphi}^x_g(X).$$
Therefore, $\overline{\varphi_g^x}\in Hom(\mathfrak{sp}_{2g}(\mathbb{Z}/d),A)^{GL_g(\mathbb{Z})}.$
Moreover, it is clear that $\Theta$ is injective.

Next we show that $\Theta$ is surjective.
Recall that, by definition of $\mathfrak{sp}_{2g}(\mathbb{Z}/d),$ the elements of
$\mathfrak{sp}_{2g}(\mathbb{Z}/d)$ are matrices
$\left(\begin{smallmatrix}
\alpha & \beta \\
\gamma & -\alpha^t
\end{smallmatrix}\right)$ where $\alpha,\beta,\gamma$ are matrices of size $g\times g$ with $\beta,\gamma$ symmetric matrices.

For $1 \leq i,j\leq 2g,$ let $E_{i,j}$ denote the matrix with a $1$ at position $(i,j)$ and $0$'s elsewhere. Define
$$S=\{A_{i,j}\;|\;1\leq i,j\leq g\}\cup \{B_i,B'_i\;|\;1\leq i\leq g\}\cup \{C_{i,j},C'_{i,j}\;|\;1\leq i<j\leq g\}\subset \mathfrak{sp}_{2g}(\mathbb{Z}/p),$$
where
$$A_{i,j}=E_{i,j}-E_{g+j,g+i},\quad C_{i,j}=E_{g+i,j}+E_{g+j,i},\quad C'_{i,j}=E_{i,g+j}+E_{j,g+i},$$
$$B_i=E_{g+i,i},\quad B'_i=E_{i,g+i}.$$
By the form of the elements of $\mathfrak{sp}_{2g}(\mathbb{Z}/d),$ it is clear that $S$ is a basis for $\mathfrak{sp}_{2g}(\mathbb{Z}/d).$

Moreover, given an element $X=\left(\begin{smallmatrix}
\alpha & \beta \\
\gamma & -\alpha^t
\end{smallmatrix}\right)\in \mathfrak{sp}_{2g}(\mathbb{Z}/d)$ and $f\in Sp_{2g}(\mathbb{Z})$ a matrix of the form $\left(\begin{smallmatrix}
G & 0 \\
0 & {}^tG^{-1}
\end{smallmatrix}\right)$ with $G\in GL_g(\mathbb{Z}),$ we have that

\begin{align*}
fXf^{-1} = & \left(\begin{matrix}
G & 0 \\
0 & {}^{t}G^{-1}
\end{matrix}\right)
\left(
\begin{matrix}
\alpha & \beta \\
\gamma & -\alpha^t
\end{matrix}\right)
\left(
\begin{matrix}
G^{-1} & 0 \\
0 & G^t
\end{matrix}\right)\\
=& \left(
\begin{matrix}
G\alpha & G\beta \\
{}^{t}G^{-1}\gamma & -{}^{t}G^{-1}\alpha^t
\end{matrix}\right)
\left(
\begin{matrix}
G^{-1} & 0 \\
0 & G^t
\end{matrix}\right) \\
=& \left(
\begin{matrix}
G\alpha G^{-1} & G\beta G^t \\
{}^{t}G^{-1}\gamma G^{-1} & -(G\alpha G^{-1})^t
\end{matrix}\right).
\end{align*}

Let $\varphi_g\in Hom(\mathfrak{sp}_{2g}(\mathbb{Z}/d),A)^{GL_g(\mathbb{Z})},$
then we have the following relations:
\begin{itemize}
\item Take $f$ the matrix $\left(\begin{smallmatrix}
G & 0 \\
0 & {}^tG^{-1}
\end{smallmatrix}\right)$ with $G=(1i)(2j)\in \mathfrak{S}_g\subset GL_g(\mathbb{Z})$ for $i\neq j.$
\begin{equation}
\label{eqp:0}
\begin{aligned}
\varphi_g(A_{i,j}) = & \varphi_g(fA_{i,j}f^{-1})=\varphi_g(f(E_{i,j}-E_{g+j,g+i})f^{-1}) \\
 = & \varphi_g(E_{1,2}-E_{g+2,g+1})=\varphi_g(A_{1,2}),
\end{aligned}
\end{equation}
\begin{equation}
\label{eqp:1}
\begin{aligned}
\varphi_g(C_{i,j}) = & \varphi_g(fC_{i,j}f^{-1})=\varphi_g(f(E_{g+i,j}+E_{g+j,i})f^{-1}) \\
 = & \varphi_g(E_{g+1,2}+E_{g+2,1})=\varphi_g(C_{1,2}),
\end{aligned}
\end{equation}
\begin{equation}
\label{eqp:2}
\begin{aligned}
\varphi_g(C'_{i,j}) = & \varphi_g(fC'_{i,j}f^{-1})=\varphi_g(f(E_{i,g+j}+E_{j,g+i})f^{-1}) \\
 = & \varphi_g(E_{1,g+2}+E_{2,g+1})=\varphi_g(C'_{1,2}),
\end{aligned}
\end{equation}
\begin{equation}
\label{eqp:6}
\begin{aligned}
\varphi_g(A_{i,i}) = & \varphi_g(fA_{i,i}f^{-1})=\varphi_g(f(E_{i,i}-E_{g+i,g+i})f^{-1}) \\
 = & \varphi_g(E_{1,1}-E_{g+1,g+1})= \varphi_g(A_{1,1}),
\end{aligned}
\end{equation}
\begin{equation}
\label{eqp:7}
\begin{aligned}
\varphi_g(B_i) =\varphi_g(fB_if^{-1})=\varphi_g(fE_{g+i,i}f^{-1})= \varphi_g(E_{g+1,1})
= \varphi_g(B_1),
\end{aligned}
\end{equation}
\begin{equation}
\label{eqp:8}
\begin{aligned}
\varphi_g(B'_i) =\varphi_g(fB'_if^{-1})=\varphi_g(fE_{i,g+i}f^{-1})= \varphi_g(E_{1,g+1})
= \varphi_g(B'_1).
\end{aligned}
\end{equation}
\item Take $f$ the matrix $\left(\begin{smallmatrix}
G & 0 \\
0 & {}^tG^{-1}
\end{smallmatrix}\right)$ with $G\in GL_g(\mathbb{Z})$ the matrix with $1$'s at positions $(1,2),$ $(2,1)$ and $(i,i)\neq (2,2) \;\forall i$  and $0$'s at the other positions.

Observe that
$$
fB'_1f^{-1} = fE_{1,g+1}f^{-1} = \left(\begin{matrix}
G & 0 \\
0 & {}^{t}G^{-1}
\end{matrix}\right)
\left(
\begin{matrix}
0 & E_{1,1} \\
0 & 0
\end{matrix}\right)
\left(
\begin{matrix}
G^{-1} & 0 \\
0 & G^t
\end{matrix}\right)
=\left(
\begin{matrix}
0 & GE_{1,1}G^t \\
0 & 0
\end{matrix}\right),
$$
where
\begin{align*}
GE_{1,1}G^t = &
\left(\begin{array}{c|c}
\begin{smallmatrix}
1 & 1 \\
1 & 0
\end{smallmatrix}
 & 0 \\
 \hline
0 & Id
\end{array}\right)\left(\begin{array}{c|c}
\begin{smallmatrix}
1 & 0 \\
0 & 0
\end{smallmatrix}
 & 0 \\
 \hline
0 & 0
\end{array}\right)\left(\begin{array}{c|c}
\begin{smallmatrix}
1 & 1 \\
1 & 0
\end{smallmatrix}
 & 0 \\
 \hline
0 & Id
\end{array}\right) \\
= &
\left(\begin{array}{c|c}
\begin{smallmatrix}
1 & 0 \\
1 & 0
\end{smallmatrix}
 & 0 \\
 \hline
0 & 0
\end{array}\right)\left(\begin{array}{c|c}
\begin{smallmatrix}
1 & 1 \\
1 & 0
\end{smallmatrix}
 & 0 \\
 \hline
0 & Id
\end{array}\right) \\
= &
\left(\begin{array}{c|c}
\begin{smallmatrix}
1 & 1 \\
1 & 1
\end{smallmatrix}
 & 0 \\
 \hline
0 & 0
\end{array}\right)=E_{1,1}+E_{2,2}+(E_{1,2}+E_{2,1}).
\end{align*}
Thus $fB'_1f^{-1} = E_{1,g+1}+E_{2,g+2}+(E_{1,g+2}+E_{2,g+1})=B'_1+B'_2+C'_{1,2}.$
Therefore,
$$ \varphi_g(B'_1)= \varphi_g(fB'_1f^{-1})=\varphi_g(B'_1+B'_2+C'_{1,2})=
\varphi_g(B'_1)+\varphi_g(B'_2)+\varphi_g(C'_{1,2}),\quad \text{i.e.} $$
\begin{equation}
\label{eqp:9} \varphi_g(B'_2)=  -\varphi_g(C'_{1,2}).
\end{equation}
Similarly, we have that
$$
f^{-1}B_1f = f^{-1}E_{g+1,1}f = \left(\begin{matrix}
G^{-1} & 0 \\
0 & G^t
\end{matrix}\right)
\left(
\begin{matrix}
0 & 0 \\
E_{1,1} & 0
\end{matrix}\right)
\left(
\begin{matrix}
G & 0 \\
0 & {}^tG^{-1}
\end{matrix}\right)
=\left(
\begin{matrix}
0 & 0 \\
G^tE_{1,1}G & 0
\end{matrix}\right),
$$
where
\begin{align*}
G^tE_{1,1}G = &
\left(\begin{array}{c|c}
\begin{smallmatrix}
1 & 1 \\
1 & 0
\end{smallmatrix}
 & 0 \\
 \hline
0 & Id
\end{array}\right)\left(\begin{array}{c|c}
\begin{smallmatrix}
1 & 0 \\
0 & 0
\end{smallmatrix}
 & 0 \\
 \hline
0 & 0
\end{array}\right)\left(\begin{array}{c|c}
\begin{smallmatrix}
1 & 1 \\
1 & 0
\end{smallmatrix}
 & 0 \\
 \hline
0 & Id
\end{array}\right) \\
= &
\left(\begin{array}{c|c}
\begin{smallmatrix}
1 & 0 \\
1 & 0
\end{smallmatrix}
 & 0 \\
 \hline
0 & 0
\end{array}\right)\left(\begin{array}{c|c}
\begin{smallmatrix}
1 & 1 \\
1 & 0
\end{smallmatrix}
 & 0 \\
 \hline
0 & Id
\end{array}\right) \\
= &
\left(\begin{array}{c|c}
\begin{smallmatrix}
1 & 1 \\
1 & 1
\end{smallmatrix}
 & 0 \\
 \hline
0 & 0
\end{array}\right)=E_{1,1}+E_{2,2}+(E_{1,2}+E_{2,1}).
\end{align*}
Thus $f^{-1}B_1f = E_{g+1,1}+E_{g+2,2}+(E_{g+1,2}+E_{g+2,1})=B_1+B_2+C_{1,2}.$
Therefore,
$$\varphi_g(B_1)= \varphi_g(f^{-1}B_1f)=\varphi_g(B_1+B_2+C_{1,2})=
\varphi_g(B_1)+\varphi_g(B_2)+\varphi_g(C_{1,2}),\quad \text{i.e.}$$
\begin{equation}
\label{eqp:10}
\varphi_g(B_2)= -\varphi_g(C_{1,2}).
\end{equation}

\item Take $f$ the matrix $\left(\begin{smallmatrix}
G & 0 \\
0 & {}^tG^{-1}
\end{smallmatrix}\right)$ with $G\in GL_g(\mathbb{Z})$ the matrix with  $1$'s at the diagonal and positions $(2,1),(3,1),$ and $0$'s elsewhere.

Observe that
$$
fB'_1f^{-1} = fE_{1,g+1}f^{-1} = \left(\begin{matrix}
G & 0 \\
0 & {}^tG^{-1}
\end{matrix}\right)
\left(
\begin{matrix}
0 & E_{1,1} \\
0 & 0
\end{matrix}\right)
\left(
\begin{matrix}
G^{-1} & 0 \\
0 & G^t
\end{matrix}\right)
=\left(
\begin{matrix}
0 & GE_{1,1}G^t \\
0 & 0
\end{matrix}\right),
$$
where
\begin{align*}
GE_{1,1}G^t = &
\left(\begin{array}{c|c}
\begin{smallmatrix}
1 & 0 & 0 \\
1 & 1 & 0 \\
1 & 0 & 1
\end{smallmatrix}
 & 0 \\
 \hline
0 & Id
\end{array}\right)\left(\begin{array}{c|c}
\begin{smallmatrix}
1 & 0 & 0 \\
0 & 0 & 0 \\
0 & 0 & 0
\end{smallmatrix}
 & 0 \\
 \hline
0 & 0
\end{array}\right)\left(\begin{array}{c|c}
\begin{smallmatrix}
1 & 1 & 1 \\
0 & 1 & 0 \\
0 & 0 & 1
\end{smallmatrix}
 & 0 \\
 \hline
0 & Id
\end{array}\right)= \\
= &
\left(\begin{array}{c|c}
\begin{smallmatrix}
1 & 0 & 0 \\
1 & 0 & 0 \\
1 & 0 & 0
\end{smallmatrix}
 & 0 \\
 \hline
0 & 0
\end{array}\right)\left(\begin{array}{c|c}
\begin{smallmatrix}
1 & 1 & 1 \\
0 & 1 & 0 \\
0 & 0 & 1
\end{smallmatrix}
 & 0 \\
 \hline
0 & Id
\end{array}\right)
=
\left(\begin{array}{c|c}
\begin{smallmatrix}
1 & 1 & 1 \\
1 & 1 & 1 \\
1 & 1 & 1
\end{smallmatrix}
 & 0 \\
 \hline
0 & 0
\end{array}\right) =\\
= & E_{1,1}+E_{2,2}+E_{3,3}+(E_{1,2}+E_{2,1})+(E_{2,3}+E_{3,2})+(E_{1,3}+E_{3,1}).
\end{align*}
Thus,
$$fB_1'f^{-1}=B_1'+B_2'+B_3'+C_{1,2}'+C_{2,3}'+C_{1,3}'.$$
Therefore,
$$\varphi_g(B_1')=\varphi_g(fB_1'f^{-1})=\varphi_g(B_1')+\varphi_g(B_2')+\varphi_g(B_3')+\varphi_g(C_{1,2}')+\varphi_g(C_{2,3}')+\varphi_g(C_{1,3}'),$$
and by relations \eqref{eqp:8}, \eqref{eqp:9}, \eqref{eqp:2}, we get that
$$\varphi_g(B'_1)=3\varphi_g(B'_1)-3\varphi_g(B'_1)=0.$$
Moreover by relations \eqref{eqp:8}, \eqref{eqp:9}, \eqref{eqp:2}, we also get that
$$0=\varphi_g(B_1')=\varphi_g(B_i')=\varphi_g(C_{i,j}').$$
Similarly, take $f$ the matrix $\left(\begin{smallmatrix}
G & 0 \\
0 & {}^tG^{-1}
\end{smallmatrix}\right)$ with $G^t\in GL_g(\mathbb{Z})$ the matrix with  $1$'s at the diagonal and positions $(2,1),(3,1),$ and $0$'s at the other positions.
Then we have that
$$
f^{-1}B_1f = f^{-1}E_{g+1,1}f = \left(\begin{matrix}
G^{-1} & 0 \\
0 & G^t
\end{matrix}\right)
\left(
\begin{matrix}
0 & 0 \\
E_{1,1} & 0
\end{matrix}\right)
\left(
\begin{matrix}
G & 0 \\
0 & {}^{t}G^{-1}
\end{matrix}\right)
=\left(
\begin{matrix}
0 & 0 \\
G^tE_{1,1}G & 0
\end{matrix}\right),
$$
where
\begin{align*}
G^tE_{1,1}G = & \left(\begin{array}{c|c}
\begin{smallmatrix}
1 & 0 & 0 \\
1 & 1 & 0 \\
1 & 0 & 1
\end{smallmatrix}
 & 0 \\
 \hline
0 & Id
\end{array}\right)\left(\begin{array}{c|c}
\begin{smallmatrix}
1 & 0 & 0 \\
0 & 0 & 0 \\
0 & 0 & 0
\end{smallmatrix}
 & 0 \\
 \hline
0 & 0
\end{array}\right)\left(\begin{array}{c|c}
\begin{smallmatrix}
1 & 1 & 1 \\
0 & 1 & 0 \\
0 & 0 & 1
\end{smallmatrix}
 & 0 \\
 \hline
0 & Id
\end{array}\right)= \\
= &  E_{1,1}+E_{2,2}+E_{3,3}+(E_{1,2}+E_{2,1})+(E_{2,3}+E_{3,2})+(E_{1,3}+E_{3,1}).
\end{align*}
Thus,
$$fB_1f^{-1}=B_1+B_2+B_3+C_{1,2}+C_{2,3}+C_{1,3}.$$
Therefore,
$$\varphi_g(B_1)=\varphi_g(fB_1f^{-1})=\varphi_g(B_1)+\varphi_g(B_2)+\varphi_g(B_3)+\varphi_g(C_{1,2})+\varphi_g(C_{2,3})+\varphi_g(C_{1,3}),$$
and by relations \eqref{eqp:7}, \eqref{eqp:10}, \eqref{eqp:1}, we have that
$$\varphi_g(B_1)=3\varphi_g(B_1)-3\varphi_g(B_1)=0.$$
Moreover by relations \eqref{eqp:7}, \eqref{eqp:10}, \eqref{eqp:1}, we have that
\begin{equation}
\label{eqp:10'}
0=\varphi_g(B_1)=\varphi_g(B_i)=\varphi_g(C_{i,j}).
\end{equation}

\item Take $f$ the matrix $\left(\begin{smallmatrix}
G & 0 \\
0 & {}^tG^{-1}
\end{smallmatrix}\right)$ with $G\in GL_g(\mathbb{Z})$ the matrix with a $1$'s at the diagonal and position $(1,2),$ and $0$'s at the other positions.

Observe that
\begin{align*}
fA_{2,1}f^{-1} = & f(E_{2,1}-E_{g+1,g+2})f^{-1} = \left(\begin{matrix}
G & 0 \\
0 & {}^{t}G^{-1}
\end{matrix}\right)
\left(
\begin{matrix}
E_{2,1} & 0 \\
0 & -E_{1,2}
\end{matrix}\right)
\left(
\begin{matrix}
G^{-1} & 0 \\
0 & G^t
\end{matrix}\right) \\
= & \left(
\begin{matrix}
GE_{2,1}G^{-1} & 0 \\
0 & -(GE_{2,1}G^{-1})^t
\end{matrix}\right),
\end{align*}
where
\begin{align*}
GE_{2,1}G^{-1} = &
\left(\begin{array}{c|c}
\begin{smallmatrix}
1 & 1 \\
0 & 1
\end{smallmatrix}
 & 0 \\
 \hline
0 & Id
\end{array}\right)\left(\begin{array}{c|c}
\begin{smallmatrix}
0 & 0 \\
1 & 0
\end{smallmatrix}
 & 0 \\
 \hline
0 & 0
\end{array}\right)\left(\begin{array}{c|c}
\begin{smallmatrix}
1 & -1 \\
0 & 1
\end{smallmatrix}
 & 0 \\
 \hline
0 & Id
\end{array}\right) \\
= &
\left(\begin{array}{c|c}
\begin{smallmatrix}
1 & 0 \\
1 & 0
\end{smallmatrix}
 & 0 \\
 \hline
0 & 0
\end{array}\right)\left(\begin{array}{c|c}
\begin{smallmatrix}
1 & -1 \\
0 & 1
\end{smallmatrix}
 & 0 \\
 \hline
0 & Id
\end{array}\right) \\
= &
\left(\begin{array}{c|c}
\begin{smallmatrix}
1 & -1 \\
1 & -1
\end{smallmatrix}
 & 0 \\
 \hline
0 & 0
\end{array}\right)=E_{1,1}-E_{2,2}-E_{1,2}+E_{2,1}.
\end{align*}
Thus
\begin{align*}
fA_{2,1}f^{-1}= & E_{1,1}-E_{2,2}-E_{1,2}+E_{2,1}-E_{g+1,g+1}+E_{g+2,g+2}+E_{g+2,g+1}-E_{g+1,g+2} \\
= & A_{1,1}-A_{2,2}+A_{1,2}-A_{2,1}.
\end{align*}
Therefore,
\begin{align*}
\varphi_g(A_{2,1})= & \varphi_g(fA_{2,1}f^{-1})=\varphi_g(A_{1,1}-A_{2,2}+A_{1,2}-A_{2,1})\\
= & \varphi_g(A_{1,1})-\varphi_g(A_{2,2})+\varphi_g(A_{1,2})-\varphi_g(A_{2,1}),
\end{align*}
and by relations \eqref{eqp:0}, \eqref{eqp:6} we obtain that
\begin{equation}
\label{eqp:11}
\varphi_g(A_{2,1})=0,
\end{equation}
and by relation \eqref{eqp:6} we get that $\varphi_g(A_{i,j})=0$ for all $1\leq i<j\leq g.$
\end{itemize}
Summarizing, the unique elements of the basis $S,$ whose image by $\varphi_g$ are not necessarily zero, are the elements $\{A_{i,i}\;|\;1\leq i\leq g\}.$ By the relation \eqref{eqp:0} we have that $\varphi_g(A_{i,i})=\varphi_g(A_{1,1}).$
Thus, any homomorphism $\varphi_g\in Hom(\mathfrak{sp}_{2g}(\mathbb{Z}/d),A)^{GL_g(\mathbb{Z})}$ only depends on the image of $A_{1,1}.$
In particular, for all $X=\left(\begin{smallmatrix}
\alpha & \beta \\
\gamma & -\alpha^t
\end{smallmatrix}\right)\in \mathfrak{sp}_{2g}(\mathbb{Z}/d),$ we have that
$$\varphi_g(X)=tr(\alpha)\varphi_g(A_{1,1}),$$
where $tr(\alpha)$ denotes the trace of the matrix $\alpha.$

Furthermore, since every element of $\mathfrak{sp}_{2g}(\mathbb{Z}/d)$ is a $d$-torsion element and $\varphi_g$ is a homomorphism, then $\varphi_g(A_{1,1})$ is also a $d$-torsion element. Therefore, $\varphi_g(A_{1,1})$ may be any element of $A_d.$
\end{proof}

As a direct consequence of Propositions \eqref{prop_iso_M[d],Sp[d]}, \eqref{prop_iso_M[d],sp(Z/d)} and Lemma \eqref{lema_elem_sp(Z/d)} we have the following result:
\begin{lema}
\label{lem_abinv_modp}
For $g\geq 3,$ $d\geq 3$ an odd integer and for $g\geq 5,$ $d\geq 2$ an even integer such that $4\nmid d,$ there is an isomorphism
\begin{align*}
A_d & \longrightarrow Hom(\mathcal{M}_{g,1}[d],A)^{\mathcal{AB}_{g,1}} \\
x & \longmapsto \varphi^x_g=\overline{\varphi}_g^x \circ abel \circ \Psi,
\end{align*}
where $\overline{\varphi}^x_g$ is defined in Lemma \eqref{lema_elem_sp(Z/d)}.
\end{lema}

Next, we show that every element of $Hom (\mathcal{M}_{g,1}[d],A)^{\mathcal{AB}_{g,1}}$ is zero on $\mathcal{A}_{g,1}[d],$
$\mathcal{B}_{g,1}[d].$
We first need to compute $\Psi(\mathcal{A}_{g,1}[d]),$ $\Psi(\mathcal{B}_{g,1}[d]).$ We also compute $\Psi(\mathcal{AB}_{g,1}[d]),$ for sake of completeness.

\begin{lema}
\label{lem_B[p]_2}
There are short exact sequences of groups:
\begin{align*}
& \xymatrix@C=10mm@R=13mm{1 \ar@{->}[r] & \mathcal{TB}_{g,1} \ar@{->}[r] & \mathcal{B}_{g,1}[d] \ar@{->}[r]^-{\phi^B\; \circ \;\Psi}  & SL_g(\mathbb{Z},d)\ltimes S_g(d\mathbb{Z}) \ar@{->}[r] & 1,} \\
& \xymatrix@C=10mm@R=13mm{1 \ar@{->}[r] & \mathcal{TA}_{g,1} \ar@{->}[r] & \mathcal{A}_{g,1}[d] \ar@{->}[r]^-{\phi^A\; \circ \;\Psi}  & SL_g(\mathbb{Z},d)\ltimes S_g(d\mathbb{Z}) \ar@{->}[r] & 1 ,} \\
& \xymatrix@C=10mm@R=13mm{1 \ar@{->}[r] & \mathcal{TAB}_{g,1} \ar@{->}[r] & \mathcal{AB}_{g,1}[d] \ar@{->}[r]^-{\phi^{AB}\; \circ \;\Psi} & SL_g(\mathbb{Z},d) \ar@{->}[r] & 1 .}
\end{align*}

\end{lema}

\begin{proof}
We only prove the first short exact sequence (the proof for the other two short exact sequences is analogous).
Observe that we have the following commutative diagram
$$
\xymatrix@C=7mm@R=10mm{   & & 1 \ar@{->}[d] & 1 \ar@{->}[d] &\\
 1 \ar@{->}[r] & \mathcal{TB}_{g,1} \ar@{->}[r] \ar@{=}[d] & \mathcal{B}_{g,1}[d] \ar@{->}[r]^-{\phi^B\; \circ \;\Psi} \ar@{->}[d] & SL_g(\mathbb{Z},d)\ltimes S_g(d\mathbb{Z}) \ar@{->}[r] \ar@{->}[d]^-{i} & 1\\
1 \ar@{->}[r] & \mathcal{TB}_{g,1} \ar@{->}[r] & \mathcal{B}_{g,1} \ar@{->}[r]^-{\phi^B \;\circ\;\Psi} \ar@{->}[d]^-{\phi_d^B\;\circ\;\Psi_d} & GL_g(\mathbb{Z})\ltimes S_g(\mathbb{Z}) \ar@{->}[r] \ar@{->}[d]^-{r'_d\times r''_d} & 1\\
 &  & GL_g(\mathbb{Z}/d)\ltimes S_g(\mathbb{Z}/d) \ar@{=}[r] & GL_g(\mathbb{Z}/d)\ltimes S_g(\mathbb{Z}/d) &  ,}
$$
where $i$ denotes the natural inclusion.
We prove that the 1st row is a short exact sequence.

We first show that the kernel of
$\phi^B\; \circ \;\Psi$ restricted to $\mathcal{B}_{g,1}[d]$ is $\mathcal{TB}_{g,1}.$

Since $\mathcal{TB}_{g,1}\subset Ker(\Psi),$  we have that
$$\mathcal{TB}_{g,1}\subset Ker((\phi^B\; \circ \;\Psi) :\mathcal{B}_{g,1}[d] \rightarrow SL_g(\mathbb{Z},d)\ltimes S_g(d\mathbb{Z}))\subset B_{g,1}\cap \mathcal{T}_{g,1}=\mathcal{TB}_{g,1}.$$

Next, we prove that
$(\phi\; \circ \;\Psi) :\mathcal{B}_{g,1}[d] \rightarrow SL_g(\mathbb{Z},d)\ltimes S_g(d\mathbb{Z})$ is surjective.
Let $x\in SL_g(\mathbb{Z},d)\ltimes S_g(d\mathbb{Z}),$ then $i(x)\in GL_g(\mathbb{Z})\ltimes S_g(\mathbb{Z}).$ By Lemma \eqref{lem_B} we know that $\phi\; \circ \;\Psi :\mathcal{B}_{g,1}\rightarrow GL_g(\mathbb{Z})\ltimes S_g(\mathbb{Z})$ is surjective and as a consequence there exists an element $\widetilde{x}\in \mathcal{B}_{g,1}$ such that $(\phi\; \circ \;\Psi)(\widetilde{x})=i(x).$
By the above commutative diagram and Lemma \eqref{lem_Sp[p]_3}, we know that the 3th column is exact and we have the following equality:
$$(\phi_d\;\circ\;\Psi_d)(\widetilde{x})=(r'_d\times r''_d)(i(x))=0.$$
Thus $(\phi_d\;\circ\;\Psi_d)(\widetilde{x})=0.$ Since by Lemma \eqref{lem_B[p]} the 2nd column is exact then we get that $\widetilde{x}\in \mathcal{B}_{g,1}[d].$
Therefore by the above commutative diagram we know that $i((\phi\; \circ \;\Psi)(\tilde{x}))=i(x)$ and since $i$ is injective we get that $(\phi\; \circ \;\Psi)(\tilde{x})=x.$
\end{proof}

Now we are ready to prove the following result:

\begin{lema}
\label{lem_B[p],A[p]}
Let $A$ be an abelian group. Fix an integer $d\geq 2,$ then for any $g\geq 3$ all $\mathcal{AB}_{g,1}$-invariant homomorphisms
$$\varphi_g: \mathcal{B}_{g,1}[d]\rightarrow A \quad \text{and} \quad \varphi_g:\mathcal{A}_{g,1}[d]\rightarrow A,$$
have to be zero.
\end{lema}

\begin{proof}
We only prove the result for $\mathcal{B}_{g,1}[d],$ since the other case is similar.

Let $\varphi_g: \mathcal{B}_{g,1}[d]\rightarrow A$ be a $\mathcal{AB}_{g,1}$-invariant homomorphism.
By Lemma \eqref{lem_TB,TA} we know that the restriction of $\varphi_g$ to $\mathcal{TB}_{g,1}$ is zero. Then, by the short exact sequence in Lemma \eqref{lem_B[p]_2}, $\varphi_g$
factors through $SL_g(\mathbb{Z},d) \ltimes_B S_g(d\mathbb{Z}).$ Thus we have an isomorphism
$$Hom(\mathcal{B}_{g,1}[d],A)^{\mathcal{AB}_{g,1}}\cong Hom(SL_g(\mathbb{Z},d) \ltimes_B S_g(d\mathbb{Z}), A)^{GL_g(\mathbb{Z})}.$$
Observe that we have the short exact sequence
$$
\xymatrix@C=7mm@R=13mm{1 \ar@{->}[r] & S_g(d\mathbb{Z}) \ar@{->}[r] & SL_{g}(\mathbb{Z},d)\ltimes_B S_g(d\mathbb{Z}) \ar@{->}[r] & SL_g(\mathbb{Z},d) \ar@{->}[r] & 1 .}
$$
By analogous relations to the relations \eqref{eqp:1}, \eqref{eqp:7},
\eqref{eqp:10}, \eqref{eqp:10'} in the proof of Lemma \eqref{lem_abinv_modp}, we have that
$Hom(S_g(d\mathbb{Z}),A)^{GL_g(\mathbb{Z})}=0.$
Thus we have an isomorphism
$$Hom(SL_g(\mathbb{Z},d) \ltimes_B S_g(d\mathbb{Z}),A)^{GL_g(\mathbb{Z})}\cong Hom(SL_g(\mathbb{Z},d),A)^{GL_g(\mathbb{Z})}.$$
In \cite{Lee}, Lee and Szczarba showed that for $g\geq 3$ and any prime $p,$ $H_1(SL_g(\mathbb{Z},p))\cong \mathfrak{sl}_g(\mathbb{Z}/p),$ as modules over $SL_g(\mathbb{Z}/p).$ Actually, following the same proof the same result holds for any integer $d.$
Then, since $A$ is an abelian group, we have that every homomorphism $Hom(SL_g(\mathbb{Z},d),A)^{GL_g(\mathbb{Z})},$ factors through $\mathfrak{sl}_g(\mathbb{Z}/d),$ i.e. we have an isomorphism
$$Hom(SL_g(\mathbb{Z},d),A)^{GL_g(\mathbb{Z})}\cong Hom(\mathfrak{sl}_g(\mathbb{Z}/d),A)^{GL_g(\mathbb{Z})}.$$
Finally, by relations \eqref{eqp:0}, \eqref{eqp:11}, \eqref{eqp:6} in proof of Lemma \eqref{lem_abinv_modp}, we get that every element of
$Hom(\mathfrak{sl}_g(\mathbb{Z}/d),A)^{GL_g(\mathbb{Z})}$ have to be zero,
obtaining the desired result.
\end{proof}

\begin{lema}
\label{lem_stab_modp}
The homomorphisms $\varphi_g^x$ defined in Lemma \eqref{lem_abinv_modp} are compatible with the stabilization map.
\end{lema}

\begin{proof}
Recall that the homomorphisms $\varphi_g^x,$ for $T\in \mathcal{M}_{g,1}[d]$ are defined as follows
$$\varphi^x_g(T)=tr(\pi_g\circ abel \circ \Psi(T))x.$$
Let $abel (\Psi(T))=\left(\begin{smallmatrix}
\alpha & \beta \\
\gamma & -\alpha^t
\end{smallmatrix}\right)\in \mathfrak{sp}_{2g}(\mathbb{Z}/d).$ Then if we take $T$ as an element of $\mathcal{M}_{g+1,1}[d],$ we have that
$$abel (\Psi(T))=\left(\begin{smallmatrix}
\alpha' & \beta' \\
\gamma' & -(\alpha')^t
\end{smallmatrix}\right)\in \mathfrak{sp}_{2(g+1)}(\mathbb{Z}/d),$$
where $\alpha'=\left(\begin{smallmatrix}
\alpha & 0 \\
0 & 0
\end{smallmatrix}\right).$

Hence, $\varphi^x_g(T)=\varphi^x_{g+1}(T)$ for every $T\in \mathcal{M}_{g,1}[d].$
\end{proof}

\begin{rem}
By Lemmas \eqref{lem_B[p],A[p]}, \eqref{lem_stab_modp}, the $\mathcal{AB}_{g,1}$-invariant homomorphisms $\{\varphi_g^x\}_g,$ in Lemma \eqref{lem_abinv_modp}, are compatible with the stabilization map and zero on $\mathcal{A}_{g,1}[d],$ $\mathcal{B}_{g,1}[d].$ Then, by bijection \eqref{bij_M[p]}, we get that, for each $x\in A_d,$ the family of homomorphisms $\{\varphi^x_g\}_g$ of Lemma \eqref{lem_abinv_modp} reassemble into an invariant of $\mathcal{S}^3[d].$
\end{rem}

A natural question that arise from this section is what $3$-manifolds distinguish the invariant $\varphi^x.$ Next we answer this question for lens spaces.
Given a lens space $L(p,q),$ since Lens spaces are $\mathbb{Q}$-homology $3$-spheres, by Theorem \eqref{teo_rat_homology_gen}, we know that there exists an integer $d\geq 2$ for which $L(p,q)\in \mathcal{S}^3[d].$ Thus, given an integer $d\geq 2$, we first need to determine which lens spaces are in $\mathcal{S}^3[d].$

\begin{prop}
A lens space $L(p,q)$ is in $\mathcal{S}^3[d]$ if and only if $p\equiv  \pm 1 (\text{mod } d).$
\end{prop}

\begin{proof}
Given a lens space $L(p,q)$, it is well known that its homology groups are:
$$H_k(L(p,q);\mathbb{Z})=\left\{\begin{array}{rl}
\mathbb{Z}, & \text{for }k=0,3, \\
\mathbb{Z}/p, & \text{for } k=1, \\
0, & \text{otherwise.} 
\end{array} \right.$$
Then $|H_1(L(p,q);\mathbb{Z})|=p.$ By Theorem \eqref{teo_rat_homology_gen}, $L(p,q)\in \mathcal{S}^3[d]$ if and only if $p\equiv \pm 1 (\text{mod }d).$
\end{proof}

As a consequence, all lens spaces in $\mathcal{S}^3[d]$ are of the form $L(\pm 1+dk,q)$ with $k,q\in \mathbb{Z}.$ Next we compute the values of the invariant $\varphi^x$ on these lens spaces. By the classification of lens spaces we know that $L(1+dk,q)\cong L(-1-dk,q).$ Then, it is enough to compute the value of the invariant $\varphi^x$ on the lens spaces $L(-1+dk,q).$

Given a lens space $L(-1+dk,q)\in \mathcal{S}^3[d],$ by definition of lens spaces, such space has a Heegaard splitting of genus $1.$ Following the proof of Theorem \eqref{teo_rat_homology_gen}, we have that there exists an element $f_d\in \mathcal{M}_{1,1}[d]$ such that
$L(-1+dk,q)\cong\mathcal{H}_1\cup_{\iota f_d}-\mathcal{H}_1.$
Then, the lens space $L(-1+dk,q)$ is homeomorphic to $L(-1+dk,-dl)$ for suitable integers $k,l,m,r,$ with
$$\Psi(f_d)= \left(\begin{matrix}
1+dr & dm \\
dl & 1-dk \end{matrix}\right)\in Sp_{2}(\mathbb{Z}).$$
Since $Sp_2(\mathbb{Z})=SL_2(\mathbb{Z}),$ $det(\Psi(f_d))=1$ and reducing modulo $d^2$ we have that
$$1=det(\Psi(f_d))=1+d(r-k)+d^2(-rk+lm)\equiv 1+d(r-k) \;(\text{mod } d^2).$$
Thus, $k\equiv r \;(\text{mod } d).$
Then,
$$\varphi^x(L(-1+pk,-pl))=\varphi^x_1(f_d)=tr(\pi_1\circ abel \circ \Psi(f_d))x=rx\equiv kx.$$
Therefore we get the following result:
\begin{prop}
The homomorphisms $\{\varphi_g^x\}_g\in Hom(\mathcal{M}_{g,1}[d],A),$ defined in Lemma \eqref{lem_abinv_modp}, take the value $kx$ on lens spaces $L(-1+dk,q)$ and $-kx$ on lens spaces $L(1+dk,q).$
\end{prop}

\section{From trivial cocycles to invariants}

Conversely, what are the conditions for a family of trivial $2$-cocycles $C_g$ on $\mathcal{M}_{g,1}[d]$ satisfying properties (1)-(3) to actually provide an invariant?
Here we follow the same arguments used in Section \ref{section_cocycles to invariants} taking $\mathcal{M}_{g,1}[d]$ instead of $\mathcal{T}_{g,1}.$

Firstly we need to check the existence of an $\mathcal{AB}_{g,1}$-invariant trivialization of each $C_g.$
Denote by $\mathcal{Q}_{C_g}$ the set of all trivializations of the cocycle $C_g:$
$$\mathcal{Q}_{C_g}=\{q:\mathcal{M}_{g,1}[d]\rightarrow A\mid q(\phi)+q(\psi)-q(\phi\psi)=C_g(\phi,\psi)\}.$$
The group $\mathcal{AB}_{g,1}$ acts on $\mathcal{Q}_g$ via its conjugation action on $\mathcal{M}_{g,1}[d].$ This action confers the set $\mathcal{Q}_{C_g}$ the structure of an affine set over the abelian group $Hom(\mathcal{M}_{g,1}[d],A).$ On the other hand, choosing an arbitrary element $q\in \mathcal{Q}_{C_g}$ the map defined as follows
\begin{align*}
\rho_q:\mathcal{AB}_{g,1} & \longrightarrow Hom(\mathcal{M}_{g,1}[d],A) \\ \phi & \longmapsto \phi \cdot q-q,
\end{align*}
induces a well-defined cohomology class
$\rho(C_g)\in H^1(\mathcal{AB}_{g,1};Hom(\mathcal{M}_{g,1}[d],A)),$
called the torsor of the cocycle $C_g,$ and we have the following result:

\begin{prop}
\label{prop_torsor}
The natural action of $\mathcal{AB}_{g,1}$ on $\mathcal{Q}_{C_g}$ admits a fixed point if and only if the associated torsor $\rho(C_g)$ is trivial.
\end{prop}

Suppose that for every $g\geq 3$ there is a fixed point $q_g$ of $\mathcal{Q}_{C_g}$ for the action of $\mathcal{AB}_{g,1}$ on $\mathcal{Q}_{C_g}.$
Since every pair of $\mathcal{AB}_{g,1}$-invariant trivialization differ by a $\mathcal{AB}_{g,1}$-invariant homomorphism, by Lemma \eqref{lem_abinv_modp}, for every $g\geq 3$ the fixed points are $q_g+\varphi_g^x$ with $x\in A_d.$

By Lemma \eqref{lem_stab_modp}, all elements of $Hom(\mathcal{M}_{g,1}[d],A)^{\mathcal{AB}_{g,1}}$ are compatible with the stabilization map. Then, given two different fixed points $q_g,$ $q'_g$ of $\mathcal{Q}_{C_g}$ for the action of $\mathcal{AB}_{g,1},$ we have that
$${q_g}_{\mid\mathcal{M}_{g-1,1}[d]}-{q'_g}_{\mid\mathcal{M}_{g-1,1}[d]}=(q_g-q'_g)_{\mid\mathcal{M}_{g-1,1}[d]} ={\varphi_g^x}_{\mid\mathcal{M}_{g-1,1}[d]}=\varphi_{g-1}^x.$$
Therefore the restriction of the trivializations of $\mathcal{Q}_{C_g}$ to $\mathcal{M}_{g-1,1}[d],$ give us a bijection between the fixed points of $\mathcal{Q}_{C_g}$ for the action of $\mathcal{AB}_{g,1}$ and the fixed points of $\mathcal{Q}_{C_{g-1}}$ for the action of $\mathcal{AB}_{g-1,1}.$

Therefore, given an $\mathcal{AB}_{g,1}$-invariant trivialization $q_g,$ for each $x\in A_d$ we get a well-defined map
$$
q+\varphi^x= \lim_{g\to \infty}q_g+\varphi^x_g: \lim_{g\to \infty}\mathcal{M}_{g,1}[d]\longrightarrow A.
$$
These are the only candidates to be $A$-valued invariants of $\mathbb{Z}/d$-homology spheres with associated family of $2$-cocycles $\{C_g\}_g.$ For these maps to be invariants, since they are already $\mathcal{AB}_{g,1}$-invariant, we only have to prove that they are constant on the double cosets $\mathcal{A}_{g,1}[d]\backslash \mathcal{M}_{g,1}[d]/\mathcal{B}_{g,1}[d].$
From property (3) of our cocycle we have that $\forall \phi\in \mathcal{M}_{g,1}[d],$
$\forall \psi_a\in \mathcal{A}_{g,1}[d]$ and $\forall \psi_b\in \mathcal{B}_{g,1}[d],$
\begin{equation}
\label{eq_A[p],B[p]_constant}
\begin{aligned}
(q_g+\varphi^x_g)(\phi)-(q_g+\varphi^x_g)(\phi \psi_b)= & -(q_g+\varphi^x_g)(\psi_b) , \\
(q_g+\varphi^x_g)(\phi)-(q_g+\varphi^x_g)(\psi_a\phi )= & -(q_g+\varphi^x_g)(\psi_a).
\end{aligned}
\end{equation}
Thus in particular, taking $\phi=\psi_a,\psi_b$ in above equations, we have that $q_g+\varphi^x_g$ with $x\in A_d,$ are homomorphisms on $\mathcal{A}_{g,1}[d],$ $\mathcal{B}_{g,1}[d].$ Then, by Lemma \eqref{lem_B[p],A[p]} we have that $q_g+\varphi^x_g$ are trivial on $\mathcal{A}_{g,1}[d]$ and $\mathcal{B}_{g,1}[d].$

Therefore, by equalities \eqref{eq_A[p],B[p]_constant}, we get that $q_g+\varphi_g^x$ with $x\in A_d$ are constant on the double cosets $\mathcal{A}_{g,1}[d]\backslash \mathcal{M}_{g,1}[d]/\mathcal{B}_{g,1}[d].$

Summarizing, we get the following result:

\begin{teo}
\label{teo_cocy_p}
Let $A$ an abelian group. For $g\geq 3,$ $d\geq 3$ an odd integer and for $g\geq 5,$ $d\geq 2$ an even integer such that $4 \nmid d.$
For each $x\in A_d,$ a family of 2-cocycles $(C_g)_{g\geq 3}$ on the $(\text{mod }d)$ Torelli groups $\mathcal{M}_{g,1}[d],$ with values in $A,$ satisfying conditions (1)-(3) provides a compatible family of trivializations $F_g+\varphi_g^x: \mathcal{M}_{g,1}[d]\rightarrow A$ that reassembles into an invariant of $\mathbb{Q}$-homology spheres $\mathcal{S}^3[d]$
$$\lim_{g\to \infty}F_g+\varphi_g^x: \mathcal{S}^3[d]\longrightarrow A$$
if and only if the following two conditions hold:
\begin{enumerate}[(i)]
\item The associated cohomology classes $[C_g]\in H^2(\mathcal{M}_{g,1}[d];A)$ are trivial.
\item The associated torsors $\rho(C_g)\in H^1(\mathcal{AB}_{g,1},Hom(\mathcal{M}_{g,1}[d],A))$ are trivial.
\end{enumerate}
\end{teo}

\section{Pull-back of $2$-cocycles over abelian groups}
\label{section: pull-back 2-cocycles}
In general, it is not easy to construct a family of $2$-cocycles $\{C_g\}_g\geq 3$ satisfying the hypothesis of Theorem \eqref{teo_cocy_p}.
The idea to construct such family of $2$-cocycles $(C_g)_{g\geq 3},$ inspired on \cite{pitsch}, is the following:
Consider a $\mathcal{AB}_{g,1}$-equivariant map $f$ from $\mathcal{M}_{g,1}[p]$ to a certain module $V.$ Then, we construct a family of bilinear forms $\{B_g\}_g$ on the module $V,$ (which are naturally $2$-cocycles on $V$), in such a way the pull-back of this bilinear forms along $f$ will be the desired family of $2$-cocycles on $\mathcal{M}_{g,1}[p].$
For the purposes of this thesis it is enough to give the construction for the first Zassenhaus mod $p$ Johnson homomorphism $\tau_1^Z,$ defined in Section \eqref{sec_johnson mod p}, and bilinear forms on $\extp^3 H_p.$

Consider the first Zassenhaus mod $p$ Johnson homomorphism $\tau_1^Z$ and a family of bilinear forms $(B_g)_{g\geq 3},$ defined on $\extp^3H_p$ satisfying the following properties:

\begin{enumerate}[(1')]
\item The $2$-cocycles $\{B_g\}_g$ are compatible with the stabilization map.
\item For every $\phi=
\left(\begin{smallmatrix}
G & 0 \\
0 & {}^tG^{-1}
\end{smallmatrix}\right)
 \in Sp_{2g}(\mathbb{Z})$ with $G\in GL_g(\mathbb{Z}),$ $B_g(\phi - \phi^{-1},\phi - \phi^{-1})=B_g(-,-),$
\item If $\phi\in \tau_1^Z(\mathcal{A}_{g,1}[d])$ or $\psi \in \tau_1^Z(\mathcal{B}_{g,1}[d])$ then $B_g(\phi, \psi)=0.$
\end{enumerate}

Since $\tau_1^Z$ is an $\mathcal{M}_{g,1}$-equivariant homomorphism and compatible with the stabilization map, the pull-back of this family of $2$-cocycles along $\tau_1^Z$ give us a family of $2$-cocycles on $\mathcal{M}_{g,1}[p]$ satisfying the properties (1)-(3) in Section \eqref{section_From invariants to trivial cocycles}.

Then we devote this section to find families of $2$-cocycles on $\extp^3H_p$ satisfying the properties (1')-(3').
To find such families, we first compute the image of $\mathcal{A}_{g,1}[d],$ $\mathcal{B}_{g,1}[d]$ under $\tau_1^Z.$


\subsection{The extension of Johnson's homomorphism}
\sectionmark{The extension of Johnson's homomorphism modulo $p$}
\label{sec_ext_jonh_hom}
We first construct, in a natural way, a crossed homomorphism $k_p: \mathcal{M}_{g,1}\rightarrow \extp^3 H_p$ which extends the Jonhson homomorphism $\tau_1$ modulo $p.$
By definition of the Zassenhaus filtration, we have a commutative diagram
\begin{equation}
\label{diag_com_M}
\xymatrix@C=10mm@R=10mm{
 1 \ar@{->}[r] & \mathcal{T}_{g,1} \ar@{->}[r] \ar@{->}[d]^{\tau_1^Z} & \mathcal{M}_{g,1} \ar@{->}[r]^{\Psi} \ar@{->}[d]^{\rho_2^Z} & Sp_{2g}(\mathbb{Z}) \ar@{->}[r] \ar@{->}[d]^{r_p} & 1\\
0  \ar@{->}[r] & \bigwedge^3 H_p \ar@{->}[r]^-{i}
 & \rho_2^Z(\mathcal{M}_{g,1}) \ar@{->}[r]^-{\psi_1^Z} & Sp_{2g}(\mathbb{Z}/p) \ar@{->}[r] & 1 }
\end{equation}
By Lemma \eqref{lema_duality_homology} and the fact that $\extp^3H_p$ is an $\mathbb{F}_p$-vector space, we have isomorphisms
$$H^*(Sp_{2g}(\mathbb{Z}/p);\extp^3H_p)\cong H^*(Sp_{2g}(\mathbb{Z}/p);(\extp^3H_p)^*)\cong (H_*(Sp_{2g}(\mathbb{Z}/p);\extp^3H_p))^*.$$
Since $-Id\in Sp_{2g}(\mathbb{Z}/p)$ acts on $\extp^3H_p$ by the multiplication of $-1,$ by the Center kills Lemma we have that
$H_*(Sp_{2g}(\mathbb{Z}/p);\extp^3H_p)=0.$
As a consequence,
$$H^*(Sp_{2g}(\mathbb{Z}/p);\extp^3H_p)=0.$$
Then the bottom extension of diagram \eqref{diag_com_M} splits with only one $\extp^3H_p$-conjugacy class of splittings.
By Proposition \eqref{prop_eqiv_splittings}, we have an isomorphism
$$f: \rho_2^Z(\mathcal{M}_{g,1})\longrightarrow \extp^3H_p\rtimes Sp_{2g}(\mathbb{Z}).$$
Recall that the operation on $\extp^3H_p\rtimes Sp_{2g}(\mathbb{Z}/p)$ is given by $(a,g)\cdot (b,h)= (a+g\cdot b, gh).$
Then, if we take $\pi:\extp^3H_p\rtimes Sp_{2g}(\mathbb{Z}/p)\rightarrow \extp^3H_p$ the projection on the first component, we get a derivation, and as a consequence we have that $k_p=(\pi \circ f \circ \rho_2^Z)$ is a crossed homomorphism
$$k_p: \mathcal{M}_{g,1}\longrightarrow \extp^3 H_p.$$
Observe that in fact $\pi$ is a retraction of the bottom extension of \eqref{diag_com_M}, then using the commutative diagram \eqref{diag_com_M} we have that $k_p$ restricted to $\mathcal{T}_{g,1}$ is the image of the Johnson homomorphism reduced modulo $p.$
Thus, we have found an extension of the Johnson's homomorphism modulo $p$ to the whole mapping class group $\mathcal{M}_{g,1}.$

Moreover, if we consider the crossed homomorphism $k_p\in H^1(\mathcal{M}_{g,1};\extp^3H_p)$ restricted to $\mathcal{M}_{g,1}[p]$ we get an extension of the Johnson's homomorphism modulo $p$ to $\mathcal{M}_{g,1}[p].$
Notice that the first Zassenhaus mod $p$ Johnson homomorphism $\tau_1^Z$ is also an extension of the Johnson's homomorphism modulo $p$ to $\mathcal{M}_{g,1}[p].$

Next we concern about the unicity of such extensions.

\begin{prop}
For any odd prime $p$ and $g\geq 4,$ there are isomorphisms
\begin{align*}
 H^1(\mathcal{M}_{g,1};\extp^3H_p)\cong & H^1(\mathcal{M}_{g,1}[p];\extp^3H_p)^{Sp_{2g}(\mathbb{Z}/p)}\cong H^1(\mathcal{T}_{g,1};\extp^3H_p)^{Sp_{2g}(\mathbb{Z}/p)} \cong \\
\cong & Hom(\extp^3 H_p, \extp^3 H_p)^{Sp_{2g}(\mathbb{Z}/p)} \cong \left(\mathbb{Z}/p\right)^2.
\end{align*}
\end{prop}

\begin{proof}

Consider the 5-term exact sequence associated to the short exact sequence
$$
\xymatrix@C=7mm@R=10mm{1 \ar@{->}[r] & \mathcal{M}_{g,1}[p] \ar@{->}[r] & \mathcal{M}_{g,1} \ar@{->}[r] & Sp_{2g}(\mathbb{Z}/p) \ar@{->}[r] & 1 }$$
and the $Sp_{2g}(\mathbb{Z}/p)$-module $\extp^3H_p.$
Then, we get the exact sequence
$$
\xymatrix@C=7mm@R=10mm{0 \ar@{->}[r] & H^1(Sp_{2g}(\mathbb{Z}/p);\extp^3H_p) \ar@{->}[r]^-{inf} & H^1(\mathcal{M}_{g,1};\extp^3H_p) \ar@{->}[r]^-{res} & H^1(\mathcal{M}_{g,1}[p];\extp^3H_p)^{Sp_{2g}(\mathbb{Z}/p)}.}$$
By the Center kills Lemma, since $-Id$ acts on $\extp^3H_p$ as the multiplication by $-1,$ we have that $H^1(Sp_{2g}(\mathbb{Z}/p);\extp^3H_p)=0.$ 
Therefore we have the injection
\begin{equation}
\label{res_inj_M[p]}
res: \; H^1(\mathcal{M}_{g,1};\extp^3H_p) \rightarrow H^1(\mathcal{M}_{g,1}[p];\extp^3H_p)^{Sp_{2g}(\mathbb{Z}/p)}.
\end{equation}
Next, consider the 5-term exact sequence associated to the short exact sequence
$$
\xymatrix@C=7mm@R=10mm{1 \ar@{->}[r] & \mathcal{T}_{g,1} \ar@{->}[r] & \mathcal{M}_{g,1}[p] \ar@{->}[r] & Sp_{2g}(\mathbb{Z},p) \ar@{->}[r] & 1 }$$
and the $Sp_{2g}(\mathbb{Z}/p)$-module $\extp^3H_p.$
Then we get the exact sequence
$$
\xymatrix@C=7mm@R=10mm{0 \ar@{->}[r] & H^1(Sp_{2g}(\mathbb{Z},p);\extp^3H_p) \ar@{->}[r]^-{inf} & H^1(\mathcal{M}_{g,1}[p];\extp^3H_p) \ar@{->}[r]^-{res} & H^1(\mathcal{T}_{g,1};\extp^3H_p).}$$
Taking $Sp_{2g}(\mathbb{Z}/p)$-invariants, we get another exact sequence
$$
0 \rightarrow H^1(Sp_{2g}(\mathbb{Z},p);\extp^3H_p)^{Sp_{2g}(\mathbb{Z}/p)} \rightarrow H^1(\mathcal{M}_{g,1}[p];\extp^3H_p)^{Sp_{2g}(\mathbb{Z}/p)}\rightarrow$$
$$\rightarrow H^1(\mathcal{T}_{g,1};\extp^3H_p)^{Sp_{2g}(\mathbb{Z}/p)}.
$$
By the Universal coefficients Theorem and Proposition \eqref{prop_iso_M[d],sp(Z/d)} we have that
$$H^1(Sp_{2g}(\mathbb{Z},p);\extp^3H_p)^{Sp_{2g}(\mathbb{Z}/p)} \cong
Hom(\mathfrak{sp}_{2g}(\mathbb{Z}/p),\extp^3H_p)^{Sp_{2g}(\mathbb{Z}/p)}.$$
Notice that, since $-Id\in Sp_{2g}(\mathbb{Z}/p)$ acts trivially on $\mathfrak{sp}_{2g}(\mathbb{Z}/p)$ and as $-1$ on $\extp^3H_p,$ then $$Hom(\mathfrak{sp}_{2g}(\mathbb{Z}/p),\extp^3H_p)^{Sp_{2g}(\mathbb{Z}/p)}=0.$$
Therefore we have the injection
\begin{equation}
\label{res_inj_T}
res: \; H^1(\mathcal{M}_{g,1}[p];\extp^3H_p) \rightarrow H^1(\mathcal{T}_{g,1};\extp^3H_p)^{Sp_{2g}(\mathbb{Z}/p)}.
\end{equation}

Moreover, by the Universal coefficients Theorem and Theorem 6.19 in \cite{farb}, we have that the first Johnson homomorphism modulo $p$ induces an isomorphism
\begin{equation}
\label{iso_T}
H^1(\mathcal{T}_{g,1};\extp^3H_p)^{Sp_{2g}(\mathbb{Z}/p)}\cong Hom(\extp^3H_p,\extp^3H_p)^{Sp_{2g}(\mathbb{Z}/p)}.
\end{equation}

Next, we compute $Hom(\extp^3H_p,\extp^3H_p)^{Sp_{2g}(\mathbb{Z})}.$
Let $c_i=a_i \text{ or } b_i$ for every $i\in \{1,\ldots ,g\},$ and $f \in Hom(\extp^3H_p,\extp^3H_p)^{Sp_{2g}(\mathbb{Z})}.$

\begin{itemize}
\item Consider the element
$\phi=\left(\begin{smallmatrix}
G & 0 \\
0 & {}^tG^{-1}
\end{smallmatrix}\right) \in Sp_{2g}(\mathbb{Z})$
with $G=(1,i)(2,j)(3,k)\in \mathfrak{S}_g.$ Then we have that
\begin{align*}
\phi \cdot f( c_1\wedge c_2\wedge c_3)= & f(\phi \cdot c_1\wedge c_2\wedge c_3)= f(c_i\wedge c_j\wedge c_k), \\
\phi \cdot f( c_1\wedge a_2\wedge b_2)= & f(\phi \cdot c_1\wedge a_2\wedge b_2)= f(c_i\wedge a_j\wedge b_j).
\end{align*}

Thus, every element of $Hom(\extp^3H_p,\extp^3H_p)^{Sp_{2g}(\mathbb{Z})}$ is determined by the images of the elements $c_1\wedge c_2\wedge c_3,$ $c_1\wedge a_2\wedge b_2$ with $c_i=a_i \text{ or } b_i.$

\item Consider the element
$\phi=\left(\begin{smallmatrix}
0 & Id \\
-Id & 0
\end{smallmatrix}\right) \in Sp_{2g}(\mathbb{Z}).$ Then we have that
\begin{align*}
\phi \cdot f( a_1\wedge a_2\wedge a_3)= & f(\phi \cdot a_1\wedge a_2\wedge a_3)= f(-b_1\wedge b_2\wedge b_3)=-f(b_1\wedge b_2\wedge b_3), \\
\phi \cdot f( a_1\wedge a_2\wedge b_3)= & f(\phi \cdot a_1\wedge a_2\wedge b_3)= f(b_1\wedge b_2\wedge a_3),\\
\phi \cdot f( a_1\wedge a_2\wedge b_2)= & f(\phi \cdot a_1\wedge a_2\wedge b_2)= f(b_1\wedge b_2\wedge a_2).
\end{align*}

\item Consider the element
$\phi \in Sp_{2g}(\mathbb{Z}),$ the matrix with $1$'s at the diagonal and position $(3,g+3),$ and $0$'s elsewhere. Then we have that
\begin{align*}
\phi \cdot f( a_1\wedge a_2\wedge b_3)= & f(\phi \cdot a_1\wedge a_2\wedge b_3)=
f(a_1\wedge a_2\wedge b_3+a_1\wedge a_2\wedge a_3)= \\
= & f(a_1\wedge a_2\wedge b_3)+f(a_1\wedge a_2\wedge a_3).
\end{align*}
So, $f(a_1\wedge a_2\wedge a_3)=\phi \cdot f( a_1\wedge a_2\wedge b_3)-f(a_1\wedge a_2\wedge b_3).$

\item Consider the element
$\phi=\left(\begin{smallmatrix}
G & 0 \\
0 & {}^tG^{-1}
\end{smallmatrix}\right) \in Sp_{2g}(\mathbb{Z}),$ with ${}^tG^{-1}\in GL_{2g}(\mathbb{Z})$ the matrix with $1$'s at positions $(3,2)$ and at the diagonal, and $0$'s elsewhere. Then we have that
\begin{align*}
\phi \cdot f( a_1\wedge a_2\wedge b_2)= & f(\phi \cdot a_1\wedge a_2\wedge b_2)=
f(a_1\wedge a_2\wedge b_2+a_1\wedge a_2\wedge b_3)= \\
= & f(a_1\wedge a_2\wedge b_2)+f(a_1\wedge a_2\wedge b_3).
\end{align*}
So, $f(a_1\wedge a_2\wedge b_3)=\phi \cdot f( a_1\wedge a_2\wedge b_2)-f(a_1\wedge a_2\wedge b_2).$
\end{itemize}

Hence, every element of $Hom(\extp^3H_p,\extp^3H_p)^{Sp_{2g}(\mathbb{Z})}$ is determined by the image of the element $a_1\wedge a_2 \wedge b_2.$
Next we study the possible images of this element.
Set
$$f(a_1\wedge a_2 \wedge b_2)=\sum_{i< j< k} m(c_i, c_j, c_k) c_i\wedge c_j\wedge c_k + \sum_{i, s} m(c_i, a_s, b_s) c_i\wedge a_s\wedge b_s,$$
with $m(c_i, c_j, c_k)\in \{1,\ldots, p-1\}.$
Let $l\neq 1.$ Consider the element $\phi=\left(\begin{smallmatrix}
G & 0 \\
0 & {}^tG^{-1}
\end{smallmatrix}\right) \in Sp_{2g}(\mathbb{Z}),$ with $G\in GL_{2g}(\mathbb{Z})$ the matrix with $-1$ at position $(l,l)$ and $1$'s at the positions $(r,r)$ with $r\neq l.$ Then for $l=g,$ we have that
\begin{align*}
 & \sum_{i< j< k} m(c_i, c_j, c_k) c_i\wedge c_j\wedge c_k + \sum_{i, s} m(c_i, a_s, b_s) c_i\wedge a_s\wedge b_s \\
= & f( a_1\wedge a_2 \wedge b_2)=f(\phi\cdot a_1\wedge a_2 \wedge b_2)=\phi\cdot f( a_1\wedge a_2 \wedge b_2)= \\
= &\sum_{i< j< k<g} m(c_i, c_j, c_k) c_i\wedge c_j\wedge c_k + \sum_{i<g, s} m(c_i, a_s, b_s) c_i\wedge a_s\wedge b_s+ \\
 & -\sum_{i< j<g} m(c_i, c_j, c_g) c_i\wedge c_j\wedge c_g - \sum_{ s} m(c_g, a_s, b_s) c_g\wedge a_s\wedge b_s.
\end{align*}
Hence $m(c_i, c_j, c_g)=0,$ for all $i,j$ such that $i< j<g,$
and $m(c_g, a_s, b_s)=0$ for all $s.$ Then,
$$f(a_1\wedge a_2 \wedge b_2)=\sum_{i< j< k<g} m(c_i, c_j, c_k) c_i\wedge c_j\wedge c_k + \sum_{i<g, s} m(c_i, a_s, b_s) c_i\wedge a_s\wedge b_s.$$
Repeating the same argument from $l=g-1$ to $l=2$ we get that
$m(c_i, c_j, c_k)=0,$ for all $i,j,k$ such that $i< j<k,$ and $m(c_i,a_s,b_s)=0$ for all $i,s$ with $i\neq 1.$ Then
$$f(a_1\wedge a_2 \wedge b_2)=\sum_{j} m(c_1, a_j, b_j) c_1\wedge a_j\wedge b_j.$$
Next, consider the element $\phi\in Sp_{2g}(\mathbb{Z})$ the matrix with $1$'s at the diagonal and position $(1,g+1).$ Then we have that
\begin{align*}
\sum_{ j} m(c_1, a_j, b_j) c_1\wedge a_j\wedge b_j = & f(a_1\wedge a_2 \wedge b_2)=\phi\cdot f(a_1\wedge a_2 \wedge b_2)= \\
= & \sum_{ j} m(c_1, a_j, b_j) c_1\wedge a_j\wedge b_j +\sum_{ j} m(b_1, a_j, b_j) a_1\wedge a_j\wedge b_j.
\end{align*}
Thus $m(b_1, a_j, b_j)=0$ for all $j,$ and we have that
$$f(a_1\wedge a_2 \wedge b_2)=\sum_{ j} m(a_1, a_j, b_j) a_1\wedge a_j\wedge b_j.$$

For a fixed $k\neq 1,2,$ consider the element $\phi=\left(\begin{smallmatrix}
G & 0 \\
0 & {}^tG^{-1}
\end{smallmatrix}\right) \in Sp_{2g}(\mathbb{Z}),$ with $G\in M_{2g}(\mathbb{Z})$ the matrix with $1$'s at the positions $(k+1,k),$ $(l,l).$
Then we have that
\begin{align*}
& \sum_{ j} m(a_1, a_j, b_j) a_1\wedge a_j\wedge b_j= f( a_1\wedge a_2 \wedge b_2)=f(\phi\cdot a_1\wedge a_2 \wedge b_2)=\phi\cdot f( a_1\wedge a_2 \wedge b_2) \\
& = m(a_1, a_k, b_k)a_1\wedge a_{k+1}\wedge b_k -m(a_1, a_{k+1}, b_{k+1})a_1\wedge a_{k+1}\wedge b_k + \\
&  +\sum_{ j} m(a_1, a_j, b_j) a_1\wedge a_j\wedge b_j.
\end{align*}
Thus $m(a_1, a_k, b_k)=m(a_1, a_{k+1}, b_{k+1}),$ for every $k\neq 1,2.$ Hence,
$$f(a_1\wedge a_2 \wedge b_2)=m_1 a_1\wedge a_2 \wedge b_2+m_2\sum^g_{ j=3}  a_1\wedge a_j\wedge b_j,$$
with $m_1,m_2\in \mathbb{Z}/p.$
Therefore $Hom(\extp^3H_p,\extp^3H_p)^{Sp_{2g}(\mathbb{Z})}$ at most has $p^2$ elements.

Observe that the identity is an element of $Hom(\extp^3H_p,\extp^3H_p)^{Sp_{2g}(\mathbb{Z})}.$
In addition, if we consider the $Sp_{2g}(\mathbb{Z}/p)$-equivariant maps $$C:\extp^3H_p\rightarrow H_p,\qquad u:H_p\rightarrow \extp^3H_p,$$ defined by
$$C(a\wedge b\wedge c)=2[\omega(b,c)a+\omega(c,a)b+\omega(a,b)c],$$
$$u(x)=x\wedge \left( \sum_{i=1}^g b_i\wedge a_i\right).$$
Then $r=\frac{1}{2} u\circ C$ is an endomorphism of $\extp^3H_p$ which is $Sp_{2g}(\mathbb{Z}/p)$-equivariant because $\omega$ is $Sp_{2g}(\mathbb{Z}/p)$-invariant and $\sum_{i=1}^g b_i\wedge a_i$ is a fixed point by the action of $Sp_{2g}(\mathbb{Z}/p).$
In addition we have that
$$
r(a_1\wedge a_2 \wedge b_2)= \frac{1}{2}u(C(a_1\wedge a_2 \wedge b_2))
=-\sum_{j=2}^ga_1\wedge a_j \wedge b_j .$$
As a consequence, the homomorphism $r$ and the identity are linearly independent and generates $Hom(\extp^3H_p,\extp^3H_p)^{Sp_{2g}(\mathbb{Z}/p)}.$
Therefore,
$$Hom(\extp^3H_p,\extp^3H_p)^{Sp_{2g}(\mathbb{Z}/p)}\cong (\mathbb{Z}/p)^{2}.$$

Finally, observe that $Id \circ k_p,$ $(u\circ C)\circ k_p$ are elements of $H^1(\mathcal{M}_{g,1};\extp^3H_p),$
whose restriction to $\mathcal{T}_{g,1}$ induce the homomorphisms $Id, (u\circ C)\in Hom(\extp^3H_p,\extp^3H_p)^{Sp_{2g}(\mathbb{Z}/p)}$ respectively. Therefore, by the injections \eqref{res_inj_M[p]}, \eqref{res_inj_T} and the isomorphism \eqref{iso_T},
we get the desired isomorphisms.
\end{proof}

As a consequence, the extension of the Johnson's homomorphism modulo $p$ to $\mathcal{M}_{g,1}$ is unique up to principal derivations and the extension of the Johnson's homomorphism modulo $p$ to $\mathcal{M}_{g,1}[p]$ is also unique.
Therefore, the homomorphism $\tau_1^Z:\mathcal{M}_{g,1}[p]\rightarrow \extp^3H_p$ coincides with the restriction of $k_p$ to $\mathcal{M}_{g,1}[p].$

\paragraph{Images of $\mathcal{A}_{g,1}[p],$ $\mathcal{B}_{g,1}[p]$ under $\tau_1^Z.$}
Next, we focus on computing the image of $\mathcal{A}_{g,1}[p],$ $\mathcal{B}_{g,1}[p]$ under $\tau_1^Z.$ First of all we need to compute the ablianization of $Sp_{2g}^B(\mathbb{Z},p)$ and $Sp_{2g}^A(\mathbb{Z},p).$

\begin{defi}
For $1\leq i,j\leq n,$ denote by $E^n_{i,j}(r)$ the $n\times n$ matrix with an $r$ at position $(i,j)$ and $0$'s elsewhere. Similarly, denote by $SE^n_{i,j}$ the $n\times n$ matrix with an $r$ at positions $(i,j)$ and $(j,i)$ and $0$'s elsewhere.
\end{defi}

\begin{defi}
For $1\leq i,j \leq g,$ denote by $\mathcal{X}^g_{i,j}(r)$ the matrix
$\left(\begin{smallmatrix}
Id_g & 0 \\
SE^g_{i,j}(r) & Id_g
\end{smallmatrix}\right),$ by $\mathcal{Y}^g_{i,j}(r)$ the matrix
$\left(\begin{smallmatrix}
Id_g & SE^g_{i,j}(r) \\
 0 & Id_g
\end{smallmatrix}\right),$
and by $\mathcal{Z}^g_{i,j}(r)$ the matrix $\left(\begin{smallmatrix}
Id_g+E^g_{i,j}(r) & 0 \\
 0 & Id_g-E^g_{i,j(r)}
\end{smallmatrix}\right).$
\end{defi}

The following Proposition is a direct consequence of results in \cite{newman} due to M. Newman and J.R. Smart, and \cite{Lee} due to R. Lee and R.H. Szczarba.

\begin{lema}
\label{lema_ses_SpB}
Let $g\geq 3,$ and $p$ an odd prime, the following sequence is a short exact sequence.
$$
\xymatrix@C=10mm@R=13mm{1 \ar@{->}[r] & Sp_{2g}^B(\mathbb{Z},p^2) \ar@{->}[r] & Sp_{2g}^B(\mathbb{Z},p) \ar@{->}[r]^-{abel} & \mathfrak{sl}_{2g}(\mathbb{Z}/p)\oplus S_g(\mathbb{Z}/p) \ar@{->}[r] & 1 }
$$
\end{lema}

\begin{proof}
Consider the homomorphism $abel: Sp_{2g}(\mathbb{Z},p)\rightarrow \mathfrak{sp}_{2g}(\mathbb{Z}/p).$ In \cite{Lee}, R. Lee and R.H. Szczarba proved that this map induces the following short exact sequence:
$$
\xymatrix@C=10mm@R=13mm{1 \ar@{->}[r] & Sp_{2g}(\mathbb{Z},p^2) \ar@{->}[r] & Sp_{2g}(\mathbb{Z},p) \ar@{->}[r]^{abel} & \mathfrak{sp}_{2g}(\mathbb{Z}/p) \ar@{->}[r] & 1 .}
$$
Observe that for any element
$\left(\begin{smallmatrix}
G & 0 \\
M & {}^tG^{-1}
\end{smallmatrix}\right) \in Sp_{2g}^B(\mathbb{Z},p)$ i.e. $G\in SL_g(\mathbb{Z},p),$ ${}^tGM\in S_g(p\mathbb{Z}),$  we have that
$$\left(\begin{matrix}
G & 0 \\
M & {}^tG^{-1}
\end{matrix}\right)=\left(\begin{matrix}
G & 0 \\
0 & {}^tG^{-1}
\end{matrix}\right)\left(\begin{matrix}
Id & 0 \\
{}^tGM & Id
\end{matrix}\right).$$
By \cite{Lee} we get that $(G-Id)/p \; \text{mod }p$ is an element of $\mathfrak{sl}_g(\mathbb{Z}/p),$ and clearly ${}^tGM/p \; \text{mod }p$ is an element of $S_g(\mathbb{Z}/p).$
So we get that
$$abel(Sp_{2g}^B(\mathbb{Z},p))\subset \mathfrak{sl}_{2g}(\mathbb{Z}/p)\oplus S_g(\mathbb{Z}/p).$$
On the other hand, in \cite{Lee}, R. Lee and R.H. Szczarba proves that the map
$abel: SL_g[p]\rightarrow \mathfrak{sl}_g(\mathbb{Z}/p)$ is surjective, and
in \cite{newman} M. Newman and J.R. Smart proves that the map $S_g(p\mathbb{Z})\rightarrow S_g(\mathbb{Z}/p)$ given by $N\mapsto N/p\text{ mod }p$ is also surjective, then it is clear that $abel:Sp_{2g}^B(\mathbb{Z},p)\rightarrow \mathfrak{sl}_{2g}(\mathbb{Z}/p)\oplus S_g(\mathbb{Z}/p)$ is surjective.
Finally, the kernel of this map is given by $Sp_{2g}(\mathbb{Z},p^2)\cap Sp_{2g}^B(\mathbb{Z},p)$ and this is by definition $Sp_{2g}^B(\mathbb{Z},p^2).$
Therefore we get the desired short exact sequence.
\end{proof}

\begin{lema}
\label{lema_comut_SpB}
Let $g\geq 3,$ and $p$ an odd prime,
$$H_1(Sp_{2g}^B(\mathbb{Z},p);\mathbb{Z}/p)=\mathfrak{sl}_g(\mathbb{Z}/p)\oplus S_g(\mathbb{Z}/p), \quad H_1(Sp_{2g}^A(\mathbb{Z},p);\mathbb{Z}/p)=\mathfrak{sl}_g(\mathbb{Z}/p)\oplus S_g(\mathbb{Z}/p).$$
\end{lema}

\begin{proof}
We only prove the result for $Sp_{2g}^B(\mathbb{Z},p)$ (the proof for $Sp_{2g}^A(\mathbb{Z},p)$ is analogous).
We show that $[Sp_{2g}^B(\mathbb{Z},p),Sp_{2g}^B(\mathbb{Z},p) ]=Sp_{2g}^B(\mathbb{Z},p^2).$ Then, by Lemma \eqref{lema_ses_SpB}, we will get the result.
Since $\mathfrak{sl}_{2g}(\mathbb{Z}/p)\oplus S_g(\mathbb{Z}/p)$ is abelian, $[Sp_{2g}^B(\mathbb{Z},p),Sp_{2g}^B(\mathbb{Z},p) ]\subset Sp_{2g}^B(\mathbb{Z},p^2).$
Next we show that
$$Sp_{2g}^B(\mathbb{Z},p^2)\subset [Sp_{2g}^B(\mathbb{Z},p),Sp_{2g}^B(\mathbb{Z},p) ].$$
Observe that for any element
$\left(\begin{smallmatrix}
G & 0 \\
M & {}^tG^{-1}
\end{smallmatrix}\right) \in Sp_{2g}^B(\mathbb{Z},p^2)$ i.e. $G\in SL_g(\mathbb{Z},p^2),$ ${}^tGM\in S_g(p^2\mathbb{Z})$  we have that
$$\left(\begin{matrix}
G & 0 \\
M & {}^tG^{-1}
\end{matrix}\right)=\left(\begin{matrix}
G & 0 \\
0 & {}^tG^{-1}
\end{matrix}\right)\left(\begin{matrix}
Id & 0 \\
{}^tGM & Id
\end{matrix}\right).$$
In \cite{Lee}, R. Lee and R. H. Szczarba proved that
$$SL_g(\mathbb{Z},p^2)=[SL_g(\mathbb{Z},p),SL_g(\mathbb{Z},p)].$$
Then, for every $G\in SL_g(\mathbb{Z},p^2),$ there exists a family of elements $\{ G_{i_1},G_{i_2}\}_{i_1,i_2}$ of $SL_g(\mathbb{Z},p)$ such that
\begin{align*}
\left(\begin{matrix}
G & 0 \\
0 & {}^tG^{-1}
\end{matrix}\right)= & \left(\begin{matrix}
\prod_{i_1,i_2}[G_{i_1},G_{i_2}] & 0 \\
0 & {}^t\left(\prod_{i_1,i_2}[G_{i_1},G_{i_2}]\right)^{-1}
\end{matrix}\right) \\
= & \prod_{i_1,i_2}\left[\left(\begin{matrix}
G_{i_1} & 0 \\
0 & {}^tG_{i_1}^{-1}
\end{matrix}\right),\left(\begin{matrix}
G_{i_2} & 0 \\
0 & {}^tG_{i_2}^{-1}
\end{matrix}\right)\right].
\end{align*}
On the other hand, we have that any element of the form $\left(\begin{smallmatrix}
Id & 0 \\
H & Id
\end{smallmatrix}\right)$ with $H\in S_g(p^2\mathbb{Z})$
is a product of elements of the family $\{\mathcal{X}_{i,j}^g(p^2)\}_{i,j},$ and moreover we have that
$$\mathcal{X}^g_{i,j}(p^2)=[\mathcal{X}^g_{i,j}(p),\mathcal{Z}^g_{i,j}(p)]\subset [Sp_{2g}^B(\mathbb{Z},p),Sp_{2g}^B(\mathbb{Z},p)].$$
Thus we get that $\left(\begin{smallmatrix}
G & 0 \\
M & {}^tG^{-1}
\end{smallmatrix}\right)\in [Sp_{2g}^B(\mathbb{Z},p),Sp_{2g}^B(\mathbb{Z},p)].$
Therefore,
$$Sp_{2g}^B(\mathbb{Z},p^2)=[Sp_{2g}^B(\mathbb{Z},p),Sp_{2g}^B(\mathbb{Z},p)],$$
as desired.
\end{proof}

\begin{prop} \label{prop_res_B}
The map $res: H^1(\mathcal{B}_{g,1}[p];\extp^3H_p)^{\mathcal{AB}_{g,1}}\rightarrow H^1(\mathcal{TB}_{g,1};\extp^3H_p)^{\mathcal{AB}_{g,1}}$ is injective.
\end{prop}

\begin{proof}
Taking the 5-term sequence associated to the short exact sequence
$$
\xymatrix@C=10mm@R=13mm{1 \ar@{->}[r] & \mathcal{TB}_{g,1} \ar@{->}[r] & \mathcal{B}_{g,1}[p] \ar@{->}[r] & Sp_{2g}^B(\mathbb{Z},p) \ar@{->}[r] & 1 },$$
and the $\mathcal{B}_{g,1}[p]$-module $\extp^3H_p$ with the trivial action, we get the exact sequence
$$
\xymatrix@C=8mm@R=10mm{0 \ar@{->}[r] & H^1(Sp_{2g}^B(\mathbb{Z},p);\extp^3H_p) \ar@{->}[r]^{\;\text{inf}} & H^1(\mathcal{B}_{g,1}[p];\extp^3H_p) \ar@{->}[r]^{\text{res}\quad} & H^1(\mathcal{TB}_{g,1};\extp^3H_p)
 .}
$$
By Universal coefficients Theorem and Lemma \eqref{lema_comut_SpB}, we have that
$$H^1(Sp_{2g}^B(\mathbb{Z},p);\extp^3H_p)=Hom(\mathfrak{sl}_{g}(\mathbb{Z}/p)\oplus S_g(\mathbb{Z}/p),\extp^3H_p),$$
and taking $\mathcal{AB}_{g,1}$-invariants, we get the exact sequence
$$
\xymatrix@C=7mm@R=5mm{0 \ar@{->}[r] & Hom(\mathfrak{sl}_{g}(\mathbb{Z}/p)\oplus S_g(\mathbb{Z}/p),\extp^3H_p)^{\mathcal{AB}_{g,1}} \ar@{->}[r]^-{\text{inf}} & H^1(\mathcal{B}_{g,1}[p];\extp^3H_p)^{\mathcal{AB}_{g,1}} \ar@{->}[r] & }
$$
$$ 
 \xymatrix@C=7mm@R=5mm{ \ar@{->}[r]^-{\text{res}} & H^1(\mathcal{TB}_{g,1};\extp^3H_p)^{\mathcal{AB}_{g,1}}.
 }
$$
Since $\Psi(\mathcal{AB}_{g,1})\cong SL^\pm_g(\mathbb{Z}/p),$ there is an element $f$ such that $\Psi(f)=-Id,$ which acts trivially on $\mathfrak{sl}_{g}(\mathbb{Z}/p)\oplus S_g(\mathbb{Z}/p)$ and as $-1$ on $\extp^3H_p.$
Hence, $Hom(\mathfrak{sl}_{g}(\mathbb{Z}/p)\oplus S_g(\mathbb{Z}/p),\extp^3H_p)^{\mathcal{AB}_{g,1}}=0.$
\end{proof}

Now we are ready to compute the images of $\mathcal{A}_{g,1}[p]$ and $\mathcal{B}_{g,1}[p]$ under $\tau_1^Z.$

Recall that we have a decomposition $H=A\oplus B,$ this induces the decomposition $\bigwedge^3 H=\bigwedge^3 A\oplus B \wedge(\bigwedge^2 A)\oplus A \wedge (\bigwedge^2 B)\oplus \bigwedge^3 B.$
Set $W_A=\bigwedge^3 A,$ $W_B=\bigwedge^3 B$ and $W_{AB}=B\wedge (\bigwedge^2 A)\oplus A \wedge (\bigwedge^2 B).$ The Johnson homomorphism computes the action of the Torelli group on the second nilpotent quotient of the fundamental group of $\Sigma_{g,1}.$ Computing on specific elements one can check that (see \cite{mor})

\begin{lema}
\label{lem_im_TA,TB}
The image of $\mathcal{TA}_{g,1}$ and $\mathcal{TB}_{g,1}$ under $\tau_1$ in $\bigwedge^3H$ are respectively $$W_A\oplus W_{AB}\quad\text{and}\quad W_B\oplus W_{AB}.$$ 
\end{lema}

Similarly as before, we have a decomposition $H_p=A_p\oplus B_p,$ and this induces the decomposition $\bigwedge^3 H_p=\bigwedge^3 A_p\oplus B _p\wedge(\bigwedge^2 A_p)\oplus A_p \wedge (\bigwedge^2 B_p)\oplus \bigwedge^3 B_p.$
Set $W_{A}^p=\bigwedge^3 A_p,$ $W_B^p=\bigwedge^3 B_p$ and $W_{AB}^p=B_p\wedge (\bigwedge^2 A_p)\oplus A_p \wedge (\bigwedge^2 B_p).$
We get the following lemma analog to \eqref{lem_im_TA,TB}.

\begin{lema}
\label{lem_im_A[p],B[p]}
The image of $\mathcal{A}_{g,1}[p]$ and $\mathcal{B}_{g,1}[p]$ under $\tau_1^Z$ in $\bigwedge^3H_p$ are respectively $$W_A^p\oplus W_{AB}^p\quad\text{and}\quad W_B^p\oplus W_{AB}^p.$$ 
\end{lema}

\begin{proof}
We only do the proof for $\mathcal{B}_{g,1}[p].$ For $\mathcal{A}_{g,1}[p]$ the argument is analogous.
Observe that we have a commutative diagram
\begin{equation}
\label{diag_com_B}
\xymatrix@C=7mm@R=10mm{
 1 \ar@{->}[r] & \mathcal{TB}_{g,1} \ar@{->}[r] \ar@{->}[d]^{\tau_1^Z} & \mathcal{B}_{g,1} \ar@{->}[r]^-{\Psi} \ar@{->}[d]^-{\rho_2^Z} & Sp_{2g}^B(\mathbb{Z}) \ar@{->}[r] \ar@{->}[d]^{r_p} & 1\\
0  \ar@{->}[r] & W_B^p\oplus W_{AB}^p\ar@{->}[r]^{i}
 & \rho_2^Z(\mathcal{B}_{g,1}) \ar@{->}[r]^-{\psi_1^Z} & Sp_{2g}^{B\pm}(\mathbb{Z}/p) \ar@{->}[r] & 1 .}
\end{equation}
By Center kills Lemma we have that
$H^*(Sp_{2g}^{B\pm}(\mathbb{Z}/p);W_B^p\oplus W_{AB}^p)=0.$
Then the bottom row of the diagram \eqref{diag_com_B} splits with only one $W_B^p\oplus W_{AB}^p$-conjugacy class of splittings.
Therefore, we have an isomorphism $f: \rho_2^Z(\mathcal{B}_{g,1}) \rightarrow (W_B^p\oplus W_{AB}^p)\rtimes Sp_{2g}^{B\pm}(\mathbb{Z}/p).$

Consider the map $\pi: (W_B^p\oplus W_{AB}^p)\rtimes Sp_{2g}^{B\pm}(\mathbb{Z}/p) \rightarrow W_B^p\oplus W_{AB}^p$
given by the projection in the first component.
Take $k_B=(\pi \circ f \circ \rho_2^Z)$ we obtain a crossed homomorphsim:
$$k_B: \mathcal{B}_{g,1}\longrightarrow W_B^p\oplus W_{AB}^p.$$
If we restrict $k_B$ on $\mathcal{B}_{g,1}[p],$ we get an $\mathcal{B}_{g,1}$-equivariant homomorphism:
$k_B: \mathcal{B}_{g,1}[p]\longrightarrow W_B^p\oplus W_{AB}^p.$
Composing $k_B$ with the natural inclusion 
$W_B^p\oplus W_{AB}^p \hookrightarrow \extp^3 H_p$ we get an $\mathcal{B}_{g,1}$-equivariant homomorphism: $$k_B: \mathcal{B}_{g,1}[p]\longrightarrow \extp^3H_p,$$
and by the commutative diagram \eqref{diag_com_B}, and the fact that $\pi$ is a retraction, we get that $k_B$ restricted to $\mathcal{TB}_{g,1}$ is the Johnson homomorphism modulo $p.$

On the other hand, if we take $k_p: \mathcal{M}_{g,1}[p]\longrightarrow \extp^3H_p$ restricted to $\mathcal{B}_{g,1}[p]$ a priori we get another 
$\mathcal{B}_{g,1}$-equivariant homomorphism, which is the Johnson homomorphism modulo $p$ on $\mathcal{TB}_{g,1},$
but by Proposition \eqref{prop_res_B} there is only one $\mathcal{B}_{g,1}$-equivariant homomorphism that coincides with the Johnson homomorphism modulo $p$ on $\mathcal{TB}_{g,1}.$
Hence, the restriction of $k_p$ to $\mathcal{B}_{g,1}[p]$ and $k_B$ have to be equal and so $k_p(\mathcal{B}_{g,1}[p])=W_B^p\oplus W_{AB}^p.$
\end{proof}

\subsection{Pull-back of 2-cocycles}

For each $g,$ the intersection form on homology induces a bilinear form $\omega:A\otimes B\rightarrow \mathbb{Z}/p.$ This in turn induces the bilinear forms $ J_g:W_A^p\otimes W_B^p \rightarrow \mathbb{Z}/p$ and $ J_g^t:W_B^p\otimes W_A^p \rightarrow \mathbb{Z}/p$ that we extend by $0$ to degenerate bilinear forms on $\bigwedge^3H_p=W^p_A\oplus W^p_{AB}\oplus W^p_{B}.$ Written as a matrices according to the decomposition $\bigwedge^3H_p=W^p_A\oplus W^p_{AB}\oplus W^p_{B}$ these are:
$$J_g:=\left(\begin{matrix}
0 & 0 & Id\\
0 & 0 & 0 \\
0 & 0 & 0 
\end{matrix}\right),
\qquad J_g^t:=\left(\begin{matrix}
0 & 0 & 0\\
0 & 0 & 0 \\
Id & 0 & 0 
\end{matrix}\right)$$
Notice that bilinear forms are naturally $2$-cocycles on abelian groups.

\begin{prop}
For each $g\geq 3,$ the $2$-cocycle $J_g^t$ is the unique cocycle (up to a multiplicative constant) on $\bigwedge^3 H_p$ whose pull-back along $k_p$ on the Torelli group modulo $p,$ $\mathcal{M}_{g,1}[p]$ satisfies conditions (2) and (3). Moreover once we have fixed a common multiplicative constant the family of pull-backed cocycles satisfies also (1).
\end{prop}

\begin{proof}
Fix an integer $n\in \mathbb{Z}/p.$ It is obvious from the definition and from Lemma \eqref{lem_im_A[p],B[p]} that the family $k^*_p(nJ_g^t)$ satisfies (1), (2) and (3).

Let $B$ denote an arbitrary 2-cocycle on $\bigwedge^3 H_p$ whose pull-back on $\mathcal{M}_{g,1}[p]$ along $\tau_1^Z$ satisfies (2) and (3).
Write each element $w\in \bigwedge^3 H_p $ as $w=w_a+w_{ab}+w_b$ according to the decomposition $W^p_A\oplus W^p_{AB}\oplus W^p_B.$

The cocycle relation together with condition (3) and Lemma \eqref{lem_im_A[p],B[p]} imply that
$$\forall v,w\in \extp^3H_p,\qquad B(v,w)=B(v_b,w_a)$$
We first prove that $B$ is bilinear. For the linearity on the first variable compute
\begin{align*}
B(u+v,w)= & B(u_b+v_b,w_a) \\
= & B(v_b,w_a)+B(u_b,v_b+w_a) -B(u_b,v_b) \\
= & B(u_b,w_a)+B(v_b,w_a) \\
= & B(u,w) +B(v,w).
\end{align*}
where in the second equality we used the cocycle relation.
A similar proof holds for the linearity on the second variable.

The equivariance of $k_p$ and the condition (3) and Lemma \eqref{lem_im_A[p],B[p]} implies that our bilinear form $B(x,y)$ should be zero for
$y\in \extp^3H_p,\; x\in W_A^p\oplus W_{AB}^p$
and
$ x\in \extp^3H_p, \; y\in W_B^p\oplus W_{AB}^p.$
By the equivariance properties of $k_p$ we have the following relations:

\begin{itemize}
\item Let $i,j,k,l,m,n\in\{1,\ldots g\}$ with $i, j\, k$ pairwise distinct, $l, m, k$ pairwise distinct and $n\neq i,j,k,l,m.$  Take $f$ the matrix
$\left(\begin{smallmatrix}
G & 0 \\
0 & {}^tG^{-1}
\end{smallmatrix}\right)$ with $G\in \mathfrak{S}_g\subset GL_g(\mathbb{Z})$ with $1$'s at the diagonal and position $(m,k).$ Then we have that
\begin{align*}
B(b_i\wedge b_j \wedge b_k,\;a_l\wedge a_m \wedge a_k)= & B(f.(b_i\wedge b_j \wedge b_k),\;f.(a_l\wedge a_m \wedge a_k))= \\
= & B(b_i\wedge b_j \wedge b_k,\;a_l\wedge a_m \wedge a_k+a_l\wedge a_m \wedge a_n)= \\
= & B(b_i\wedge b_j \wedge b_k,\;a_l\wedge a_m \wedge a_k)\\
 & +B(b_i\wedge b_j \wedge b_k,\;a_l\wedge a_m \wedge a_n)
\end{align*}
Thus $B(b_i\wedge b_j \wedge b_k,\;a_l\wedge a_m \wedge a_n)=0,$ where $i,j$ not necessarily equal to $l,m.$
\item Take $f$ the matrix $\left(\begin{smallmatrix}
G & 0 \\
0 & {}^tG^{-1}
\end{smallmatrix}\right)$ with $G=(1i)(2j)(3k)\in \mathfrak{S}_g\subset GL_g(\mathbb{Z})$ for $i\neq j \neq k.$ Then we have that
\begin{align*}
B(b_1\wedge b_2 \wedge b_3,\;a_1\wedge a_2 \wedge a_3)= & B(f.(b_1\wedge b_2 \wedge b_3),\;f.(a_1\wedge a_2 \wedge a_3))= \\
= & B(b_i\wedge b_j \wedge b_k,\;a_i\wedge a_j \wedge a_k)
\end{align*}
\end{itemize}
Thus $B$ must be $nJ_g^t$ for some $n\in \mathbb{Z}/p.$
\end{proof}

\chapter{Obstruction to Perron's conjecture}
\label{chapter: obtruction}

In this chapter, our aim is to give an obstruction to the Perron's conjecture. The main ideas of this chapter are inspired in the computations in \cite{mor_linear} due to S. Morita.

As a starting point, in the first section, we state the Perron's conjecture.
In the second section, in order to study the Perron's conjecture, we make some cohomological computations with coefficients in $\mathbb{Z}/p,$ which are similar to the cohomological computations with rational coefficients done in \cite{mor_linear} by S. Morita.
Finally, in the last section, we will use the computations of the second section to give an obstruction to Perron's conjecture.
As we will show in the last section, such obstruction comes from the fact that if the conjectured Perron's invariant is well-defined, then the cohomology class of the associated 2-cocycle is equal to the restriction of the first characteristic class of surface bundles $e_1\in H^2(\mathcal{M}_{g,1};\mathbb{Z}),$ defined in \cite{mor_char}, to $\mathcal{M}_{g,1}[p]$ reduced modulo $p.$
But such cohomology class is not zero.

Throughout this chapter we denote by $\lambda$ the Casson's invariant, and by $S^3_f$ the Heegaard splitting $\mathcal{H}_g \cup_{\iota_g f} -\mathcal{H}_g$ with $f\in \mathcal{M}_{g,1}.$

\section{Perron's conjecture}

In \cite{per} B. Perron stated the following conjectures:
\begin{conj}
\label{conj1}
Let $g\geq 3$ and $p>3$ be a prime number.
For every $f\in \mathcal{T}_{g,1}\cap \mathcal{D}_{g,1}[p],$ we have that
$\lambda(S^3_f)\equiv 0 \;(\text{mod }p).$
\end{conj}
\begin{conj}
\label{conj2}
Let $g\geq 3,$ $p>3$ be a prime number and $S^3_\varphi \in \mathcal{S}^3[p]$ with $\varphi=f\circ m,$ where $f\in \mathcal{T}_{g,1},$ $m\in D_{g,1}[p].$ Then the map $\gamma_p:\mathcal{M}_{g,1}[p] \rightarrow \mathbb{Z}/p$ given by
$$\gamma_p(\varphi)=\lambda(S^3_f) \; (\text{mod }p)$$
is a well defined invariant of $\mathcal{S}^3[p].$
\end{conj}
\begin{rem}
In particular in \cite{per} B. Perron states that the conjecture \eqref{conj1} implies the conjecture \eqref{conj2}.
In addition, he asserts that we can reformulate the conjectures \eqref{conj1}, \eqref{conj2} for all integer coprime with $6$ instead of a prime $p>3.$ Unfortunately, the proof of such results are not available in the literature.
\end{rem}

\section{Preliminary results}
Before to prove the obstruction's result to the Perron's conjecture, we need to do some cohomological computations with $\mathbb{Z}/p$ coefficients.
We devote this section to do such computations, some of them quite similar to the cohomological computations with rational coefficients that S. Morita have done in \cite{mor_linear}.

\begin{prop}
\label{prop_inj_MGC}
For $g\geq 4$ and $p$ an odd prime, the restriction maps
\begin{align*}
res: & \;H^2(Sp_{2g}(\mathbb{Z});\mathbb{Z}/p) \rightarrow H^2(Sp_{2g}(\mathbb{Z},p);\mathbb{Z}/p), \\
res: & \;H^2(\mathcal{M}_{g,1};\mathbb{Z}/p) \rightarrow H^2(\mathcal{M}_{g,1}[p];\mathbb{Z}/p),
\end{align*}
are injective homomorphisms.
\end{prop}

\begin{proof}
Consider the short exact sequence
$$
\xymatrix@C=7mm@R=10mm{1 \ar@{->}[r] & Sp_{2g}(\mathbb{Z},p) \ar@{->}[r] & Sp_{2g}(\mathbb{Z}) \ar@{->}[r]^{r_p} & Sp_{2g}(\mathbb{Z}/p) \ar@{->}[r] & 1 },$$
where the surjectivity of $r_p$ was been proved by  M. Newmann and J. R. Smart in Theorem 1 in \cite{newman}.
Taking the 5-term exact sequence associated to the above short exact sequence we get the exact sequence
$$
\xymatrix@C=6mm@R=10mm{0 \ar@{->}[r] & H^1(Sp_{2g}(\mathbb{Z}/p);\mathbb{Z}/p) \ar@{->}[r]^-{\text{inf}} & H^1(Sp_{2g}(\mathbb{Z});\mathbb{Z}/p) \ar@{->}[r]^-{\text{res}} & H^1(Sp_{2g}(\mathbb{Z},p);\mathbb{Z}/p)^{Sp_{2g}(\mathbb{Z}/p)}  \ar@{->}[r]^-{\text{tr}} & }
$$
$$
\xymatrix@C=6mm@R=10mm{  H^2(Sp_{2g}(\mathbb{Z}/p);\mathbb{Z}/p) \ar@{->}[r] & H^2(Sp_{2g}(\mathbb{Z});\mathbb{Z}/p)_1 \ar@{->}[r] & H^1(Sp_{2g}(\mathbb{Z}/p);H^1(Sp_{2g}(\mathbb{Z},p);\mathbb{Z}/p)).}
$$
By Lemma 3.7 and Theorem 3.8 in \cite{Putman}, we know that $H_2(Sp_{2g}(\mathbb{Z}/p);\mathbb{Z})=0,$ $H_1(Sp_{2g}(\mathbb{Z}/p);\mathbb{Z})=0.$
Then, by Universal coefficients Theorem, we have that
\begin{align*}
H^2(Sp(\mathbb{Z}/p);\mathbb{Z}/p)\cong & Hom(H_2(Sp(\mathbb{Z}/p);\mathbb{Z}),\mathbb{Z}/p)\oplus \text{Ext}^1_\mathbb{Z}(H_1(Sp(\mathbb{Z}/p);\mathbb{Z}),\mathbb{Z}/p) \\
= & Hom(0,\mathbb{Z}/p)\oplus \text{Ext}^1_\mathbb{Z}(0,\mathbb{Z}/p)=0.
\end{align*}
Moreover, by Universal coefficients Theorem we have that
\begin{align*}
H^1(Sp_{2g}(\mathbb{Z},p);\mathbb{Z}/p)= & Hom(H_1(Sp_{2g}(\mathbb{Z},p),\mathbb{Z}/p)\oplus \text{Ext}^1_\mathbb{Z}(H_0(Sp_{2g}(\mathbb{Z},p);\mathbb{Z}),\mathbb{Z}/p) \\
= & (\mathfrak{sp}_{2g}(\mathbb{Z}/p))^*\oplus \text{Ext}^1_\mathbb{Z}(\mathbb{Z},\mathbb{Z}/p)=(\mathfrak{sp}_{2g}(\mathbb{Z}/p))^*.
\end{align*}
Then
$H^1(Sp_{2g}(\mathbb{Z}/p);H^1(Sp_{2g}(\mathbb{Z},p);\mathbb{Z}/p))=
H^1(Sp_{2g}(\mathbb{Z}/p);(\mathfrak{sp}_{2g}(\mathbb{Z}/p))^*).$

By Lemma \eqref{lema_duality_homology} we have that
$$H^1(Sp_{2g}(\mathbb{Z}/p);(\mathfrak{sp}_{2g}(\mathbb{Z}/p))^*)\cong
(H_1(Sp_{2g}(\mathbb{Z}/p);\mathfrak{sp}_{2g}(\mathbb{Z}/p)))^*.$$
By Theorem G in \cite{Putman},
for $g\geq 3$ and $L\geq 2$ such that $4\nmid L,$
$$H_1(Sp_{2g}(\mathbb{Z}/L);\mathfrak{sp}_{2g}(\mathbb{Z}/L))=0.$$
Thus $H^1(Sp_{2g}(\mathbb{Z}/p);(\mathfrak{sp}_{2g}(\mathbb{Z}/p))^*)=0,$
and as a consequence we get that $$H^2(Sp_{2g}(\mathbb{Z});\mathbb{Z}/p)_1=0.$$
Therefore we get an injective homomorphism 
$$res:\;H^2(Sp_{2g}(\mathbb{Z});\mathbb{Z}/p)\rightarrow H^2(Sp_{2g}(\mathbb{Z},p);\mathbb{Z}/p).$$

Consider the commutative diagram
$$
\xymatrix@C=7mm@R=10mm{1 \ar@{->}[r] & \mathcal{T}_{g,1} \ar@{=}[d] \ar@{->}[r] & \mathcal{M}_{g,1}[p] \ar@{->}[r] \ar@{^{(}->}[d] & Sp_{2g}(\mathbb{Z},p) \ar@{^{(}->}[d] \ar@{->}[r] & 1 \\
1 \ar@{->}[r] & \mathcal{T}_{g,1} \ar@{->}[r] & \mathcal{M}_{g,1} \ar@{->}[r] & Sp_{2g}(\mathbb{Z}) \ar@{->}[r] & 1,}$$
where the rows are short exact sequences.
By naturality of the 5-term exact sequence, we get the commutative diagram
$$
\xymatrix@C=7mm@R=10mm{0 \ar@{->}[r] & H^1(Sp_{2g}(\mathbb{Z});\mathbb{Z}/p) \ar@{->}[d]^-{res} \ar@{->}[r]^-{\text{inf}} & H^1(\mathcal{M}_{g,1};\mathbb{Z}/p) \ar@{->}[d]^-{res} \ar@{->}[r]^-{\text{res}} & H^1(\mathcal{T}_{g,1};\mathbb{Z}/p)^{Sp_{2g}(\mathbb{Z})} \ar@{->}[d]^-{res} \ar@{->}[r] & \\
0 \ar@{->}[r] & H^1(Sp_{2g}(\mathbb{Z},p);\mathbb{Z}/p) \ar@{->}[r]^-{\text{inf}} & H^1(\mathcal{M}_{g,1}[p];\mathbb{Z}/p) \ar@{->}[r]^-{\text{res}} & H^1(\mathcal{T}_{g,1};\mathbb{Z}/p)^{Sp_{2g}(\mathbb{Z},p)} \ar@{->}[r] & }
$$
$$
\xymatrix@C=7mm@R=10mm{ \ar@{->}[r] & H^2(Sp_{2g}(\mathbb{Z});\mathbb{Z}/p) \ar@{->}[d]^-{res} \ar@{->}[r]^-{\text{inf}} & H^2(\mathcal{M}_{g,1};\mathbb{Z}/p) \ar@{->}[d]^-{res} \\
 \ar@{->}[r] & H^2(Sp_{2g}(\mathbb{Z},p);\mathbb{Z}/p) \ar@{->}[r]^-{\text{inf}} & H^2(\mathcal{M}_{g,1}[p];\mathbb{Z}/p). }
$$
Next we prove that $inf: H^2(Sp_{2g}(\mathbb{Z},p);\mathbb{Z}/p)\rightarrow H^2(\mathcal{M}_{g,1}[p];\mathbb{Z}/p)$ is injective. By exactness of the 5-term exact sequence, it is enough to prove that the homomorphism $res: H^1(\mathcal{M}_{g,1}[p],\mathbb{Z}/p) \rightarrow H^1(\mathcal{T}_{g,1},\mathbb{Z}/p)$ is surjective.

By Universal coefficients Theorem we have isomorphisms
\begin{align*}
& H^1(\mathcal{M}_{g,1}[p],\mathbb{Z}/p)\cong  Hom(H_1(\mathcal{M}_{g,1}[p]),\mathbb{Z}/p),\\
& H^1(\mathcal{T}_{g,1},\mathbb{Z}/p)^{Sp_{2g}(\mathbb{Z},p)}\cong  Hom(H_1(\mathcal{T}_{g,1}),\mathbb{Z}/p)^{Sp_{2g}(\mathbb{Z},p)}\cong Hom(\extp^3 H_p,\mathbb{Z}/p),\\
& H^1(Sp_{2g}(\mathbb{Z},p);\mathbb{Z}/p)\cong  Hom(H_1(Sp_{2g}(\mathbb{Z},p)),\mathbb{Z}/p)\cong Hom(\mathfrak{sp}_{2g}(\mathbb{Z}/p),\mathbb{Z}/p).
\end{align*}

By Theorem 5.12 in \cite{coop} we have a split extension of $\mathbb{Z}/p$-modules
$$\xymatrix@C=7mm@R=10mm{0 \ar@{->}[r] & \extp^3 H_p \ar@{->}[r] & H_1(\mathcal{M}_{g,1}[p])\ar@{->}[r] & \mathfrak{sp}_{2g}(\mathbb{Z}/p)\ar@{->}[r] & 0.}$$
Then, applying the contravariant functor $Hom(\; , \mathbb{Z}/p),$ we get a short exact sequence of $\mathbb{Z}/p$-modules
$$\xymatrix@C=6mm@R=10mm{0 \ar@{->}[r] & Hom(\mathfrak{sp}_{2g}(\mathbb{Z}/p), \mathbb{Z}/p) \ar@{->}[r] &  Hom(H_1(\mathcal{M}_{g,1}[p]), \mathbb{Z}/p)\ar@{->}[r] &  Hom(\extp^3 H_p, \mathbb{Z}/p)\ar@{->}[r] & 0.}$$
Therefore the homomorphism $res: H^1(\mathcal{M}_{g,1}[p],\mathbb{Z}/p) \rightarrow H^1(\mathcal{T}_{g,1},\mathbb{Z}/p)$ is surjective.

Next we show that $inf: H^2(Sp_{2g}(\mathbb{Z});\mathbb{Z}/p)\rightarrow H^2(\mathcal{M}_{g,1};\mathbb{Z}/p)$ is an isomorphism.
By Universal coefficients Theorem we have that
$$H^1(\mathcal{T}_{g,1},\mathbb{Z}/p)^{Sp_{2g}(\mathbb{Z})}\cong  Hom(H_1(\mathcal{T}_{g,1}),\mathbb{Z}/p)^{Sp_{2g}(\mathbb{Z})}\cong Hom(\extp^3 H_p,\mathbb{Z}/p)^{Sp_{2g}(\mathbb{Z})}=0.$$
Thus, by exactness of the 5-term exact sequence, the map $inf: H^2(Sp_{2g}(\mathbb{Z});\mathbb{Z}/p)\rightarrow H^2(\mathcal{M}_{g,1};\mathbb{Z}/p)$ is injective.
Moreover, by the Universal coefficients Theorem, Theorem 5.1 in \cite{Putman} and Theorem 5.8 in \cite{farb}, we get that $H^2(Sp_{2g}(\mathbb{Z});\mathbb{Z}/p)\cong \mathbb{Z}/p$ and $H^2(\mathcal{M}_{g,1};\mathbb{Z}/p)\cong \mathbb{Z}/p.$
Hence the map $inf$ is an isomorphism.

As a consequence, we have the following commutative diagram:
\begin{equation}\label{diag_com_cocycle}
\xymatrix@C=7mm@R=10mm{ H^2(Sp_{2g}(\mathbb{Z});\mathbb{Z}/p) \ar@{^{(}->}[d]^{res} \ar@{->}[r]^-{inf}_-{\cong} & H^2(\mathcal{M}_{g,1};\mathbb{Z}/p) \ar@{->}[d]^{res} \\
H^2(Sp_{2g}(\mathbb{Z},p);\mathbb{Z}/p) \ar@{^{(}->}[r]^-{inf} & H^2(\mathcal{M}_{g,1}[p];\mathbb{Z}/p).}
\end{equation}
Therefore, the homomorphism
$$res:\;H^2(\mathcal{M}_{g,1};\mathbb{Z}/p) \rightarrow H^2(\mathcal{M}_{g,1}[p];\mathbb{Z}/p)$$
is also injective.
\end{proof}

\begin{defi}
\label{def_sym_power}
Denote by $S^2(\extp^2H_p)$ the submodule of $\extp^2H_p\otimes \extp^2H_p$ generated by all elements of the form:
\begin{enumerate}[i)]
\item $(a\wedge b)\otimes (a\wedge b),$ also denoted by $(a\wedge b)^{\otimes 2},$
\item $(a\wedge b)\otimes (c\wedge d) +(c\wedge d)\otimes(a\wedge b),$ also denoted by $a\wedge b\leftrightarrow c\wedge d.$
\end{enumerate}
(where $a,b,c,d\in H_p).$
\end{defi}

\begin{rem}
Observe that in Definition \eqref{def_sym_power}, if $p$ is an odd prime,
the elements of the form $ii)$ generate $S^2(\extp^2H_p)$ since
$(a\wedge b)^{\otimes 2}=(p+1)(a\wedge b)^{\otimes 2}=\frac{p+1}{2} (a\wedge b\leftrightarrow a\wedge b)$
\end{rem}

In Theorem 3.1 in \cite{mor1}, S. Morita considered the 2-cocycle associated to following central extension
\begin{equation}
\label{eq_exact_seq_of rho3}
\xymatrix@C=10mm@R=13mm{0 \ar@{->}[r] & Im(\tau_2) \ar@{->}[r] & \rho_3(\mathcal{T}_{g,1}) \ar@{->}[r]^{\tau_1} & \extp^3 H \ar@{->}[r] & 1 }
\end{equation}
and he determined the explicit corresponding map of $Hom(\extp^2(\extp^3H); Im(\tau_2))$ as $\pi \circ \chi$ where $\chi \in Hom(\extp^2(\extp^3H); S^2(\extp^2H))$ is given by
\begin{align*}
\chi(\xi \wedge \eta)= & -(a\cdot d)b\wedge c \leftrightarrow e\wedge f  -(a\cdot e)b\wedge c \leftrightarrow f\wedge d  -(a\cdot f)b\wedge c \leftrightarrow d\wedge e \\
& -(b\cdot d)c\wedge a \leftrightarrow e\wedge f -(b\cdot e)c\wedge a \leftrightarrow f\wedge d  -(b\cdot f)c\wedge a \leftrightarrow d\wedge e \\
& -(c\cdot d)a\wedge b \leftrightarrow e\wedge f  -(c\cdot e)a\wedge b \leftrightarrow f\wedge d  -(c\cdot f)a\wedge b \leftrightarrow d\wedge e,
\end{align*}
whith $\xi=a\wedge b\wedge c,$ $\eta=d\wedge e\wedge f \in \extp^3H$
$(a,b,c,d,e,f\in H).$
And $\pi$ is the following natural projection map of $Hom(S^2(\extp^2H),H\otimes \mathcal{L}_3)$ given by
$$\pi(a\wedge b \leftrightarrow c\wedge d)= a\otimes [b,[c,d]]-b\otimes [a,[c,d]]+c\otimes [d,[a,b]]-d\otimes [c,[a,b]].$$
We will denote by $\chi_p\in Hom(\extp^2(\extp^3H_p); Im(\tau^Z_2))$ and $\pi_p\in Hom(S^2(\extp^2H_p),H_p\otimes \mathcal{L}^Z_3)$ the reduction modulo $p$ of the
respective functions $\chi$ and $\pi.$

As a direct consequence of Section 5 in \cite{mor1}, we have that
\begin{prop}
\label{prop_ker_pi}
The kernel of $\pi_p: S^2(\extp^2H_p)\rightarrow H_p\otimes \mathcal{L}^Z_3$ is generated by the elements of the form
$$a\wedge b\leftrightarrow c\wedge d -a\wedge c\leftrightarrow b\wedge d + a\wedge d \leftrightarrow b\wedge c.$$

\end{prop}

\begin{prop} \label{prop_hom_2sym}
Let $p$ be an odd prime, then $Hom(S^2(\extp^2H_p),\mathbb{Z}/p)^{\mathcal{AB}_{g,1}}\cong (\mathbb{Z}/p)^3$ with a basis
\begin{align*}
d_1(a\wedge b\leftrightarrow c \wedge d)= & \omega(a,b)\omega(c,d),\\
d_2(a\wedge b\leftrightarrow c \wedge d)= & \omega(a,c)\omega(b,d)-\omega(a,d)\omega(b,c),\\
d_3(a\wedge b\leftrightarrow c \wedge d)= & \varpi(a,c)\varpi(b,d)-\varpi(a,d)\varpi(b,c),
\end{align*}
where $\omega$ is the intersection form and $\varpi$ is the form associated to the matrix $\left(\begin{smallmatrix}
0 & Id \\
Id & 0
\end{smallmatrix}\right).$
\end{prop}

\begin{proof}
We first show that every element $f\in Hom(S^2(\extp^2H_p),\mathbb{Z}/p)^{\mathcal{AB}_{g,1}}$ is completely determined by its value on the following elements:
$$\{b_1\wedge b_2 \leftrightarrow a_1 \wedge a_2,\; b_1\wedge a_2 \leftrightarrow a_1\wedge b_2,\; a_1\wedge b_1\leftrightarrow a_2\wedge b_2\}.$$

\begin{itemize}
\item Let $i,j,k,l$ be pairwise distinct sub-indexes.
Consider the element
$\phi=\left(\begin{smallmatrix}
G & 0 \\
0 & {}^tG^{-1}
\end{smallmatrix}\right) \in Sp_{2g}(\mathbb{Z})$
with $G=(1,i)(2,j)(3,k)(4,l)\in \mathfrak{S}_g.$ Then we have that
$$f(c_i\wedge c'_j\leftrightarrow d_k\wedge d'_l)=f(\phi\cdot (c_i\wedge c'_j\leftrightarrow d_k\wedge d'_l))=f(c_1\wedge c'_2\leftrightarrow d_3\wedge d'_4).$$
Consider the element
$\phi=\left(\begin{smallmatrix}
G & 0 \\
0 & {}^tG^{-1}
\end{smallmatrix}\right) \in Sp_{2g}(\mathbb{Z})$
with $G=(1,2)\in \mathfrak{S}_g.$ Then we have that
\begin{align*}
f(c_1\wedge c_2\leftrightarrow d_3\wedge d'_4)= & f(\phi\cdot(c_1\wedge c_2\leftrightarrow d_3\wedge d'_4))= f(c_2\wedge c_1\leftrightarrow d_3\wedge d'_4) = \\
= & -f(c_1\wedge c_2\leftrightarrow d_3\wedge d'_4).
\end{align*}
Thus $f(c_1\wedge c_2\leftrightarrow d_3\wedge d'_4)=0.$

Consider the element
$\phi=\left(\begin{smallmatrix}
G & 0 \\
0 & {}^tG^{-1}
\end{smallmatrix}\right) \in Sp_{2g}(\mathbb{Z})$
with $G\in GL_g(\mathbb{Z}),$ with $1$'s at the diagonal and position $(2,1),$ and $0$'s elsewhere. Then we have that
\begin{align*}
f(a_1\wedge b_1\leftrightarrow d_3\wedge d'_4)= & f(\phi\cdot(a_1\wedge b_1\leftrightarrow d_3\wedge d'_4))= \\
= & f(a_1\wedge b_1\leftrightarrow d_3\wedge d'_4)+f(a_2\wedge b_1\leftrightarrow d_3\wedge d'_4).
\end{align*}
Thus $f(a_2\wedge b_1\leftrightarrow d_3\wedge d'_4)=0.$

Consider the element
$\phi=\left(\begin{smallmatrix}
G & 0 \\
0 & {}^tG^{-1}
\end{smallmatrix}\right) \in Sp_{2g}(\mathbb{Z})$
with $G=(1,2)\in \mathfrak{S}_g.$ Then we get that
$$f(a_1\wedge b_2\leftrightarrow d_3\wedge d'_4)=f(\phi\cdot(a_1\wedge b_2\leftrightarrow d_3\wedge d'_4))=f(a_2\wedge b_1\leftrightarrow d_3\wedge d'_4)=0.$$
Therefore $f(c_i\wedge c'_j\leftrightarrow d_k\wedge d'_l)=0$ with $i,j,k,l$ pairwise distinct.

\item Suppose that there are two sub-indexes equal and the other are different.
Then we have to study the image of $f\in Hom(S^2(\extp^2H_p),\mathbb{Z}/p)^{\mathcal{AB}_{g,1}}$ on the elements of the form:
\begin{align*}
c_i\wedge c'_i\leftrightarrow d_j\wedge d'_k, & \qquad c_i\wedge c'_j\leftrightarrow d_i\wedge d'_k, \qquad c_i\wedge c'_j\leftrightarrow d_k\wedge d'_i, \\
c_j\wedge c'_i\leftrightarrow d_i\wedge d'_k, & \qquad c_j\wedge c'_i\leftrightarrow d_k\wedge d'_i, \qquad c_j\wedge c'_k\leftrightarrow d_i\wedge d'_i.
\end{align*}
But observe that
$$c_i\wedge c'_i\leftrightarrow d_j\wedge d'_k=d_j\wedge d'_k \leftrightarrow c_i\wedge c'_i,$$
$$c_i\wedge c'_j\leftrightarrow d_i\wedge d'_k=-c_i\wedge c'_j\leftrightarrow d'_k\wedge d_i=c'_j\wedge c_i\leftrightarrow d'_k\wedge d_i=-c'_j\wedge c_i\leftrightarrow d_i\wedge d'_k.$$
Thus, it is enough to study image of $f\in Hom(S^2(\extp^2H_p),\mathbb{Z}/p)^{\mathcal{AB}_{g,1}}$ on the elements of the form:
$$c_i\wedge c'_i\leftrightarrow d_j\wedge d'_k, \qquad c_i\wedge c'_j\leftrightarrow d_i\wedge d'_k.$$

Consider the element
$\phi=\left(\begin{smallmatrix}
G & 0 \\
0 & {}^tG^{-1}
\end{smallmatrix}\right) \in Sp_{2g}(\mathbb{Z})$
with $G=(1i)(2j)(3k)\in \mathfrak{S}_g.$ Then we get that
\begin{align*}
f(c_i\wedge c'_i\leftrightarrow d_j\wedge d'_k)= & f(\phi\cdot (c_i\wedge c'_i\leftrightarrow d_j\wedge d'_k))=f(c_1\wedge c'_1\leftrightarrow d_2\wedge d'_3), \\
f(c_i\wedge c'_j\leftrightarrow d_i\wedge d'_k)= & f(\phi\cdot (c_i\wedge c'_j\leftrightarrow d_i\wedge d'_k))=f(c_1\wedge c'_2\leftrightarrow d_1\wedge d'_3).
\end{align*}
Consider the element
$\phi=\left(\begin{smallmatrix}
G & 0 \\
0 & {}^tG^{-1}
\end{smallmatrix}\right) \in Sp_{2g}(\mathbb{Z})$
with $G\in GL_g(\mathbb{Z}),$ with $1$'s at the diagonal and position $(2,3),$ and $0$'s elsewhere. Then we have that

\begin{align*}
f(c_1\wedge c'_1\leftrightarrow a_3\wedge b_3)= & f(\phi\cdot(c_1\wedge c'_1\leftrightarrow a_3\wedge b_3))= f(c_1\wedge c'_1\leftrightarrow (a_2+a_3)\wedge b_3)= \\
= & f(c_1\wedge c'_1\leftrightarrow a_2\wedge b_3)+f(c_1\wedge c'_1\leftrightarrow a_3\wedge b_3), \\
f(c_1\wedge a_3\leftrightarrow d_1\wedge b_3)= & f(\phi\cdot(c_1\wedge a_3\leftrightarrow d_1\wedge b_3))=f(c_1\wedge (a_2+a_3)\leftrightarrow d_1\wedge b_3)= \\
= & f(c_1\wedge a_2\leftrightarrow d_1\wedge b_3)+f(c_1\wedge a_3\leftrightarrow d_1\wedge b_3).
\end{align*}
Thus,
\begin{align*}
f(c_1\wedge c'_1\leftrightarrow a_2\wedge b_3)=0, \quad &
f(c_1\wedge a_2\leftrightarrow d_1\wedge b_3)=0.
\end{align*}
Taking the element
$\phi=\left(\begin{smallmatrix}
G & 0 \\
0 & {}^tG^{-1}
\end{smallmatrix}\right) \in Sp_{2g}(\mathbb{Z})$
with $G=(23)\in\mathfrak{S}_g$ we also get that 
\begin{align*}
f(c_1\wedge c'_1\leftrightarrow b_2\wedge a_3)=0, \quad &
f(c_1\wedge b_2\leftrightarrow d_1\wedge a_3)=0.
\end{align*}
Moreover taking the same element we have that
\begin{align*}
f(c_1\wedge c'_1\leftrightarrow d_2\wedge d_3)= &f(\phi\cdot(c_1\wedge c'_1\leftrightarrow d_2\wedge d_3))=f(c_1\wedge c'_1\leftrightarrow d_3\wedge d_2)= \\
= & -f(c_1\wedge c'_1\leftrightarrow d_2\wedge d_3).
\end{align*}
Thus $f(c_1\wedge c'_1\leftrightarrow d_2\wedge d_3)=0.$ Therefore $f(c_1\wedge c'_1\leftrightarrow d_2\wedge d'_3)=0.$

Then, it is enough to check the value of $f(c_1\wedge a_2\leftrightarrow d_1\wedge a_3),$ and $f(c_1\wedge b_2\leftrightarrow d_1\wedge b_3).$ 

Consider the element
$\phi=\left(\begin{smallmatrix}
G & 0 \\
0 & {}^tG^{-1}
\end{smallmatrix}\right) \in Sp_{2g}(\mathbb{Z})$
with $G\in GL_g(\mathbb{Z}),$ with $1$'s at the diagonal and position $(2,3),$ and $0$'s elsewhere. Then we have that
\begin{align*}
f(c_1\wedge a_3 \leftrightarrow c'_1 \wedge a_3)= & f(\phi\cdot (c_1\wedge a_3 \leftrightarrow c'_1 \wedge a_3)) =\\
= & f(c_1\wedge (a_2+a_3) \leftrightarrow c'_1 \wedge (a_2+a_3))= \\
= & f(c_1\wedge a_2 \leftrightarrow c'_1 \wedge a_2)+f(c_1\wedge a_2 \leftrightarrow c'_1 \wedge a_3)+ \\
 & +f(c_1\wedge a_3 \leftrightarrow c'_1 \wedge a_2)+f(c_1\wedge a_3 \leftrightarrow c'_1 \wedge a_3).
\end{align*}
Thus $f(c_1\wedge a_2 \leftrightarrow c'_1 \wedge a_2)=-f(c_1\wedge a_2 \leftrightarrow c'_1 \wedge a_3)-f(c_1\wedge a_3 \leftrightarrow c'_1 \wedge a_2).$

Now using the element
$\phi=\left(\begin{smallmatrix}
G & 0 \\
0 & {}^tG^{-1}
\end{smallmatrix}\right) \in Sp_{2g}(\mathbb{Z})$
with $G=(2,3)\in \mathfrak{S}_g$ we have that $f(c_1\wedge a_2 \leftrightarrow c'_1 \wedge a_3)=f(c_1\wedge a_3 \leftrightarrow c'_1 \wedge a_2).$
And as a consequence we get:
$$f(c_1\wedge a_2 \leftrightarrow c'_1 \wedge a_2)=-2f(c_1\wedge a_2 \leftrightarrow c'_1 \wedge a_3).$$
Consider the element
$\phi=\left(\begin{smallmatrix}
G & 0 \\
0 & {}^tG^{-1}
\end{smallmatrix}\right) \in Sp_{2g}(\mathbb{Z}),$
with $G\in GL_g(\mathbb{Z}),$ with $1$'s at the diagonal and position $(2,3),$ and $0$'s elsewhere. Then we get that
$$f(c_1\wedge a_2 \leftrightarrow c'_1 \wedge a_3)=f(c_1\wedge a_2 \leftrightarrow c'_1 \wedge (a_2+a_3)).$$
Thus $f(c_1\wedge a_2 \leftrightarrow c'_1 \wedge a_2)=0,$ and therefore
$f(c_1\wedge a_2 \leftrightarrow c'_1 \wedge a_3)=0.$
Similarly switching $G$ by ${}^tG^{-1}$ and the $a$'s by $b$'s, we also get that
$f(c_1\wedge b_2 \leftrightarrow c'_1 \wedge b_3)=0.$

\item Suppose that three indexes of $i,j,k,l$ are equal and the other one is different. Then taking a suitable element
$\phi=\left(\begin{smallmatrix}
G & 0 \\
0 & {}^tG^{-1}
\end{smallmatrix}\right) \in Sp_{2g}(\mathbb{Z})$
with $G\in\mathfrak{S}_g,$ we get that
$$f(c_i\wedge c'_i\leftrightarrow d_i\wedge d'_j)=f(c_1\wedge c'_1\leftrightarrow d_1\wedge d'_2).$$
Consider the element
$\phi=\left(\begin{smallmatrix}
G & 0 \\
0 & {}^tG^{-1}
\end{smallmatrix}\right) \in Sp_{2g}(\mathbb{Z})$
with $G\in GL_g(\mathbb{Z})$ with $1$'s at the diagonal and position $(2,3),$ and $0$'s elsewhere. Then we get that
$$f(c_1\wedge c'_1\leftrightarrow d_1\wedge d'_3)=f(c_1\wedge c'_1\leftrightarrow d_1\wedge d'_3)+f(c_1\wedge c'_1\leftrightarrow d_1\wedge d'_2).$$
Thus $f(c_1\wedge c'_1\leftrightarrow d_1\wedge d'_2)=0.$
\end{itemize}

Summarizing, an element $f\in Hom(S^2(\extp^2H_p),\mathbb{Z}/p)^{\mathcal{AB}_{g,1}}$ is completely determined by the image of $f$ on the following elements:
$$c_i\wedge c'_j\leftrightarrow d_i\wedge d'_j,\quad
c_i\wedge c'_i\leftrightarrow d_j\wedge d'_j,\quad
c_i\wedge c'_i\leftrightarrow d_i\wedge d'_i.$$
In particular if we take the element
$\phi=\left(\begin{smallmatrix}
G & 0 \\
0 & {}^tG^{-1}
\end{smallmatrix}\right) \in Sp_{2g}(\mathbb{Z})$
with $G=(1,i)(2,j)\in \mathfrak{S}_g.$ Then we get that an element $f\in Hom(S^2(\extp^2H_p),\mathbb{Z}/p)^{\mathcal{AB}_{g,1}}$ is completely determined by the image of f on the following elements:
$$c_1\wedge c'_2\leftrightarrow d_1\wedge d'_2,\quad
c_1\wedge c'_1\leftrightarrow d_2\wedge d'_2,\quad
c_1\wedge c'_1\leftrightarrow d_1\wedge d'_1.$$
\begin{itemize}
\item Now take the element
$\phi=\left(\begin{smallmatrix}
G & 0 \\
0 & {}^tG^{-1}
\end{smallmatrix}\right) \in Sp_{2g}(\mathbb{Z})$
with $G\in GL_g(\mathbb{Z}),$ with $1$'s at the diagonal and position $(2,3),$ and $0$'s elsewhere. Then we get that
\begin{align*}
f(c_1\wedge a_3 \leftrightarrow c'_1 \wedge a_3)= & f(\phi\cdot (c_1\wedge a_3 \leftrightarrow c'_1 \wedge a_3)) =\\
= & f(c_1\wedge (a_2+a_3) \leftrightarrow c'_1 \wedge (a_2+a_3))= \\
= & f(c_1\wedge a_2 \leftrightarrow c'_1 \wedge a_2)+f(c_1\wedge a_2 \leftrightarrow c'_1 \wedge a_3)+ \\
 & +f(c_1\wedge a_3 \leftrightarrow c'_1 \wedge a_2)+f(c_1\wedge a_3 \leftrightarrow c'_1 \wedge a_3)= \\
 = & f(c_1\wedge a_2 \leftrightarrow c'_1 \wedge a_2)+f(c_1\wedge a_3 \leftrightarrow c'_1 \wedge a_3),
\end{align*}
since by above relations we know that $f(c_1\wedge a_3 \leftrightarrow c'_1 \wedge a_2)=f(c_1\wedge a_2 \leftrightarrow c'_1 \wedge a_3)=0.$
Thus we get that 
$$f(c_1\wedge a_2 \leftrightarrow c'_1 \wedge a_2)=f(a_2\wedge c_1 \leftrightarrow a_2 \wedge c'_1)=0.$$
Analogously, we get that $f(c_1\wedge b_2 \leftrightarrow c'_1 \wedge b_2)=f(b_2\wedge c_1 \leftrightarrow b_2 \wedge c'_1)=0.$

Therefore the only elements of the form $c_1\wedge c'_2\leftrightarrow d_1\wedge d'_2$ with $f(c_1\wedge c'_2\leftrightarrow d_1\wedge d'_2)\neq 0$ are
$$b_1\wedge b_2 \leftrightarrow a_1 \wedge a_2, \qquad b_1\wedge a_2 \leftrightarrow a_1\wedge b_2.$$

\item Now observe that the elements of the form $c_1\wedge c'_1\leftrightarrow d_2\wedge d'_2$ are:
\begin{align*}
a_1\wedge b_1\leftrightarrow a_2\wedge b_2, &\quad b_1\wedge a_1\leftrightarrow a_2\wedge b_2, \\
a_1\wedge b_1\leftrightarrow b_2\wedge a_2, &\quad b_1\wedge a_1\leftrightarrow b_2\wedge a_2.
\end{align*}
Moreover we have that
\begin{align*}
&f(a_1\wedge b_1\leftrightarrow a_2\wedge b_2)=- f(b_1\wedge a_1\leftrightarrow a_2\wedge b_2)= \\
& =-f(a_1\wedge b_1\leftrightarrow b_2\wedge a_2)= f(b_1\wedge a_1\leftrightarrow b_2\wedge a_2).
\end{align*}

\item Now observe that the elements of the form $c_1\wedge c'_1\leftrightarrow d_1\wedge d'_1$ are:
$$ a_1\wedge b_1\leftrightarrow a_1\wedge b_1,\quad a_1\wedge b_1\leftrightarrow b_1\wedge a_1. $$
Moreover we have that
$f(a_1\wedge b_1\leftrightarrow a_1\wedge b_1)=-f( a_1\wedge b_1\leftrightarrow b_1\wedge a_1).$

Consider the element
$\phi=\left(\begin{smallmatrix}
G & 0 \\
0 & {}^tG^{-1}
\end{smallmatrix}\right) \in Sp_{2g}(\mathbb{Z})$
with $G\in GL_g(\mathbb{Z}),$ with $1$'s at the diagonal and position $(1,2),$ and $0$'s elsewhere. By the previous relations, we have that
\begin{align*}
0=f(a_1\wedge b_2 \leftrightarrow a_1 \wedge b_2)= & f(\phi\cdot (a_1\wedge b_2 \leftrightarrow a_1 \wedge b_2))= \\
= &f((a_1+a_2)\wedge (-b_1+b_2) \leftrightarrow (a_1+a_2) \wedge (-b_1+b_2))= \\
=& f(a_1\wedge b_1 \leftrightarrow a_1\wedge b_1)-f(a_1\wedge b_1 \leftrightarrow a_2\wedge b_2)+ \\
& -f(a_1\wedge b_2 \leftrightarrow a_2\wedge b_1)-f(a_2\wedge b_1 \leftrightarrow a_1\wedge b_2)+ \\
& -f(a_2\wedge b_2 \leftrightarrow a_1\wedge b_1)+f(a_2\wedge b_2 \leftrightarrow a_2\wedge b_2) = \\
=& f(a_1\wedge b_1 \leftrightarrow a_1\wedge b_1) + f(a_2\wedge b_2 \leftrightarrow a_2\wedge b_2) + \\
& -2f(a_1\wedge b_1 \leftrightarrow a_2\wedge b_2)-2f(a_1\wedge b_2 \leftrightarrow a_2\wedge b_1)
.
\end{align*}
Thus, by the previous relations and the fact that $2$ is invertible in $\mathbb{Z}/p,$ we get that
$$f(a_1\wedge b_1 \leftrightarrow a_1\wedge b_1)=f(a_1\wedge b_1 \leftrightarrow a_2\wedge b_2)+f(a_1\wedge b_2 \leftrightarrow a_2\wedge b_1).$$
\end{itemize}
Therefore every element $f\in Hom(S^2(\extp^2H_p),\mathbb{Z}/p)^{\mathcal{AB}_{g,1}}$ is completely determined by its image on the following elements:
$$b_1\wedge b_2 \leftrightarrow a_1 \wedge a_2, \qquad b_1\wedge a_2 \leftrightarrow a_1\wedge b_2,$$
$$a_1\wedge b_1\leftrightarrow a_2\wedge b_2.$$
Next consider the following maps:
\begin{align*}
d_1(a\wedge b\leftrightarrow c \wedge d)= & \omega(a,b)\omega(c,d),\\
d_2(a\wedge b\leftrightarrow c \wedge d)= & \omega(a,c)\omega(b,d)-\omega(a,d)\omega(b,c),\\
d_3(a\wedge b\leftrightarrow c \wedge d)= & \varpi(a,c)\varpi(b,d)-\varpi(a,d)\varpi(b,c),
\end{align*}
where $\omega$ is the intersection form and $\varpi$ is the form associated to the matrix $\left(\begin{smallmatrix}
0 & Id \\
Id & 0
\end{smallmatrix}\right).$
Notice that these maps are $\mathcal{AB}_{g,1}$-invariant homomorphisms. Moreover, observe that
\begin{align*}
d_1(b_1\wedge b_2 \leftrightarrow a_1 \wedge a_2)= & \omega(b_1, b_2)\omega(a_1, a_2)=0,\\
d_1(b_1\wedge a_2 \leftrightarrow a_1\wedge b_2)= & \omega(b_1, a_2)\omega(a_1, b_2)=0,\\
d_1(a_1\wedge b_1\leftrightarrow a_2\wedge b_2)= & \omega(a_1, b_1)\omega(a_2, b_2)=1.
\end{align*}
\begin{align*}
d_2(b_1\wedge b_2 \leftrightarrow a_1 \wedge a_2)= & \omega(b_1,a_1)\omega(b_2,a_2)-\omega(b_1,a_2)\omega(b_2,a_1)=1,\\
d_2(b_1\wedge a_2 \leftrightarrow a_1\wedge b_2)= & \omega(b_1,a_1)\omega(a_2,b_2)-\omega(b_1,b_2)\omega(a_2,a_1)=-1,\\
d_2(a_1\wedge b_1\leftrightarrow a_2\wedge b_2)= & \omega(a_1,a_2)\omega(b_1,b_2)-\omega(a_1,b_2)\omega(b_1,a_2)=0.
\end{align*}
\begin{align*}
d_3(b_1\wedge b_2 \leftrightarrow a_1 \wedge a_2)= & \varpi(b_1,a_1)\varpi(b_2,a_2)-\varpi(b_1,a_2)\varpi(b_2,a_1)=1,\\
d_3(b_1\wedge a_2 \leftrightarrow a_1\wedge b_2)= & \varpi(b_1,a_1)\varpi(a_2,b_2)-\varpi(b_1,b_2)\varpi(a_2,a_1)=1,\\
d_3(a_1\wedge b_1\leftrightarrow a_2\wedge b_2)= & \varpi(a_1,a_2)\varpi(b_1,b_2)-\varpi(a_1,b_2)\varpi(b_1,a_2)=0.
\end{align*}
Therefore $d_1,d_2,d_3$  are linearly independent $\mathcal{AB}_{g,1}$-invariant homomorphisms. Hence,
$$Hom(S^2(\extp^2H_p),\mathbb{Z}/p)^{\mathcal{AB}_{g,1}}\cong (\mathbb{Z}/p)^3,$$
and $\{d_1,d_2,d_3\}$ is a basis of this $\mathbb{F}_p$-vector space.
\end{proof}

\begin{prop}
\label{prop_morf_d}
The homomorphisms $d_3$ and $2d_1+d_2$ of Proposition \eqref{prop_hom_2sym} factor through the homomorphism $\pi_p.$
\end{prop}

\begin{proof}
We show that $d_3$ is zero on $Ker(\pi_p),$ and so it factors through $\pi_p.$
By Proposition \eqref{prop_ker_pi}, it is enough to check that $d_3$ is zero on the elements of the form
$$a\wedge b\leftrightarrow c\wedge d -a\wedge c\leftrightarrow b\wedge d + a\wedge d \leftrightarrow b\wedge c.$$
Notice that, since $\varpi$ is symmetric,
\begin{align*}
& d_3(a\wedge b\leftrightarrow c\wedge d -a\wedge c\leftrightarrow b\wedge d + a\wedge d \leftrightarrow b\wedge c)= \\
& =d_3(a\wedge b\leftrightarrow c\wedge d) -d_3(a\wedge c\leftrightarrow b\wedge d) + d_3(a\wedge d \leftrightarrow b\wedge c)= \\
& =\varpi(a,c)\varpi(b,d)-\varpi(a,d)\varpi(b,c)-\varpi(a,b)\varpi(c,d)+\\
&+\varpi(a,d)\varpi(c,b)+\varpi(a,b)\varpi(d,c)-\varpi(a,c)\varpi(d,b)=0.
\end{align*}
In addition we have that, since $\omega$ is skew-symmetric,
\begin{align*}
& (2d_1+d_2)(a\wedge b\leftrightarrow c\wedge d -a\wedge c\leftrightarrow b\wedge d + a\wedge d \leftrightarrow b\wedge c)= \\
& =(2d_1+d_2)(a\wedge b\leftrightarrow c\wedge d) -(2d_1+d_2)(a\wedge c\leftrightarrow b\wedge d) + (2d_1+d_2)(a\wedge d \leftrightarrow b\wedge c)= \\
& =2\omega(a,b)\omega(c,d)+\omega(a,c)\omega(b,d)-\omega(a,d)\omega(b,c)+\\
&-2\omega(a,c)\omega(b,d)-\omega(a,b)\omega(c,d)+\omega(a,d)\omega(c,b)+\\
&+2\omega(a,d)\omega(b,c)+\omega(a,b)\omega(d,c)-\omega(a,c)\omega(d,b)=0.
\end{align*}
Therefore $d_3$ and $2d_1+d_2$ factor through $\pi_p.$
\end{proof}

\begin{prop}
The homomorphisms $d_1,$ $d_2$ do not factor through $\pi_p.$
\end{prop}

\begin{proof}
Consider the element of $Ker(\pi_p),$
$$a_1\wedge b_1\leftrightarrow a_2\wedge b_2-a_1\wedge a_2\leftrightarrow b_1\wedge b_2 + a_1\wedge b_2 \leftrightarrow b_1\wedge a_2.$$
Observe that
\begin{align*}
&d_1(a_1\wedge b_1\leftrightarrow a_2\wedge b_2-a_1\wedge a_2\leftrightarrow b_1\wedge b_2 + a_1\wedge b_2 \leftrightarrow b_1\wedge a_2)=1\neq 0, \\
&d_2(a_1\wedge b_1\leftrightarrow a_2\wedge b_2-a_1\wedge a_2\leftrightarrow b_1\wedge b_2 + a_1\wedge b_2 \leftrightarrow b_1\wedge a_2)=-2\neq 0.
\end{align*}
\end{proof}

\begin{prop}
\label{prop_antisim_hom}
For any odd prime $p$ and $g\geq 3,$ there is an isomorphism
$$Hom(\extp^2(\extp^3H_p), \mathbb{Z}/p)^{GL_g(\mathbb{Z})}\cong (\mathbb{Z}/p)^3,$$
and every element is uniquely determined by its values on
$$(a_1\wedge a_2 \wedge a_3) \wedge (b_1\wedge b_2 \wedge b_3),\quad
(a_1\wedge a_2 \wedge b_2) \wedge (b_1\wedge a_2 \wedge b_2),$$
$$(a_1\wedge a_2 \wedge b_2) \wedge (b_1\wedge a_3 \wedge b_3).$$
Moreover it has a basis formed by bilinear maps: $Q_g,$ $\Theta_g,$ $(J^t_g-J_g),$ where $\Theta_g,$ $Q_g$ are defined on the basis of $\extp^3H_p$ as
\begin{align*}
\Theta_g(c_i\wedge c_j\wedge c_k\otimes c'_i\wedge c'_j\wedge c'_k) & =\sum_{\sigma\in \mathfrak{S}_3}(-1)^{\tau(\sigma)}(\omega(c_i,c'_{\sigma(i)})\omega(c_j,c'_{\sigma(j)})\omega(c_k,c'_{\sigma(k)})),\\
Q_g(c_i\wedge c_j\wedge c_k\otimes c'_i\wedge c'_j\wedge c'_k) & =\omega(C(c_i\wedge c_j\wedge c_k),C(c'_i\wedge c'_j\wedge c'_k)),
\end{align*}
where $\tau(\sigma)$ denotes the sign of the permutation $\sigma.$
\end{prop}

\begin{proof}
Suppose that $B_g$ is an element of $Hom\left(\extp^3H_p\wedge \extp^3H_p;\mathbb{Z}/p\right)^{GL_g(\mathbb{Z})}.$ Then $B_g$ satisfies the following equalities:

\begin{itemize}
\item Let $i,j,k,l,m,n\in \{1,\ldots g\}$ with $i\neq j\neq k.$ Consider the element
$\phi=\left(\begin{smallmatrix}
G & 0 \\
0 & {}^tG^{-1}
\end{smallmatrix}\right) \in Sp_{2g}(\mathbb{Z})$
with $G=(1,i)(2,j)(3,k)\in \mathfrak{S}_g.$ Then we get that
\begin{align*}
B( (c_i\wedge c_j\wedge c_k)\wedge(c'_l\wedge c'_m\wedge c'_n))= & B( \phi \cdot (c_i\wedge c_j\wedge c_k)\wedge \phi \cdot (c'_l\wedge c'_m\wedge c'_n))= \\
= & B( (c_1\wedge c_2\wedge c_3) \wedge (c'_{G(l)}\wedge c'_{G(m)}\wedge c'_{G(n)}),
\end{align*}
\begin{align*}
B( (c_i\wedge a_j\wedge b_j)\wedge (c'_l\wedge c'_m\wedge c'_n))= & B( \phi \cdot (c_i\wedge a_j\wedge b_j)\wedge \phi \cdot (c'_l\wedge c'_m\wedge c'_n))= \\
= & B( (c_1\wedge a_2\wedge b_2)\wedge (c'_{G(l)}\wedge c'_{G(m)}\wedge c'_{G(n)})).
\end{align*}

\item Suppose that $l\neq 1,2,3,$ and $l\neq n\neq m.$ Now take the element
$\phi=\left(\begin{smallmatrix}
G & 0 \\
0 & {}^tG^{-1}
\end{smallmatrix}\right) \in Sp_{2g}(\mathbb{Z}),$ with $G\in GL_{g}(\mathbb{Z})$ the matrix with a $-1$ at position $(l,l)$ and $1$'s at positions $(i,i)$ for $i\neq l,$ and $0$'s elsewhere. Then we get that
\begin{align*}
B( (c_1\wedge c_2\wedge c_3)\wedge (c'_l\wedge c'_m\wedge c'_n))= & B(\phi \cdot (c_1\wedge c_2\wedge c_3)\wedge \phi \cdot (c'_l\wedge c'_m\wedge c'_n))= \\
= & B( (c_1\wedge c_2\wedge c_3)\wedge (-c'_l\wedge c'_m\wedge c'_n))=\\
= & -B( (c_1\wedge c_2\wedge c_3)\wedge (c'_l\wedge c'_m\wedge c'_n))
\end{align*}
Observe that this argument also works for $m,n$ instead of $l.$
So we get that
$$B( (c_1\wedge c_2\wedge c_3)\wedge (c'_l\wedge c'_m\wedge c'_n))=0$$
for all $l,m,n\in \{1,\ldots g\},$ such that $l\neq n\neq m,$ with $l$ or $n$ or $m$ different of $1,2,3.$

Now suppose that $l\neq n\neq m.$ With out lose of generality we can suppose that $l\neq 1,$ then
\begin{align*}
B( (c_1\wedge a_2\wedge b_2)\wedge (c'_l\wedge c'_m\wedge c'_n))= & B(\phi \cdot (c_1\wedge a_2\wedge b_2)\wedge \phi \cdot (c'_l\wedge c'_m\wedge c'_n))= \\
= & B( (c_1\wedge a_2\wedge b_2)\wedge (-c'_l\wedge c'_m\wedge c'_n))=\\
= & -B( (c_1\wedge a_2\wedge b_2)\wedge (c'_l\wedge c'_m\wedge c'_n))
\end{align*}
Thus $B( (c_1\wedge a_2\wedge b_2)\wedge (c'_l\wedge c'_m\wedge c'_n))=0$ for any $l,m,n\in \{1,\ldots g\}$ such that $l\neq n\neq m.$

Now suppose that $l\neq 1,$ then
\begin{align*}
B( (c_1\wedge a_2\wedge b_2)\wedge (c'_l\wedge a_m\wedge b_m))= & B(\phi \cdot (c_1\wedge a_2\wedge b_2)\wedge \phi \cdot (c'_l\wedge a_m\wedge b_m))= \\
= & B( (c_1\wedge a_2\wedge b_2)\wedge (-c'_l\wedge a_m\wedge b_m))=\\
= & -B( (c_1\wedge a_2\wedge b_2)\wedge (c'_l\wedge a_m\wedge b_m))
\end{align*}
Thus $B( (c_1\wedge a_2\wedge b_2)\wedge (c'_l\wedge a_m\wedge b_m))=0$ for any $l,m\in \{1,\ldots g\}$ such that $1\neq l,$ $l\neq m.$

\item Suppose that $l\neq m,$ and with out lose of generality we can assume that $1\neq l,m.$ Now take the element
$\phi=\left(\begin{smallmatrix}
G & 0 \\
0 & {}^tG^{-1}
\end{smallmatrix}\right) \in Sp_{2g}(\mathbb{Z}),$ with $G\in M_{2g}(\mathbb{Z})$ the matrix with a $-1$ at position $(1,1)$ and $1$'s at positions $(i,i)$ for $i\in \{2,\ldots g\},$ and $0$'s elsewhere. Then we get that
\begin{align*}
B( (c_1\wedge c_2\wedge c_3)\wedge (c'_l\wedge a_m\wedge b_m))= & B(\phi \cdot (c_1\wedge c_2\wedge c_3)\wedge \phi \cdot (c'_l\wedge a_m\wedge b_m))= \\
= & B( (-c_1\wedge c_2\wedge c_3)\wedge (c'_l\wedge a_m\wedge b_m))=\\
= & -B( (c_1\wedge c_2\wedge c_3)\wedge (c'_l\wedge a_m\wedge b_m))
\end{align*}
Thus $B( (c_1\wedge c_2\wedge c_3)\wedge (c'_l\wedge a_m\wedge b_m))=0$ for all $l,m,n\in \{1,\ldots g\}.$

\item Let $k\in \{1,\ldots, g\}$ such that $k\neq 2.$ Consider the element
$\phi=\left(\begin{smallmatrix}
G & 0 \\
0 & {}^tG^{-1}
\end{smallmatrix}\right) \in Sp_{2g}(\mathbb{Z})$
with $G=(3,k)\in \mathfrak{S}_g.$ Then we get that
\begin{align*}
B((c_1\wedge a_2\wedge b_2)\wedge (c'_1\wedge a_k\wedge b_k))= & B(\phi \cdot (c_1\wedge a_2\wedge b_2)\wedge \phi \cdot (c'_1\wedge a_k\wedge b_k))= \\
= & B((c_1\wedge a_2\wedge b_2)\wedge (c'_1\wedge a_3\wedge b_3)).
\end{align*}
\end{itemize}
Summarizing, every element $B\in Hom(\extp^2(\extp^3H_p), \mathbb{Z}/p)^{GL_g(\mathbb{Z})}$ is completely determined by its image on the following elements:
$$(c_1\wedge c_2 \wedge c_3) \wedge (c'_1\wedge c'_2 \wedge c'_3),\quad
(c_1\wedge a_2 \wedge b_2) \wedge (c'_1\wedge a_2 \wedge b_2),$$
$$(c_1\wedge a_2 \wedge b_2) \wedge (c'_1\wedge a_3 \wedge b_3).$$
\begin{itemize}
\item Observe that $B((c_1\wedge a_2 \wedge b_2) \wedge (c'_1\wedge a_2 \wedge b_2))= -B((c'_1\wedge a_2 \wedge b_2) \wedge (c_1\wedge a_2 \wedge b_2)).$
Thus
\begin{align*}
B((c_1\wedge a_2 \wedge b_2) \wedge (c_1\wedge a_2 \wedge b_2)) &=0 \text{ for } c_1=a_1 \text{ or } b_1, \\
B((a_1\wedge a_2 \wedge b_2) \wedge (b_1\wedge a_2 \wedge b_2)) &= -B((b_1\wedge a_2 \wedge b_2) \wedge (a_1\wedge a_2 \wedge b_2)).
\end{align*}
\item Consider the element
$\phi=\left(\begin{smallmatrix}
G & 0 \\
0 & {}^tG^{-1}
\end{smallmatrix}\right) \in Sp_{2g}(\mathbb{Z})$
with $G=(2,3)\in \mathfrak{S}_g.$ Then we get that
\begin{align*}
B((c_1\wedge a_2 \wedge b_2) \wedge (c'_1\wedge a_3 \wedge b_3))= & B(\phi\cdot (c_1\wedge a_2 \wedge b_2) \wedge \phi\cdot (c'_1\wedge a_3 \wedge b_3))= \\
B((c_1\wedge a_3 \wedge b_3)\wedge(c'_1\wedge a_2 \wedge b_2) )= & -B((c'_1\wedge a_2 \wedge b_2)\wedge (c_1\wedge a_3 \wedge b_3) )
\end{align*}
Thus
\begin{align*}
B((c_1\wedge a_2 \wedge b_2) \wedge (c_1\wedge a_3 \wedge b_3)) &=0 \text{ for } c_1=a_1 \text{ or } b_1, \\
B((a_1\wedge a_2 \wedge b_2) \wedge (b_1\wedge a_3 \wedge b_3)) &= -B((b_1\wedge a_2 \wedge b_2) \wedge (a_1\wedge a_3 \wedge b_3)).
\end{align*}
\item Observe that $B((c_1\wedge c_2 \wedge c_3) \wedge (c_1\wedge c_2 \wedge c_3))= -B((c_1\wedge c_2 \wedge c_3) \wedge (c_1\wedge c_2 \wedge c_3)).$
Thus
\begin{align*}
B((c_1\wedge c_2 \wedge c_3) \wedge (c_1\wedge c_2 \wedge c_3)) & =0 \text{ for } c_i=a_i \text{ or } b_i, \\
B((a_1\wedge a_2 \wedge a_3) \wedge (b_1\wedge b_2 \wedge b_3)) & =-B((b_1\wedge b_2 \wedge b_3) \wedge (a_1\wedge a_2 \wedge a_3)), \\
B((a_1\wedge b_2 \wedge a_3) \wedge (b_1\wedge a_2 \wedge b_3)) & =-B((b_1\wedge a_2 \wedge b_3) \wedge (a_1\wedge b_2 \wedge a_3)),
\end{align*}
Now suppose that there exist elements $i\neq j\in \{1,2,3\}$ such that $c_i=c_i'$ and $c_j\neq c'_j.$
Taking the element
$\phi=\left(\begin{smallmatrix}
G & 0 \\
0 & {}^tG^{-1}
\end{smallmatrix}\right) \in Sp_{2g}(\mathbb{Z})$
with $G=(1,i)(2,j)$ Then without lose of generality we can suppose $c_1=c'_1$ and $c_2\neq c'_2.$
Consider the element
$\phi=\left(\begin{smallmatrix}
G & 0 \\
0 & {}^tG^{-1}
\end{smallmatrix}\right) \in Sp_{2g}(\mathbb{Z})$
with $G\in \mathfrak{S}_g$ the matrix with a $-1$ at position $(2,1),$ $1$'s at the diagonal and $0$'s elsewhere. Then we get that

$$B((a_1\wedge a_2 \wedge c_3) \wedge (a_1\wedge b_2 \wedge c'_3))= B(\phi\cdot(a_1\wedge a_2 \wedge c_3) \wedge \phi\cdot(a_1\wedge b_2 \wedge c'_3))=$$ 
$$B((a_1\wedge a_2 \wedge c_3) \wedge ((a_1-a_2)\wedge (b_2+b_1) \wedge c'_3))$$
Therefore we get that
\begin{align*}
B((a_1\wedge a_2 \wedge c_3) \wedge (a_1\wedge b_1 \wedge c'_3))= & B((a_1\wedge a_2 \wedge c_3) \wedge (a_2\wedge b_2 \wedge c'_3))\\
& +B((a_1\wedge a_2 \wedge c_3) \wedge (a_2\wedge b_1 \wedge c'_3))
\end{align*} 
But the above relations we get that $B((a_1\wedge a_2 \wedge c_3) \wedge (a_2\wedge b_1 \wedge c'_3))=0$

\item Now taking the element
$\phi=\left(\begin{smallmatrix}
G & 0 \\
0 & {}^tG^{-1}
\end{smallmatrix}\right) \in Sp_{2g}(\mathbb{Z}),$ with $G\in GL_{g}(\mathbb{Z})$ the matrix with $1$'s at the diagonal and position $(3,2),$ and $0$'s elsewhere. Then we get that
\begin{align*}
0= & B((a_1\wedge a_2\wedge b_2) \wedge ( b_1\wedge a_2\wedge b_3))= \\
= & B((a_1\wedge (a_2+a_3)\wedge b_2) \wedge ( b_1\wedge (a_2+a_3)\wedge (b_3-b_2)))= \\
= & B((a_1\wedge a_2\wedge b_2) \wedge ( b_1\wedge a_2\wedge b_3))-B((a_1\wedge a_2\wedge b_2) \wedge ( b_1\wedge a_2\wedge b_2))+\\
&+ B((a_1\wedge a_2\wedge b_2) \wedge ( b_1\wedge a_3\wedge b_3))-B((a_1\wedge a_2\wedge b_2) \wedge ( b_1\wedge a_3\wedge b_2))+\\
&+ B((a_1\wedge a_3\wedge b_2) \wedge ( b_1\wedge a_2\wedge b_3))-B((a_1\wedge a_3\wedge b_2) \wedge ( b_1\wedge a_2\wedge b_2))+\\
&+ B((a_1\wedge a_3\wedge b_2) \wedge ( b_1\wedge a_3\wedge b_3))-B((a_1\wedge a_3\wedge b_2) \wedge ( b_1\wedge a_3\wedge b_2))=\\
=& - B((a_1\wedge a_2\wedge b_2) \wedge ( b_1\wedge a_2\wedge b_2))+B((a_1\wedge a_2\wedge b_2) \wedge ( b_1\wedge a_3\wedge b_3))+\\
&+ B((a_1\wedge a_3\wedge b_2) \wedge ( b_1\wedge a_2\wedge b_3)).
\end{align*}
Thus
\begin{align*}
& B((a_1\wedge b_2\wedge a_3) \wedge ( b_1\wedge a_2\wedge b_3))= \\
& =- B((a_1\wedge a_2\wedge b_2) \wedge ( b_1\wedge a_2\wedge b_2))+B((a_1\wedge a_2\wedge b_2) \wedge ( b_1\wedge a_3\wedge b_3)).
\end{align*}
 
\end{itemize}
Summarizing we have that, every element $B\in Hom(\extp^2(\extp^3H_p), \mathbb{Z}/p)^{GL_g(\mathbb{Z})}$ is completely determined by its image on the following elements:
$$(a_1\wedge a_2 \wedge a_3) \wedge (b_1\wedge b_2 \wedge b_3),\quad
(a_1\wedge a_2 \wedge b_2) \wedge (b_1\wedge a_2 \wedge b_2),$$
$$(a_1\wedge a_2 \wedge b_2) \wedge (b_1\wedge a_3 \wedge b_3).$$
Observe that $Q_g,$ $\Theta_g,$ $(J^t_g-J_g)$ are elements of $Hom(\extp^2(\extp^3H_p), \mathbb{Z}/p)^{GL_g(\mathbb{Z})}$ and in addition these elements are linearly independent because
\begin{align*}
(J^t_g-J_g)((a_1\wedge a_2 \wedge a_3) \wedge (b_1\wedge b_2 \wedge b_3))=-1, \\
(J^t_g-J_g)((a_1\wedge a_2 \wedge b_2) \wedge (b_1\wedge a_2 \wedge b_2))=0,\\
(J^t_g-J_g)((a_1\wedge a_2 \wedge b_2) \wedge (b_1\wedge a_3 \wedge b_3))=0.
\end{align*}
\begin{align*}
Q_g((a_1\wedge a_2 \wedge a_3) \wedge (b_1\wedge b_2 \wedge b_3))=0, \\
Q_g((a_1\wedge a_2 \wedge b_2) \wedge (b_1\wedge a_2 \wedge b_2))=-4,\\
Q_g((a_1\wedge a_2 \wedge b_2) \wedge (b_1\wedge a_3 \wedge b_3))=-4.
\end{align*}
\begin{align*}
\Theta_g((a_1\wedge a_2 \wedge a_3) \wedge (b_1\wedge b_2 \wedge b_3))=-1, \\
\Theta_g((a_1\wedge a_2 \wedge b_2) \wedge (b_1\wedge a_2 \wedge b_2))=-1,\\
\Theta_g((a_1\wedge a_2 \wedge b_2) \wedge (b_1\wedge a_3 \wedge b_3))=0.
\end{align*}
\end{proof}

\begin{cor}
\label{cor_bil_Sp}
For any odd prime $p$ and $g\geq 4,$ there is an isomorphism
$$Hom\left(\extp^2(\extp^3H_p);\mathbb{Z}/p\right)^{Sp_{2g}(\mathbb{Z}/p)}\cong (\mathbb{Z}/p)^2.$$
Moreover it has a basis formed by the bilinear maps $\Theta_g,$ $Q_g.$
\end{cor}
\begin{proof}
By Proposition \eqref{prop_antisim_hom} we know that $Hom\left(\extp^2(\extp^3H_p);\mathbb{Z}/p\right)^{GL_g(\mathbb{Z})}$ has a basis given by $(J^t_g-J_g),$ $\Theta_g,$ $Q_g.$

By definition, it is clear that $\Theta_g,$ $Q_g$ are $Sp_{2g}(\mathbb{Z})$-invariant. Next we show that $(J^t_g-J_g)$ is not $Sp_{2g}(\mathbb{Z})$-invariant.

Suppose that $(J^t_g-J_g)$ was $Sp_{2g}(\mathbb{Z})$-invariant and consider the element $\phi\in Sp_{2g}(\mathbb{Z})$ the matrix with $1$'s at the diagonal and position $(1,g+1),$ and $0$'s elsewhere. Then we would have that
\begin{align*}
0= & (J^t_g-J_g)((b_1\wedge a_2 \wedge a_3)\wedge(b_1\wedge b_2 \wedge b_3))= \\
= &(J^t_g-J_g)(\phi\cdot(b_1\wedge a_2 \wedge a_3)\wedge\phi\cdot(b_1\wedge b_2 \wedge b_3)) \\
= &(J^t_g-J_g)(((a_1+b_1)\wedge a_2 \wedge a_3)\wedge((a_1+b_1)\wedge b_2 \wedge b_3))= \\
= & (J^t_g-J_g)((a_1 \wedge a_2 \wedge a_3)\wedge(b_1\wedge b_2 \wedge b_3))=-1,
\end{align*}
which is a contradiction.
Therefore, $Hom\left(\extp^2(\extp^3H_p);\mathbb{Z}/p\right)^{Sp_{2g}(\mathbb{Z}/p)}\cong \mathbb{Z}/p^2$ with a basis given by the elements $\Theta_g,$ $Q_g,$
as desired.
\end{proof}

\begin{cor}
\label{cor_H2_extp}
For any odd prime $p$ and $g\geq 4,$ there are isomorphisms
$$H^2(\extp^3H_p;\mathbb{Z}/p)^{GL_{g}(\mathbb{Z}/p)}\cong (\mathbb{Z}/p)^3, \qquad H^2(\extp^3H_p;\mathbb{Z}/p)^{Sp_{2g}(\mathbb{Z}/p)}\cong (\mathbb{Z}/p)^2.$$
\end{cor}

\begin{proof}
Since $\extp^3 H_p$ is abelian, the Universal Coefficients Theorem gives us a short exact sequence
$$
\xymatrix@C=7mm@R=13mm{ 0 \ar@{->}[r] & Ext^1_\mathbb{Z}(H_1(\extp^3H_p);\mathbb{Z}/p) \ar@{->}[r] & H^2(\extp^3H_p; \mathbb{Z}/p) \ar@{->}[r] & Hom(\extp^2(\extp^3H_p);\mathbb{Z}/p) \ar@{->}[r] & 1 .}
$$
A direct computation shows that $Ext^1_\mathbb{Z}(H_1(\extp^3H_p);\mathbb{Z}/p)=Hom(\extp^3H_p,\mathbb{Z}/p).$ Then the above short exact sequence becomes
$$
\xymatrix@C=7mm@R=13mm{ 0 \ar@{->}[r] & Hom(\extp^3H_p,\mathbb{Z}/p) \ar@{->}[r] & H^2(\extp^3H_p; \mathbb{Z}/p) \ar@{->}[r] & Hom(\extp^2(\extp^3H_p);\mathbb{Z}/p) \ar@{->}[r] & 1 .}
$$
Taking $GL_{g}(\mathbb{Z}/p)$-invariants and $Sp_{2g}(\mathbb{Z})$-invariants in the above short exact sequence, we get exact sequences
\begin{align*}
0 \rightarrow & Hom(\extp^3H_p,\mathbb{Z}/p)^{GL_{g}(\mathbb{Z}/p)} \rightarrow H^2(\extp^3H_p; \mathbb{Z}/p)^{GL_{g}(\mathbb{Z}/p)} \rightarrow \\
\rightarrow & Hom(\extp^2(\extp^3H_p);\mathbb{Z}/p)^{GL_{g}(\mathbb{Z}/p)} \rightarrow H^1(GL_{g}(\mathbb{Z}/p);(\extp^3H_p)^*),
\end{align*}
\begin{align*}
0 \rightarrow & Hom(\extp^3H_p,\mathbb{Z}/p)^{Sp_{2g}(\mathbb{Z}/p)} \rightarrow H^2(\extp^3H_p; \mathbb{Z}/p)^{Sp_{2g}(\mathbb{Z}/p)} \rightarrow \\
\rightarrow & Hom(\extp^2(\extp^3H_p);\mathbb{Z}/p)^{Sp_{2g}(\mathbb{Z}/p)} \rightarrow H^1(Sp_{2g}(\mathbb{Z}/p);(\extp^3H_p)^*).
\end{align*}
Observe that taking the action of $-Id,$ we get that
$$Hom(\extp^3H_p,\mathbb{Z}/p)^{GL_g(\mathbb{Z}/p)}=Hom(\extp^3H_p,\mathbb{Z}/p)^{Sp_{2g}(\mathbb{Z}/p)}=0.$$
In addition, by the Center kills Lemma, since $-Id$ acts on $(\extp^3 H_p)^*$ by the multiplication by $-1$, we get that 
$$H^1(GL_{g}(\mathbb{Z}/p);(\extp^3H_p)^*)=H^1(Sp_{2g}(\mathbb{Z}/p);(\extp^3H_p)^*)=0.$$
Therefore we obtain isomorphisms
\begin{align*}
H^2(\extp^3H_p; \mathbb{Z}/p)^{GL_{g}(\mathbb{Z}/p)} & \cong Hom(\extp^2(\extp^3H_p);\mathbb{Z}/p)^{GL_{g}(\mathbb{Z}/p)},\\
H^2(\extp^3H_p; \mathbb{Z}/p)^{Sp_{2g}(\mathbb{Z}/p)} & \cong Hom(\extp^2(\extp^3H_p);\mathbb{Z}/p)^{Sp_{2g}(\mathbb{Z}/p)}.
\end{align*}
Finally, by Proposition \eqref{prop_antisim_hom} and Corollary \eqref{cor_bil_Sp}, we get the desired result.
\end{proof}

\section{The obstruction to Perron's conjecture}

Now we are ready to give an obstruction to Perron's conjecture.
We split the argument in the following two parts:
\begin{enumerate}
\item First of all we show that the triviality of the $2$-cocycle $(\tau_1^Z)^*(-2J_g^t)$ is a necessary condition for Conjecture \eqref{conj2} to be true.
\item Secondly, we show that, up to a non zero multiplicative constant modulo $p,$ the cohomology class of $(\tau_1^Z)^*(J_g^t)$ in $H^2(\mathcal{M}_{g,1}[p];\mathbb{Z}/p)$ is equal to the restriction of the first characteristic class of surface bundles $e_1\in H^2(\mathcal{M}_{g,1};\mathbb{Z}),$ defined in \cite{mor_char}, to $\mathcal{M}_{g,1}[p]$ reduced modulo $p.$
\end{enumerate}

As a consequence, we will get that the $2$-cocycle $(\tau_1^Z)^*(-2J_g^t)$ on $\mathcal{M}_{g,1}[p]$ is not trivial. Therefore we will get an obstruction to the Perron's conjecture.

\textbf{1)}
We show that the triviality of $(\tau_1^Z)^*(-2J_g^t)$ is a necessary condition for the Perron's conjecture to be true.

\begin{prop}
\label{prop_obst_perron}
If the Conjecture \eqref{conj2} is true, i.e. $\gamma_p$ is an invariant of $\mathcal{S}^3[p],$ then $(\tau_1^Z)^*(-2J_g^t)$ must be trivial with trivialization $\gamma_p.$
\end{prop}

\begin{proof}
We show that given $\varphi_1,\varphi_2\in \mathcal{M}_{g,1}[p]$ such that $\varphi_1=f_1\circ m_1,$ and $\varphi_2=f_2\circ m_2$ with $f_i\in \mathcal{T}_{g,1}$ and $m_i\in D_{g,1}[p]$ for $i=1,2,$ then the following equality holds
$$(\tau_1^Z)^*(-2J_g^t)(\varphi_1,\varphi_2)= \gamma_p(\varphi_1)+\gamma_p(\varphi_2)-\gamma_p(\varphi_1\varphi_2)$$
Observe that
\begin{align*}
(\tau_1^Z)^*(-2J_g^t)(\varphi_1,\varphi_2)= & (\tau_1^Z)^*(-2J_g^t)(f_1m_1,f_2m_2)= 
-2J_g^t(\tau_1^Z(f_1m_1),\tau_1^Z(f_2m_2))= \\
= & -2J_g^t(\tau_1^Z(f_1),\tau_1^Z(f_2))=-2J_g^t(\tau_1(f_1),\tau_1(f_2)) \;(\text{mod }p),
\end{align*}
\begin{align*}
\gamma_p(\varphi_1)+\gamma_p(\varphi_2)-\gamma_p(\varphi_1\varphi_2)= &
\gamma_p(f_1m_1)+\gamma_p(f_2m_2)-\gamma_p(f_1m_1f_2m_2)= \\
= & \gamma_p(f_1)+\gamma_p(f_2)-\gamma_p(f_1f_2(f_2^{-1}m_1f_2)m_2)= \\
= & \gamma_p(f_1)+\gamma_p(f_2)-\gamma_p(f_1f_2).
\end{align*}
On the other hand, in \cite{pitsch}, W. Pitsch proved that $\tau_1^*(-2J_g^t)$ is a trivial $2$-cocycle with trivialization the Casson invariant. In other words, he proved that
for every $f_1,f_2\in \mathcal{T}_{g,1}$ the following equality holds
$$-2J_g^t(\tau_1(f_1),\tau_1(f_2))= \lambda (f_1)+\lambda(f_2)- \lambda(f_1f_2).$$
Therefore
$$-2J_g^t(\tau_1(f_1),\tau_1(f_2))= \lambda(f_1)+\lambda(f_2)- \lambda(f_1f_2)\;(\text{mod }p),$$
and as a consequence,
$$(\tau_1^Z)^*(-2J_g^t)(\varphi_1,\varphi_2)= \gamma_p(\varphi_1)+\gamma_p(\varphi_2)-\gamma_p(\varphi_1\varphi_2).$$
\end{proof}

\textbf{2)} We show that the $2$-cocycle $(\tau_1^Z)^*(J^t_g)$ is not trivial in $H^2(\mathcal{M}_{g,1}[p];\mathbb{Z}/p).$ First of all we rewrite the $2$-cocycle $(\tau_1^Z)^*(J^t_g)$ as
$$(\tau_1^Z)^*(J^t_g)=(\tau_1^Z)^*\left(J^t_g+\frac{1}{48}Q_g\right) -\frac{1}{48}(\tau_1^Z)^*(Q_g).$$
We will see the following two facts:
\begin{enumerate}[i)]
\item The cohomology class of $(\tau_1^Z)^*(J^t_g+\frac{1}{48}Q_g)$ is zero in $H^2(\mathcal{M}_{g,1}[p];\mathbb{Z}/p).$
\item The cohomology class of $(\tau_1^Z)^*(Q_g)$ is the restriction of the first characteristic class of surface bundles $e_1\in H^2(\mathcal{M}_{g,1};\mathbb{Z})$ to $\mathcal{M}_{g,1}[p]$ reduced modulo $p.$ Therefore, the cohomology class of $(\tau_1^Z)^*(Q_g)$ is not zero in $H^2(\mathcal{M}_{g,1}[p];\mathbb{Z}/p).$
\end{enumerate}

\textbf{i)} 
Consider the short exact sequence
\begin{equation*}
\xymatrix@C=7mm@R=10mm{0 \ar@{->}[r] & \mathcal{I}_{g,1}^Z(2) \ar@{->}[r] & \mathcal{M}_{g,1}[p] \ar@{->}[r]^-{\tau_1^Z} & \extp^3 H_p \ar@{->}[r] & 0.}
\end{equation*}
Recall that $\mathcal{I}_{g,1}^Z(k)=Ker(\rho^Z_k)$ and $\rho^Z_{k+1}$ restricted to $\mathcal{I}_{g,1}^Z(k)$ is $\tau^Z_k.$

Since $\mathcal{I}_{g,1}^Z(3)\subset \mathcal{I}_{g,1}^Z(2)\subset \mathcal{M}_{g,1}[p],$ then $Im(\tau^Z_2)\cong \frac{\mathcal{I}^Z_{g,1}(2)}{\mathcal{I}_{g,1}^Z(3)}$ and $\rho^Z_3(\mathcal{M}_{g,1}[p])\cong \frac{\mathcal{M}_{g,1}[p]}{\mathcal{I}_{g,1}^Z(3)}.$
Therefore taking the quotient by $\mathcal{I}_{g,1}^Z(3)$ in the above short exact sequence, we get another short exact sequence
\begin{equation}
\label{eq_exact_seq_modp}
\xymatrix@C=7mm@R=10mm{0 \ar@{->}[r] & Im(\tau_2^Z) \ar@{->}[r] & \rho_3^Z(\mathcal{M}_{g,1}[p]) \ar@{->}[r]^-{\psi_2^Z} & \extp^3 H_p \ar@{->}[r] & 0, }
\end{equation}
which is central by Corollary \eqref{cor_cent_ext_IA}.
By Section V.3 in \cite{brown}, we know that there is an associated $2$-cocycle $c\in H^2(\extp^3H_p;Im(\tau_2^Z))$ for the central extension \eqref{eq_exact_seq_modp} which is defined by
$c(x,y)=s(x)s(y)s(xy)^{-1}$ where $s: \extp^3H_p \rightarrow \rho_2^Z(\mathcal{M}_{g,1}[p])$ is a set theoretic section.
Let $f=-6d_3-2d_2-4d_1,$ by Proposition \eqref{prop_morf_d} we know that $f$ factors through $\pi_p.$ Denote by $\overline{f}$ the homomorphism satisfying that
$f=\pi_p\circ \overline{f}.$

Consider the short exact sequence obtained in the proof of Corollary \eqref{cor_H2_extp}:
$$
\xymatrix@C=7mm@R=10mm{ 0 \ar@{->}[r] & Hom(\extp^3H_p,\mathbb{Z}/p) \ar@{->}[r]^-{i} & H^2(\extp^3H_p; \mathbb{Z}/p) \ar@{->}[r]^-{\theta} & Hom(\extp^2(\extp^3H_p);\mathbb{Z}/p) \ar@{->}[r] & 1 .}
$$
Next we show that $[J^t_g+\frac{1}{48}Q_g]=[\overline{f}_*(c)].$
\begin{prop}
\label{prop_equality_theta}
The following equality holds $\theta(J^t_g+\frac{1}{48}Q_g)=\theta(\frac{1}{24}\overline{f}_*(c)).$
\end{prop}

\begin{proof}
Observe that $\theta$ and $\overline{f}_*$ commutes so we have that $\theta(\overline{f}_*(c))=\overline{f}_*(\theta(c)).$ Now we show that $\theta(c)=\pi_p\circ \chi_p.$
Recall that by definition of $c$ we know that $c(x,y)=s(x)s(y)s(xy)^{-1}$ where $s$ is a theoretic section of the central extension \eqref{eq_exact_seq_modp}. So we have that
$$\theta(c)(x\wedge y)=c([x|y])-c([y|x])=[s(x),s(y)].$$
Since $s(x),s(y)\in \rho_3^Z(\mathcal{M}_{g,1}[p]),$ we know that there exist elements $\xi, \eta \in \mathcal{M}_{g,1}[p]$ such that $\rho_3^Z(\xi)=x, \rho_3^Z(\eta)=y$, in particular, since $\mathcal{M}_{g,1}[p]=\mathcal{T}_{g,1}D_{g,1}[p] ,$ by Proposition \eqref{prop_tau2_Dp} without lose of generality we can suppose that $\xi, \eta \in \mathcal{T}_{g,1}.$ So $\theta(c)(x\wedge y)=\rho_3^Z([\xi,\eta]).$ But we know that $[\mathcal{T}_{g,1},\mathcal{T}_{g,1}]<\mathcal{K}_{g,1},$ then $\rho_3^Z([\xi,\eta])=\tau_2^Z([\xi,\eta])=\tau_2([\xi,\eta]) \;(\text{ mod }p).$ By Theorem 3.1 in \cite{mor1}, we know that $\tau_2([\xi,\eta])=\pi(\chi(\tau_1(\xi)\wedge \tau_1(\eta)))=\pi(\chi(x\wedge y)),$ and by definition we have that $\pi(\chi(x\wedge y))$ mod $p$ is equal to $\pi_p(\chi_p(x\wedge y)).$ Thus $\theta(\frac{1}{24}\overline{f}_*(c))=\frac{1}{24}\overline{f}_*(\pi_p\circ \chi_p).$

Moreover we have that $\overline{f},$ $\pi_p$ and $\chi_p$ are $\mathcal{AB}_{g,1}$-equivariant, so in particular $\overline{f}_*(\pi_p\circ \chi_p)$ is an element of $Hom(\extp^2(\extp^3H_p),\mathbb{Z}/p)^{\mathcal{AB}_{g,1}}.$
On the other hand we have that $\theta(J^t_g+\frac{1}{48}Q_g)=J^t_g-J_g+\frac{1}{24}Q_g$ and this is also an element of $Hom(\extp^2(\extp^3H_p),\mathbb{Z}/p)^{\mathcal{AB}_{g,1}}.$
So applying Proposition \eqref{prop_antisim_hom} we get that it is enough to check that these two elements take the same values on the following elements
$$(a_1\wedge a_2 \wedge a_3) \wedge (b_1\wedge b_2 \wedge b_3),\quad
(a_1\wedge a_2 \wedge b_2) \wedge (b_1\wedge a_2 \wedge b_2),$$
$$(a_1\wedge a_2 \wedge b_2) \wedge (b_1\wedge a_3 \wedge b_3).$$
One can check that:
\begin{align*}
&\frac{1}{24}\overline{f}((\pi_p\circ \chi_p)((a_1\wedge a_2 \wedge a_3) \wedge (b_1\wedge b_2 \wedge b_3)))=\\
&= \frac{1}{24}\overline{f}((\pi_p( a_2 \wedge a_3 \leftrightarrow b_2 \wedge b_3 + a_3\wedge a_1 \leftrightarrow b_3 \wedge b_1 + a_1\wedge a_2 \leftrightarrow b_1 \wedge b_2 )= \\
&= \frac{1}{8}f(a_1\wedge a_2 \leftrightarrow b_1 \wedge b_2)=\frac{1}{8}(-6-2+0)=-1,
\end{align*}
\begin{align*}
&\frac{1}{24}\overline{f}((\pi_p\circ \chi_p)((a_1\wedge a_2 \wedge b_2) \wedge (b_1\wedge a_2 \wedge b_2)))=\\
&= \frac{1}{24}\overline{f}((\pi_p( a_2 \wedge b_2 \leftrightarrow a_2 \wedge b_2 + b_2\wedge a_1 \leftrightarrow b_1 \wedge a_2 - a_1\wedge a_2 \leftrightarrow b_2 \wedge b_1 )= \\
&= \frac{1}{24}[f( a_2 \wedge b_2 \leftrightarrow a_2 \wedge b_2) +f( b_2\wedge a_1 \leftrightarrow b_1 \wedge a_2) - f(a_1\wedge a_2 \leftrightarrow b_2 \wedge b_1 )]= \\
&= \frac{1}{24}[(6-2-4)+(6-2+0)-(6+2+0)]=-\frac{1}{6},
\end{align*}
\begin{align*}
&\frac{1}{24}\overline{f}((\pi_p\circ \chi_p)((a_1\wedge a_2 \wedge b_2) \wedge (b_1\wedge a_3 \wedge b_3)))=\\
&=\frac{1}{24}f(a_2\wedge b_2 \leftrightarrow a_2 \wedge b_3)=-\frac{1}{6},
\end{align*}

\begin{align*}
(J^t_g-J_g+\frac{1}{24}Q_g)&((a_1\wedge a_2 \wedge a_3) \wedge (b_1\wedge b_2 \wedge b_3))=-1, \\
(J^t_g-J_g+\frac{1}{24}Q_g)&((a_1\wedge a_2 \wedge b_2) \wedge (b_1\wedge a_2 \wedge b_2))=-\frac{1}{6},\\
(J^t_g-J_g+\frac{1}{24}Q_g)&((a_1\wedge a_2 \wedge b_2) \wedge (b_1\wedge a_3 \wedge b_3))=-\frac{1}{6}.
\end{align*}
\end{proof}

By Proposition \eqref{prop_equality_theta}, the cohomology class of the $2$-cocycle $J^t_g+\frac{1}{48}Q_g-\overline{f}_*(c)$ corresponds to the image of an element $h\in Hom(\extp^3H_p,\mathbb{Z}/p)$ under $i.$
In \cite{McLane}, S. MacLane described how to get this element $h.$
In particular he established the following natural isomorphism:
$$\nu: Ext(T;G)\cong Hom(T; G/pG) \qquad (pT=0);$$
In his result, the isomorphism $\nu$ may be described as follows. Write each element E of Ext(T,G) as a short exact sequence
$$\xymatrix@C=10mm@R=13mm{0 \ar@{->}[r] & G \ar@{->}[r] & B \ar@{->}[r]^-{\mu} & T \ar@{->}[r] & 0 }.$$
The corresponding homomorphism $\nu E: T\rightarrow G/pG$ is then given as follows:
to each $t\in T$ choose a representative $b(t)\in B$ with $\mu b(t)=0$ and then set $(\nu E)(t)=pb(t)+pG.$
Now we apply this in our case. Let $G_g=J^t_g+\frac{1}{48}Q_g-\overline{f}_*(c)$ we take the central abelian extension associated to $G_g:$
$$\xymatrix@C=10mm@R=13mm{0 \ar@{->}[r] & \mathbb{Z}/p \ar@{->}[r]^-{j} & \mathbb{Z}/p\times_{G_g} \extp^3H_p \ar@{->}[r]^-{\mu} & \extp^3H_p \ar@{->}[r] & 0 }.$$
Then the corresponding homomorphism $h: \extp^3H_p\rightarrow \mathbb{Z}/p$ is given by $h(x)=j^{-1}((0,x)^p).$

\begin{prop}
The homomorphism $h\in Hom(\extp^3H_p,\mathbb{Z}/p)$ described above is zero.
\end{prop}

\begin{proof}
If we expand the expression $(0,x)^p,$ since $J^t_g+\frac{1}{48}Q_g$ is a bilinear map, we get that
\begin{align*}
j^{-1}((0,x)^p)= & j^{-1}((G_g(x,x)+G_g(2x,x)+\cdots +G_g((p-1)x,x),0)) \\
=& \sum_{i=1}^{p-1} \overline{f}_*(c)(ix,x) +\sum_{i=1}^{p-1} (J^t_g-\frac{1}{48}Q_g)(ix,x)\\
=& \sum_{i=1}^{p-1} \overline{f}_*(c)(ix,x) +\sum_{i=1}^{p-1}i ((J^t_g-\frac{1}{48}Q_g)(x,x))\\
=& \sum_{i=1}^{p-1} \overline{f}_*(c)(ix,x)= \overline{f}(\sum_{i=1}^{p-1} c(ix,x)).
\end{align*}
Thus, it is enough to prove that $\sum_{i=1}^{p-1} c(ix,x)$ is zero.
Recall that $c$ is the $2$-cocycle associated to the central extension
$$\xymatrix@C=10mm@R=13mm{0 \ar@{->}[r] & Im(\tau_2^Z) \ar@{->}[r] & \rho_3^Z(\mathcal{M}_{g,1}[p]) \ar@{->}[r]^-{\psi_2^Z} & \extp^3 H_p \ar@{->}[r] & 1 .}$$
Observe that if we take $(0,x)\in Im(\tau_2^Z)\times_c \extp^3H_p$ and we expand the $p^{th}$-power of $(0,x)$ we get that
$$(0,x)^p=(\sum_{i=1}^{p-1}c(ix,x),0).$$
Hence, it is enough to show that $(0,x)^p\in Im(\tau_2^Z)\times_c \extp^3H_p$ is zero. Since $[c]$ is the cohomology class associated to the central extension \eqref{eq_exact_seq_modp}, there is an isomorphism 
$\varphi:Im(\tau_2^Z)\times_c \extp^3H_p \cong \rho^Z_3(\mathcal{M}_{g,1}[p]),$ so there is an element $\tilde{x}\in\rho^Z_3(\mathcal{M}_{g,1}[p])$ such that $(0,x)^p=\tilde{x}^p.$ As a consequence, there is an element $\xi \in \mathcal{M}_{g,1}[p]$ such that $\varphi(\rho^Z_3(\xi^p))=(0,x)^p.$ Moreover, by Proposition \eqref{prop_tau2_Dp}, we know that $\rho^Z_3(D_{g,1}[p])=\tau^Z_2(D_{g,1}[p])=0$. Therefore, without lose of generality, we can assume that $\xi\in \mathcal{T}_{g,1}.$

Now we show that $\rho^Z_3(\xi^p)=0.$
Since $\xi\in \mathcal{T}_{g,1},$ we can express $\xi^p$ as a product of BP-maps, i.e.
$$\xi^p=((T_{\beta_1}T^{-1}_{\beta'_1})(T_{\beta_2}T^{-1}_{\beta'_2})\cdots (T_{\beta_l}T^{-1}_{\beta'_l}))^p.$$
To prove that $\rho^Z_3(\xi^p)=0$ we proceed by induction on the number $l$ of BP-maps.

First of all, observe that given $f,g\in \mathcal{T}_{g,1}$ we have that
\begin{equation}
\label{eq_torelli_power_p}
\begin{aligned}
(fg)^p= & (fg)^{p-1}gf[f^{-1},g^{-1}]=(fg)^{p-2}fg^2f[f^{-1},g^{-1}]=\\
= & (fg)^{p-2}g^2f[f^{-1},g^{-2}]f[f^{-1},g^{-1}]=\\
& \vdots  \\
= & g^pf[f^{-1},g^{-p}]f[f^{-1},g^{-(p-1)}]\cdots f[f^{-1},g^{-2}]f[f^{-1},g^{-1}] \\
= & g^p\left(\prod_{k=1}^{p-3}f[f^{-1},g^{k-p-1}]\right)f^2[f^{-1},g^{-2}]\big[[f^{-1},g^{-2}]^{-1},f^{-1}\big][f^{-1},g^{-1}]\\
& \vdots  \\
= & g^pf^p[f^{-1},g^{-p}]\prod_{k=2}^{p}\big[[f^{-1},g^{k-p-2}]^{-1},f^{k-p-1}\big][f^{-1},g^{k-p-1}].
\end{aligned}
\end{equation}

Now take $l=1.$ In this case, since $\tau_2^Z(D_{g,1}[p])=0,$ we have that
$$\rho_3^Z\left( (T_{\beta_1}T^{-1}_{\beta'_1})^p\right)=\rho_3^Z(T_{\beta_1}^pT^{-p}_{\beta'_1})=\tau_2^Z(T_{\beta_1}^pT^{-p}_{\beta'_1}))=0.$$ 
Suppose that the assertion is true for $l-1$ and we will prove it for $l.$
So assume that
$$\xi=(T_{\beta_1}T^{-1}_{\beta'_1})(T_{\beta_2}T^{-1}_{\beta'_2})\cdots (T_{\beta_l}T^{-1}_{\beta'_l})=fg,$$
with $f=T_{\beta_1}T^{-1}_{\beta'_1}$ and $g=(T_{\beta_2}T^{-1}_{\beta'_2})\cdots (T_{\beta_l}T^{-1}_{\beta'_l}).$
Now using the formula \eqref{eq_torelli_power_p} and the induction's hypothesis we have that
\begin{align*}
\rho_3^Z((fg)^p)= & \rho_3^Z([f^{-1},g^{-p}])\prod_{k=2}^{p}\rho_3^Z(\big[[f^{-1},g^{k-p-2}]^{-1},f^{k-p-1}\big])\rho_3^Z([f^{-1},g^{k-p-1}]) \\
= & \tau_2^Z([f^{-1},g^{-p}])+\sum_{k=2}^{p}(\tau_2^Z(\big[[f^{-1},g^{k-p-2}]^{-1},f^{k-p-1}\big])+\tau_2^Z([f^{-1},g^{k-p-1}]))
\end{align*}
But, by construction, we know that $\tau_2^Z([\xi,\eta])=(\pi_p\circ\chi_p)(\tau_1^Z(\xi)\wedge \tau_1^Z(\eta)).$
Then,
\begin{align*}
\rho_3^Z((fg)^p)= & \sum_{k=1}^{p}(p+1-k)(\pi_p\circ\chi_p)(\tau_1^Z(f)\wedge \tau_1^Z(g))= \\
= & \left(\sum_{k=1}^{p}k\right)(\pi_p\circ\chi_p)(\tau_1^Z(f)\wedge \tau_1^Z(g))=0 \;(\text{ mod } p)
\end{align*}
Thus $\rho^Z_3(\xi^p)=0,$ and therefore the homomorphism $h$ is zero.

\end{proof}

Therefore we get that $[J^t_g+\frac{1}{48}Q_g]=[\overline{f}_*(c)]$ and as a consequence
$[(\tau_1^Z)^*(J^t_g+\frac{1}{48}Q_g)]=[(\tau_1^Z)^*(\overline{f}_*(c))].$

Thus it is enough to prove the following proposition

\begin{prop}
The cohomology class of the $2$-cocycle $(\tau_1^Z)^*(\overline{f}_*(c))$ is zero.
\end{prop}
\begin{proof}
Recall that we have the following central extension
\begin{equation}
\label{ext_rho3}
\xymatrix@C=10mm@R=13mm{0 \ar@{->}[r] & Im(\tau_2^Z) \ar@{->}[r] & \rho_3^Z(\mathcal{M}_{g,1}[p]) \ar@{->}[r]^-{\psi_{2}^Z} & \extp^3 H_p \ar@{->}[r] & 1 .}
\end{equation}
Notice that $(\psi_{2}^Z)^*$ commutes with $f_*$ Moreover, on $\mathcal{M}_{g,1}[p],$ $\psi_{2}^Z \circ \rho_3^Z=\tau_1^Z.$ Then we have a commutative diagram
\begin{equation}
\label{diag-triv-cocy}
\xymatrix@C=10mm@R=13mm{H^2(\extp^3H_p; Im(\tau_2^Z)) \ar@{->}[d]^{f_*} \ar@{->}[r]^-{(\psi_{2}^Z)^*} &  H^2(\rho_3^Z(\mathcal{M}_{g,1}[p]);Im(\tau_2^Z)) \ar@{->}[d]^{f_*} \\
H^2(\extp^3H_p; \mathbb{Z}/p)\ar@{->}[d]^{(\tau_1^Z)^*} \ar@{->}[r]^-{(\psi_{2}^Z)^*} & H^2(\rho_3^Z(\mathcal{M}_{g,1}[p]);\mathbb{Z}/p) \ar@{->}[d]^{(\rho_3^Z)^*} \\
H^2(\mathcal{M}_{g,1}[p];\mathbb{Z}/p) \ar@{=}[r] & H^2(\mathcal{M}_{g,1}[p];\mathbb{Z}/p)}
\end{equation}
As a consequence, we have that $[(\tau_1^Z)^*(\overline{f}_*(c))]=[(\rho_3^Z)^*\overline{f}_*(\psi_{2}^Z)^*(c)].$
On the other hand, applying Lemma \eqref{lem-cocy-pull-back} to the central extension \eqref{ext_rho3} we get that $[(\psi_{2}^Z)^*(c)]$ is zero. Therefore $[(\tau_1^Z)^*(\overline{f}_*(c))]$ is zero.
\end{proof}

\textbf{ii)} We show that the cohomology class of the $2$-cocycle $(\tau_1^Z)^*(Q_g)$ is not zero.

In \cite{mor_ext}, S. Morita gave a crossed homomorphism
$k: \mathcal{M}_{g,1}\rightarrow \frac{1}{2}\extp^3 H$ (denoted by $\tilde{k}$ in \cite{mor_ext}) that extends the first Johnson homomorphism $\tau_1:\mathcal{T}_{g,1}\rightarrow \extp^3 H.$
Moreover, in \cite{mor}, he defined a $2$-cocycle $\varsigma_g$ on $\mathcal{M}_{g,1}$ (denoted by $c$ in \cite{mor}) given by
$$\varsigma_g(\phi, \psi)=\omega((k\circ C)(\phi),(k\circ C)(\psi^{-1})),$$
where $\omega$ is the intersection form and $C$ is the contraction map.

Now, if we restrict the $2$-cocycle $\varsigma_g$ on $\mathcal{M}_{g,1}[p]$ and take values in $\mathbb{Z}/p$ we have the $2$-cocycle
\begin{align*}
&\omega((\tau_1^Z\circ C)(\phi),(\tau_1^Z\circ C)(\psi^{-1}))=\omega((\tau_1^Z\circ C)(\phi),-(\tau_1^Z\circ C)(\psi))= \\
& = -\omega((\tau_1^Z\circ C)(\phi),(\tau_1^Z\circ C)(\psi))=-(\tau_1^Z)^*(Q_g)(\phi, \psi).
\end{align*}
Because $k$ restricted to $\mathcal{M}_{g,1}[p]$ with values in $\mathbb{Z}/p$ is also an extension of the first Johnson homomorphism modulo $p$ and by Section \eqref{sec_ext_jonh_hom} we know that this coincide with $\tau_1^Z.$

In \cite{harer} J. Harer proved that the second cohomology group $H^2(\mathcal{M}_{g,1};\mathbb{Z})$ of $\mathcal{M}_{g,1}$ is an infinite cyclic group (for $g\geq 3$) generated by the first Chern class
$$c_1\in H^2(\mathcal{M}_{g,1};\mathbb{Z}).$$
Now let $e_1\in H^2(\mathcal{M}_{g,1};\mathbb{Z})$ the first characteristic class of surface bundles defined in \cite{mor_char}.
S. Morita proved that
$$e_1=12c_1,$$
and that $\varsigma_g$ is a $2$-cocycle which represents $e_1.$ 
Therefore, the image of $\varsigma_g$ is $12\mathbb{Z}$ and so, for $p>3,$ the cohomology class of $\varsigma_g$ reduced modulo $p$ is not zero.
As a consequence, since the restriction map $res:\; H^2(\mathcal{M}_{g,1};\mathbb{Z}/p)\rightarrow H^2(\mathcal{M}_{g,1}[p];\mathbb{Z}/p)$
is injective, we have that the cohomology class of $(\tau_1^Z)^*(Q_g)$ is not zero.

\newpage

\bibliography{biblio}{}
\bibliographystyle{abbrv}

\newpage

\end{document}